\documentclass[reqno,11pt]{article}

\usepackage[english]{babel}
\usepackage[T1]{fontenc}

\usepackage{enumerate}
\usepackage{enumitem}

\usepackage{a4wide}
\usepackage{color}
\usepackage{mathrsfs}
\usepackage{mathtools}
\usepackage{amsmath}
\usepackage{amsthm}
\usepackage{amssymb}
\usepackage{bbm}
\usepackage{tikz-cd}
\usepackage{esint}
\usepackage{nicefrac}
\numberwithin{equation}{section}
\usepackage[colorlinks,citecolor=green,linkcolor=red]{hyperref}
\usepackage{cleveref}
\usepackage{longtable}
\usepackage{setspace}

\usepackage[nottoc]{tocbibind}
\usepackage{bbm}

\usepackage[utf8]{inputenc}

\makeatletter
\newcommand{\subjclass}[2][1991]{%
  \let\@oldtitle\@title%
  \gdef\@title{\@oldtitle\footnotetext{#1 \emph{Mathematics subject classification.} #2}}%
}
\newcommand{\keywords}[1]{%
  \let\@@oldtitle\@title%
  \gdef\@title{\@@oldtitle\footnotetext{\emph{Key words and phrases.} #1.}}%
}
\makeatother

\newcommand{\nchi}{{\raise.3ex\hbox{\(\chi\)}}}
\newcommand{\N}{\mathbb{N}}
\newcommand{\R}{\mathbb{R}}

\newcommand{\W}{{\rm W}}
\newcommand{\sfd}{{\sf d}}
\renewcommand{\d}{{\mathrm d}}
\newcommand{\e}{{\rm e}}
\newcommand{\X}{{\rm X}}
\newcommand{\Y}{{\rm Y}}
\newcommand{\Z}{{\rm Z}}

\newcommand{\mm}{\mathfrak{m}}
\newcommand{\1}{\mathbbm 1}
\newcommand{\Mod}{{\rm Mod}}

\newcommand{\LIP}{{\rm LIP}}
\newcommand{\Lip}{{\rm Lip}}
\newcommand{\lip}{{\rm lip}}

\newcommand{\ppi}{{\mbox{\boldmath\(\pi\)}}}

\newcommand{\sppi}{{\mbox{\scriptsize\boldmath\(\pi\)}}}

\newcommand{\limi}{\varliminf}
\newcommand{\lims}{\varlimsup}
\renewcommand{\div}{{\rm div}}

\newtheorem{theorem}{Theorem}[section]
\newtheorem{corollary}[theorem]{Corollary}
\newtheorem{lemma}[theorem]{Lemma}
\newtheorem{proposition}[theorem]{Proposition}
\newtheorem{definition}[theorem]{Definition}

\newtheorem{remark}[theorem]{Remark}

\title{Metric Sobolev spaces II: \\ dual energies and divergence measures}


\author{Luigi Ambrosio \thanks{Scuola Normale Superiore, Piazza dei Cavalieri 7, 56126 Pisa.
\textit{Email:} {\sf luigi.ambrosio@sns.it}}
\and Toni Ikonen \thanks{Department of Mathematics, University of Fribourg, Chemin du Musée 23, 1700 Fribourg, Switzerland. \textit{Email:} {\sf toni.ikonen@unifr.ch}}
\and Danka Lu\v{c}i\'{c} \thanks{University of Jyvaskyla, Department of Mathematics and Statistics,
P.O.\ Box 35 (MaD) FI-40014, Finland. \textit{Email:} {\sf danka.d.lucic@jyu.fi}}
\and Enrico Pasqualetto \thanks{University of Jyvaskyla, Department of Mathematics and Statistics,
P.O.\ Box 35 (MaD) FI-40014, Finland. \textit{Email:} {\sf enrico.e.pasqualetto@jyu.fi}}}
\begin{document}
\date{\today}
\maketitle

\begin{abstract}
This is the second of two works concerning the Sobolev calculus on metric measure spaces and its applications. In this work, we focus on several approaches
to vector calculus in the non-smooth setting of complete and separable metric spaces equipped with a boundedly-finite Borel measure. More precisely, we study
different notions of (co)vector fields and derivations appearing in the literature, as well as their mutual relation. We also carry forward a thorough investigation
of gradients, divergence measures, and Laplacian measures, together with their applications in potential analysis (for example, regarding the condenser capacity)
and in the study of duality properties of Sobolev spaces. Most of the results are obtained
for the full range of exponents \(p\in[1,\infty)\) and without finiteness assumption on the measure.
\end{abstract}
\bigskip

\tableofcontents
\section{Introduction}
\subsection{General overview}
This paper is the follow-up to our previous manuscript \cite{Amb:Iko:Luc:Pas:24,Amb:Iko:Luc:Pas:correction}, where we provided a systematic overview of the Sobolev space theory in the metric measure space setting,
whose developments in the last 20 years are based -- among others -- on \cite{Ch:99,Sha:00,Amb:Gig:Sav:14,Amb:Gig:Sav:13,DiMaPhD:14}; see also the monographs \cite{HKST:15,Bj:Bj:11}.
In our first paper \cite{Amb:Iko:Luc:Pas:24}, we completed the proof of equivalence of several definitions of Sobolev space appearing in the above-mentioned literature, in the setting of arbitrary
metric measure spaces (without doubling or Poincar\'e assumptions involved) and for the whole range of exponents \(p\in[1,\infty)\), and we focussed on the study of fine properties
of Sobolev functions. Building on \cite{Amb:Iko:Luc:Pas:24}, in the present paper we focus on the vector calculus instead, in the same setting and covering -- in most of the
results -- the full range of exponents \(p\in[1,\infty)\). More specifically, we study the following topics:
\begin{itemize}
\item [1)] The notions of {\bf cotangent module} and the related concept of Sobolev differential, as well as {\bf Lipschitz/Sobolev derivations} and their relation to the {\bf tangent module};
the latter allows for the notion of {\bf gradient} of a Sobolev function.  
\item [2)] The duality relation between \textbf{plans} (with barycenter), derivations, and \textbf{metric \(1\)-currents}.
\item [3)] The concept of {\bf divergence measure}, arising in the study of duals of the energy functionals that naturally appear in the Sobolev space theory,
as well as several related applications.
\item [4)] A new notion of {\bf \(p\)-Laplacian} and its properties, together with some applications (e.g.\ to the study of the {\bf condenser capacity}).
\item [5)] In the case \(p>1\), characterisations of the {\bf dual} and of the {\bf predual} of the Sobolev space, and new characterizations of the {\bf reflexivity} of the Sobolev space.
\end{itemize}
We next mention some of the main sources for our investigation in this manuscript. At the end of the paper (Section \ref{sec:bib_notes}), we provide a more detailed bibliographical information.
The contents related to item 1) are based on the theories developed in \cite{Gig:15, Gig:18} and \cite{DiMar:14,DiMaPhD:14}, and on some refinements contained in \cite{Amb:Iko:Luc:Pas:24}.
For the concepts appearing in 2), we rely on the machinery developed in \cite{Sav:22} (and on the relation between plans with barycenters and derivations studied in \cite{Amb:Iko:Luc:Pas:24}),
while \cite{AK:00,Pao:Ste:12,Pao:Ste:13} are our main sources concerning metric currents. The results in 3) provide new characterizations of dual energies studied in \cite{Amb:Sav:21,Sav:22}.
The notion of \(p\)-Laplacian in 4) is novel -- it provides new insights into various problems in metric Potential Theory (e.g.\ on the condenser capacity); see also \cite{Bj:Bj:11}.
The contents of 5) complement the study of dual and predual spaces initiated in \cite{Amb:Sav:21}, as well as the reflexivity questions addressed in \cite{Amb:Col:DiMa:15,Gig:18,Gig:15}. 
In the remaining part of the introduction, we give a more detailed description of the contents of this paper.
\subsection{Contents of this work}
\subsubsection*{Metric Sobolev functions and their quasicontinuous representatives}
Fix a metric measure space \((\X,\sfd,\mm)\) and an exponent \(p\in[1,\infty)\). Our first object of study is the
\[
\textbf{Sobolev space}\quad W^{1,p}(\X),
\]
which is the set of all \(\mm\)-a.e.\ equivalence classes of functions in the \textbf{Newtonian space} \(\bar{N}^{1,p}(\X)\).
The latter consists of those \(p\)-integrable \(\mm\)-measurable functions \(f\colon\X\to\R\) having a \(p\)-integrable
\textbf{weak upper gradient} (in the sense of Definition \ref{def:wug}). A useful variant is the so-called \textbf{Dirichlet
space} \(D^{1,p}_{\mm}(\X)\), which we define as the set of \(\mm\)-a.e.\ equivalence classes of \(\mm\)-measurable functions
(with no integrability assumptions) having a \(p\)-integrable weak upper gradient. See Definition \ref{def:Sobolev} for
the details. We remind that several equivalent definitions of metric Sobolev spaces are studied in the literature; for a proof of such equivalence,
we refer to \cite{Amb:Iko:Luc:Pas:24} and the references therein.
In this paper, we opted for the definition in terms of the \textbf{\(p\)-modulus} \({\rm Mod}_p\) (see Definition
\ref{def:modulus}), which is the appropriate one for introducing the Newtonian space \(\bar{N}^{1,p}(\X)\). The
approaches via approximation with Lipschitz functions and via Lipschitz derivations are presented in Theorems
\ref{thm:density_nrg_Lip} and \ref{thm:Sob_via_der}, respectively. The notion via plans is not used explicitly here.
\medskip

Every Newtonion/Sobolev/Dirichlet function \(f\) on \(\X\) is canonically associated with a
\[
\textbf{minimal weak upper gradient}\quad|Df|\in L^p(\mm)^+.
\]
Then \(\|f\|_{\bar{N}^{1,p}(\X)}\coloneqq\big(\int|f|^p+|Df|^p\,\d\mm\big)^{1/p}\) defines a seminorm on \(\bar{N}^{1,p}(\X)\).
Similarly, \(W^{1,p}(\X)\) becomes a Banach space if endowed with the quotient norm \(\|\cdot\|_{W^{1,p}(\X)}\)
induced by \(\|\cdot\|_{\bar{N}^{1,p}(\X)}\). As in the classical Euclidean setting, a fundamental object associated
to \(\bar{N}^{1,p}(\X)\) is the
\[
\textbf{Sobolev capacity}\quad{\rm Cap}_p.
\]
We remind that \({\rm Cap}_p\) is an outer measure on \(\X\) satisfying \(\mm\ll{\rm Cap}_p\), cf.\ with Definition \ref{def:cap_p}.
See Remarks \ref{rmk:equiv_Cap} and \ref{eq:equiv_Cap_bis} for two equivalent definitions of \({\rm Cap}_p\).
Shortly said, whilst a generic \(L^p(\mm)\)-function is defined only up to \(\mm\)-negligible sets, Sobolev functions
are defined \({\rm Cap}_p\)-almost everywhere. More precisely, if two Newtonian functions \(f,g\in \bar{N}^{1,p}(\X)\) agree
\(\mm\)-a.e., then they agree \({\rm Cap}_p\)-a.e.\ (see Proposition \ref{prop:mm_implies_Cap}). Furthermore, it has
been recently proved in \cite{EB:PC:23} (cf.\ with Theorem \ref{thm:N_are_qc}) that each Newtonian function is
\textbf{quasicontinuous} (in the sense of Definition \ref{def:qc_function}). By building on top of this fact, as well as
of the machinery and results developed in \cite{Deb:Gig:Pas:21}, we prove in Theorem \ref{thm:qc_repr} that there is a
well-defined, continuous \textbf{quasicontinuous-representative map}
\[
{\sf qcr}\colon W^{1,p}(\X)\to\mathcal{QC}(\X),
\]
where \(\mathcal{QC}(\X)\) denotes the set of \({\rm Cap}_p\)-a.e.\ equivalence classes of quasicontinuous functions
(see Definition \ref{def:QC(X)}) equipped with a distance \(\sfd_{\rm qu}\) associated with the \textbf{local quasiuniform
convergence} (Definition \ref{def:qu_conv}). The map \({\sf qcr}\) also enjoys a useful continuity property with
respect to the `convergence in energy' in \(W^{1,p}(\X)\), see Proposition \ref{prop:good_repres_Sob}.
\subsubsection*{Cotangent and tangent modules}
This work focuses on the differential structures induced by metric Sobolev spaces. Despite the lack of a linear and/or smooth structure,
it is still possible to develop an effective vector calculus over an arbitrary metric measure space, thus providing meaningful notions
of `measurable \(1\)-forms' and `measurable vector fields'. Several different approaches have been recently studied in the literature:
in this paper, we deepen our understanding of them and of their mutual relations.
\medskip

Given a metric measure space \((\X,\sfd,\mm)\) and an exponent \(p\in[1,\infty)\), a notion of
\[
\textbf{cotangent module}\quad L^p(T^*\X)
\]
has been introduced in \cite{Gig:18} (see also \cite{Gig:17,Gig:Pas:19}). The elements of \(L^p(T^*\X)\) can be thought of as `\(p\)-integrable
\(1\)-forms' on \(\X\). The cotangent module \(L^p(T^*\X)\) is canonically associated with a
\[
\textbf{differential}\quad \d\colon W^{1,p}(\X)\to L^p(T^*\X),
\]
see Definition \ref{def:cotg_mod}. Technically speaking, the cotangent module is a Banach space endowed with some additional structure:
its elements \(\omega\) can be multiplied by bounded measurable functions (i.e.\ \(L^p(T^*\X)\) is a module over \(L^\infty(\mm)\)) and are
associated with a `pointwise norm' \(|\omega|\in L^p(\mm)^+\). This kind of structure, called an \textbf{\(L^p(\mm)\)-Banach \(L^\infty(\mm)\)-module}
(Definition \ref{def:Lp-Ban_Linfty_mod}) was introduced and thoroughly studied in \cite{Gig:18} with the aim of developing a vector calculus
on metric measure spaces. The differential \(\d\colon W^{1,p}(\X)\to L^p(T^*\X)\), which is characterised by the identity \(|\d f|=|Df|\)
for every \(f\in W^{1,p}(\X)\), satisfies the expected calculus rules (see Proposition \ref{prop:clos_diff} and Theorem \ref{thm:calculus_rules_diff}).
Notice that -- a priori -- a `cotangent bundle' \(T^*\X\) underlying the cotangent module is not given; nevertheless, a posteriori \(L^p\)-Banach \(L^\infty\)-modules
can be represented as section spaces of suitable measurable Banach bundles in several ways (see \cite{Lu:Pa:19,DM:Lu:Pas:2021,Lu:Pas:Voj:24} and the
references therein). These fiberwise descriptions are crucial in addressing some theoretical issues concerning duals and/or pullbacks of \(L^p\)-Banach
\(L^\infty\)-modules that are relevant in analysis on metric measure spaces \cite{Gig:Luc:Pas:22}, but they will not play a role in this paper.
\medskip

A notion of `\(q\)-integrable vector field' (where \(q\) denotes the conjugate exponent of \(p\)) is then obtained by duality with the cotangent module.
Namely, in accordance with \cite{Gig:18}, we define the
\[
\textbf{tangent module}\quad L^q(T\X)
\]
as the dual of \(L^p(T^*\X)\) in the sense of Banach modules, see Definition \ref{def:tg_mod}. Another approach for defining vector fields is in terms
of the concept of \textbf{derivation}. We focus on two notions:
\begin{itemize}
\item Following \cite[Definition 4.4]{Amb:Iko:Luc:Pas:24}, we introduce in Definition \ref{def:Lip_tg_mod} the \textbf{Lipschitz tangent module} \(L^q_\Lip(T\X)\)
as the space generated (in the sense of \(L^q\)-Banach \(L^\infty\)-modules) by the family \({\rm Der}^q_{\mathfrak M}(\X)\) of all `\(q\)-integrable
derivations having measure-valued divergence' (in the sense of \cite{DiMar:14,DiMaPhD:14}). The suffix `Lip' here indicates that the elements of \(L^q_{\Lip}(T\X)\)
are determined by their action on Lipschitz functions. We refer to \cite[Section 4]{Amb:Iko:Luc:Pas:24} for a detailed account of the theory of Lipschitz derivations
and of their role in metric measure geometry.
\item The \textbf{Sobolev tangent module} \(L^q_{\rm Sob}(T\X)\) is the space of `Sobolev \(q\)-derivations' on \(\X\), which are linear operators
\(\delta\colon W^{1,p}(\X)\to L^1(\mm)\) satisfying \(|\delta(f)|\leq g|Df|\) for every \(f\in W^{1,p}(\X)\), for some \(g\in L^q(\mm)^+\). The pointwise \(\mm\)-minimal
such function \(g\) is the so-called pointwise norm \(|\delta|\in L^q(\mm)^+\) of \(\delta\). See Definition \ref{def:Sob_der}, which we slightly adapted from \cite{Gig:18}.
\end{itemize}
A great deal of effort is devoted to studying the compatibility among these spaces. We prove that
\[
L^q_\Lip(T\X)\hookrightarrow L^q_{\rm Sob}(T\X)\cong L^q(T\X).
\]
Namely, the Sobolev tangent module \(L^q_{\rm Sob}(T\X)\) can be canonically identified with the tangent module \(L^q(T\X)\) (Theorem \ref{thm:Lq(TX)=LqSob(TX)}),
while the Lipschitz tangent module \(L^q_{\rm Lip}(T\X)\) is a subspace of the Sobolev tangent module (Theorem \ref{thm:Lip_and_Sob_tg_mod}). The former result is
essentially taken from \cite{Gig:18}, while the latter result is completely new. We also prove (in Theorem \ref{thm:predual_W1p}) that \(L^q_\Lip(T\X)=L^q_{\rm Sob}(T\X)\)
if and only if \(W^{1,p}(\X)\) is reflexive. In particular, since there are metric measure spaces for which $W^{1,p}(X)$ is not reflexive, not even for $p=2$
(by \cite[Proposition 44]{Amb:Col:DiMa:15}),
there are metric measure spaces with \(L^q_\Lip(T\X)\neq L^q_{\rm Sob}(T\X)\).
\medskip

Along the way, we also obtain in Theorem \ref{thm:plans_vs_der_vs_curr} a novel identification result between \textbf{dynamic plans with \(L^q\)-barycenter} (Definition \ref{def:barycenter})
and \(q\)-integrable Lipschitz derivations whose divergence is a finite measure. As an auxiliary step of independent interest, we show that this class of Lipschitz derivations can be
identified with a suitable class of \textbf{normal metric \(1\)-currents} (in the sense of \cite{Amb:Kir:00}), and the proof of such identification relies on the \textbf{superposition principle}
for normal metric \(1\)-currents \cite{Pao:Ste:12,Pao:Ste:13}. In order to drop the integrability assumption on the Lipschitz derivations under consideration, it would be interesting to
understand their relation also with the notion of metric current studied in \cite{La:We:11}; we leave it for a future investigation.
\subsubsection*{Potential theory in metric measure spaces}
We also provide several applications of the differential calculus we discussed above.
\medskip

First, we study the concept of \textbf{divergence measure}, i.e.\ the divergence \(\boldsymbol\div(b)\) of a Lipschitz derivation \(b\in{\rm Der}^q_{\mathfrak M}(\X)\).
We point out that -- for the sake of applications -- we deal with a rather vast class of measures, namely with (possibly 
unbounded) \textbf{boundedly-finite signed Borel measures}
$\mu$ in the sense of Definition \ref{def:bf_sgn_Borel_meas}. The \textbf{dual dynamic cost} of a divergence measure \(\mu\) is
\[
{\sf D}_q(\mu)\coloneqq\inf\Big\{\|b\|_{{\rm Der}^q(\X)}\;\Big|\;b\in{\rm Der}^q_{\mathfrak M}(\X),\,\boldsymbol\div(b)=\mu\Big\}.
\]
We prove in Theorem \ref{thm:F_p=D_q} that
\[
{\sf D}_q(\mu)={\sf F}_p(\mu),
\]
where \({\sf F}_p(\mu)\) denotes the \textbf{dual \(p\)-energy}, which is the conjugate of the \textbf{pre-Cheeger energy functional} \(\mathcal E_{p,\lip}\)
defined in \eqref{eq:pre-Cheeger}, cf.\ Remark \ref{rmk:comments_def_F_p} i). As a byproduct of the identity \({\sf D}_q={\sf F}_p\), we obtain that
\[
\mu\ll{\rm Cap}_p\quad\text{ for every divergence measure }\mu,
\]
see Proposition \ref{prop:mu_ll_Cap}. In turn, this implies (Theorem \ref{thm:L_mu}) that each divergence measure \(\mu\) induces an element \(L_\mu\colon W^{1,p}(\X)\to\R\)
in the dual of \(W^{1,p}(\X)\). Whenever $f \in W^{1,p}(\X) \cap L^{\infty}(\mm)$ is boundedly-supported, \(L_\mu\) acts as
\[
L_\mu(f)=\int{\sf qcr}(f)\,\d\mu\quad\text{ for every }f\in W^{1,p}(\X)\cap L^\infty(\mm),
\]
cf.\ Remark \ref{rmk:L_mu_for_mu_finite}. The inclusion \(L^q_\Lip(T\X)\subseteq L^q_{\rm Sob}(T\X)\) is obtained in Theorem \ref{thm:Lip_and_Sob_tg_mod} by using \(L_\mu\).

In light of the aforementioned connection between Lipschitz derivations with divergence
and dynamic plans with \(L^q\)-barycenter, we also show (under the additional assumption that the measures \(\mm\) and \(\mu\) are finite) that
\[
{\sf D}_q(\mu)={\sf B}_q(\mu),
\]
where \({\sf B}_q(\mu)\) stands for the dual dynamic cost induced by dynamic plans with \(L^q\)-barycenter \eqref{eq:def_B_q_mu}; see Proposition \ref{prop:B_q=D_q}.
This sort of result cannot hold for arbitrary (infinite) reference measures \(\mm\), as we show in Remark \ref{rmk:m_inf_B_q_diff}.
\medskip

Another concept we will encounter is that of a \textbf{gradient}. As it was originally proved in \cite{Gig:18}, each Dirichlet function \(f\in D^{1,p}_{\mm}(\X)\) is associated
with a (non-empty) collection
\[
{\sf Grad}_q(f)\subseteq L^q(T\X)
\]
of \(q\)-gradients, see Definition \ref{def:gradient}. However, gradients are not necessarily unique. If \(p,q\in(1,\infty)\), metric measure spaces where \(q\)-gradients
are unique are said to be \textbf{\(q\)-infinitesimally strictly convex}, see Definition \ref{def:inf_strict_conv} and Corollary \ref{cor:equiv_inf_strict_conv}; this class of 
spaces was introduced in \cite{Gig:15}. In the case where \(p=q=2\), a distinguished subclass is that of \textbf{infinitesimally Hilbertian} metric measure spaces, also introduced
in \cite{Gig:15}. This condition -- which amounts to asking for the Hilbertianity of the Sobolev space \(W^{1,2}(\X)\) (see Definition \ref{def:inf_Hilb}) --
had a fundamental role in the development of the theory of spaces with synthetic Ricci curvature bounds, the so-called \textbf{\(\sf RCD\) spaces} \cite{AmbICM,Gig:23}.
\medskip

Furthermore, we propose and study a notion of \textbf{Laplacian measure}, see Definition \ref{def:laplacian}. We say that a boundedly-finite signed Borel measure \(\mu\) on \(\X\)
is a \(p\)-Laplacian of a function \(f\in D^{1,p}_{\mm}(\X)\), with \(p\in[1,\infty)\), provided it holds that
\[
-\int\psi\,\d\mu\leq\int_{\{|Df|>0\}}\frac{|D(f+\psi)|^p-|Df|^p}{p}\,\d\mm\quad\text{ for every }\psi\in\LIP_{bs}(\X).
\]
When \(p \in (1,\infty)\), the above condition is equivalent to requiring that \(-\int\psi\,\d\mu\leq\mathcal E_p(f+\psi)-\mathcal E_p(f)\) for every \(f\in\LIP_{bs}(\X)\), where \(\mathcal E_p\)
denotes the \textbf{Cheeger \(p\)-energy functional} (Definition \ref{rmk:equiv_Lapl_p>1}), but this is no longer true when \(p=1\); cf.\ with Remark \ref{rmk:equiv_Lapl_p>1}. We denote by
\[
\boldsymbol\Delta_p(f)\quad\text{the set of all }p\text{-Laplacians of }f\in D^{1,p}_{\mm}(\X).
\]
In Theorem \ref{thm:Lapl_are_div} and Lemma \ref{lem:grad_impl_Lapl}, we prove that every \(p\)-Laplacian of a given Dirichlet function \(f\) is the divergence of some gradient of \(f\),
and vice versa. Namely,
\[
\boldsymbol\Delta_p(f)=\big\{\boldsymbol\div(v)\;\big|\;v\in{\sf Grad}_q(f)\big\}\quad\text{ for every }f\in D^{1,p}_{\mm}(\X).
\]
In particular, in the setting of \(q\)-infinitesimally strictly convex spaces, \(p\)-Laplacians are unique, in the sense that every Dirichlet function has at most one \(p\)-Laplacian
(Corollary \ref{cor:str_convex_Delta_implies_div}).  In the forthcoming work \cite{Amb:Iko:Luc:Pas:3},
we will address the question about the `surjectivity of the \(p\)-Laplacian' and show that under additional assumptions on the space (satisfied, for instance, by the class of \(p\)-PI spaces),
it holds that \emph{every} divergence measure is the \(p\)-Laplacian of some Sobolev function.
\medskip

As an application of our study of divergence measures, gradients, and Laplacians, we obtain a new result regarding \textbf{condenser capacities}.
Given two disjoint sets \(E,F\subseteq\X\), we denote by
\[
{\rm Cap}_p(E,F)\coloneqq\inf\bigg\{p\,\mathcal E_p(f)\;\bigg|\;f\in \bar{N}^{1,p}(\X),\,0\leq f\leq 1,\,f=0\text{ on }E,\,f=1\text{ on }F\bigg\}
\]
the \(p\)-condenser capacity of \((E,F)\), cf.\ with Definition \ref{def:condens_cap}. Assuming \(p \in (1,\infty)\), we recall the following folklore fact: if \(E\) and \(F\) are Souslin sets, \({\rm Cap}_p(E,F)\) coincides with the \(p\)-modulus of the family \(\Gamma(E,F)\) of curves joining \(E\) and \(F\) (Proposition \ref{prop:Mod_inf_E_p}). As an application of the $p$-Laplacian above, we prove that whenever the measure \(\mm\) is finite and ${\rm Cap}_p( E, F ) \in (0,\infty)$, the minimisation problem defining \({\rm Cap}_p(E,F)\) admits a minimum \(f\) which has a finite \(p\)-Laplacian \(\mu\in\boldsymbol\Delta_p(f)\) whose positive and negative parts are concentrated on \(E\) and \(F\), respectively (Theorem \ref{thm:condens_cap}). 
The study of such a capacitary potential plays a prominent role in the recent uniformization theorem for metric surfaces \cite{Raj:17}; in the duality of capacities results in the plane \cite{AB:50} and on metric surfaces \cite{Raj:17,RR:19,Es:Iko:Raj}; and their higher-dimensional analogs \cite{Zie:67,Loh:21,EB:PC:22:reciprocal}.
\subsubsection*{Duality and reflexivity of Sobolev spaces}
Leveraging the results discussed above and classical tools in convex analysis, we study duals and preduals of
metric Sobolev spaces. Given a metric measure space \((\X,\sfd,\mm)\) and \(p\in(1,\infty)\), we identify a subspace \(W^{-1,q}_{pd}(\X)\) of the dual
\(W^{-1,q}(\X)\) of \(W^{1,p}(\X)\) that is isomorphic to some \textbf{predual} of \(W^{1,p}(\X)\) (see Definition \ref{def:pd_Sob} and Theorem \ref{thm:char_pred_Sob}),
thus extending previous results of \cite{Amb:Sav:21}. Furthermore, in Theorem \ref{thm:der_and_W-1q-0} we clarify the relation between the Sobolev/Lipschitz
tangent module and the dual/predual of the Sobolev space (we omit to state the result here, as it is rather technical and it requires some additional terminology).
As a consequence, the space of \(q\)-integrable derivations with \(L^q\)-divergence is dense in the Lipschitz tangent module (see Corollary \ref{cor:density_vf_with_div}).
\medskip

Finally, we obtain a new result (i.e.\ Theorem \ref{thm:predual_W1p}) on the \textbf{reflexivity} of the Sobolev space:
we prove that the reflexivity of \(W^{1,p}(\X)\) is equivalent to
the identification \(L^q_\Lip(T\X)\cong L^q_{\rm Sob}(T\X)\), as well as to 
the density of \(q\)-integrable derivations with \(L^q\)-divergence in \(L^q(T\X)\).
\paragraph{Acknowledgements.} 
The authors wish to thank Tero Kilpel\"{a}inen for several helpful discussions on the contents of this paper.
L.A.\ and E.P.\ have been supported by the MIUR-PRIN 202244A7YL project “Gradient Flows and Non-Smooth
Geometric Structures with Applications to Optimization and Machine Learning”. E.P.\ has been supported also by the Research Council of Finland, grant number 362898. 
T.I.\ has been supported by the Academy of Finland, project numbers 308659 and 332671, by the Vilho, Yrj\"{o} and Kalle V\"{a}is\"{a}l\"{a} Foundation, and by the Swiss National Science Foundation grant 212867. D.L.\ has been supported by the Research Council of Finland, grant number 362689.

\section{Preliminaries}
Let us first fix some notation and terminology that will be used throughout the whole paper.
\medskip

By \(\mathscr L^n\) we mean the \textbf{Lebesgue measure} on \(\R^n\).
Given a set \(\X\) and two functions \(f,g\colon\X\to\R\), we define \(f\vee g\) and \(f\wedge g\) as the \textbf{maximum} and \textbf{minimum}
between \(f\) and \(g\), respectively. Namely,
\[
(f\vee g)(x)\coloneqq\max\{f(x),g(x)\},\quad(f\wedge g)(x)\coloneqq\min\{f(x),g(x)\}\quad\text{ for every }x\in\X.
\]
The \textbf{positive part} and the \textbf{negative part} of \(f\) are given by \(f^+\coloneqq f\vee 0\) and \(f^-\coloneqq-f\wedge 0\), respectively.
The \textbf{characteristic function} \(\1_E\colon\X\to\{0,1\}\) of a set \(E\subseteq\X\) is defined as
\[
\1_E(x)\coloneqq\left\{\begin{array}{ll}
1\\
0
\end{array}\quad\begin{array}{ll}
\text{ if }x\in E,\\
\text{ if }x\in\X\setminus E.
\end{array}\right.
\]
Given any exponent \(p\in[1,\infty]\), we always tacitly denote by \(q\in[1,\infty]\) its \textbf{conjugate exponent}, and vice versa.
By \(\ell^p\) we mean the Banach space of \(p\)-summable real-valued sequences, and by \(c_0\) the closed subspace of \(\ell^\infty\)
consisting of those sequences that vanish at infinity. Given a normed space \(\mathbb V\), we denote by \(\mathbb V'\) its \textbf{(continuous) dual}, which is
a Banach space. The duality pairing between \(\mathbb V'\) and \(\mathbb V\) will be occasionally denoted by \(\mathbb V'\times\mathbb V\ni(\omega,v)\mapsto\langle\omega,v\rangle\).
Given a vector space \(V\) and a subset \(S\) of \(V\), we denote by \({\rm conv}(S)\subseteq V\) the \textbf{convex hull} of \(S\) in \(V\), i.e.
\[
{\rm conv}(S)\coloneqq\bigg\{\sum_{i=1}^n\lambda_i v_i\;\bigg|\;n\in\N,\,(\lambda_i)_{i=1}^n\subseteq[0,1],\,\sum_{i=1}^n\lambda_i=1,\,(v_i)_{i=1}^n\subseteq S\bigg\}.
\]
A \textbf{convex set} $C \subseteq V$ is a set that coincides with ${\rm conv}(C)$. A \textbf{convex cone} $C \subseteq V$ is a convex set such that $\lambda C \subseteq C$
for every $\lambda>0$.
If \(C\) is a convex cone of \(V\) and \(f\colon C\to\R\) is given, then we say that \(f\) is \textbf{additive} if it satisfies \(f(v+w)=f(v)+f(w)\) for every \(v,w\in C\),
while we say that \(f\) is \textbf{positively \(1\)-homogeneous} if \(f(\lambda v)=\lambda f(v)\) for every \(\lambda\in\R\) with \(\lambda\geq 0\) and \(v\in C\).
\subsection{Metric measure spaces}
In this work, the main objects of study are \emph{metric measure spaces}. We adopt this definition:
\begin{definition}[Metric measure space]
By a \textbf{metric measure space} we mean a triple \((\X,\sfd,\mm)\),
\[\begin{split}
(\X,\sfd)&\quad\text{ is a complete and separable metric space,}\\
\mm\geq 0&\quad\text{ is a boundedly-finite Borel measure on }\X,
\end{split}\]
where by \textbf{boundedly-finite} we mean that \(\mm(B)<+\infty\) whenever \(B\subseteq\X\) is a bounded Borel set.
\end{definition}

In the rest of this section, we will collect a number of useful definitions and results in the theory of metric/measure spaces.
The presentation is essentially taken from \cite[Section 2]{Amb:Iko:Luc:Pas:24}.
\subsubsection*{Metric geometry}
Let \((\X,\sfd_\X)\) and \((\Y,\sfd_\Y)\) be metric spaces. We denote by \(C(\X;\Y)\) the space of \emph{continuous maps} from \(\X\)
to \(\Y\). The subspace \(C_b(\X;\Y)\) of \(C(\X;\Y)\), which is defined as the set of all bounded elements of \(C(\X;\Y)\), is a metric
space if endowed with the \emph{supremum distance}
\[
\sfd_{C_b(\X;\Y)}(f,g)\coloneqq\sup_{x\in\X}\sfd_\Y(f(x),g(x))\quad\text{ for every }f,g\in C_b(\X;\Y).
\]
If \(\Y\) is complete, then \(C_b(\X;\Y)\) is complete. If \(\X\) is compact and \(\Y\) is separable, then the space \(C_b(\X;\Y)=C(\X;\Y)\) is separable.
When \(\Y=\R\) together with the Euclidean distance, we adopt the shorthand notations \(C(\X)\coloneqq C(\X;\R)\) and \(C_b(\X)\coloneqq C_b(\X;\R)\).
Moreover, we define various spaces of \emph{Lipschitz maps} as follows:
\[\begin{split}
\LIP(\X;\Y)&\coloneqq\big\{f\in C(\X;\Y)\;\big|\;f\text{ is Lipschitz}\big\},\\
\LIP(\X)&\coloneqq\LIP(\X;\R),\\
\LIP_b(\X)&\coloneqq\LIP(\X)\cap C_b(\X),\\
\LIP_{bs}(\X)&\coloneqq\big\{f\in\LIP_b(\X)\;\big|\;f\text{ has bounded support}\big\}.
\end{split}\]
Given any map \(f\in\LIP(\X;\Y)\), we denote by \(\Lip(f)\in[0,\infty)\) its \emph{(global) Lipschitz constant}
\begin{equation}\label{eq:global_lip}
\Lip(f)\coloneqq\sup\bigg\{\frac{\sfd_\Y(f(x),f(y))}{\sfd_\X(x,y)}\;\bigg|\;x,y\in\X\text{ with }x\neq y\bigg\}.
\end{equation}
Furthermore, the \textbf{slope} \(\lip(f)\colon\X\to[0,\infty)\) and the \textbf{asymptotic slope} \(\lip_a(f)\colon\X\to[0,\infty)\) are
\begin{equation}\label{eq:local_asym_lip}\begin{split}
\lip(f)(x)&\coloneqq\limsup_{y\to x}\frac{\sfd_\Y(f(x),f(y))}{\sfd_\X(x,y)},\\
\lip_a(f)(x)&\coloneqq\limsup_{y,z\to x}\frac{\sfd_\Y(f(y),f(z))}{\sfd_\X(y,z)}=\inf_{r>0}\Lip(f|_{B_r(x)})\\
\end{split}\end{equation}
if \(x\in\X\) is an accumulation point of $\X \setminus \{x\}$, and \(\lip(f)(x)=\lip_a(f)(x)\coloneqq 0\) otherwise. It is easy to check that
\begin{equation}\label{eq:convex_ineq_slope}
\lip_a(fg)\leq|f|\lip_a(g)+|g|\lip_a(f)\quad\text{ for every }f,g\in\LIP_b(\X).
\end{equation}
Given any exponent \(p\in[1,\infty)\), we define the \textbf{pre-Cheeger \(p\)-energy functional} \(\mathcal E_{p,\lip}\) as
\begin{equation}\label{eq:pre-Cheeger}
\mathcal E_{p,\lip}(f)\coloneqq\frac{1}{p}\int\lip_a(f)^p\,\d\mm\quad\text{ for every }f\in\LIP_{bs}(\X).
\end{equation}
\subsubsection*{Non-negative measures}
Let \((\X,\Sigma,\mm)\) be a measure space. Then we define the spaces \(\mathcal L^0_{\rm ext}(\mm)\) and \(\mathcal L^0(\mm)\) as
\[\begin{split}
\mathcal L^0_{\rm ext}(\mm)&\coloneqq\big\{f\colon\X\to[-\infty,+\infty]\;\big|\;f\text{ is $\Sigma$-measurable}\big\},\\
\mathcal L^0(\mm)&\coloneqq\big\{f\in\mathcal L^0_{\rm ext}(\mm)\;\big|\;f(x)\in\R\text{ for every }x\in\X\big\}.
\end{split}\]
In fact, \(\mathcal L^0_{\rm ext}(\mm)\) and \(\mathcal L^0(\mm)\) depend exclusively on the \(\sigma\)-algebra \(\Sigma\) (and not
on \(\mm\)), but we prefer to keep a shorter and unified notation for all the Lebesgue-type spaces. Most often
$\Sigma$ will be tacitly defined (in most cases, the Borel $\sigma$-algebra of the metric space $\X$). The space \(\mathcal L^0(\mm)\) is a commutative
algebra with respect to the natural pointwise operations. Given any \(p\in[1,\infty)\), we define
\[\begin{split}
\mathcal L^p_{\rm ext}(\mm)&\coloneqq\bigg\{f\in\mathcal L^0_{\rm ext}(\mm)\;\bigg|\;\|f\|_{\mathcal L^p(\mm)}\coloneqq\bigg(\int|f|^p\,\d\mm\bigg)^{1/p}<+\infty\bigg\},\\
\mathcal L^p(\mm)&\coloneqq\mathcal L^p_{\rm ext}(\mm)\cap\mathcal L^0(\mm),\\
\mathcal L^\infty(\mm)&\coloneqq\bigg\{f\in\mathcal L^0(\mm)\;\bigg|\;\sup_\X|f|<+\infty\bigg\}.
\end{split}\]
Then \(\mathcal L^p(\mm)\) is a vector subspace of \(\mathcal L^0(\mm)\), while \(\mathcal L^\infty(\mm)\) is a subalgebra of \(\mathcal L^0(\mm)\)
(and depends only on \(\Sigma\), not on \(\mm\)). Moreover, \((\mathcal L^p(\mm),\|\cdot\|_{\mathcal L^p(\mm)})\) is a seminormed space,
while \(\mathcal L^\infty(\mm)\) is a Banach space with respect to the norm \(f\mapsto\sup_\X|f|\). Denoting by \(\sim_\mm\) the equivalence relation
on \(\mathcal L^0_{\rm ext}(\mm)\) that identifies two functions if they agree \(\mm\)-almost everywhere, we define the spaces \(L^0_{\rm ext}(\mm)\)
and \(L^0(\mm)\) as
\begin{equation}\label{eq:ext_measurable}
L^0_{\rm ext}(\mm)\coloneqq\mathcal L^0_{\rm ext}(\mm)/\sim_\mm,\qquad L^0(\mm)\coloneqq\mathcal L^0(\mm)/\sim_\mm.
\end{equation}
We will denote by
\[
f^\mm\in L^0_{\rm ext}(\mm)\quad\text{ the equivalence class of }f\in\mathcal L^0_{\rm ext}(\mm)\text{ with respect to }\sim_\mm.
\]
For example, we will often consider characteristic functions of $\Sigma$-measurable sets, namely
\[
\1_E^\mm\in L^0(\mm)\quad\text{ for any }E\in\Sigma.
\]
Moreover, given any subspace \(S\) of \(\mathcal L^0(\mm)\), we denote
\[
S^\mm\coloneqq\big\{f^\mm\;\big|\;f\in S\big\}\subseteq L^0(\mm).
\]

For any exponent \(p\in[1,\infty]\), we consider the \textbf{\(p\)-Lebesgue space} \(L^p(\mm)\), which is given by
\begin{equation}\label{eq:Lebesgue_space}
L^p(\mm)\coloneqq\mathcal L^p(\mm)/\sim_\mm=\mathcal L^p(\mm)^\mm.
\end{equation}
The vector space \(L^p(\mm)\) inherits a complete norm \(\|\cdot\|_{L^p(\mm)}\) from \((\mathcal L^p(\mm),\|\cdot\|_{\mathcal L^p(\mm)})\).
\medskip

We also recall that \(\mathcal L^0_{\rm ext}(\mm)\) is a lattice (i.e.\ a partially ordered set admitting all finite suprema and infima)
with respect to the pointwise order, i.e.\ \(f,g\in\mathcal L^0_{\rm ext}(\mm)\)
satisfy \(f\leq g\) if and only if \(f(x)\leq g(x)\) for every \(x\in\X\). Similarly, \(L^0_{\rm ext}(\mm)\) is a lattice with respect to the \(\mm\)-a.e.\ pointwise
order, i.e.\ \(f,g\in L^0_{\rm ext}(\mm)\) satisfy \(f\leq g\) if and only if \(f(x)\leq g(x)\) for \(\mm\)-a.e.\ \(x\in\X\). Moreover, for any exponent
\(p\in\{0\}\cup[1,\infty]\) we have that \(\mathcal L^p(\mm)\) and \(L^p(\mm)\) are sublattices of \(\mathcal L^0(\mm)\) and
\(L^0(\mm)\), respectively, as well as Riesz spaces.
We recall that a Riesz space is a vector space \(V\) equipped with a lattice structure \(\leq\) that satisfies the following compatibility conditions:
\[\begin{split}
u+w\leq v+w&\quad\text{ for every }u,v,w\in V\text{ with }u\leq v,\\
\lambda u\geq 0&\quad\text{ for every }\lambda\in(0,\infty)\text{ and }u\in V\text{ with }u\geq 0.
\end{split}\]

\begin{remark}{\rm
In the case where \((\X,\Sigma,\mm)\) is a \(\sigma\)-finite measure space, the lattice \(L^0_{\rm ext}(\mm)\)
-- thus also \(L^p(\mm)\) for any \(p\in\{0\}\cup[1,\infty]\) -- is \emph{Dedekind complete}. This means that every given
family \(\{f_i\}_{i\in I}\subseteq L^0_{\rm ext}(\mm)\) has a (unique) supremum \(\bigvee_{i\in I}f_i\in L^0_{\rm ext}(\mm)\),
characterised as follows:
\begin{itemize}
\item \(f_j\leq\bigvee_{i\in I}f_i\) for every \(j\in I\),
\item if \(g\in L^0_{\rm ext}(\mm)\) satisfies \(f_j\leq g\) for every \(j\in I\), then \(\bigvee_{i\in I}f_i\leq g\).
\end{itemize}
Similarly, the family \(\{f_i\}_{i\in I}\subseteq L^0_{\rm ext}(\mm)\) has a (unique) infimum
\(\bigwedge_{i\in I}f_i\in L^0_{\rm ext}(\mm)\). See e.g.\ \cite{Fre:12}.
}\end{remark}
Given any Riesz space \(V\), we denote by \(V^+\coloneqq\{v\in V\,:\,v\geq 0\}\) its \emph{positive cone}. For example,
\[
\mathcal L^p(\mm)^+\coloneqq\big\{f\in\mathcal L^p(\mm)\;\big|\;f\geq 0\big\},
\qquad L^p(\mm)^+\coloneqq\big\{f\in L^p(\mm)\;\big|\;f\geq 0\big\}.
\]
We also denote \(\mathcal L^p_{\rm ext}(\mm)^+\coloneqq\big\{f\in\mathcal L^p_{\rm ext}(\mm)\;\big|\;f\geq 0\big\}\).
\begin{remark}\label{rmk:L0_bar_mm_ident}{\rm
We denote by \(\Sigma_\mm\) the \(\sigma\)-algebra of all \(\mm\)-measurable subsets of \(\X\) (i.e.\ the completion
of the \(\sigma\)-algebra \(\Sigma\) with respect to \(\mm\)) and by \(\bar\mm\colon\Sigma_\mm\to[0,\infty]\) the
completion of the measure \(\mm\).
Observe that \(\mathcal L^0_{\rm ext}(\mm)\) is contained in \(\mathcal L^0_{\rm ext}(\bar\mm)\),
but the two spaces might differ. However, \(L^0_{\rm ext}(\bar\mm)\) can be canonically identified with \(L^0_{\rm ext}(\mm)\)
by means of the map
\begin{equation}\label{eq:def_i_m}
{\rm i}_\mm(f^\mm)\coloneqq f^{\bar\mm}\in L^0_{\rm ext}(\bar\mm)\quad\text{ for every }f\in\mathcal L^0_{\rm ext}(\mm)
\end{equation}
which defines a bijection \({\rm i}_\mm\colon L^0_{\rm ext}(\mm)\to L^0_{\rm ext}(\bar\mm)\). Also, \({\rm i}_\mm(L^p(\mm))=L^p(\bar\mm)\)
for any \(p\in\{0\}\cup[1,\infty]\).
}\end{remark}
Given any measure space \((\X,\Sigma,\mm)\), the elements of \(\mathcal L^0_{\rm ext}(\bar\mm)\) are called the \textbf{\(\mm\)-measurable}
functions on \(\X\). We also define \(\pi_\mm\colon\mathcal L^0_{\rm ext}(\bar\mm)\to L^0_{\rm ext}(\mm)\) as the inverse of the
operator \({\rm i}_\mm\) in \eqref{eq:def_i_m}, namely
\begin{equation}\label{eq:def_pi_m}
\pi_\mm(f)\coloneqq{\rm i}_\mm^{-1}(f^{\bar\mm})\in L^0_{\rm ext}(\mm)\quad\text{ for every }f\in\mathcal L^0_{\rm ext}(\bar\mm),
\end{equation}
so that $\pi_\mm$ maps $\mathcal L^0(\bar\mm)$ to $L^0(\mm)$.
\medskip

In the setting of metric measure spaces, we shall also consider local Lebesgue spaces. Given a metric measure space \((\X,\sfd,\mm)\) and an exponent \(p\in[1,\infty]\),
we define
\[
\mathcal L^p_{loc}(\X,\sfd,\bar\mm)\coloneqq\Big\{f\in\mathcal L^0(\bar\mm)\;\Big|\;\text{for any }x\in\X\text{ there is }r_x>0\text{ such that }\1_{B_{r_x}(x)}f\in\mathcal L^p(\bar\mm)\Big\}.
\]
We then define the space \(L^p_{loc}(\X,\sfd,\mm)\) as
\[
L^p_{loc}(\X,\sfd,\mm)\coloneqq\big\{\pi_\mm(f)\;\big|\;f\in\mathcal L^p_{loc}(\X,\sfd,\bar\mm)\big\}\subseteq L^0(\mm).
\]
We declare that a sequence \((f_n)_n\subseteq L^p_{loc}(\X,\sfd,\mm)\) \textbf{converges in \(L^p_{loc}(\X,\sfd,\mm)\)} to the function \(f_\infty\in L^p_{loc}(\X,\sfd,\mm)\)
if for any \(x\in\X\) there exists a radius \(r_x>0\) such that \(\1_{B_{r_x}(x)}^\mm f_n\in L^p(\mm)\) for every \(n\in\N\cup\{\infty\}\) and
\(\|\1_{B_{r_x}(x)}^\mm f_n-\1_{B_{r_x}(x)}^\mm f_\infty\|_{L^p(\mm)}\to 0\) as \(n\to\infty\).
\subsubsection*{Signed measures}
Besides non-negative measures, we shall also consider signed Borel measures. Given any complete and separable
metric space \((\X,\sfd)\), we denote by \(\mathcal M(\X)\) the vector space of all \textbf{finite signed Borel measures}
on \(\X\) and by \(\mathcal M_+(\X)\) its positive cone, consisting of (non-negative) finite Borel measures
on \(\X\). Given any \(\mu\in\mathcal M(\X)\), we denote by \(\mu^+,\mu^-\in\mathcal M_+(\X)\) its positive and
negative parts, respectively, and by \(|\mu|\coloneqq\mu^+ +\mu^-\in\mathcal M_+(\X)\) its \textbf{total variation measure}.
\medskip

More generally, we will consider boundedly-finite signed Borel measures. With \(\mathscr B(\X)\) we denote the Borel
\(\sigma\)-algebra of a metric space \((\X,\sfd)\), while \(\mathscr B_b(\X)\) stands for the \(\delta\)-ring of all
bounded Borel subsets of \(\X\). Following \cite[Definition 2.5]{Amb:Iko:Luc:Pas:24}, we then give the ensuing definition:
\begin{definition}[Boundedly-finite signed Borel measure]\label{def:bf_sgn_Borel_meas}
Let \((\X,\sfd)\) be a complete and separable metric space. Then we say that a function \(\mu\colon\mathscr B_b(\X)\to\R\)
is a \textbf{boundedly-finite signed Borel measure} on \(\X\) if \(\mu_B\in\mathcal M(\X)\) holds for every \(B\in\mathscr B_b(\X)\),
where we define
\[
\mu_B(E)\coloneqq\mu(E\cap B)\quad\text{ for every }E\in\mathscr B(\X).
\]
We denote by \(\mathfrak M(\X)\) the space of all boundedly-finite signed Borel measures on \(\X\).
\end{definition}

The \textbf{positive part} \(\mu^+\) of \(\mu\in\mathfrak M(\X)\) is the unique Borel measure \(\mu^+\geq 0\) on \(\X\) such that
\[
\mu^+(E)=\mu^+_B(E)\quad\text{ for every }B,E\in\mathscr B_b(\X)\text{ with }E\subseteq B.
\]
The \textbf{negative part} of \(\mu\) is \(\mu^-\coloneqq(-\mu)^+\). The \textbf{total variation measure} of \(\mu\) is then defined as
\[
|\mu|\coloneqq\mu^+ +\mu^-.
\]
Given any \(f\in\mathcal L^1(|\mu|)\) or \(f\in L^1(|\mu|)\), we can define its \textbf{Lebesgue integral} \(\int f\,\d\mu\in\R\) as
\[
\int f\,\d\mu\coloneqq\int f\,\d\mu^+ -\int f\,\d\mu^-.
\]
Notice that \(f\in\mathcal L^1(|\mu|)\) for every bounded Borel function \(f\colon\X\to\R\) having bounded support.
\medskip

If \(\mu\in\mathfrak M(\X)\) is given and \(\nu\) is an outer measure on \(\X\), then we say that \(\mu\) is
\textbf{absolutely continuous} with respect to \(\nu\), and we write \(\mu\ll\nu\), if it holds that
\[
|\mu|(N)=0\quad\text{ for every }N\in\mathscr B(\X)\text{ such that }\nu(N)=0.
\]
\subsubsection*{Modulus and plans}
Let \((\X,\sfd)\) be a complete and separable metric space. By a \textbf{curve} in \(\X\) we mean a continuous map
\(\gamma\colon[a,b]\to\X\), where \(a,b\in\R\) satisfy \(a<b\). The collection of all curves in \(\X\) is denoted by
\[
\mathscr C(\X)\coloneqq\bigcup_{\substack{a,b\in\R:\\a<b}}C([a,b];\X).
\]
Given any \(\gamma\in\mathscr C(\X)\), we denote by \([a_\gamma,b_\gamma]\) its interval of definition.
The value of \(\gamma\) at \(t\in[a_\gamma,b_\gamma]\) will be denoted by \(\gamma(t)\) or \(\gamma_t\).
For any \(t\in[0,1]\), the \textbf{evaluation map} \(\e_t\colon C([0,1];\X)\to\X\) is
\begin{equation}\label{eq:def_e_t}
\e_t(\gamma)\coloneqq\gamma_t\quad\text{ for every }\gamma\in C([0,1];\X).
\end{equation}
Notice that \(\e_t\) is \(1\)-Lipschitz. The \textbf{length} \(\ell(\gamma)\in[0,\infty]\) of a curve
\(\gamma\in\mathscr C(\X)\) is defined as
\[
\ell(\gamma)\coloneqq\sup\bigg\{\sum_{i=1}^n\sfd(\gamma_{t_i},\gamma_{t_{i-1}})\;\bigg|\;
n\in\N,\,a_\gamma=t_0<t_1<\ldots<t_n=b_\gamma\bigg\}.
\]
We say that \(\gamma\) is \textbf{rectifiable} if \(\ell(\gamma)<+\infty\). For any \(a,b\in\R\) with \(a<b\),
we denote by \(R([a,b];\X)\) the set of all rectifiable curves \(\gamma\colon[a,b]\to\X\). Since
\(C([a,b];\X)\ni\gamma\mapsto\ell(\gamma)\) is lower semicontinuous, we have that \(R([a,b];\X)\) is an \(F_\sigma\) subset
(thus, a Borel subset) of \(C([a,b];\X)\). Moreover, we denote
\[
\mathscr R(\X)\coloneqq\bigcup_{\substack{a,b\in\R:\\a<b}}R([a,b];\X)\subseteq\mathscr C(\X).
\]
Each \(\gamma\in\mathscr R(\X)\) has a unique \textbf{constant-speed reparametrisation}
\(\gamma^{\sf cs}\in\LIP([0,1];\X)\), thus
\[
\ell(\gamma^{\sf cs}|_{[s,t]})=\ell(\gamma)(t-s)\quad\text{ for every }s,t\in[0,1]\text{ with }s<t.
\]
Given a Borel function \(\rho\colon\X\to[0,\infty]\) and \(\gamma\in\mathscr C(\X)\), we define the
\textbf{path integral} of \(\rho\) over \(\gamma\) as
\begin{equation}\label{eq:def_path_int}
\int_\gamma\rho\,\d s\coloneqq\ell(\gamma)\int_0^1\rho(\gamma^{\sf cs}(t))\,\d t\in[0,\infty],
\end{equation}
where we adopt the convention that \(\int_\gamma\rho\,\d s\coloneqq+\infty\) if \(\gamma\notin\mathscr R(\X)\).
We also have that
\begin{equation}\label{eq:meas_path_int}
C([0,1];\X)\ni\gamma\mapsto\int_\gamma\rho\,\d s\in[0,\infty]\quad\text{ is Borel measurable},
\end{equation}
as it is proved e.g.\ in \cite[Corollary 2.23]{Amb:Iko:Luc:Pas:24}. In the next definition, $\mathcal L^p_{\rm ext}(\mm)^+$
is canonically referred to the Borel $\sigma$-algebra of $\X$.
\begin{definition}[Modulus]\label{def:modulus}
Let \((\X,\sfd,\mm)\) be a metric measure space and \(p\in[1,\infty)\). Let \(\Gamma\subseteq\mathscr C(\X)\) be given.
Then we define the \textbf{\(p\)-modulus} \({\rm Mod}_p(\Gamma)\in[0,\infty]\) of \(\Gamma\) as
\[
{\rm Mod}_p(\Gamma)\coloneqq\inf\bigg\{\int\rho^p\,\d\mm\;\bigg|\;\rho\in\mathcal L^p_{\rm ext}(\mm)^+
\text{ such that }\int_\gamma\rho\,\d s\geq 1\text{ for every }\gamma\in\Gamma\bigg\}.
\]
\end{definition}

Let us collect some properties of the \(p\)-modulus:
\begin{itemize}
\item \({\rm Mod}_p\) is an outer measure on \(\mathscr C(\X)\).
\item \({\rm Mod}_p(\Gamma)=0\) for every \(\Gamma\subseteq\mathscr C(\X)\setminus\mathscr R(\X)\).
\item \({\rm Mod}_p(\Gamma)=+\infty\) if \(\Gamma\subseteq\mathscr C(\X)\) contains some constant curve.
\end{itemize}
\begin{definition}[Plan]
Let \((\X,\sfd,\mm)\) be a metric measure space. Then by a \textbf{plan} on \(\X\)
we mean a measure \(\ppi\in\mathcal M_+(C([0,1];\X))\) that is concentrated on \(R([0,1];\X)\).
\end{definition}
\begin{definition}[Barycenter of a plan]\label{def:barycenter}
Let \((\X,\sfd,\mm)\) be a metric measure space and \(q\in(1,\infty]\). Let \(\ppi\) be a plan on \(\X\).
Then we say that \(\ppi\) has \textbf{barycenter} in \(L^q(\mm)\) if there exists \({\rm Bar}(\ppi)\in L^q(\mm)^+\)
such that
\[
\int\!\!\!\int_\gamma\rho\,\d s\,\d\ppi(\gamma)=\int\rho\,{\rm Bar}(\ppi)\,\d\mm\quad\text{ for every }\rho\in\LIP_{bs}(\X)^+.
\]
The barycenter \({\rm Bar}(\ppi)\) is uniquely determined thanks to the strong density of \(\LIP_{bs}(\X)^\mm\) in \(L^p(\mm)\).
In addition, we say that \(\ppi\) is a \textbf{dynamic plan} if \({\rm Bar}(\ppi)\in L^1(\mm)\).
We denote by \(\mathcal B_q(\X)\) the positive cone of \(\mathcal M_+(C([0,1];\X))\) consisting of all dynamic plans on \(\X\) with barycenter in \(L^q(\mm)\).
\end{definition}
The terminology `dynamic plan' comes from the optimal transport theory and has appeared firstly in \cite{Lott-Villani}, where we recall that plans concentrated on geodesics have finite cost (when the cost function in the optimal
transport problem is precisely the distance function). Indeed, having \({\rm Bar}(\ppi)\in L^1(\mm)\) implies 
$$
\int\sfd(x,y)\,d(\e_0,\e_1)_\#\gamma(x,y)\leq\int\ell(\gamma)\,\d\ppi(\gamma)=\|{\rm Bar}(\ppi)\|_{L^1(\mm)}<+\infty.
$$
\medskip

Given a plan \(\ppi\) with barycenter in \(L^q(\mm)\) and a \(\ppi\)-measurable set \(\Gamma\subseteq C([0,1];\X)\), we have
\[
\ppi(\Gamma)\leq\|{\rm Bar}(\ppi)\|_{L^q(\mm)}{\rm Mod}_p(\Gamma)^{1/p},
\]
see e.g.\ \cite[Remark 3.12 1)]{Amb:Iko:Luc:Pas:24}. In particular, it holds that
\begin{equation}\label{eq:plan_ll_Mod}
\ppi\ll{\rm Mod}_p.
\end{equation}
\subsection{\texorpdfstring{\(L^p\)}{Lp}-Banach \texorpdfstring{\(L^\infty\)}{Linfty}-modules}
In this section, we briefly present the theory of \(L^p\)-Banach \(L^\infty\)-modules, which was introduced by Gigli in \cite{Gig:18,Gig:17} and
provides a theoretical framework for defining `measurable tensor fields'.
\begin{definition}[\(L^p(\mm)\)-Banach \(L^\infty(\mm)\)-module]\label{def:Lp-Ban_Linfty_mod}
Let \((\X,\Sigma,\mm)\) be a \(\sigma\)-finite measure space and \(p\in[1,\infty]\). Then a module \(\mathscr M\) over
\(L^\infty(\mm)\) is said to be an \textbf{\(L^p(\mm)\)-Banach \(L^\infty(\mm)\)-module} if it is endowed with a functional
\(|\cdot|\colon\mathscr M\to L^p(\mm)^+\), called a \textbf{pointwise norm} on \(\mathscr M\), such that:
\begin{itemize}
\item[\(\rm i)\)] {\sc Pointwise norm axioms.} It holds that
\[\begin{split}
|v|\neq 0&\quad\text{ for every }v\in\mathscr M\text{ with }v\neq 0,\\
|v+w|\leq|v|+|w|&\quad\text{ for every }v,w\in\mathscr M,\\
|fv|=|f||v|&\quad\text{ for every }f\in L^\infty(\mm)\text{ and }v\in\mathscr M.
\end{split}\]
\item[\(\rm ii)\)] {\sc Completeness.} The norm \(\|v\|_{\mathscr M}\coloneqq\||v|\|_{L^p(\mm)}\) on \(\mathscr M\) is complete.
\item[\(\rm iii)\)] {\sc Glueing property.} Given a countable partition \((E_n)_n\subseteq\Sigma\) of the space \(\X\) and a sequence
\((v_n)_n\subseteq\mathscr M\) satisfying \(\big(\|\1_{E_n}^\mm v_n\|_{\mathscr M}\big)_n\in\ell^p\), there exists \(v\in\mathscr M\),
denoted \(\sum_{n\in\N}\1_{E_n}^\mm v_n\), such that
\[
\1_{E_n}^\mm v=\1_{E_n}^\mm v_n\quad\text{ for every }n\in\N.
\]
We denote by \({\rm Adm}(\mathscr M)\) the set of all `admissible' sequences \(((E_n,v_n))_n\) as above.
\end{itemize}
\end{definition}

An example of \(L^p(\mm)\)-Banach \(L^\infty(\mm)\)-module is the space \(L^p(\mm)\) itself.
\begin{remark}\label{rmk:comments_Ban_mod}{\rm
Some comments on Definition \ref{def:Lp-Ban_Linfty_mod} are in order:
\begin{itemize}
\item[\(\rm i)\)] Definition \ref{def:Lp-Ban_Linfty_mod} is fully equivalent to the notion of \emph{\(L^p(\mm)\)-normed \(L^\infty(\mm)\)-module} introduced in
\cite[Definitions 1.2.1 and 1.2.10]{Gig:18}. We prefer to use the term `\(L^p(\mm)\)-Banach \(L^\infty(\mm)\)-module' to underline that we assume completeness,
consistently with previous works (e.g.\ \cite{Pas:24,Pas:24-2}).
\item[\(\rm ii)\)] For any \(L^p(\mm)\)-Banach \(L^\infty(\mm)\)-module \((\mathscr M,|\cdot|)\), we have that \((\mathscr M,\|\cdot\|_{\mathscr M})\) is a Banach space.
\item[\(\rm iii)\)] The pointwise norm axioms imply the validity of the \emph{locality property}: if \((E_n)_n\subseteq\Sigma\) is a countable partition of
\(\X\) and \(v,w\in\mathscr M\) satisfy \(\1_{E_n}^\mm v=\1_{E_n}^\mm w\) for every \(n\in\N\), then \(v=w\). In particular, the `glued element' \(\sum_{n\in\N}\1_{E_n}^\mm v_n\) appearing
in Definition \ref{def:Lp-Ban_Linfty_mod} iii) is uniquely determined, and thus the notation is not ambiguous.
\item[\(\rm iv)\)] In the case \(p\in[1,\infty)\), asking for the validity of the glueing property is redundant, since it follows automatically from the other requirements
in Definition \ref{def:Lp-Ban_Linfty_mod}: indeed, we have that the series \(\sum_{n\in\N}\1_{E_n}^\mm v_n\) converges in \(\mathscr M\) and that
\[
\bigg\|\sum_{n=1}^N\1_{E_n}^\mm v_n-\sum_{n\in\N}\1_{E_n}^\mm v_n\bigg\|_{\mathscr M}\to 0\quad\text{ as }N\to\infty,
\]
thanks to the dominated convergence theorem.
\item[\(\rm v)\)] The situation is completely different when \(p=\infty\). First of all, in this case the glueing property does not follow from the other requirements
in Definition \ref{def:Lp-Ban_Linfty_mod}: for example, considering the measure \(\mm\coloneqq\sum_{n\in\N}\delta_n\) on \(\N\), we have that the Banach space \(c_0\)
is a module over \(L^\infty(\mm)\cong\ell^\infty\) equipped with the natural pointwise norm, but the glueing property fails (just consider the canonical elements
\((e^k)_k\) of \(c_0\), where \(e^k_k\coloneqq 1\) and \(e^k_n\coloneqq 0\) if \(n\neq k\)). Moreover, \(\sum_{n\in\N}\1_{E_n}^\mm v_n\) is a formal expression, in the
sense that -- in general -- it is not the limit of partial sums: for example, consider the glued element \(\sum_{k\in\mathbb Z}\1_{[k,k+1)}^{\mathscr L^1}=\1_\R^{\mathscr L^1}\)
of \(L^\infty(\mathscr L^1)\).
\end{itemize}
}\end{remark}

Let us recall some other concepts and constructions related to \(L^p(\mm)\)-Banach \(L^\infty(\mm)\)-modules.
A \textbf{homomorphism of \(L^p(\mm)\)-Banach \(L^\infty(\mm)\)-modules} is an operator \(T\colon\mathscr M\to\mathscr N\) that is both \(L^\infty(\mm)\)-linear and continuous.
Equivalently, we can require that \(T\colon\mathscr M\to\mathscr N\) is linear and there exists \(g\in L^\infty(\mm)^+\) such that \(|T(v)|\leq g|v|\) for every \(v\in\mathscr M\).
Moreover, we say that \(T\colon\mathscr M\to\mathscr N\) is an \textbf{isomorphism of \(L^p(\mm)\)-Banach \(L^\infty(\mm)\)-modules} if it is a bijective homomorphism such that
\(|T(v)|=|v|\) for every \(v\in\mathscr M\). In this case, we say that \(\mathscr M\) and \(\mathscr N\) are \textbf{isomorphic}.
\medskip

Following \cite[Definition 1.2.6]{Gig:18}, we give the ensuing definition of \emph{dual} \(L^q\)-Banach \(L^\infty\)-module:
\begin{definition}[Dual of an \(L^p(\mm)\)-Banach \(L^\infty(\mm)\)-module]
Let \((\X,\Sigma,\mm)\) be a \(\sigma\)-finite measure space and \(p\in[1,\infty]\). Let \(\mathscr M\) be
an \(L^p(\mm)\)-Banach \(L^\infty(\mm)\)-module. Then we define the \textbf{dual}
\(\mathscr M^*\) of \(\mathscr M\) as the space of all \(L^\infty(\mm)\)-linear and continuous maps from \(\mathscr M\) to \(L^1(\mm)\).
\end{definition}

The space \(\mathscr M^*\) is an \(L^q(\mm)\)-Banach \(L^\infty(\mm)\)-module if endowed with the pointwise operations
\[\begin{split}
(\omega+\eta)(v)&\coloneqq\omega(v)+\eta(v)\quad\text{ for every }\omega,\eta\in\mathscr M^*\text{ and }v\in\mathscr M,\\
(f\omega)(v)&\coloneqq f\,\omega(v)\quad\text{ for every }f\in L^\infty(\mm),\,\omega\in\mathscr M^*,\text{ and }v\in\mathscr M
\end{split}\]
and with the pointwise norm \(|\cdot|\colon\mathscr M^*\to L^q(\mm)^+\) defined as
\[
|\omega|\coloneqq\bigvee\big\{\omega(v)\;\big|\;v\in\mathscr M,\,|v|\leq 1\big\}
=\bigvee_{v\in\mathscr M}\1_{\{|v|>0\}}^\mm\frac{\omega(v)}{|v|}\quad\text{ for every }\omega\in\mathscr M^*.
\]
The fact that \(|\omega|\in L^q(\mm)^+\) follows from \cite[Proposition 1.2.12 v)]{Gig:18}.
\medskip

Bearing in mind that \(L^p(\mm)\)-Banach \(L^\infty(\mm)\)-modules are in particular Banach spaces (item ii) of Remark \ref{rmk:comments_Ban_mod}),
we have both the dual \(\mathscr M^*\) as a Banach module and the dual \(\mathscr M'\) as a Banach space. In the case \(p\in[1,\infty)\), these
two objects can be canonically identified in the following way:
\begin{proposition}\label{prop:map_Int}
Let \((\X,\Sigma,\mm)\) be a \(\sigma\)-finite measure space and \(p\in[1,\infty)\). Let \(\mathscr M\) be an \(L^p(\mm)\)-Banach \(L^\infty(\mm)\)-module.
Let us define the operator \(\text{\sc Int}_{\mathscr M}\colon\mathscr M^*\to\mathscr M'\) as
\[
\text{\sc Int}_{\mathscr M}(\omega)(v)\coloneqq\int\omega(v)\,\d\mm\quad\text{ for every }\omega\in\mathscr M^*\text{ and }v\in\mathscr M.
\]
Then \(\text{\sc Int}_\mathscr M\) is an isomorphism of Banach spaces.
\end{proposition}

The above result is taken from \cite[Proposition 1.2.13]{Gig:18}. The corresponding statement for \(p=\infty\) is false.
For example, the dual \(L^\infty(\mm)^*\) of \(L^\infty(\mm)\) as a Banach module is isomorphic to \(L^1(\mm)\), while the
dual \(L^\infty(\mm)'\) as a Banach space is (typically) not isomorphic to \(L^1(\mm)\).
\begin{remark}\label{rmk:conseq_HB}{\rm
The classical Hahn--Banach theorem has a number of useful consequences even in the setting of \(L^p(\mm)\)-Banach \(L^\infty(\mm)\)-modules, as it
was observed in \cite{Gig:18} in several instances. Below we recall one of such consequences, which we will need later in this work (to get \eqref{eq:grad_exist}).

\emph{Let \((\X,\Sigma,\mm)\) be a \(\sigma\)-finite measure space and \(p\in[1,\infty)\). Let \(\mathscr M\) be an \(L^p(\mm)\)-Banach
\(L^\infty(\mm)\)-module. Let \(v\in\mathscr M\) be given. Then there exists an element \(\omega_v\in\mathscr M^*\) such that
\[
\omega_v(v)=|v|^p,\qquad|\omega_v|=|v|^{p/q},
\]
where we adopt the convention that \(|v|^{p/q}\coloneqq\1_{\{|v|>0\}}^\mm\) if \(q=\infty\).
}

This property appears implicitly in many results of \cite{Gig:18}, as in \cite[Proposition 1.2.15]{Gig:18}.
Moreover, it is a particular case of \cite[Corollary 3.31]{Luc:Pas:23}, where a larger class of modules is studied.
}\end{remark}

Following \cite[Definition 1.2.20]{Gig:18}, we say that an \(L^2(\mm)\)-Banach \(L^\infty(\mm)\)-module \(\mathscr H\) is a \textbf{Hilbert module} if
it is a Hilbert space when considered as a Banach space, or equivalently if
\[
|v+w|^2+|v-w|^2=2|v|^2+2|w|^2\quad\text{ for every }v,w\in\mathscr H.
\]
The above formula is referred to as the \textbf{pointwise parallelogram rule}.
\subsubsection*{Modules generated by pointwise sublinear maps}
Later on, we will need the existence result for \(L^p\)-Banach \(L^\infty\)-modules that we are going to present.
Such result is a particular case of \cite[Theorem 3.19]{Luc:Pas:23}, which encompasses several constructions that
were originally performed in the works \cite{Gig:18,Gig:17}. We first introduce some auxiliary terminology.
\begin{itemize}
\item Let \((\X,\Sigma,\mm)\) be a \(\sigma\)-finite measure space and \(\mathscr M\) an \(L^p(\mm)\)-Banach
\(L^\infty(\mm)\)-module, for some \(p\in[1,\infty]\). Let \(\mathcal V\) be a vector subspace of \(\mathscr M\).
Then we say that \(\mathcal V\) \textbf{generates \(\mathscr M\) (in the sense of \(L^p(\mm)\)-Banach \(L^\infty(\mm)\)-modules)} provided
\[
{\rm cl}_{\mathscr M}\bigg(\bigg\{\sum_{n\in\N}\1_{E_n}^\mm v_n\;\bigg|\;((E_n,v_n))_n\in{\rm Adm}(\mathscr M)
\text{ with }(v_n)_n\subseteq\mathcal V\bigg\}\bigg)=\mathscr M.
\]
More generally, we say that a set \(\mathcal S\subseteq\mathscr M\) generates \(\mathscr M\) if the linear span of
\(\mathcal S\) generates \(\mathscr M\).
\item Let \(V\) be a vector space. Then we say that a map \(\psi\colon V\to L^p(\mm)^+\) is \textbf{even} if
\[
\psi(-v)=\psi(v)\quad\text{ for every }v\in V.
\]
\item Let \(V\) be a vector space. Then we say that \(\psi\colon V\to L^p(\mm)^+\) is \textbf{pointwise sublinear} if
\[\begin{split}
\psi(\lambda v)=\lambda\,\psi(v)&\quad\text{ for every }v\in V\text{ and }\lambda\in\R\text{ with }\lambda>0,\\
\psi(v+w)\leq\psi(v)+\psi(w)&\quad\text{ for every }v,w\in V.
\end{split}\]
\end{itemize}
\begin{theorem}[Module generated by a pointwise sublinear map]\label{thm:module_gen_sublin_map}
Let \((\X,\Sigma,\mm)\) be a \(\sigma\)-finite measure space and \(p\in[1,\infty]\). Let \(V\) be a vector space. Let \(\psi\colon V\to L^p(\mm)^+\) be an even, pointwise sublinear map. Then there exists a couple \((\mathscr M_{\langle\psi\rangle},T_{\langle\psi\rangle})\) having the following properties:
\begin{itemize}
\item[\(\rm i)\)] \(\mathscr M_{\langle\psi\rangle}\) is an \(L^p(\mm)\)-Banach \(L^\infty(\mm)\)-module and \(T_{\langle\psi\rangle}\colon V\to\mathscr M_{\langle\psi\rangle}\)
is a linear operator.
\item[\(\rm ii)\)] \(|T_{\langle\psi\rangle}(v)|=\psi(v)\) for every \(v\in V\).
\item[\(\rm iii)\)] \(T_{\langle\psi\rangle}(V)\) generates \(\mathscr M_{\langle\psi\rangle}\) in the sense of \(L^p(\mm)\)-Banach \(L^\infty(\mm)\)-modules.
\end{itemize}
Moreover, \((\mathscr M_{\langle\psi\rangle},T_{\langle\psi\rangle})\) is unique up to a unique isomorphism, in the following sense: given any couple
\((\mathscr M,T)\) having the same properties, there exists a unique isomorphism of \(L^p(\mm)\)-Banach \(L^\infty(\mm)\)-modules \(\Phi\colon\mathscr M_{\langle\psi\rangle}\to\mathscr M\)
such that the diagram
\[\begin{tikzcd}
V \arrow[r,"T_{\langle\psi\rangle}"] \arrow[dr,swap,"T"] & \mathscr M_{\langle\psi\rangle} \arrow[d,"\Phi"] \\
& \mathscr M
\end{tikzcd}\]
commutes. We say that \(\mathscr M_{\langle\psi\rangle}\) and \(T_{\langle\psi\rangle}\) are \textbf{generated by} (or \textbf{induced by}) the map \(\psi\).
\end{theorem}

We also extract the following property from the more general statement \cite[Proposition 3.20]{Luc:Pas:23},
which is a universal property that characterises \((\mathscr M_{\langle\psi\rangle},T_{\langle\psi\rangle})\) in a categorical sense.
\begin{proposition}\label{prop:univ_prop_mod_gen}
Let \((\X,\Sigma,\mm)\) be a \(\sigma\)-finite measure space and \(p\in[1,\infty]\). Let \(V\) be a vector space
and \(\psi\colon V\to L^p(\mm)^+\) an even, pointwise sublinear map. Assume that \(S\colon V\to L^1(\mm)\) is a
linear map having the following property: there exists \(g\in L^q(\mm)^+\) such that \(|S(v)|\leq g\,\psi(v)\) for all
\(v\in V\). Then there exists a unique \(L^\infty(\mm)\)-linear continuous map \(\hat S\colon\mathscr M_{\langle\psi\rangle}\to L^1(\mm)\)
such that
\[\begin{tikzcd}
V \arrow[r,"T_{\langle\psi\rangle}"] \arrow[dr,swap,"S"] & \mathscr M_{\langle\psi\rangle} \arrow[d,"\hat S"] \\
& L^1(\mm)
\end{tikzcd}\]
is a commutative diagram. Moreover, it holds that \(|\hat S(v)|\leq g|v|\) for every \(v\in V\).
\end{proposition}

Another consequence of Theorem \ref{thm:module_gen_sublin_map} is the following extension result (cf.\ \cite[Corollary 3.21]{Luc:Pas:23}).
\begin{corollary}\label{cor:conseq_univ_prop}
Let \((\X,\Sigma,\mm)\) be a \(\sigma\)-finite measure space and \(p\in[1,\infty]\). Let \(V\) be a vector space. Let \(\psi\colon V\to L^p(\mm)^+\) be an even, pointwise sublinear map.
Let \(\mathscr N\) be an \(L^p(\mm)\)-Banach \(L^\infty(\mm)\)-module. Assume that \(T\colon V\to\mathscr N\) is a
linear map having the following property: there exists \(g\in L^\infty(\mm)^+\) such that \(|T(v)|\leq g\,\psi(v)\) for all
\(v\in V\). Then there exists a unique homomorphism of \(L^p(\mm)\)-Banach \(L^\infty(\mm)\)-modules \(\hat T\colon\mathscr M_{\langle\psi\rangle}\to\mathscr N\) such that
\[\begin{tikzcd}
V \arrow[r,"T_{\langle\psi\rangle}"] \arrow[dr,swap,"T"] & \mathscr M_{\langle\psi\rangle} \arrow[d,"\hat T"] \\
& \mathscr N
\end{tikzcd}\]
is a commutative diagram. Moreover, it holds that \(|\hat T(v)|\leq g|v|\) for every \(v\in V\).
\end{corollary}
We will use Theorem \ref{thm:module_gen_sublin_map} for instance to construct the cotangent module of a metric measure space, which 
provides us with a notion of `space of \(p\)-integrable \(1\)-forms' that underlies the Sobolev space; see Definition \ref{def:cotg_mod}.
Proposition \ref{prop:univ_prop_mod_gen} and Corollary \ref{cor:conseq_univ_prop} will be then used to investigate the properties of the cotangent module.
\subsection{Choquet integral}\label{ss:Choquet_int}
We recall Choquet's theory of integration with respect to an outer measure. We refer to \cite{Den:10} for a
thorough presentation of this topic. We also mention that the same notion of integral is considered in
\cite{Amb:Til:04}, under the name `Archimedean integral', borrowing a terminology from De Giorgi.
The following presentation is taken from \cite{Deb:Gig:Pas:21}.
\medskip

Let \(\mu\) be an outer measure on a set \(\X\neq\varnothing\). Let \(f\colon\X\to[0,\infty]\) be a
given function. Then we define the \textbf{Choquet integral} of \(f\) with respect to \(\mu\) on a set
\(E\subseteq\X\) via \emph{Cavalieri's formula}, as
\[
\int_E f\,\d\mu\coloneqq\int_0^{+\infty}\mu(E\cap\{f>t\})\,\d t\in[0,\infty]
\]
and when \(E=\X\) we use the shorthand notation \(\int f\,\d\mu\coloneqq\int_\X f\,\d\mu\). Observe that
the above definition is well-posed because the function \([0,\infty)\ni t\mapsto\mu(E\cap\{f>t\})\in[0,\infty]\)
is non-increasing, and thus Lebesgue measurable. Next, we collect a few basic properties of outer measures and of
the Choquet integral, for whose proof we refer to \cite[Proposition 2.1]{Deb:Gig:Pas:21}. Let \(f,g\colon\X\to[0,\infty]\) be given functions.
\begin{itemize}
\item If \(f\leq g\), then \(\int f\,\d\mu\leq\int g\,\d\mu\).
\item {\sc Positive \(1\)-homogeneity.} \(\int\lambda f\,\d\mu=\lambda\int f\,\d\mu\) for every \(\lambda>0\).
\item \(\int f\,\d\mu=0\) if and only \(\mu(\{f>0\})=0\).
\item If \(\mu(\{f\neq g\})=0\), then \(\int f\,\d\mu=\int g\,\d\mu\).
However, it is worth pointing out that it can happen that \(f\leq g\) and \(\int f\,\d\mu=\int g\,\d\mu\),
but \(\mu(\{f<g\})>0\); cf.\ with \cite[Example 2.2]{Deb:Gig:Pas:21}.
\item {\sc Chebyshev's inequality.} It holds that \(\mu(\{f\geq\lambda\})\leq\lambda^{-1}\int_{\{f\geq\lambda\}}f\,\d\mu\)
for every \(\lambda>0\).
\item {\sc Monotone convergence theorem.} If \((f_n)_n\) is a sequence of functions \(f_n\colon\X\to[0,\infty]\) such that
\(f_n(x)\nearrow f(x)\) for \(\mu\)-a.e.\ \(x\in\X\), then \(\int f_n\,\d\mu\to\int f\,\d\mu\). On the other hand, the
analogue of the dominated convergence theorem might fail, cf.\ with \cite[Remark 2.3]{Deb:Gig:Pas:21}.
\item {\sc Borel--Cantelli lemma.} If \((E_n)_n\) is a sequence of subsets of \(\X\) with \(\sum_{n\in\N}\mu(E_n)<+\infty\),
then it holds that \(\mu\big(\bigcap_{n\in\N}\bigcup_{k\geq n}E_k\big)=0\).
\end{itemize}
An outer measure \(\mu\) on \(\X\) is said to be \textbf{submodular} provided it satisfies
\[
\mu(E\cup F)+\mu(E\cap F)\leq\mu(E)+\mu(F)\quad\text{ for every }E,F\subseteq\X.
\]
For our purposes, the relevance of the notion of submodularity is due to the fact that it captures exactly
the situations in which the Choquet integral is subadditive (and thus sublinear):
\begin{theorem}[Subadditivity theorem]\label{thm:subadd_Choquet_int}
Let \(\mu\) be an outer measure on a set \(\X\neq\varnothing\). Then \(\mu\) is submodular if and only if
the Choquet integral induced by \(\mu\) is subadditive, i.e.
\[
\int (f+g)\,\d\mu\leq\int f\,\d\mu+\int g\,\d\mu\quad\text{ for every }f,g\colon\X\to[0,\infty].
\]
\end{theorem}

For a proof of the above result, we refer to \cite[Theorem 2.5]{Deb:Gig:Pas:21} or \cite[Chapter 6]{Den:10}.
\subsubsection*{The space \texorpdfstring{\(L^0(\mu)\)}{L0(mu)}}
Fix a metric measure space \((\X,\sfd,\mm)\) and a submodular outer measure \(\mu\) on \(\X\) such that
\(\mm\ll\mu\). Assume that \(\mu\) is \textbf{boundedly-finite}, i.e.\ \(\mu(B)<+\infty\) for every
bounded subset \(B\) of \(\X\). The outer measure \(\mu\) induces an equivalence relation on \(\mathcal L^0(\bar\mm)\):
given any \(f,g\in\mathcal L^0(\bar\mm)\), we declare that \(f\sim_\mu g\) if and only if \(f=g\) holds \(\mu\)-a.e.\ on \(\X\),
and the definition is well posed because $\bar\mm\ll\mu$ as well. We then define the space \(L^0(\mu)\) as
\[
L^0(\mu)\coloneqq\mathcal L^0(\bar\mm)/\sim_\mu.
\]
We denote by \(\pi_\mu\colon\mathcal L^0(\bar\mm)\to L^0(\mu)\) the canonical projection operator.
Notice that when \(\mu\) is the outer measure \(\mm^*\) induced by \(\mm\) via the Carath\'{e}odory's construction,
we have that \(L^0(\mm^*)=L^0(\bar\mm)\) and \(\pi_{\mm^*}\) coincides with (the restriction to \(\mathcal L^0(\bar\mm)\) of)
the map \(\pi_{\bar\mm}\) defined in \eqref{eq:def_pi_m}.
\medskip

We now define a distance on \(L^0(\mu)\). Fix an increasing sequence \((U_k)_k\) of bounded open subsets
of \(\X\) having the following property: if \(B\subseteq\X\) is bounded, then \(B\subseteq U_{k(B)}\) for
some \(k(B)\in\N\). In particular, we have that \(\bigcup_k U_k=\X\). We equip the space \(L^0(\mu)\)
with the following distance:
\begin{equation}\label{eq:d_L0}
\sfd_{L^0(\mu)}(f,g)\coloneqq\sum_{k\in\N}\frac{1}{2^k(\mu(U_k)\vee 1)}\int_{U_k}|f-g|\wedge 1\,\d\mu
\quad\text{ for every }f,g\in L^0(\mu).
\end{equation}
Notice that the above definition is well-posed, since it does not depend on the specific choice of the representatives
of \(f\) and \(g\) (which is the reason why we have written \(\int_{U_k}|f-g|\wedge 1\,\d\mu\), even though \(|f-g|\wedge 1\)
is an equivalence class rather than a function). Moreover, the fact that \(\sfd_{L^0(\mu)}\) satisfies the triangle
inequality -- and thus it is a distance -- follows from the subadditivity of the integral, which is guaranteed by
Theorem \ref{thm:subadd_Choquet_int}. Next, we collect some more properties of \(L^0(\mu)\):
\begin{itemize}
\item \(L^0(\mu)\) is a commutative algebra with respect to the natural pointwise operations.
\item Given that \(\mm\ll\mu\), we have a natural projection operator
\begin{equation}\label{eq:Pr_mu}
{\rm Pr}_{\mu,\mm}\colon L^0(\mu)\to L^0(\mm),
\end{equation}
which sends the \(\sim_\mu\)-equivalence class of an \(\mm\)-measurable function \(f\colon\X\to\R\) to \(\pi_\mm(f)\).
\item \((L^0(\mu),\sfd_{L^0(\mu)})\) is a complete metric space. Moreover, \(L^0(\mu)\) is a topological vector space
with respect to the topology induced by \(\sfd_{L^0(\mu)}\).
\item The distance \(\sfd_{L^0(\mu)}\) metrises the `convergence in \(\mu\)-measure on bounded sets'. Namely, given any
\((f_n)_n\subseteq L^0(\mu)\) and \(f\in L^0(\mu)\), it holds that \(\sfd_{L^0(\mu)}(f_n,f)\to 0\) if and only if
\[
\lim_{n\to\infty}\mu\big(B\cap\{|f_n-f|>\varepsilon\}\big)=0\quad\text{ for every }B\subseteq\X\text{ bounded and }\varepsilon>0.
\]
In particular, the choice of \((U_k)_k\) affects the distance \(\sfd_{L^0(\mu)}\), but not the induced topology.
\item If \((f_n)_n\subseteq L^0(\mu)\) and \(f\in L^0(\mu)\) satisfy \(\sfd_{L^0(\mu)}(f_n,f)\to 0\) as \(n\to\infty\),
then we can extract a subsequence \((n_j)_j\) such that \(f_{n_j}(x)\to f(x)\) as \(j\to\infty\) for \(\mu\)-a.e.\ \(x\in\X\).
The converse implication might fail: there are examples of \(\mu\)-a.e.\ converging sequences that do not converge with respect
to \(\sfd_{L^0(\mu)}\), cf.\ with \cite[Remark 2.14]{Deb:Gig:Pas:21}.
\end{itemize}
Proofs of the above properties can be extracted from \cite{Deb:Gig:Pas:21}. Therein, only a specific submodular outer measure (i.e.\ the
Sobolev \(2\)-capacity \(\rm Cap\)) is considered, but one can easily check that the arguments rely only on the fact that \(\rm Cap\) is
submodular and boundedly finite.
\medskip
Let us also define the space \(L^\infty(\mu)\subseteq L^0(\mu)\) of all \textbf{\(\mu\)-essentially bounded} functions on \(\X\) as
\[
L^\infty(\mu)\coloneqq\bigg\{f\in L^0(\mu)\;\bigg|
\;\|f\|_{L^\infty(\mu)}\coloneqq\inf_{\bar f\in\pi_\mu^{-1}(f)}\sup_\X|\bar f|<+\infty\bigg\}.
\]
One can readily check that \((L^\infty(\mu),\|\cdot\|_{L^\infty(\mu)})\) is a Banach space.
\medskip
We will be concerned with two classes of submodular, boundedly-finite outer measures:
\begin{itemize}
\item
In the case where \(\mu=\mm^*\), the distance \(\sfd_{L^0(\mm^*)}\) metrises the `pointwise
\(\bar\mm\)-a.e.\ convergence up to a subsequence', see e.g.\ \cite[Proposition 1.1.21]{Gig:Pas:20}.
\item The \emph{Sobolev \(p\)-capacity} \({\rm Cap}_p\), for any \(p\in[1,\infty)\), on a metric measure space; see Section \ref{ss:capacity}.
\end{itemize}
\subsection{Metric \texorpdfstring{\(1\)}{1}-currents}
Let us recall a few basic concepts in the theory of \emph{metric currents}, which were introduced in \cite{Amb:Kir:00}.
\begin{definition}[Metric \(1\)-current]
Let \((\X,\sfd)\) be a complete, separable metric space. Then we say that a bilinear functional
\(T\colon\LIP_b(\X)\times\LIP(\X)\to\R\) is a \textbf{\(1\)-current} on \(\X\) if the following hold:
\begin{itemize}
\item[\(\rm i)\)] If \((g_n)_{n\in\N\cup\{\infty\}}\subseteq\LIP(\X)\) satisfy \(\sup_{n\in\N}\Lip(g_n)<+\infty\)
and \(g_n(x)\to g_\infty(x)\) for every \(x\in\X\), then \(T(f,g_n)\to T(f,g_\infty)\) for every \(f\in\LIP_b(\X)\).
\item[\(\rm ii)\)] If \((f,g)\in\LIP_b(\X)\times\LIP(\X)\) and \(g\) is constant on a neighbourhood of
\(\{f\neq 0\}\), then it holds that \(T(f,g)=0\).
\item[\(\rm iii)\)] There exists a measure \(\mu\in\mathcal M_+(\X)\) such that
\begin{equation}\label{eq:mass_current}
|T(f,g)|\leq\Lip(g)\int|f|\,\d\mu\quad\text{ for every }(f,g)\in\LIP_b(\X)\times\LIP(\X).
\end{equation}
\end{itemize}
The minimal measure \(\mu\in\mathcal M_+(\X)\) satisfying \eqref{eq:mass_current} is denoted by \(\|T\|\) and is
called the \textbf{mass measure} of \(T\). We denote by \({\bf M}_1(\X)\) the set of all \(1\)-currents on \(\X\).
\end{definition}
\begin{remark}[Local currents] \label{rem_local_currents} {\rm The theory of metric currents can be adapted to the case when the mass measure
$\|T\|$ is boundedly finite, along the lines of \cite{La:We:11} (see also \cite{Lang}). It suffices to consider duality with $\LIP_{bs}(\X)\times\LIP_b(\X)$, with obvious changes in the other definitions. We restricted ourselves to currents with finite mass to focus on the main ideas.
}
\end{remark}

The space \({\bf M}_1(\X)\) is a Banach space if endowed with the pointwise vector space operations and with the norm
\(T\mapsto\|T\|(\X)\). The \textbf{boundary} \(\partial T\colon\LIP_b(\X)\to\R\) of \(T\in{\bf M}_1(\X)\) is defined as
\[
\partial T(f)\coloneqq T(\1_\X,f)\quad\text{ for every }f\in\LIP_b(\X).
\]
If there exists \(\nu\in\mathcal M(\X)\) such that \(\partial T(f)=\int f\,\d\nu\) for every \(f\in\LIP_b(\X)\),
then we say that \(T\) is a \textbf{normal \(1\)-current}. The measure \(\nu\) is uniquely determined and, with an
abuse of notation, we will denote it by \(\partial T\). We define \({\bf N}_1(\X)\) as the set of normal
\(1\)-currents on \(\X\). The space \({\bf N}_1(\X)\) is a vector subspace of \({\bf M}_1(\X)\) and it is a
Banach space if endowed with the norm \(T\mapsto\|T\|(\X)+|\partial T|(\X)\).
\medskip

Every plan \(\ppi\) on \(\X\) induces a normal \(1\)-current \([\![\ppi]\!]\in{\bf N}_1(\X)\), which is defined as follows:
\begin{equation}\label{eq:current_induced_by_plan}
[\![\ppi]\!](f,g)\coloneqq\int\!\!\!\int_0^1 f(\gamma^{\sf cs}(t))\frac{\d}{\d t}g(\gamma^{\sf cs}(t))\,\d t\,\d\ppi(\gamma)\quad\text{ for every }(f,g)\in\LIP_b(\X)\times\LIP(\X).
\end{equation}
It can be readily checked that \(\partial[\![\ppi]\!]=(\e_1)_\#\ppi-(\e_0)_\#\ppi\) and that
\begin{equation}\label{eq:ineq_curr_ind_plan}
\int f\,\d\|[\![\ppi]\!]\|\leq\int\!\!\!\int_\gamma f\,\d s\,\d\ppi(\gamma)\quad\text{ for every }f\in\LIP_b(\X)^+.
\end{equation}
Remarkably, the converse holds as well, namely \emph{every} normal \(1\)-current \(T\) can be written as \([\![\ppi]\!]\) for some plan \(\ppi\).
In addition, the plan \(\ppi\) can be chosen to be optimal, in the sense that the inequality in \eqref{eq:ineq_curr_ind_plan} is saturated. This is the
content of Paolini--Stepanov's version for metric currents \cite{Pao:Ste:12,Pao:Ste:13} of Smirnov's superposition principle \cite{Smi:93}, which we
report in the following result (that is a reformulation of \cite[Corollary 3.3]{Pao:Ste:12} taken from \cite[Theorem 4.9]{DiMar:Gig:Pas:Sou:20}).
\begin{theorem}[Superposition principle for normal \(1\)-currents]\label{thm:Paolini-Stepanov}
Let \((\X,\sfd)\) be a complete, separable metric space. Let \(T\in{\bf N}_1(\X)\) be given.
Then there exists a plan \(\ppi\) on \(\X\) that is concentrated on non-constant Lipschitz curves
having constant speed and satisfying
\[
T=[\![\ppi]\!],\qquad\int f\,\d\|T\|=\int\!\!\!\int_\gamma f\,\d s\,\d\ppi(\gamma)\quad\text{ for every }f\in\LIP_b(\X)^+.
\]
If \(T\) is acyclic, then the plan \(\ppi\) can be chosen so that \((\partial T)^+=(\e_1)_\#\ppi\) and \((\partial T)^-=(\e_0)_\#\ppi\).
\end{theorem}
\begin{remark}\label{rem:superposition:localcurrents}{\rm
A superposition principle for local currents was recently proved by the first named author and coauthors in \cite{Amb:Renzi:Vitillaro:2026}.
}    
\end{remark}
\subsection{Metric Sobolev spaces}\label{ss:metric_Sob_spaces}
\begin{definition}[Weak upper gradient]\label{def:wug}
Let \((\X,\sfd,\mm)\) be a metric measure space and \(p\in[1,\infty)\). Let $f \colon \X \to [-\infty,\infty]$ be $\mm$-measurable.
Then we say that \(\rho\in\mathcal L^p_{\rm ext}(\mm)^+\) is a \textbf{weak \(p\)-upper gradient} of \(f\) if
\[
\big|f(\gamma(b_\gamma))-f(\gamma(a_\gamma))\big|\leq\int_\gamma\rho\,\d s\quad\text{ for }{\rm Mod}_p\text{-a.e.\ nonconstant }\gamma\in\mathscr R(\X).
\]
We denote by \({\rm WUG}_p(f)\) the set of all weak \(p\)-upper gradients of \(f\).
\end{definition}

Following \cite[Definition 5.22]{Amb:Iko:Luc:Pas:24}, we define
\[
\bar{D}^{1,p}( \X ) \coloneqq
\big\{f\colon \X\to [-\infty, \infty]\;\mm\text{-measurable}\,\big|\, {\rm WUG}_p(f) \neq \emptyset \big\}.
\]
We recall basic properties of the collection ${\rm WUG}_p(f)$ proved e.g.\ in \cite[Section 6]{Amb:Iko:Luc:Pas:24}. If the set \({\rm WUG}_p(f)\) is not empty, then it is a closed convex sublattice of \(\mathcal L^p_{\rm ext}(\mm)^+\),
thus it admits an \(\mm\)-a.e.\ unique element \(\rho_f\) of minimal \(\mathcal L^p(\mm)\)-seminorm, which is also
minimal in the \(\mm\)-a.e.\ sense. Moreover, if $f \in \bar{D}^{1,p}( \X )$ and $g \colon \X \to \R$ is defined as
\[
g(x)\coloneqq\left\{\begin{array}{ll}
f(x)\\
0
\end{array}\quad\begin{array}{ll}
\text{ if }|f(x)|<+\infty,\\
\text{ if }|f(x)|=+\infty,
\end{array}\right.
\]
then ${\rm WUG}_p(g) = {\rm WUG}_p( f )$ and ${\rm WUG}_p( f - g ) \neq \varnothing$. Furthermore, if $\rho \in {\rm WUG}_p( f )$ and $\rho' \in \mathcal L^p_{\rm ext}(\mm)^+$ such that $\rho = \rho'$ $\mm$-almost everywhere, then $\rho' \in {\rm WUG}_p( f )$; this is because the class of curves that have positive length in an $\mm$-negligible set is $\Mod_p$-negligible; see e.g.\ \cite[Lemma 6.8]{Amb:Iko:Luc:Pas:24}. We can therefore unambiguously write \(|Df|\) whenever \(f\in L^0(\mm)\) has a representative $\bar f \in \bar{D}^{1,p}( \X )$. We then define the \textbf{minimal weak \(p\)-upper gradient} of \(f\) as
\[
|Df|\coloneqq\rho_f^\mm\in L^p(\mm)^+.
\]
\begin{definition}[Newtonian, Sobolev, and Dirichlet spaces]\label{def:Sobolev}
Let \((\X,\sfd,\mm)\) be a metric measure space and \(p\in[1,\infty)\). Then:
\begin{itemize}
\item[\(\rm i)\)] We define the \textbf{Newtonian space} \(\bar{N}^{1,p}(\X)\) as
\[
    \bar{N}^{1,p}(\X)\coloneqq\big\{f\in\mathcal L^p(\bar\mm)\;\big|\;{\rm WUG}_p(f)\neq\varnothing\big\}.
\]
We endow \(\bar{N}^{1,p}(\X)\) with the seminorm \(\|f\|_{\bar{N}^{1,p}(\X)}\coloneqq\big(\|f\|_{\mathcal L^p(\bar\mm)}^p+\||Df|\|_{L^p(\mm)}^p\big)^{1/p}\).
\item[\(\rm ii)\)] We define the \textbf{Sobolev space} \(W^{1,p}(\X)\) as
\[
W^{1,p}(\X)\coloneqq\big\{\pi_\mm(f)\;\big|\;f\in \bar{N}^{1,p}(\X)\big\}\subseteq L^p(\mm).
\]
We endow \(W^{1,p}(\X)\) with the complete norm \(\|f\|_{W^{1,p}(\X)}\coloneqq\big(\|f\|_{L^p(\mm)}^p+\||Df|\|_{L^p(\mm)}^p\big)^{1/p}\).
\item[\(\rm iii)\)] We define the \textbf{Dirichlet space} \(D^{1,p}_{\mm}(\X)\) as
\[
D^{1,p}_{\mm}(\X)\coloneqq\big\{\pi_\mm(f)\;\big|\;f\in\mathcal L^0(\bar\mm),\,{\rm WUG}_p(f)\neq\varnothing\big\}\subseteq L^0(\mm).
\]
We endow $\bar{D}^{1,p}(\X)$ with the seminorm $\|f\|_{D^{1,p}(\X)} \coloneqq \| |Df| \|_{ L^{p}(\mm) }$ and $D^{1,p}_{\mm}(\X)$ with the induced seminorm $\|f\|_{D^{1,p}(\X)} \coloneqq \| |Df| \|_{ L^{p}(\mm) }$.
\end{itemize}
\end{definition}
\begin{remark}\label{remark:aboutDirichletspaces}{\rm
The definition of the Dirichlet space here differs from the one in \cite[Definition 5.22]{Amb:Iko:Luc:Pas:24}. Indeed, therein we defined the Dirichlet space, denoting it by $D^{1,p}( \X )$, as the quotient space of $\bar{D}^{1,p}( \X )$ under the seminorm $\| f \|_{ D^{1,p}( \X ) } = \| |Df| \|_{ L^{p}( \mm ) }$ in $\bar{D}^{1,p}(\X)$. That is, $f_1, f_2 \in \bar{D}^{1,p}( \X )$ satisfies $f_1 \sim f_2$ if and only if $\| f_2 - f_1 \|_{ D^{1,p}( \X ) } = 0$. Thus, \(\|\cdot\|_{D^{1,p}(\X)}\) is a norm on \(D^{1,p}(\X)\).

It follows from \cite[Theorem 6.1 ii)]{Amb:Iko:Luc:Pas:24} that every element in $D^{1,p}( \X )$ has a representative $f\in\mathcal L^0(\bar\mm) \cap \bar{D}^{1,p}(\X)$. Combining this with \cite[Theorem 6.1 iii)]{Amb:Iko:Luc:Pas:24}, we see that $D^{1,p}( \X )$ is a vector space over $\mathbb{R}$ and that the canonical map $\tau\colon D^{1,p}_\mm( \X ) \to D^{1,p}( \X )$ is a linear and surjective map.

Because of the aforementioned properties of the canonical map $\tau\colon D^{1,p}_\mm( \X ) \to D^{1,p}( \X )$ and as $D^{1,p}_{\mm}( \X ) \subseteq L^{0}( \mm )$ by construction, we formulate most of the results regarding the Dirichlet space in terms of $D^{1,p}_{\mm}(\X)$. We emphasize the difference between $D^{1,p}_{\mm}(\X)$ and $D^{1,p}( \X )$ when appropriate.}
\end{remark}

Notice that \(\bar{N}^{1,p}(\X)\), \(W^{1,p}(\X)\), and \(D^{1,p}_{\mm}(\X)\) are vector subspaces of \(\mathcal L^p(\bar\mm)\),
\(L^p(\mm)\), and \(L^0(\mm)\), respectively. It also holds that \(\LIP_{bs}(\X)\subseteq \bar{N}^{1,p}(\X)\) and \(|Df|\leq\lip(f)^\mm\) for every \(f\in\LIP_{bs}(\X)\).
\begin{definition}[Cheeger energy]\label{def:Cheeger_energy}
Let \((\X,\sfd,\mm)\) be a metric measure space and \(p\in[1,\infty)\). Then we define the \textbf{Cheeger \(p\)-energy functional}
\(\mathcal E_p\colon L^0(\mm)\to[0,\infty]\) on \(\X\) as
\[
\mathcal E_p(f)\coloneqq\frac{1}{p}\int|Df|^p\,\d\mm\quad\text{ for every }f\in D^{1,p}_{\mm}(\X)
\]
and \(\mathcal E_p(f)\coloneqq+\infty\) for every \(f\in L^0(\mm)\setminus D^{1,p}_{\mm}(\X)\).
\end{definition}
\begin{proposition}[Closure properties of minimal weak upper gradients]\label{prop:clos_wug}
Let \((\X,\sfd,\mm)\) be a metric measure space and \(p\in[1,\infty)\). Assume that \((f_n)_n\subseteq W^{1,p}(\X)\) and \(f,G\in L^p(\mm)\) satisfy
\(f_n\rightharpoonup f\) and \(|Df_n|\rightharpoonup G\) weakly in \(L^p(\mm)\). Then it holds that \(f\in W^{1,p}(\X)\) and \(|Df|\leq G\).
\end{proposition}
For a proof of the above, see for instance \cite[Lemma 7.5]{Amb:Iko:Luc:Pas:24}. Furthermore, the following calculus rules hold:
\begin{theorem}[Calculus rules for minimal weak upper gradients]\label{thm:calculus_rules_mwug}
Let \((\X,\sfd,\mm)\) be a metric measure space and \(p\in[1,\infty)\). Then the following properties are verified:
\begin{itemize}
\item[\(\rm i)\)] {\sc Locality}. Let \(f,g\in D^{1,p}_{\mm}(\X)\) be given. Then
\[
|Df|=|Dg|\quad\text{ holds }\mm\text{-a.e.\ on }\{f=g\}.
\]
\item[\(\rm ii)\)] {\sc Chain rule.} Let \(f\in D^{1,p}_{\mm}(\X)\) be given. Let \(N\subseteq\R\) be a Borel set with \(\mathscr L^1(N)=0\). Then
\[
|Df|=0\quad\text{ holds }\mm\text{-a.e.\ on }f^{-1}(N).
\]
Moreover, given any function \(\varphi\in\LIP(\R)\), it holds that \(\varphi\circ f\in D^{1,p}_{\mm}(\X)\) and
\[
|D(\varphi\circ f)|=\lip(\varphi)\circ f\,|Df|.
\]
\item[\(\rm iii)\)] {\sc Leibniz rule.} Let \(f,g\in D^{1,p}_{\mm}(\X)\cap L^\infty(\mm)\) be given. Then \(fg\in D^{1,p}_{\mm}(\X)\cap L^\infty(\mm)\) and
\[
|D(fg)|\leq|f||Dg|+|g||Df|.
\]
\item[\(\rm iv)\)]
{\sc Lattice property.} 
Let \(f,g\in D^{1,p}_{\mm}(\X)\) be given. Then \(f\vee g,f\wedge g\in D^{1,p}_{\mm}(\X)\) and
\[
|D(f\vee g)|=\1_{\{f\geq g\}}^\mm|Df|+\1_{\{f<g\}}^\mm|Dg|,\qquad|D(f\wedge g)|=\1_{\{f\leq g\}}^\mm|Df|+\1_{\{f>g\}}^\mm|Dg|.
\]
In particular, it holds that \(|f|\in D^{1,p}_{\mm}(\X)\) and \(|D|f||=|Df|\).
\end{itemize}
\end{theorem}
\begin{proof}
For a proof of i), ii), and iii), we refer to \cite[Theorem 6.1]{Amb:Iko:Luc:Pas:24}. Let us now show that iv) holds. Fix \(f,g\in D^{1,p}_{\mm}(\X)\).
Define \(\varphi\in\LIP(\R)\) as \(\varphi(t)\coloneqq t\vee 0\) for every \(t\in\R\). The chain rule ensures that \((f-g)\vee 0=\varphi\circ(f-g)\in D^{1,p}_{\mm}(\X)\),
thus \(f\vee g=\varphi\circ(f-g)+g\in D^{1,p}_{\mm}(\X)\) as well. The locality property then guarantees that \(|D(f\vee g)|=|Df|\) \(\mm\)-a.e.\ on \(\{f\geq g\}\)
and \(|D(f\vee g)|=|Dg|\) \(\mm\)-a.e.\ on \(\{f<g\}\). Finally, to deduce that \(f\wedge g,|f|\in D^{1,p}_{\mm}(\X)\) and the formulas for their minimal weak upper
gradients, just observe that \(f\wedge g=-(-f)\vee(-g)\) and \(|f|=f\vee(-f)\).
\end{proof}
\begin{theorem}[Sobolev space via relaxation]\label{thm:Sob_via_rel}
Let \((\X,\sfd,\mm)\) be a metric measure space and \(p\in[1,\infty)\). Let \(f\in L^p(\mm)\) be given.
Then it holds that \(f\in W^{1,p}(\X)\) if and only if there exist a sequence \((f_n)_n\subseteq\LIP_{bs}(\X)\)
and a function \(G\in L^p(\mm)^+\) such that
\begin{equation}\label{eq:def_Sob_rel}
f_n^\mm\rightharpoonup f,\quad\lip_a(f_n)^\mm\rightharpoonup G\quad\text{ weakly in }L^p(\mm).
\end{equation}
Moreover, for any function \(f\in W^{1,p}(\X)\) we have that
\begin{equation}\label{eq:def_Sob_rel2}
|Df|=\bigwedge\big\{G\in L^p(\mm)^+\;\big|\;\eqref{eq:def_Sob_rel}\text{ holds}\big\}.
\end{equation}
\end{theorem}
\begin{proof}
The claim follows from \cite[Theorem 7.1]{Amb:Iko:Luc:Pas:24}.
\end{proof}

The minimal weak upper gradient \(|Df|\) is a competitor for the right-hand side of \eqref{eq:def_Sob_rel2},
meaning that there exists \((f_n)_n\subseteq\LIP_{bs}(\X)\) satisfying \(f_n^\mm\rightharpoonup f\) and
\(\lip_a(f_n)^\mm\rightharpoonup|Df|\) weakly in \(L^p(\mm)\). In fact, a stronger approximation result holds, as we will
recall next. First we introduce some notation: we define the `convergence-in-energy' distance \(\sfd_{\rm en}\) on \(W^{1,p}(\X)\) as
\[
\sfd_{\rm en}(f,g)\coloneqq\Big(\|f-g\|_{L^p(\mm)}^p+\big\||Df|-|Dg|\big\|_{L^p(\mm)}^p\Big)^{1/p}\quad\text{ for every }f,g\in W^{1,p}(\X).
\]
\begin{theorem}[Density in energy of Lipschitz functions]\label{thm:density_nrg_Lip}
Let \((\X,\sfd,\mm)\) be a metric measure space and \(p\in[1,\infty)\). Then \((W^{1,p}(\X),\sfd_{\rm en})\) is a separable metric space.
Moreover, given any function \(f\in W^{1,p}(\X)\), there exists a sequence \((f_n)_n\subseteq\LIP_{bs}(\X)\) such that
\[
\sfd_{\rm en}(f_n^\mm,f)\to 0
\]
and \(\lip_a(f_n)^\mm\to|Df|\) strongly in \(L^p(\mm)\). In particular, \(\LIP_{bs}(\X)^\mm\) is dense in \((W^{1,p}(\X),\sfd_{\rm en})\).
\end{theorem}
\begin{proof}
The map \(W^{1,p}(\X)\ni f\mapsto(f,|Df|)\in L^p(\mm)\times L^p(\mm)\) is an isometry if we endow its domain
with \(\sfd_{\rm en}\) and its codomain with the \(p\)-product norm \(\|(f,G)\|_p\coloneqq
\big(\|f\|_{L^p(\mm)}^p+\|G\|_{L^p(\mm)}^p\big)^{1/p}\). It follows that \((W^{1,p}(\X),\sfd_{\rm en})\) is a separable metric space.
Moreover, for any \(f\in W^{1,p}(\X)\) we know e.g.\ from \cite[Theorem 6.16]{Amb:Iko:Luc:Pas:24} that there exists a sequence \((\tilde f_n)_n\subseteq\LIP_{bs}(\X)\)
such that \(\tilde f_n^\mm\to f\) and \(\lip_a(\tilde f_n)^\mm\to|Df|\) strongly in \(L^p(\mm)\). By Mazur's lemma, we can find \(f_n\in{\rm conv}(\{\tilde f_i\,:\,i\geq n\})\)
and \(g_n\in L^p(\mm)^+\) with \(g_n\geq\lip_a(f_n)^\mm\) such that \(f_n^\mm\to f\) and \(g_n\to|Df|\) strongly in \(L^p(\mm)\). Thanks to \cite[Lemma 6.15]{Amb:Iko:Luc:Pas:24},
we conclude that \(|Df_n|\to|Df|\) and \(\lip_a(f_n)^\mm\to|Df|\) strongly in \(L^p(\mm)\).
\end{proof}

It is not known whether Lipschitz functions are always strongly dense in the Sobolev space. However, the
following result -- recently proved by Eriksson-Bique and Poggi-Corradini \cite[Theorem 1.6]{EB:PC:23} --
says that \emph{continuous} Sobolev functions are always strongly dense in the Sobolev space.
\begin{theorem}[Strong density of continuous Sobolev functions]\label{thm:cont_Sob_dense}
Let \((\X,\sfd,\mm)\) be a metric measure space and \(p\in[1,\infty)\). Then \(W^{1,p}(\X)\cap C(\X)^\mm\) is a strongly dense subspace of \(W^{1,p}(\X)\).
\end{theorem}
The following technical result is concerned with cut-off and truncation of Dirichlet functions:
\begin{lemma}\label{lem:trunc_cut-off}
Let \((\X,\sfd,\mm)\) be a metric measure space and \(p\in[1,\infty)\). Fix \(\bar x\in\X\).
For any \(k\in\N\), define \(\eta_k\coloneqq\big(1-\sfd(\cdot,B_k(\bar x))\big)\vee 0\in\LIP_{bs}(\X)\)
and \(\psi_k(t)\coloneqq(t\wedge k)\vee(-k)\) for every \(t\in\R\). Let
\[
\mathcal T_k f\coloneqq\eta_k^\mm(\psi_k\circ f)\quad\text{ for every }f\in D^{1,p}_{\mm}(\X).
\]
Then \(\mathcal T_k f\) belongs to \(W^{1,p}(\X)\cap L^\infty(\mm)\) and has bounded support. Moreover, it holds that
\begin{equation}\label{eq:trunc_cut-off_estimates}
|\mathcal T_k f-f|\leq\1_{\{|f|>k\}\cup(\X\setminus B_k(\bar x))}^\mm|f|,\qquad
|D(\mathcal T_k f-f)|\leq\1_{\{|f|>k\}\cup(\X\setminus B_k(\bar x))}^\mm(|f|+|Df|).
\end{equation}
In particular, for every \(f\in W^{1,p}(\X)\) it holds that \(\|\mathcal T_k f-f\|_{W^{1,p}(\X)}\to 0\) as \(k\to\infty\).
\end{lemma}
\begin{proof}
Notice that \(|\mathcal T_k f|\leq k\) and \({\rm spt}(\mathcal T_k f)\subseteq\bar B_{k+1}(\bar x)\). Since
\(\psi_k\circ f\in D^{1,p}_{\mm}(\X)\cap L^\infty(\mm)\) (by the chain rule) and \(\eta_k\in\LIP_{bs}(\X)\), we have that \(\mathcal T_k f\in W^{1,p}(\X)\)
(by the Leibniz rule). Also, the \(\mm\)-a.e.\ inequalities in \eqref{eq:trunc_cut-off_estimates} can be easily obtained by
applying the Leibniz rule and the chain rule.

Now fix \(f\in W^{1,p}(\X)\) and denote \(E_k\coloneqq\{|f|>k\}\cap(\X\setminus B_k(\bar x))\).
Since \(\1_{E_k}^\mm(|f|+|Df|)\to 0\) in \(L^p(\mm)\) by the dominated convergence theorem,
\eqref{eq:trunc_cut-off_estimates} gives that \(\mathcal T_k f\to f\) in \(W^{1,p}(\X)\).
\end{proof}

We conclude this section by introducing local versions of the Newtonian and Sobolev spaces:
\begin{definition}[Local Newtonian and Sobolev spaces]\label{def:loc_Sobolev}
Let \((\X,\sfd,\mm)\) be a metric measure space and \(p\in[1,\infty)\). Then:
\begin{itemize}
\item[\(\rm i)\)] The \textbf{local Newtonian space} \(\bar{N}^{1,p}_{loc}(\X)\) is the set of all \(f\in\mathcal L^p_{loc}(\X,\sfd,\bar\mm)\)
such that for any \(x\in\X\) there is \(\eta^x\in\LIP_{bs}(\X;[0,1])\) with \(\eta^x=1\) on a neighbourhood of \(x\) and \(\eta^x f\in \bar{N}^{1,p}(\X)\).
\item[\(\rm ii)\)] We define the \textbf{local Sobolev space} \(W^{1,p}_{loc}(\X)\) as
\[
W^{1,p}_{loc}(\X)\coloneqq\big\{\pi_\mm(f)\;\big|\;f\in \bar{N}^{1,p}_{loc}(\X)\big\}\subseteq L^p_{loc}(\X,\sfd,\mm).
\]
\end{itemize}
\end{definition}
\begin{lemma}\label{lem:cut-off_loc_Sob}
Let \((\X,\sfd,\mm)\) be a metric measure space and \(p\in[1,\infty)\). Let \(f\in \bar{N}^{1,p}_{loc}(\X)\) be given.
Then there exists a non-decreasing sequence of cut-off functions \((\eta_n)_n\subseteq\LIP_{bs}(\X;[0,1])\) such that
\((\eta_n f)_n\subseteq \bar{N}^{1,p}(\X)\) and \(\X=\bigcup_{n\in\N}U_n\), where \(U_n\) denotes the interior of \(\{\eta_n=1\}\).
\end{lemma}
\begin{proof}
Letting \(\{\eta^x:x\in\X\}\) be as in Definition \ref{def:loc_Sobolev} i), we know (since \((\X,\sfd)\) is separable)
that for some sequence \((x_n)_n\subseteq\X\) it holds that \(\X=\bigcup_{n\in\N}\tilde U_n\), where \(\tilde U_n\)
denotes the interior of $\left\{ \tilde\eta_n = 1 \right\}$ for \(\tilde\eta_n\coloneqq\eta^{x_n}\). Define \(\eta_n\coloneqq\tilde\eta_1\vee\cdots\vee\tilde\eta_n\in\LIP_{bs}(\X;[0,1])\) for every \(n\in\N\).
Notice that \((\eta_n)_n\) is a non-decreasing sequence of functions and the interior \(U_n\) of \(\{\eta_n=1\}\)
contains \(\tilde U_1\cup\cdots\cup\tilde U_n\), thus \(\X=\bigcup_{n\in\N}U_n\). Finally, observe that
\(\eta_n f=(\tilde\eta_1 f)^+\vee\cdots\vee(\tilde\eta_n f)^+-(\tilde\eta_1 f)^-\vee\cdots\vee(\tilde\eta_n f)^-\in \bar{N}^{1,p}(\X)\)
for every \(n\in\N\) by the lattice property of minimal weak upper gradients (i.e.\ Theorem \ref{thm:calculus_rules_mwug} iv)).
\end{proof}
Given any \(f\in \bar{N}^{1,p}_{loc}(\X)\), there exists a unique function \(G_f\in L^p_{loc}(\X,\sfd,\mm)\) satisfying
\[
G_f=|D(\eta f)|\quad\mm\text{-a.e.\ on }\{\eta=1\},\text{ for every }\eta\in\LIP_{bs}(\X;[0,1])\text{ with }\eta f\in \bar{N}^{1,p}(\X).
\]
The well-posedness and the uniqueness of the above definition are ensured by the locality property of minimal weak upper gradients,
i.e.\ Theorem \ref{thm:calculus_rules_mwug} i), while the existence follows from the existence of \((\eta_n)_n\) as in
Lemma \ref{lem:cut-off_loc_Sob}.
The sets \(\bar{N}^{1,p}_{loc}(\X)\) and \(W^{1,p}_{loc}(\X)\) are vector subspaces of \(\mathcal L^p_{loc}(\X,\sfd,\bar\mm)\) and \(L^p_{loc}(\X,\sfd,\mm)\),
respectively. Note that \(\bar{N}^{1,p}(\X),\LIP(\X)\subseteq \bar{N}^{1,p}_{loc}(\X)\), and \(W^{1,p}(\X)\subseteq W^{1,p}_{loc}(\X)\).
\begin{remark}\label{rmk:mwug_loc_Sob_unamb}{\rm
We claim that if \(f\in \bar{N}^{1,p}_{loc}(\X)\) satisfies \(G_f\in L^p(\mm)\), then $f \in \bar{D}^{1,p}( \X )$ and $G_f = |Df|$. To prove it, fix a sequence \((\eta_n)_n\) as in Lemma \ref{lem:cut-off_loc_Sob}.
Having fixed an \(\mm\)-a.e.\ minimal weak \(p\)-upper gradient \(\rho_n\) of \(\eta_n f\) for each \(n\in\N\),
we can find a set \(\mathcal N\subseteq\mathscr R(\X)\) such that \({\rm Mod}_p(\mathcal N)=0\) and
\begin{equation}\label{eq:mwug_loc_Sob_unamb}
\big|(\eta_n f)(\gamma(b_\gamma))-(\eta_n f)(\gamma(a_\gamma))\big|\leq\int_\gamma\rho_n\,\d s
\quad\text{ for every }n\in\N\text{ and }\gamma\in\mathscr R(\X)\setminus\mathcal N.
\end{equation}
Now define the function \(\rho\in\mathcal L^0_{\rm ext}(\mm)\) as
\(\rho\coloneqq\sum_{n=1}^\infty\1_{U_n\setminus U_{n-1}}\bigvee_{j\geq n}\rho_j\),
where \(U_0\coloneqq\varnothing\) and \(U_n\) denotes the interior of \(\{\eta_n=1\}\) for any \(n\geq 1\).
Given any curve \(\gamma\in\mathscr R(\X)\setminus\mathcal N\), we have that \(\gamma([a_\gamma,b_\gamma])\) is compact,
so that \(\gamma([a_\gamma,b_\gamma])\subseteq U_{n_\gamma}\) for some \(n_\gamma\in\N\) and accordingly
\eqref{eq:mwug_loc_Sob_unamb} implies that \(\big|f(\gamma(b_\gamma))-f(\gamma(a_\gamma))\big|\leq\int_\gamma\rho_{n_\gamma}\,\d s\leq\int_\gamma\rho\,\d s\).
Given that \(G_f=\rho\) \(\mm\)-a.e.\ on \(\X\) by the locality of minimal weak upper gradients, one can deduce that
\({\rm WUG}_p(f)\neq\varnothing\). Since \(G_f=|D(\eta_n f)|=|Df|\) on \(U_n\) for every \(n\in\N\) (again by locality),
we also have that \(|Df|=G_f\), thus proving the claim.
}\end{remark}
In view of Remark \ref{rmk:mwug_loc_Sob_unamb}, setting
\[
|Df|\coloneqq G_f\quad\text{ for every }f\in \bar{N}^{1,p}_{loc}(\X)
\]
will not cause any ambiguity of notation.
\medskip

Dirichlet spaces and the local Sobolev space $W^{1,p}_{loc}( \X )$ are connected through truncations. More precisely, we have the following.
\begin{lemma}
Let \((\X,\sfd,\mm)\) be a metric measure space and \(p\in[1,\infty)\). If $f \in L^{0}( \mm )$, then $f \in D^{1,p}_{\mm}( \X )$ if and only if the truncations $f_k = \psi_k \circ f$, where \(\psi_k(t)\coloneqq(t\wedge k)\vee(-k)\) for every \(t\in\R\), satisfy $f_k \in W^{1,p}_{loc}( \X )$ for every \(k\in \N\), and there exists $G \in L^{p}( \mm )$ such that $|Df_k| \leq G$ for every $k \in \mathbb{N}$.
\end{lemma}
\begin{proof}
The `only if'-direction follows from \Cref{lem:trunc_cut-off}. The `if'-direction is a consequence of the following argument, similar to \Cref{rmk:mwug_loc_Sob_unamb}. We let $\bar{f}$ be a representative of $f$. For every \(k\in \N\) let $\bar{f}_k \in \bar{N}^{1,p}_{loc}( \X )$ be a representative of $f_k$ that has a Borel representative $G_k$ of $G$ as a $p$-weak upper gradient. Then $G_\infty = \bigvee_{ m \in \mathbb{N} } G_m \in {\rm WUG}_p( \bar{f}_k )$ for every $k \in \mathbb{N}$. Moreover, for $\Mod_p$-almost every nonconstant constant speed $\gamma \colon [0,1] \to \X$, we have that $( \bar{f}_k \circ \gamma )_{ k \in \mathbb{N} }$ is absolutely continuous and, for every $0 \leq s < t \leq 1$, we have
\begin{equation*}
    \sup_k | \bar{f}_k( \gamma_t ) - \bar{f}_k( \gamma_s ) | \leq \int_{ \gamma|_{ [s,t] } } G_\infty \,\d s < \infty,
\end{equation*}
and that $( \bar{f}_k( \gamma_t ) )_{ k \in \mathbb{N} }$ converges to $\bar{f}( \gamma_t )$ for $\mathcal{L}^{1}$-almost every $t \in (0,1)$. These facts imply that $( \bar{f}_k( \gamma_t ) )_{ k \in \mathbb{N} }$ is Cauchy in $\mathbb{R}$ for every $t \in [0,1]$ and if we define $\bar{f}_\infty$ to be equal to $\lim_{ k \rightarrow \infty } f_k(x)$ whenever the sequence $( \bar{f}_k( x ) )_{ k \in \mathbb{N} }$ is Cauchy and equal to $\bar{f}$ otherwise, we find that $\bar{f}_\infty$ is a representative of $f$ with $G_\infty \in {\rm WUG}_p( \bar{f}_\infty )$. The claim follows.
\end{proof}
\section{Quasicontinuity of Sobolev functions}
This section is entirely devoted to various (equivalent) notions of Sobolev \(p\)-capacity \({\rm Cap}_p\) and to the quasicontinuity properties of metric Sobolev functions. 
\subsection{Sobolev capacities}\label{ss:capacity}
\begin{definition}[Sobolev \(p\)-capacity]\label{def:cap_p}
Let \((\X,\sfd,\mm)\) be a metric measure space and \(p\in[1,\infty)\). Let \(E\subseteq\X\) be a given set.
Then we define the quantity \({\rm Cap}_p(E)\in[0,\infty]\) as
\[
{\rm Cap}_p(E)\coloneqq\inf\Big\{\|f\|_{\bar{N}^{1,p}(\X)}^p\;\Big|\;f\in \bar{N}^{1,p}(\X)\text{ such that }f\geq 1\text{ on }E\Big\}.
\]
We say that \({\rm Cap}_p(E)\) is the \textbf{Sobolev \(p\)-capacity}, or just the \textbf{\(p\)-capacity}, of the set \(E\).
\end{definition}

We collect a few properties of the Sobolev \(p\)-capacity:
\begin{itemize}
\item[\(\rm i)\)] \({\rm Cap}_p\) is a submodular, boundedly-finite outer measure on \(\X\); see e.g.\ \cite{HKST:15,Bj:Bj:11}.
\item[\(\rm ii)\)] Given any \(E\subseteq\X\), it holds that
\[
{\rm Cap}_p(E)\coloneqq\inf\Big\{\|f\|_{\bar{N}^{1,p}(\X)}^p\;\Big|\;f\in \bar{N}^{1,p}(\X)\text{ with }0\leq f\leq 1\text{, and }f=1\text{ on }E\Big\}.
\]
\item[\(\rm iii)\)] It holds that \(\mm(E)\leq{\rm Cap}_p(E)\) for every \(\mm\)-measurable set \(E\subseteq\X\), thus in particular
\begin{equation}\label{eq:mm_ll_Cap}
\mm\ll{\rm Cap}_p.
\end{equation}
\item[\(\rm iv)\)] As proved in \cite[Theorem 1.8]{EB:PC:23}, for any set \(E\subseteq\X\) we have that
\begin{equation}\label{eq:Cap_outer_regular}
{\rm Cap}_p(E)=\inf\big\{{\rm Cap}_p(U)\;\big|\;U\subseteq\X\text{ open and }E\subseteq U\big\}.
\end{equation}
\end{itemize}
\begin{remark}\label{rmk:equiv_Cap}{\rm
In the literature (see e.g.\ \cite{Hei:Kil:Mar:12,Kin:Mar:00}), the `neighbourhood capacity' is often considered:
\[
\overline{\rm Cap}_p(E)\coloneqq\inf\Big\{\|f\|_{\bar{N}^{1,p}(\X)}^p\;\Big|\;f\in \bar{N}^{1,p}(\X)\text{ such that }
f\geq 1\text{ on some neighbourhood of }E\Big\}.
\]
As it was pointed out in \cite[Remark 1.10]{EB:PC:23}, it follows from \eqref{eq:Cap_outer_regular} that
\begin{equation}\label{eq:equiv_Cap}
{\rm Cap}_p(E)=\overline{\rm Cap}_p(E)\quad\text{ for every }E\subseteq\X.
\end{equation}
To prove it, just observe that \({\rm Cap}_p(U)=\overline{\rm Cap}_p(U)\) for every open set \(U\subseteq\X\).
}\end{remark}
\begin{proposition}\label{prop:mm_implies_Cap}
Let \((\X,\sfd,\mm)\) be a metric measure space and \(p\in[1,\infty)\). Let \(U\subseteq\X\) be open. Assume that
\(f,g\in \bar{N}^{1,p}_{loc}(\X)\) satisfy \(\bar\mm(U\cap\{f\neq g\})=0\). Then it holds \({\rm Cap}_p(U\cap\{f\neq g\})=0\).
\end{proposition}
\begin{proof}
For \(f,g\in \bar{N}^{1,p}(\X)\) and \(U=\X\), see e.g.\ \cite[Proposition 6.12]{Amb:Iko:Luc:Pas:24}.
For the general case, one can argue as follows. Take a sequence \((\eta_n)_n\) of boundedly-supported Lipschitz
functions \(\eta_n\colon\X\to[0,1]\) such that \(\{\eta_n>0\}\subseteq U\) and \(\eta_n(f-g)\in \bar{N}^{1,p}(\X)\)
for all \(n\in\N\), and \(\bigcup_{n\in\N}\{\eta_n=1\}=U\). Given that \(h_n\coloneqq\eta_n|f-g|\in \bar{N}^{1,p}(\X)\)
(by the lattice property of minimal weak upper gradients) and \(\bar\mm(\{h_n\neq 0\})=0\), we deduce that
\({\rm Cap}_p(\{h_n\neq 0\})=0\) for every \(n\in\N\). We conclude that
\[
{\rm Cap}_p(U\cap\{f\neq g\})\leq\sum_{n=1}^\infty{\rm Cap}_p\big(\{\eta_n=1\}\cap\{h_n\neq 0\}\big)=0,
\]
thus proving the statement.
\end{proof}
\begin{remark}\label{rmk:N1p_Cap-ae_inv}{\rm
If \(f\in \bar{N}^{1,p}_{loc}(\X)\) and \(\tilde f\in\mathcal L^p_{loc}(\X,\sfd,\bar\mm)\) satisfy \(f=\tilde f\)
in the \({\rm Cap}_p\)-a.e.\ sense, then \(\tilde f\in \bar{N}^{1,p}_{loc}(\X)\). If also
\(f\in \bar{N}^{1,p}(\X)\), then \(\tilde f\in \bar{N}^{1,p}(\X)\). See e.g.\ the proof of \cite[Proposition 7.2.8]{HKST:15}.}
\end{remark}
\begin{remark}\label{rmk:equiv_Cap_bis}{\rm
In the work \cite{Deb:Gig:Pas:21}, the following notion of capacity is considered (for \(p=2\)):
\begin{equation}\label{eq:othercap}
\overline{\rm Cap}_p^\mm(E)\coloneqq\inf\Big\{\|f\|_{W^{1,p}(\X)}^p\;\Big|\;f\in W^{1,p}(\X),\text{ }f\geq 1\text{ \(\mm\)-a.e.\ on open set $U \supset E$}\Big\}.
\end{equation}
Since we will generalize several results from \cite{Deb:Gig:Pas:21}, it is convenient to show that
\begin{equation}\label{eq:equiv_Cap_bis}
\overline{\rm Cap}_p^\mm(E)={\rm Cap}_p(E)\quad\text{ for every }E\subseteq\X.
\end{equation}
First of all, notice that in \eqref{eq:othercap} the replacement of $W^{1,p}(\X)$ with $\bar{N}^{1,p}(\X)$ leaves the infimum unchanged,
hence \(\overline{\rm Cap}_p(E)\geq\overline{\rm Cap}_p^\mm(E)\), and then \({\rm Cap}_p(E)\geq\overline{\rm Cap}_p^\mm(E)\) by \eqref{eq:equiv_Cap}.
On the other hand, fix any function \(f\in \bar{N}^{1,p}(\X)\) such that \(f\geq 1\) holds \(\bar\mm\)-a.e.\ on some open set \(U\subseteq\X\) containing \(E\). This means that
\(\bar\mm(U\cap\{\tilde f\neq 1\})=0\), where we denote \(\tilde f\coloneqq f\wedge 1\in \bar{N}^{1,p}(\X)\), so that \({\rm Cap}_p(U\cap\{\tilde f\neq 1\})=0\) by Proposition \ref{prop:mm_implies_Cap}.
In particular, \(N\coloneqq E\cap\{\tilde f\neq 1\}\) satisfies \({\rm Cap}_p(N)=0\). Since \(f\) is a competitor for \({\rm Cap}_p(E\setminus N)\), we get
\(\|f\|_{W^{1,p}(\X)}^p\geq{\rm Cap}_p(E\setminus N)={\rm Cap}_p(E)\), whence it follows that \(\overline{\rm Cap}_p^\mm(E)\geq{\rm Cap}_p(E)\) by the arbitrariness of \(f\).
All in all, \eqref{eq:equiv_Cap_bis} is proved.
}\end{remark}
In Section \ref{ss:divergence_measures} we will use the following result, whose proof is inspired by \cite[Lemma 3.3]{Bre:Gig:24}:
\begin{lemma}\label{lem:Cap_with_Lip}
Let \((\X,\sfd,\mm)\) be a metric measure space and \(p\in[1,\infty)\). Let \(K\subseteq\X\) be compact. Then it holds that
\begin{equation}\label{eq:Cap_on_compact_sets}
{\rm Cap}_p(K)=\inf\bigg\{\int|f|^p+\lip_a(f)^p\,\d\mm\;\bigg|\;f\in\LIP_{bs}(\X;[0,1])\text{ such that }\,f=1\text{ on }K\bigg\}.
\end{equation}
\end{lemma}
\begin{proof}
The inequality 
\(\leq\) in \eqref{eq:Cap_on_compact_sets} follows from \(\LIP_{bs}(\X)\subseteq \bar{N}^{1,p}(\X)\) and
\(|Df|\leq\lip_a(f)^\mm\) for every \(f\in\LIP_{bs}(\X)\). To show the opposite inequality, we argue as follows. Fix any
\(\varepsilon>0\) and \(f\in \bar{N}^{1,p}(\X)\) such that \(\|f\|_{\bar{N}^{1,p}(\X)}\leq{\rm Cap}_p(K)+\varepsilon\) and
\(f=1\) on an open neighbourhood \(U\) of \(K\); here, we applied \eqref{eq:equiv_Cap}.
Take \((f_n)_n\subseteq\LIP_{bs}(\X)\) such that \(\sfd_{\rm en}(f_n^\mm,f)\to 0\) and \(\lip_a(f_n)^\mm\to|Df|\)
strongly in \(L^p(\mm)\). Now fix \(\eta\in\LIP_{bs}(\X;[0,1])\) with \(\eta=1\) on \(K\) and \({\rm spt}(\eta)\subseteq U\).
We then define
\[
g_n\coloneqq\big((\eta+(1-\eta)f_n)\vee 0\big)\wedge 1\in\LIP_{bs}(\X;[0,1])\quad\text{ for every }n\in \N.
\]
Since \(g_n\leq\1_{\X\setminus K}|f_n|+\1_K=\1_{\X\setminus K}|f_n|+\1_K|f|\) for every \(n\in\N\), we can estimate
\[
\lims_{n\to\infty}\int g_n^p\,\d\mm\leq\lims_{n\to\infty}\int_{\X\setminus K}|f_n|^p\,\d\mm
+\lims_{n\to\infty}\int_K|f|^p\,\d\mm =\int|f|^p\,\d\mm.
\]
Moreover, \eqref{eq:convex_ineq_slope} gives \(\lip_a(g_n)\leq(1-\eta)\lip_a(f_n)+\lip_a(\eta)|f_n-1|
\leq\lip_a(f_n)+\1_U|f_n-f|\). Given that \(\lip_a(f_n)^\mm\to|Df|\) in \(L^p(\mm)\) and \(\1_U|f_n-f|\to 0\) in \(\mathcal L^p(\mm)\),
we deduce that \(\lims_n\int\lip_a(g_n)^p\,\d\mm\leq\||Df|\|_{L^p(\mm)}\). All in all,
\(\lims_n\int g_n^p+\lip_a(g_n)^p\,\d\mm\leq\|f\|_{\bar{N}^{1,p}(\X)}\leq{\rm Cap}_p(K)+\varepsilon\).
Given that \(\varepsilon>0\) is arbitrary, also the inequality \(\geq\) in \eqref{eq:Cap_on_compact_sets} is proved.
Therefore, the proof is complete.
\end{proof}
Since \({\rm Cap}_p\) is submodular and boundedly finite, we can consider the associated space \(L^0({\rm Cap}_p)\) as in Section \ref{ss:Choquet_int}.
Moreover, given that \(\mm\ll{\rm Cap}_p\) by \eqref{eq:mm_ll_Cap}, we have that the projection operator
\begin{equation}\label{eq:Pr_Cap}
{\rm Pr}_\mm\coloneqq{\rm Pr}_{{\rm Cap}_p,\mm}\colon L^0({\rm Cap}_p)\to L^0(\mm)
\end{equation}
is well-posed, where the map \({\rm Pr}_{{\rm Cap}_p,\mm}\) is defined as in \eqref{eq:Pr_mu}.
\subsection{Quasicontinuity and quasiuniform convergence}
\begin{definition}[Quasicontinuity]\label{def:qc_function}
Let \((\X,\sfd,\mm)\) be a metric measure space and \(p\in[1,\infty)\). Then a function \(f\colon\X\to\R\) is said to be \textbf{\(p\)-quasicontinuous} if for every
\(\varepsilon>0\) there exists a set \(E\subseteq\X\) such that \({\rm Cap}_p(E)<\varepsilon\) and \(f|_{\X\setminus E}\colon\X\setminus E\to\R\) is continuous.
\end{definition}

By virtue of Remark \ref{rmk:equiv_Cap}, a given function \(f\colon\X\to\R\) is \(p\)-quasicontinuous
if and only if for every \(\varepsilon>0\) there exists an \emph{open} set \(U_\varepsilon\subseteq\X\)
such that \({\rm Cap}_p(U_\varepsilon)<\varepsilon\) and \(f|_{\X\setminus U_\varepsilon}\) is continuous.
It follows that \(N\coloneqq\bigcap_{n\in\N}U_{1/n}\) is a \(G_\delta\) set with \({\rm Cap}_p(N)=0\) (thus \(\mm(N)=0\)
by \eqref{eq:mm_ll_Cap}) and the function \(f\) is Borel when restricted to the set \(\X\setminus N\). In particular,
each \emph{\(p\)-quasicontinuous function is \(\mm\)-measurable} and \({\rm Cap}_p\)-a.e.\ equivalent to a Borel function.
\begin{definition}[The space \(\mathcal{QC}(\X)\)]\label{def:QC(X)}
Let \((\X,\sfd,\mm)\) be a metric measure space and \(p\in[1,\infty)\). Then we define the subalgebra \(\mathcal{QC}(\X)\) of \(L^0({\rm Cap}_p)\) as
\[
\mathcal{QC}(\X)\coloneqq\big\{\pi_{{\rm Cap}_p}(f)\;\big|\;f\in\mathcal L^0(\bar\mm)\text{ is \(p\)-quasicontinuous}\big\}.
\]
\end{definition}
The next result is taken from \cite[Theorem 1.6]{EB:PC:23}:
\begin{theorem}[Local Sobolev functions are quasicontinuous]\label{thm:N_are_qc}
Let \((\X,\sfd,\mm)\) be a metric measure space and \(p\in[1,\infty)\). Then every function \(f\in \bar{N}^{1,p}_{loc}(\X)\) is \(p\)-quasicontinuous.
\end{theorem}
\begin{proof}
The statement is proved in \cite[Theorem 1.6]{EB:PC:23} for \(f\in \bar{N}^{1,p}(\X)\). Let us now briefly discuss how to deduce it for an arbitrary function \(f\in \bar{N}^{1,p}_{loc}(\X)\).
Let \((\eta_n)_n\), \((U_n)_n\) be as in Lemma \ref{lem:cut-off_loc_Sob}. Fix any \(\varepsilon>0\). Since each function
\(\eta_n f\) is \(p\)-quasicontinuous, we can find a set \(E_n\subseteq\X\) such that \({\rm Cap}_p(E_n)\leq\varepsilon/2^n\)
and \((\eta_n f)|_{\X\setminus E_n}\) is continuous. Let us then define \(E\coloneqq\bigcup_{n\in\N}E_n\cap U_n\). Notice that
\({\rm Cap}_p(E)\leq\sum_{n=1}^\infty{\rm Cap}_p(E_n)\leq\varepsilon\). Moreover, we claim that \(f|_{\X\setminus E}\)
is continuous. To prove it, fix any \(x\in\X\setminus E\). Having chosen \(n\in\N\) so that \(x\in U_n\), we have that
\(f|_{\X\setminus E}=(\eta_n f)|_{\X\setminus E}\) on \(U_n\) and that \(\X\setminus E\subseteq\X\setminus E_n\), so that
\(f|_{\X\setminus E}\) is continuous at \(x\). Therefore, \(f\) is \(p\)-quasicontinuous.
\end{proof}
\subsubsection*{Quasiuniform convergence}
A concept that is strictly related to quasicontinuity is that of \emph{(local) quasiuniform convergence}:
\begin{definition}[Local quasiuniform convergence]\label{def:qu_conv}
Let \((\X,\sfd,\mm)\) be a metric measure space and \(p\in[1,\infty)\). Let \(f_n\colon\X\to\R\), with \(n\in\N\cup\{\infty\}\), be given functions.
Then we say that \(f_n\) \textbf{locally \(p\)-quasiuniformly converges} to \(f_\infty\) if for every \(B\subseteq\X\) bounded and \(\varepsilon>0\)
there exists a set \(E\subseteq\X\) with \({\rm Cap}_p(E)<\varepsilon\) such that \(f_n\to f_\infty\) uniformly on \(B\setminus E\).
\end{definition}

In order to metrize, to some extent, quasiuniform convergence we introduce the following distance on the space \(L^0({\rm Cap}_p)\) (the sets \(U_k\) are chosen as in the definition of the distance \({\sf d}_{L^0(\mu)}\) in \eqref{eq:d_L0}):
\[
\sfd_{\rm qu}(f,g)\coloneqq\inf_{E\subseteq\X}\sum_{k\in\N}\bigg(\frac{{\rm Cap}_p(E\cap U_k)}{2^k({\rm Cap}_p(U_k)\vee 1)}+
\frac{\|\pi_{{\rm Cap}_p}(\1_{U_k\setminus E})(f-g)\|_{L^\infty({\rm Cap}_p)}\wedge 1}{2^k}\bigg)
\]
for every \(f,g\in L^0({\rm Cap}_p)\).
Next, we collect many results about quasicontinuous functions and the distance \(\sfd_{{\rm qu}}\), which -- taking Remark \ref{rmk:equiv_Cap_bis} into account -- can be
proved by repeating verbatim the proofs in \cite[Section 2.4]{Deb:Gig:Pas:21}.
\begin{proposition}
Let \((\X,\sfd,\mm)\) be a metric measure space and \(p\in[1,\infty)\). The following hold:
\begin{itemize}
\item[\(\rm i)\)] Let \((f_n)_n\subseteq L^0({\rm Cap}_p)\) and \(f\in L^0({\rm Cap}_p)\) be given. If \(f_n\to f\) locally \(p\)-quasiuniformly,
then \(\sfd_{\rm qu}(f_n,f)\to 0\). Conversely, if \(\sfd_{\rm qu}(f_n,f)\to 0\), then there exists a subsequence \((n_i)_i\) such that
\(f_{n_i}\to f\) locally \(p\)-quasiuniformly.
\item[\(\rm ii)\)] \((\mathcal QC(\X),\sfd_{\rm qu})\) is a complete metric space.
\item[\(\rm iii)\)] It holds that
\[
\sfd_{L^0({\rm Cap}_p)}(f,g)\leq\sfd_{\rm qu}(f,g)\leq 2\sqrt{\sfd_{L^0({\rm Cap}_p)}(f,g)}\quad\text{ for every }f,g\in\mathcal{QC}(\X).
\]
\item[\(\rm iv)\)] \(\mathcal QC(\X)\) coincides with the closure of \(\pi_{{\rm Cap}_p}(C(\X))\) in \(\big(L^0({\rm Cap}_p),\sfd_{L^0({\rm Cap}_p)}\big)\).
\end{itemize}
\end{proposition}
\begin{proposition}[Uniqueness of the quasicontinuous representative]\label{prop:uniq_qcr}
Let \((\X,\sfd,\mm)\) be a metric measure space and \(p\in[1,\infty)\). Let \({\rm Pr}_\mm\colon L^0({\rm Cap}_p)\to L^0(\mm)\) 
be as in \eqref{eq:Pr_Cap}. Then it holds that
\[
{\rm Pr}_\mm|_{\mathcal{QC}(\X)}\colon\mathcal{QC}(\X)\to L^0(\mm)\quad\text{ is injective.}
\]
\end{proposition}
In other words, if two \(p\)-quasicontinuous functions agree \(\bar\mm\)-a.e., then they agree \({\rm Cap}_p\)-a.e..
\begin{lemma}\label{lem:cont_d_qu}
Let \((\X,\sfd,\mm)\) be a metric measure space and \(p\in[1,\infty)\). Then it holds that
\begin{equation}\label{eq:ineq_d_qu}
\sfd_{\rm qu}(f,g)\leq c_p\|\bar f-\bar g\|_{\bar{N}^{1,p}(\X)}^{p/(p+1)}\quad\text{ for every }\bar f,\bar g\in \bar{N}^{1,p}(\X),
\end{equation}
where \(f\coloneqq\pi_{{\rm Cap}_p}(\bar f)\), \(g\coloneqq\pi_{{\rm Cap}_p}(\bar g)\), and the constant \(c_p>0\) is given by \(c_p\coloneqq p^{1/(p+1)}+p^{-p/(p+1)}\).
\end{lemma}
\begin{proof}
Define \(E_\lambda\coloneqq\{|\bar f-\bar g|\geq\lambda\}\) for every \(\lambda>0\). Since \(\lambda^{-1}|\bar f-\bar g|\in \bar{N}^{1,p}(\X)\) and \(\lambda^{-1}|\bar f-\bar g|\geq 1\) on \(E_\lambda\),
we have that \({\rm Cap}_p(E_\lambda)\leq\lambda^{-p}\|\bar f-\bar g\|_{\bar{N}^{1,p}(\X)}^p\). Fix any \(a>\|\bar f-\bar g\|_{\bar{N}^{1,p}(\X)}\). Then
\begin{equation}\label{eq:ineq_d_qu_aux}
\sfd_{\rm qu}(f,g)\leq\lambda+\frac{\|\bar f-\bar g\|_{\bar{N}^{1,p}(\X)}^p}{\lambda^p}<\lambda+\frac{a^p}{\lambda^p}.
\end{equation}
The minimum of the function \((0,\infty)\ni\lambda\mapsto\lambda+\lambda^{-p}a^p\) is attained at \(\lambda_a\coloneqq(pa^p)^{p-1}\).
Plugging \(\lambda_a\) into \eqref{eq:ineq_d_qu_aux}, we get \(\sfd_{\rm qu}(f,g)\leq c_p a^{p/(p+1)}\). Letting \(a\searrow\|\bar f-\bar g\|_{\bar{N}^{1,p}(\X)}\), we thus obtain \eqref{eq:ineq_d_qu}.
\end{proof}
\begin{theorem}[Quasicontinuous-representative map]\label{thm:qc_repr}
Let \((\X,\sfd,\mm)\) be a metric measure space and \(p\in[1,\infty)\). Then there exists a unique map \({\sf qcr}\colon W^{1,p}_{loc}(\X)\to\mathcal{QC}(\X)\), which we call the
\textbf{\(p\)-quasicontinuous-representative map}, such that the diagram
\[\begin{tikzcd}
W^{1,p}_{loc}(\X) \arrow[r,hook] \arrow[d,swap,"{\sf qcr}"] & L^0(\mm) \\
\mathcal{QC}(\X) \arrow[ur,swap,"{\rm Pr}_\mm"] &
\end{tikzcd}\]
commutes, where \(W^{1,p}_{loc}(\X)\hookrightarrow L^0(\mm)\) denotes the inclusion map. The map \({\sf qcr}\colon W^{1,p}_{loc}(\X)\to\mathcal{QC}(\X)\) is a linear operator.
Moreover, \({\sf qcr}|_{W^{1,p}(\X)}\colon(W^{1,p}(\X),\|\cdot\|_{W^{1,p}(\X)})\to(\mathcal{QC}(\X),\sfd_{\rm qu})\) is continuous.
\end{theorem}
\begin{proof}
Given any \(f\in W^{1,p}_{loc}(\X)\), pick \(\bar f\in \bar{N}^{1,p}_{loc}(\X)\) so that \(\pi_\mm(\bar f)=f\) and define \({\sf qcr}(f)\coloneqq\pi_{{\rm Cap}_p}(\bar f)\).
Proposition \ref{prop:mm_implies_Cap} guarantees that the definition of \({\sf qcr}(f)\) is well-posed, in the sense that it does not depend on the specific choice
of \(\bar f\), while Theorem \ref{thm:N_are_qc} implies that \({\sf qcr}(f)\in\mathcal{QC}(\X)\). Notice also that \({\rm Pr}_\mm\circ{\sf qcr}\) coincides with the inclusion
map \(W^{1,p}_{loc}(\X)\hookrightarrow L^0(\mm)\), and that the uniqueness of \({\sf qcr}\) follows from Proposition \ref{prop:uniq_qcr}. The linearity of
\({\sf qcr}\) is clear by construction. Finally, Lemma \ref{lem:cont_d_qu} ensures that \({\sf qcr}|_{W^{1,p}(\X)}\) is a continuous map between
\((W^{1,p}(\X),\|\cdot\|_{W^{1,p}(\X)})\) and \((\mathcal{QC}(\X),\sfd_{\rm qu})\).
\end{proof}

The quasicontinuous-representative map \(\sf qcr\) satisfies also the following continuity property:
\begin{proposition}[Pointwise continuity in energy of \(\sf qcr\)]\label{prop:good_repres_Sob}
Let \((\X,\sfd,\mm)\) be a metric measure space and \(p\in[1,\infty)\). Assume that \((f_n)_{n\in\N\cup\{\infty\}}\subseteq W^{1,p}(\X)\)
and \((g_n)_{n\in\N\cup\{\infty\}}\subseteq L^p(\mm)^+\) are given sequences with \(\lim_n\|f_n-f_\infty\|_{L^p(\mm)}=\lim_n\|g_n-g_\infty\|_{L^p(\mm)}=0\)
and \(|Df_n|\leq g_n\) for every \(n\in\N\). Then we can extract a subsequence \((n_k)_k\) such that
\[
{\sf qcr}(f_\infty)(x)=\lim_{k\to\infty}{\sf qcr}(f_{n_k})(x)\quad\text{ for }{\rm Cap}_p\text{-a.e.\ }x\in\X.
\]
\end{proposition}
\begin{proof}
The claim follows from \cite[Proposition 6.13]{Amb:Iko:Luc:Pas:24}.
\end{proof}
We also have the following local version of the above result:
\begin{corollary}\label{cor:good_repres_Sob_loc}
Let \((\X,\sfd,\mm)\) be a metric measure space and \(p\in[1,\infty)\). Let \((f_n)_n\subseteq W^{1,p}_{loc}(\X)\)
and \((g_n)_n\subseteq L^p_{loc}(\mm)^+\) be given sequences such that:
\begin{itemize}
\item[\(\rm i)\)] \(|Df_n|\leq g_n\) for every \(n\in\N\).
\item[\(\rm ii)\)] For any \(x\in\X\), there is a function \(\eta_x\in\LIP_{bs}(\X)\) such that \(\eta_x=1\) on a neighbourhood of \(x\) and
\((\eta_x^\mm f_n)_{n\in\N}\subseteq W^{1,p}(\X)\).
\item[\(\rm iii)\)] There exist \(f_\infty,g_\infty\in L^p_{loc}(\mm)\) such that \(f_n\to f_\infty\) and \(g_n\to g_\infty\) in the \(L^p_{loc}(\mm)\) sense.
\end{itemize}
Then \(f_\infty\in W^{1,p}_{loc}(\X)\) and \(|Df_\infty|\leq g_\infty\). Moreover, we can extract a subsequence \((n_k)_k\) such that
\begin{equation}\label{eq:good_repres_Sob_loc_claim}
{\sf qcr}(f_\infty)(x)=\lim_{k\to\infty}{\sf qcr}(f_{n_k})(x)\quad\text{ for }{\rm Cap}_p\text{-a.e.\ }x\in\X.
\end{equation}
\end{corollary}
\begin{proof}
By ii), iii), and the separability of \((\X,\sfd)\), we can find a sequence \((\eta_j)_j\subseteq\LIP_{bs}(\X;[0,1])\) satisfying
\(\lim_n\|\1_{B_j}^\mm f_n-\1_{B_j}^\mm f_\infty\|_{L^p(\mm)}=\lim_n\|\1_{B_j}^\mm g_n-\1_{B_j}^\mm g_\infty\|_{L^p(\mm)}=0\) for every \(j\in\N\),
\((\eta_j^\mm f_n)_{n,j\in\N}\subseteq W^{1,p}(\X)\), and \(\X=\bigcup_{j\in\N}U_j\), where we set \(B_j\coloneqq\{\eta_j>0\}\), while \(U_j\) denotes
the interior of \(\{\eta_j=1\}\). In particular, we have \(\lim_n\|\eta_j^\mm f_n-\eta_j^\mm f_\infty\|_{L^p(\mm)}=0\) for every \(j\in\N\).
Now fix \(j\in\N\). We can estimate
\[
|D(\eta_j^\mm f_n)|\leq\eta_j^\mm|Df_n|+|f_n||D\eta_j^\mm|\leq\1_{B_j}^\mm g_n+\Lip(\eta_j)\1_{B_j\setminus U_j}^\mm|f_n|\quad\text{ for every }n\in\N.
\]
Given that \(\1_{B_j}^\mm g_n\to\1_{B_j}^\mm g_\infty\) and \(\1_{B_j}^\mm|f_n|\to\1_{B_j}^\mm|f_\infty|\) strongly in \(L^p(\mm)\) as \(n\to\infty\),
we obtain that (up to a non-relabelled subsequence) \(|D(\eta_j^\mm f_n)|\rightharpoonup G_j\) weakly in \(L^p(\mm)\) as \(n\to\infty\), for some
\(G_j\in L^p(\mm)\) satisfying \(G_j\leq\1_{B_j}^\mm g_\infty+\Lip(\eta_j)\1_{B_j\setminus U_j}^\mm|f_\infty|\). Recalling Proposition \ref{prop:clos_wug},
we deduce that \(\eta_j^\mm f_\infty\in W^{1,p}(\X)\) and \(|D(\eta_j^\mm f_\infty)|\leq G_j\), thus in particular \(f_\infty\in W^{1,p}_{loc}(\X)\).
Moreover, we have that \(\1_{U_j}^\mm|D f_\infty|\leq\1_{U_j}^\mm G_j\leq\1_{U_j}^\mm g_\infty\) for every \(j\in\N\), whence it follows that \(|Df_\infty|\leq g_\infty\).
Finally, by Proposition \ref{prop:good_repres_Sob} and a diagonalisation argument, we can extract a subsequence \((n_k)_k\) such that
\[
\lim_{k\to\infty}{\sf qcr}(\eta_j^\mm f_{n_k})(x)={\sf qcr}(\eta_j^\mm f_\infty)(x)\quad\text{ for every }j\in\N
\text{ and }{\rm Cap}_p\text{-a.e.\ }x\in U_j.
\]
Since \({\sf qcr}(\eta_j^\mm f_n)={\sf qcr}(f_n)\) holds \({\rm Cap}_p\)-a.e.\ on \(U_j\) by Proposition
\ref{prop:mm_implies_Cap}, we obtain \eqref{eq:good_repres_Sob_loc_claim}.
\end{proof}
\section{Derivations and tangent modules}
\subsection{Tangent and cotangent modules}\label{ss:tg_cotg_mod}
The concept of \emph{cotangent module}, which was introduced by Gigli in \cite[Definition 2.2.1]{Gig:18} (see also \cite[Theorem/Definition 2.8]{Gig:17}),
provides a notion of `space of measurable \(1\)-forms' over an arbitrary metric measure space. Let us now give a seemingly different (but equivalent) definition
of cotangent module, whose consistency with Gigli's one will be discussed in Remark \ref{rmk:consist_cotg_mod}.
\begin{definition}[Cotangent module]\label{def:cotg_mod}
Let \((\X,\sfd,\mm)\) be a metric measure space and \(p\in[1,\infty)\). Let us denote by \(\psi_p\colon D^{1,p}_{\mm}(\X)\to L^p(\mm)^+\) the 
even, pointwise sublinear map given by
\[
\psi_p(f)\coloneqq|Df|\quad\text{ for every }f\in D^{1,p}_{\mm}(\X).
\]
Then we define the \textbf{\(p\)-cotangent module} \(L^p(T^*\X)\) and \textbf{differential} \(\d\colon D^{1,p}_{\mm}(\X)\to L^p(T^*\X)\) as
\[
(L^p(T^*\X),\d)\coloneqq(\mathscr M_{\langle\psi_p\rangle},T_{\langle\psi_p\rangle}),
\]
where the couple \((\mathscr M_{\langle\psi_p\rangle},T_{\langle\psi_p\rangle})\) is the one given by Theorem \ref{thm:module_gen_sublin_map}.
\end{definition}
\begin{remark}\label{remark:Dirichletspace:differential}{\rm
For the reader's convenience, we elaborate on the properties that characterise the couple \((L^p(T^*\X),\d)\). The space \(L^p(T^*\X)\) is an \(L^p(\mm)\)-Banach \(L^\infty(\mm)\)-module
and the differential \(\d\colon D^{1,p}_{\mm}(\X)\to L^p(T^*\X)\) is a linear map satisfying \(|\d f|=|Df|\) for every \(f\in D^{1,p}_{\mm}(\X)\). Moreover, we have that
\[
\big\{\d f\;\big|\,f\in D^{1,p}_{\mm}(\X)\big\}\quad\text{ generates }L^p(T^*\X).
\]It follows from these properties that there exists a unique differential $\underline{\d} \colon D^{1,p}( \X ) \to L^{p}( T^{*}\X )$ for which the diagram
\begin{equation*}\begin{tikzcd}
D^{1,p}_{\mm}(\X) \arrow[d,swap,"\tau"] \arrow[r,"\d"] & L^{p}( T^{*}\X ) \\
D^{1,p}(\X) \arrow[ur,swap,"\underline{\d}"] &
\end{tikzcd}\end{equation*}
is commutative, where the map $\tau\colon D^{1,p}_{\mm}( \X ) \to D^{1,p}( \X )$ is the canonical one (cf.\ \Cref{remark:aboutDirichletspaces}). The key observation is that given $f_1, f_2 \in D^{1,p}_{\mm}(\X)$, then $\d f_2 = \d f_1$ if and only if $| D( f_2 - f_1 ) | = 0$ in $L^{p}( \mm )$. That is, if and only if $f_1, f_2 \in D^{1,p}_{\mm}(\X)$ define the same equivalence class in $D^{1,p}( \X )$.
}\end{remark}

In the following two results, we collect the main properties of the differential:
\begin{proposition}[Closure properties of the differential]\label{prop:clos_diff}
Let \((\X,\sfd,\mm)\) be a metric measure space and \(p\in[1,\infty)\). Let \((f_n)_n\subseteq W^{1,p}(\X)\) and \(f\in L^p(\mm)\) be such that
\(f_n\rightharpoonup f\) weakly in \(L^p(\mm)\) and \(\d f_n\rightharpoonup\omega\) weakly in \(L^p(T^*\X)\), for some \(\omega\in L^p(T^*\X)\).
Then it holds that \(f\in W^{1,p}(\X)\) and \(\omega=\d f\).
\end{proposition}
\begin{proof}
By Mazur's lemma, we may consider \(\tilde f_n\in{\rm conv}(\{f_i\,:\,i\geq n\})\subseteq W^{1,p}(\X)\)
so that \(\tilde f_n\to f\) strongly in \(L^p(\mm)\) and \(\d\tilde f_n\to\omega\) strongly in \(L^p(T^*\X)\). Then
\(f\in W^{1,p}(\X)\) by Proposition \ref{prop:clos_wug}. Moreover, for any given \(k\in\N\) we have that
\(\tilde f_n-\tilde f_k\to f-\tilde f_k\) strongly in \(L^p(\mm)\) and \(\d(\tilde f_n-\tilde f_k)\to\omega-\d\tilde f_k\)
strongly in \(L^p(T^*\X)\) as \(n\to\infty\). Again, by Proposition \ref{prop:clos_wug}, we deduce that
\(|\d f-\d\tilde f_k|\leq|\omega-\d\tilde f_k|\to 0\) strongly in \(L^p(\mm)\) as \(k\to\infty\), whence
it follows that \(\omega=\d f\).
\end{proof}
\begin{theorem}[Calculus rules for the differential]\label{thm:calculus_rules_diff}
Let \((\X,\sfd,\mm)\) be a metric measure space and \(p\in[1,\infty)\). Then the following properties hold:
\begin{itemize}
\item[\(\rm i)\)] {\sc Locality}. Let \(f,g\in D^{1,p}_{\mm}(\X)\) be given. Then
\begin{equation}\label{eq:diff_local_0}
\1_{\{f=g\}}^\mm\d f=\1_{\{f=g\}}^\mm\d g.
\end{equation}
\item[\(\rm ii)\)] {\sc Chain rule.} Let \(f\in D^{1,p}_{\mm}(\X)\) be given. Let \(N\subseteq\R\) be a Borel set with \(\mathscr L^1(N)=0\). Then
\begin{equation}\label{eq:diff_local_1}
\1_{f^{-1}(N)}^\mm\d f=0.
\end{equation}
Moreover, given any function \(\varphi\in\LIP(\R)\), it holds that \(\varphi\circ f\in D^{1,p}_{\mm}(\X)\) and
\begin{equation}\label{eq:diff_local_2}
\d(\varphi\circ f)= (\varphi'\circ f)\,\d f,
\end{equation}
where \(\varphi'\) is defined arbitrarily at the non-differentiability points of \(\varphi\).
\item[\(\rm iii)\)] {\sc Leibniz rule.} Let \(f,g\in D^{1,p}_{\mm}(\X)\cap L^\infty(\mm)\) be given. Then \(fg\in D^{1,p}_{\mm}(\X)\cap L^\infty(\mm)\) and
\begin{equation}\label{eq:diff_Leibniz}
\d(fg)=f\,\d g+g\,\d f.
\end{equation}
\item[\(\rm iv)\)]
{\sc Lattice property.} Let \(f,g\in D^{1,p}_{\mm}(\X)\) be given. Then \(f\vee g,f\wedge g\in D^{1,p}_{\mm}(\X)\) and
\begin{equation}\label{eq:diff_lattice}
\d(f\vee g)=\1_{\{f\geq g\}}^\mm \d f+\1_{\{f<g\}}^\mm\d g,\qquad\d(f\wedge g)=\1_{\{f\leq g\}}^\mm\d f+\1_{\{f>g\}}^\mm\d g.
\end{equation}
In particular, it holds that \(|f|\in D^{1,p}_{\mm}(\X)\) and \(\d|f|={\rm sgn}(f)\d f\).
\end{itemize}
\end{theorem}
\begin{proof}
\ \\
\(\bf i)\) Notice that \(|\1_{\{f=g\}}^\mm\d f-\1_{\{f=g\}}^\mm\d g|=\1_{\{f=g\}}^\mm|D(f-g)|=0\) by Theorem
\ref{thm:calculus_rules_mwug} i), proving \eqref{eq:diff_local_0}.\\
\(\bf ii)\) First, \(|\1_{f^{-1}(N)}^\mm\d f|=\1_{f^{-1}(N)}^\mm|Df|=0\) by Theorem \ref{thm:calculus_rules_mwug} ii),
proving \eqref{eq:diff_local_1}. Next, recall that \(\varphi\) is \(\mathscr L^1\)-a.e.\ differentiable by Rademacher's
theorem, thus the right-hand side of \eqref{eq:diff_local_2} is independent of how we define \(\varphi'\)
at the non-differentiability points of \(\varphi\), thanks to \eqref{eq:diff_local_1}. Given any \(n\in\N\) and letting \(t^n_k\coloneqq\frac{k}{2^n}\)
for every \(k\in\mathbb Z\), we define the piecewise affine interpolation \(\varphi_n\colon\R\to\R\) of \(\varphi\) as
\[
\varphi_n(t)\coloneqq\varphi(t^n_k)+\big(\varphi(t^n_{k+1})-\varphi(t^n_k)\big)(2^n t-k)\quad\text{ for every }k\in\mathbb Z\text{ and }t\in[t^n_k,t^n_{k+1}).
\]
Notice that \(\Lip(\varphi_n)\leq\Lip(\varphi)\) for every \(n\in\N\). Letting \(I^n_k\coloneqq[t^n_k,t^n_{k+1})\) for every \(k\in\mathbb Z\), we have
\[
\varphi'_n=\sum_{k\in\mathbb Z}2^n\1_{I^n_k}(\varphi(t^n_{k+1})-\varphi(t^n_k))=\sum_{k\in\mathbb Z}\1_{I^n_k}\fint_{I^n_k}\varphi'\,\d\mathscr L^1
\]
holds on \(\R\setminus\{t^n_k\,:\,k\in\mathbb Z\}\) (thus, \(\mathscr L^1\)-a.e.\ on \(\R\)). Thanks to Doob's martingale convergence theorem (cf.\ with the proof
of \cite[Theorem A]{Bre:Gig:24b}), we deduce that \(\varphi'_n\to\varphi'\) strongly in \(L^1_{loc}(\R)\). Up to a non-relabelled subsequence, we thus have that
\(\varphi'_n(t)\to\varphi'(t)\) for \(\mathscr L^1\)-a.e.\ \(t\in\R\). Now observe that \(\varphi\circ f\in D^{1,p}_{\mm}(\X)\) by Theorem \ref{thm:calculus_rules_mwug} ii).
It only remains to check that \eqref{eq:diff_local_2} is verified. Notice that \(\varphi_n\circ f\in D^{1,p}_{\mm}(\X)\) for every \(n\in\N\) by Theorem \ref{thm:calculus_rules_mwug} ii).
Moreover, we have that
\[\begin{split}
\1_{f^{-1}(I^n_k)}^\mm\d(\varphi_n\circ f)&=\1_{f^{-1}(I^n_k)}^\mm
\d\big(\varphi(t^n_k)+(\varphi(t^n_{k+1})-\varphi(t^n_k))(2^n f-k)\big)\\
&=(\varphi(t^n_{k+1})-\varphi(t^n_k))2^n\1_{f^{-1}(I^n_k)}^\mm\d f=\1_{f^{-1}(I^n_k)}^\mm((\varphi'_n\circ f)\,\d f)
\end{split}\]
for every \(k\in\mathbb Z\) by \eqref{eq:diff_local_0} and \eqref{eq:diff_local_1}, whence it follows that \(\d(\varphi_n\circ f)=(\varphi'_n\circ f)\,\d f\) for
every \(n\in\N\). Using \eqref{eq:diff_local_1}, Theorem \ref{thm:calculus_rules_mwug} ii), and the fact that \(\varphi'_n\to\varphi'\) holds \(\mathscr L^1\)-a.e.\ on \(\R\), we get that
\[
|\d(\varphi_n\circ f)-\d(\varphi\circ f)|=|\varphi'_n-\varphi'|\circ f\,|Df|\to 0\quad\mm\text{-a.e.\ on }\X\text{ as }n\to\infty.
\]
Since \(|\d(\varphi_n\circ f)-\d(\varphi\circ f)|\leq 2\,\Lip(\varphi)|Df|\in L^p(\mm)^+\) for every \(n\in\N\), by the dominated convergence theorem
we obtain that \(\d(\varphi_n\circ f)\to\d(\varphi\circ f)\) in \(L^p(T^*\X)\). Similarly, \(|\varphi'_n\circ f\,\d f-\varphi'\circ f\,\d f|\to 0\)
pointwise \(\mm\)-a.e.\ and \(|\varphi'_n\circ f\,\d f-\varphi'\circ f\,\d f|\leq 2\,\Lip(\varphi)|Df|\) for all \(n\in\N\), thus
\((\varphi'_n\circ f)\,\d f\to(\varphi'\circ f)\,\d f\) in \(L^p(T^*\X)\). It follows that \(\d(\varphi\circ f)=(\varphi'\circ f)\,\d f\), which proves the validity of
\eqref{eq:diff_local_2}.\\
\(\bf iii)\) We know from Theorem \ref{thm:calculus_rules_mwug} iii) that \(fg\in D^{1,p}_{\mm}(\X)\cap L^\infty(\mm)\). Now, fix a constant \(c\in\R\) so that
\(\tilde f\coloneqq f+c\geq 1\) and \(\tilde g\coloneqq g+c\geq 1\). Given that \(\log^+\in\LIP(\R)\), we deduce from \eqref{eq:diff_local_2} that
\[\begin{split}
\d(fg)&=\d(\tilde f\tilde g-c(f+g)-c^2)=\tilde f\tilde g\frac{\d(\tilde f\tilde g)}{\tilde f\tilde g}-c\,\d(f+g)
=\tilde f\tilde g\,\d(\log^+\circ(\tilde f\tilde g))-c\,\d(f+g)\\
&=\tilde f\tilde g\,\d(\log^+\circ\tilde f)+\tilde f\tilde g\,\d(\log^+\circ\tilde g)-c\,\d(f+g)=\tilde g\,\d\tilde f+\tilde f\,\d\tilde g-c\,\d(f+g)\\
&=(g+c)\d f+(f+c)\d g-c\,\d(f+g)=g\,\d f+f\,\d g,
\end{split}\]
which proves the validity of \eqref{eq:diff_Leibniz}.\\
\(\bf iv)\) We know from Theorem \ref{thm:calculus_rules_mwug} iv) that \(f\vee g,f\wedge g,|f|\in D^{1,p}_{\mm}(\X)\). By using the locality
property \eqref{eq:diff_local_0}, one then easily obtains \eqref{eq:diff_lattice} and the identity \(\d|f|={\rm sgn}(f)\d f\). The proof is complete.
\end{proof}
\begin{remark}\label{rmk:consist_cotg_mod}{\rm
We claim that
\begin{equation}\label{eq:d_Sob_gen}
\big\{\d f\;\big|\,f\in W^{1,p}(\X)\big\}\quad\text{ generates }L^p(T^*\X).
\end{equation}
To prove it, fix \(f\in D^{1,p}_{\mm}(\X)\), and notice that Lemma \ref{lem:trunc_cut-off} yields the existence of a sequence of functions \((f_k)_k\subseteq W^{1,p}(\X)\)
and of a Borel partition \((E_k)_k\) of \(\X\) such that \(f=f_k\) holds \(\mm\)-a.e.\ on \(E_k\) for all \(k\in\N\). By the locality of the differential, i.e.\ Theorem \ref{thm:calculus_rules_diff} i),
we get \(\d f=\sum_{k\in\N}\1_{E_k}^\mm\d f_k\), which gives \eqref{eq:d_Sob_gen}.
Consequently, our definition of cotangent module is equivalent (for \(p=2\)) to the original one by Gigli \cite{Gig:18,Gig:17} and it is consistent with the ones for \(p\neq 2\)
defined in \cite{Gig:Pas:19}.
}\end{remark}
\begin{lemma}\label{lem:gen_cotg_mod_crit}
Let \((\X,\sfd,\mm)\) be a metric measure space and \(p\in[1,\infty)\). Let \(S\) be a strongly dense subset of \(W^{1,p}(\X)\).
Then it holds that \(\{\d f\,:\,f\in S\}\) generates \(L^p(T^*\X)\).
\end{lemma}
\begin{proof}
Given any \(f\in W^{1,p}(\X)\), consider a sequence \((f_n)_n\subseteq S\) such that \(\|f_n-f\|_{W^{1,p}(\X)}\to 0\), whence it follows that
\(\d f_n\to\d f\) in \(L^p(T^*\X)\). This shows that \(\d f\) belongs to the closure of \(\{\d g\,:\,g\in S\}\) in \(L^p(T^*\X)\). Since
\(\{\d f\,:\,f\in W^{1,p}(\X)\}\) generates \(L^p(T^*\X)\) by Remark \ref{rmk:consist_cotg_mod}, we conclude that the set \(\{\d g\,:\,g\in S\}\)
generates \(L^p(T^*\X)\), as desired.
\end{proof}
\begin{corollary}
Let \((\X,\sfd,\mm)\) be a metric measure space and \(p\in[1,\infty)\). Then \(L^p(T^*\X)\) is separable if and only if \(W^{1,p}(\X)\) is separable.
\end{corollary}
\begin{proof}
First, assume that \(L^p(T^*\X)\) is separable. The map \(\phi\colon W^{1,p}(\X)\to L^p(\mm)\times L^p(T^*\X)\), given by \(\phi(f)\coloneqq(f,\d f)\) for
every \(f\in L^p(\mm)\), is an isometry when its codomain is endowed with the norm \((f,\omega)\mapsto\|(f,\omega)\|\coloneqq\big(\|f\|_{L^p(\mm)}^p+\|\omega\|_{L^p(T^*\X)}^p\big)^{1/p}\).
Since \(\big(L^p(\mm)\times L^p(T^*\X),\|\cdot\|\big)\) is separable (being the \(p\)-product of separable Banach spaces), we deduce that \(W^{1,p}(\X)\) is separable as well.

Conversely, assume that \(W^{1,p}(\X)\) is separable. Let \(S\) be a countable dense subset of \(W^{1,p}(\X)\). Hence, \(\{\d f\,:\,f\in S\}\) generates \(L^p(T^*\X)\)
by Lemma \ref{lem:gen_cotg_mod_crit}. Since \((\X,\sfd)\) is complete and separable, we can find a countable family \(\mathcal C\) of Borel subsets of \(\X\) such that
\(\inf_{E'\in\mathcal C}\mm(E\Delta E')=0\) holds for every \(E\subseteq\X\) Borel for the symmetric difference $E\Delta E' = ( E \setminus E' ) \cup (E' \setminus E)$. By taking also Remark \ref{rmk:comments_Ban_mod} iv) into account (here, the assumption
\(p<\infty\) plays a role), it can then be readily checked that \(\mathcal F\) is dense in \(L^p(T^*\X)\), where we set
\[
\mathcal F\coloneqq\bigg\{\sum_{i=1}^n q_i\1_{E_i}^\mm\d f_i\;\bigg|\;n\in\N,\,(q_i)_{i=1}^n\subseteq\mathbb Q,\,(E_i)_{i=1}^n\subseteq\mathcal C,\,(f_i)_{i=1}^n\subseteq S\bigg\}.
\]
Since \(\mathcal F\) is a countable set, we conclude that \(L^p(T^*\X)\) is separable. The proof is complete.
\end{proof}

In analogy with \cite[Definition 2.3.1]{Gig:18}, we then define the \emph{tangent module} $L^q(T\X)$ as the dual of \(L^p(T^*\X)\):
\begin{definition}[Tangent module]\label{def:tg_mod}
Let \((\X,\sfd,\mm)\) be a metric measure space and \(q\in(1,\infty]\). Then the \textbf{\(q\)-tangent module} \(L^q(T\X)\)
of \(\X\) is the dual of the cotangent module \(L^p(T^*\X)\), i.e.
\[
L^q(T\X)\coloneqq L^p(T^*\X)^*.
\]
\end{definition}

Even though \(L^q(T\X)\) is the dual of \(L^p(T^*\X)\) -- not the other way round --
we will denote the duality pairing between \(v\in L^q(T\X)\) and \(\omega\in L^p(T^*\X)\) by \(\omega(v)\in L^1(\mm)\),
and not by \(v(\omega)\). The aim of this choice is to be consistent with the notation used in
Riemannian geometry (and in \cite{Gig:18}).
\subsubsection*{Vector fields with divergence}
Let us now introduce two notions of divergence for elements of the tangent module \(L^q(T\X)\). We begin with a definition
of \emph{measure-valued divergence}, along the lines of \cite[Definition 2.4]{Gig:Pas:21} or \cite[Definition 3.2]{Gig:Mon:13}.
\begin{definition}[Vector fields with divergence]\label{def:vf_with_div}
Let \((\X,\sfd,\mm)\) be a metric measure space and \(q\in(1,\infty]\). Let \(v\in L^q(T\X)\).
Then we say that \(\boldsymbol\div_q(v)\in\mathfrak M(\X)\) is the \textbf{\(q\)-divergence} of \(v\) if
\[
\int\d f(v)\,\d\mm=-\int f\,\d\boldsymbol\div_q(v)\quad\text{ for every }f\in\LIP_{bs}(\X).
\]
The  measure \(\boldsymbol\div_q(v)\) is uniquely determined by duality with the space \(\LIP_{bs}(\X)\).
We denote by \(D(\boldsymbol\div_q;\X)\) the space of all those elements
of \(L^q(T\X)\) having \(q\)-divergence.
\end{definition}

It holds that \(D(\boldsymbol\div_q;\X)\) is a vector subspace of \(L^q(T\X)\) and
that \(\boldsymbol\div_q\colon D(\boldsymbol\div_q;\X)\to\mathfrak M(\X)\) is linear.
Next, let us introduce a notion of \emph{\(L^q\)-divergence}, which extends \cite[Definition 2.3.11]{Gig:18}:
\begin{definition}[\(L^q\)-divergence of a vector field]\label{def:vf_with_Lq_div}
Let \((\X,\sfd,\mm)\) be a metric measure space and \(q\in(1,\infty]\). Let \(v\in L^q(T\X)\). Then we say that \(\div_q(v)\in L^q(\mm)\) is the \textbf{\(L^q\)-divergence} of \(v\) if
\[
\int\d f(v)\,\d\mm=-\int f\,\div_q(v)\,\d\mm\quad\text{ for every }f\in W^{1,p}(\X).
\]
The function \(\div_q(v)\) is uniquely determined by the strong density of \(W^{1,p}(\X)\) in \(L^p(\mm)\). We denote by \(D(\div_q;\X)\) the space of all those elements
of \(L^q(T\X)\) having \(L^q\)-divergence.
\end{definition}

It holds that \(D(\div_q;\X)\) is a vector subspace of \(L^q(T\X)\) and that \(\div_q\colon D(\div_q;\X)\to L^q(\mm)\) is linear.
It is worth highlighting that \(D(\boldsymbol\div_q;\X)\) is defined in duality with \(\LIP_{bs}(\X)\), while \(D(\div_q;\X)\) in duality with \(W^{1,p}(\X)\).
As \(\LIP_{bs}(\X) \subseteq W^{1,p}(\X) \), it is easy to check that
\begin{equation}\label{eq:incl_between_D_div_q}
D(\div_q;\X)\subseteq\bigg\{v\in D(\boldsymbol\div_q;\X)\;\bigg|\;\boldsymbol\div_q(v)\ll\mm,\,\frac{\d\boldsymbol\div_q(v)}{\d\mm}\in L^q(\mm)\bigg\}
\end{equation}
and that \(\boldsymbol\div_q(v)=\div_q(v)\mm\) for every \(v\in D(\div_q;\X)\) which shows the consistency between \(\boldsymbol\div_q\) and \(\div_q\).
The inclusion \eqref{eq:incl_between_D_div_q} turns out be an equality, as a consequence of Lemma \ref{lem:Lip_vs_Sob_div} and Lemma
\ref{lem:relation_div_q} that we will prove below.
\subsection{Lipschitz and Sobolev derivations}\label{ss:Lip_and_Sob_der}
In Section \ref{ss:tg_cotg_mod}, we have seen that a notion of (measurable) vector field over a metric measure space can be given in terms of the
tangent module \(L^q(T\X)\). Another kind of approach, which we will present in this section, is based on the concept of \emph{derivation}.
\subsubsection*{Lipschitz derivations}
Let us begin with the definition of (Lipschitz) derivation, which is due to Di Marino \cite{DiMar:14,DiMaPhD:14}:
\begin{definition}[Lipschitz derivation]\label{def:Lip_der}
Let \((\X,\sfd,\mm)\) be a metric measure space and \(q\in[1,\infty]\). Then we say that a linear map
\(b\colon\LIP_{bs}(\X)\to L^q(\mm)\) is a \textbf{Lipschitz \(q\)-derivation} on \(\X\) provided:
\begin{itemize}
\item[\(\rm i)\)] {\sc Weak locality.} There exists a function \(g\in L^q(\mm)^+\) such that
\begin{equation}\label{eq:def_Lip_der_ineq}
|b(f)|\leq g\,\lip_a(f)\quad\mm\text{-a.e.\ on }\X,\text{ for every }f\in\LIP_{bs}(\X).
\end{equation}
\item[\(\rm ii)\)] {\sc Leibniz rule.} It holds that \(b(fg)=f\,b(g)+g\,b(f)\) for every \(f,g\in\LIP_{bs}(\X)\).
\end{itemize}
We denote by \({\rm Der}^q(\X)\), or by \({\rm Der}^q(\X,\mm)\) when the measure needs to be emphasized, the space of all Lipschitz \(q\)-derivations on \(\X\).
\end{definition}

To any derivation \(b\in{\rm Der}^q(\X)\) we associate the function \(|b|\in L^q(\mm)^+\), which we define as
\begin{equation}\label{eq:pointwise_norm_Lip_der}
|b|\coloneqq\bigwedge\big\{g\in L^q(\mm)^+\;\big|\;\eqref{eq:def_Lip_der_ineq}\text{ holds}\big\}
=\bigvee_{f\in\LIP_{bs}(\X)}\1_{\{\lip_a(f)>0\}}^\mm\frac{|b(f)|}{\lip_a(f)^\mm}.
\end{equation}
The space \({\rm Der}^q(\X)\) is a module over \(L^\infty(\mm)\) (thus, in particular, a vector space) if endowed with
\[\begin{split}
(b+\tilde b)(f)&\coloneqq b(f)+\tilde b(f)\quad\text{ for every }b,\tilde b\in{\rm Der}^q(\X)\text{ and }f\in\LIP_{bs}(\X),\\
(hb)(f)&\coloneqq h\,b(f)\quad\text{ for every }h\in L^\infty(\mm),\,b\in{\rm Der}^q(\X),\text{ and }f\in\LIP_{bs}(\X).
\end{split}\]
Moreover, one can readily check that \({\rm Der}^q(\X)\) is an \(L^q(\mm)\)-Banach \(L^\infty(\mm)\)-module if equipped
with the pointwise norm \(|\cdot|\colon{\rm Der}^q(\X)\to L^q(\mm)^+\) defined in \eqref{eq:pointwise_norm_Lip_der}.
\begin{remark}{\rm
Definition \ref{def:Lip_der} corresponds to the notion of `derivation in the sense of Di Marino' we considered
in \cite[Definition 4.4]{Amb:Iko:Luc:Pas:24}. Indeed, in \cite{Amb:Iko:Luc:Pas:24} we studied a larger class of
Lipschitz derivations, of which Di Marino derivations (and Weaver derivations) are particular cases. However,
in this paper it is sufficient to work with Di Marino derivations, which we call just `Lipschitz derivations'.
}\end{remark}

We now define Lipschitz derivations with \emph{(measure-valued) divergence} (see \cite[Definition 4.2]{Amb:Iko:Luc:Pas:24}):
\begin{definition}[Lipschitz derivations with divergence]\label{def:Lip_derivation_divergence}
Let \((\X,\sfd,\mm)\) be a metric measure space and \(q\in(1,\infty]\). Let \(b\in{\rm Der}^q(\X)\).
Then we say that \(\boldsymbol\div(b)\in\mathfrak M(\X)\) is the \textbf{\(q\)-divergence} of \(b\) if
\[
\int b(f)\,\d\mm=-\int f\,\d\boldsymbol\div(b)\quad\text{ for every }f\in\LIP_{bs}(\X).
\]
The measure \(\boldsymbol\div(b)\) is uniquely determined by duality with \(\LIP_{bs}(\X)\).
We denote by \({\rm Der}^q_{\mathfrak M}(\X)\) the space of all
elements of \({\rm Der}^q(\X)\) having \(q\)-divergence. Moreover, we define
\[\begin{split}
{\rm Der}^q_{\mathcal M}(\X)&\coloneqq\big\{b\in{\rm Der}^q_{\mathfrak M}(\X)\;\big|\;\boldsymbol\div(b)\in\mathcal M(\X)\big\},\\
{\rm Der}^q_q(\X)&\coloneqq\bigg\{b\in{\rm Der}^q_{\mathfrak M}(\X)\;\bigg|\;
\boldsymbol\div(b)\ll\mm,\,\div(b)\coloneqq\frac{\d\boldsymbol\div(b)}{\d\mm}\in L^q(\mm)\bigg\}.
\end{split}\]
\end{definition}

It holds that \({\rm Der}^q_{\mathfrak M}(\X)\),  ${\rm Der}^q_{\mathcal M}(\X)$ and \({\rm Der}^q_q(\X)\) are \(\LIP_{bs}(\X)\)-submodules
(and thus, in particular, vector subspaces) of \({\rm Der}^q(\X)\), where we identify \(\LIP_{bs}(\X)\)
with a subring of \(L^\infty(\mm)\) in the canonical way. Moreover,
\({\rm Der}^q_{\mathfrak M}(\X)\ni b\mapsto\boldsymbol\div(b)\in\mathfrak M(\X)\) and
\({\rm Der}^q_q(\X)\ni b\mapsto\div(b)\in L^q(\mm)\) are linear maps.
The divergence satisfies the following \textbf{Leibniz rule}, cf.\ with \cite[Proposition 4.3]{Amb:Iko:Luc:Pas:24}:
\begin{equation}\label{eq:Leibniz_Lip_der}
{\bf div}(\varphi b)=\varphi\,{\bf div}(b)+b(\varphi)\mm\quad\text{ for every }
\varphi\in\LIP_{bs}(\X)\text{ and }b\in{\rm Der}^q_{\mathfrak M}(\X).
\end{equation}
Following \cite{Amb:Iko:Luc:Pas:24, Amb:Iko:Luc:Pas:correction}, we then define the \emph{Lipschitz tangent module} of a metric measure space as follows:
\begin{definition}[Lipschitz tangent module]\label{def:Lip_tg_mod}
Let \((\X,\sfd,\mm)\) be a metric measure space and \(q\in(1,\infty]\). Then we define the \textbf{Lipschitz \(q\)-tangent module}
\(L^q_\Lip(T\X)\) of \(\X\) as
\[
L^q_\Lip(T\X)\coloneqq{\rm cl}_{{\rm Der}^q(\X)}\bigg(\bigg\{\sum_{n\in\N}\1_{E_n}^\mm b_n\;\bigg|\;((E_n,b_n))_n\in{\rm Adm}({\rm Der}^q(\X)),\,(b_n)_n\subseteq{\rm Der}^q_{\mathfrak M}(\X))\bigg\}\bigg).
\]
\end{definition}

We anticipate that much later in the paper, namely in Corollary \ref{cor:density_vf_with_div}, we will show that in the case \(q<\infty\)
it actually holds that \({\rm Der}^q_q(\X)\) is dense in \(L^q_\Lip(T\X)\).
\begin{remark}{\rm
The above definition is taken from \cite{Amb:Iko:Luc:Pas:correction}, and it is equivalent with the one in \cite{Amb:Iko:Luc:Pas:24} in the case \(q<\infty\).
Instead, the definition of \(L^\infty_\Lip(T\X)\) that was originally given in \cite{Amb:Iko:Luc:Pas:24}
is not appropriate, as it does not ensure the \(L^\infty(\mm)\)-module property; we refer to \cite{Amb:Iko:Luc:Pas:correction}
for further comments on this.
}\end{remark}
Since any function in \(L^\infty(\mm)\) is the pointwise \(\mm\)-a.e.\ limit of a sequence of functions
in \(\LIP_{bs}(\X)\) that are uniformly bounded in \(L^\infty(\mm)\), one can readily check that
\[
L^q_\Lip(T\X)\quad\text{ is an }L^q(\mm)\text{-Banach }L^\infty(\mm)\text{-submodule of }{\rm Der}^q(\X).
\]
Moreover, it is easy to check that, letting \(i\colon{\rm Der}^q_{\mathfrak M}(\X)\hookrightarrow{\rm Der}^q(\X)\)
be the inclusion map, one has
\[
(L^q_\Lip(T\X),i)\cong(\mathscr M_{\langle\psi\rangle},T_{\langle\psi\rangle}),
\]
with \(\psi\colon{\rm Der}^q_{\mathfrak M}(\X)\to L^q(\mm)^+\) given by \(\psi(b)\coloneqq|b|\)
and \((\mathscr M_{\langle\psi\rangle},T_{\langle\psi\rangle})\) as in Theorem \ref{thm:module_gen_sublin_map}.
\medskip

We recall the following extension result, taken from \cite[Corollary 4.7]{Amb:Iko:Luc:Pas:24} 
(taking into account \cite{Amb:Iko:Luc:Pas:correction}), which we will need later:
\begin{proposition}\label{prop:cont_Lip_der}
Let \((\X,\sfd,\mm)\) be a metric measure space and \(q\in(1,\infty]\). Let \(b\in L^q_\Lip(T\X)\) be given.
Then \(b\) can be uniquely extended to a linear operator \(b\colon\LIP(\X)\to L^q(\mm)\) such that
\[
b(fg)=f\,b(g)+g\,b(f)\quad\text{ for every }f,g\in\LIP(\X)
\]
and \(|b(f)|\leq|b|\,\lip_a(f)\) for every \(f\in\LIP(\X)\). Moreover, if \((f_n)_{n\in\N\cup\{\infty\}}\subseteq\LIP(\X)\) satisfies
\(\sup_{n\in\N}\Lip(f_n)<+\infty\) and \(f_n(x)\to f_\infty(x)\) for all \(x\in\X\), then \(b(f_n)\rightharpoonup b(f_\infty)\) weakly\(^*\) in \(L^q(\mm)\).
\end{proposition}
In order to investigate the relation between plans with barycenter and Lipschitz derivations
(see Theorem \ref{thm:plans_vs_der_vs_curr} below), we will need the following statement (which we attribute to \cite[Proposition 4.10]{Amb:Iko:Luc:Pas:24}):
\begin{proposition}[Lipschitz derivation induced by a plan]\label{prop:der_induced_by_plan}
Let \((\X,\sfd,\mm)\) be a metric measure space and \(q\in(1,\infty]\). Let \(\ppi\) be a plan on \(\X\) with barycenter \({\rm Bar}(\ppi)\)
in \(L^q(\mm)\). Then there exists a unique Lipschitz derivation \(b_\sppi\in{\rm Der}^q_{\mathcal M}(\X)\) such that
\begin{equation}\label{eq:der_induced_by_plan}
\int g\,b_\sppi(f)\,\d\mm=\int\!\!\!\int_0^1 g(\gamma^{\sf cs}(t))\frac{\d}{\d t}f(\gamma^{\sf cs}(t))\,\d t\,\d\ppi(\gamma)
\quad\text{ for all }(f,g)\in\LIP_{bs}(\X)\times\LIP_b(\X).
\end{equation}
Moreover, it holds that \(|b_\sppi|\leq{\rm Bar}(\ppi)\) and \({\bf div}(b_\sppi)=(\e_0)_\#\ppi-(\e_1)_\#\ppi\).
\end{proposition}
Derivations in the sense of Di Marino were introduced in \cite{DiMar:14,DiMaPhD:14} as a key tool for defining a notion
of metric \(p\)-Sobolev space based on an integration-by-parts formula. The resulting Sobolev space is equivalent to
the other notions we presented in Section \ref{ss:metric_Sob_spaces}. The equivalence was first proved for \(p\in(1,\infty)\)
in \cite{DiMar:14,DiMaPhD:14} and then extended also to the exponent \(p=1\) in \cite{Amb:Iko:Luc:Pas:24}. Namely:
\begin{theorem}[Sobolev space via Lipschitz derivations]\label{thm:Sob_via_der}
Let \((\X,\sfd,\mm)\) be a metric measure space and \(p\in[1,\infty)\). Let \(f\in L^p(\mm)\) be given.
Then it holds that \(f\in W^{1,p}(\X)\) if and only if there exists a $\LIP_{bs}(\X)$-linear operator \({\sf L}_f\colon{\rm Der}^q_q(\X)\to L^1(\mm)\)
satisfying the following properties:
\begin{itemize}
\item[\(\rm i)\)] There exists a function \(G\in L^p(\mm)^+\) such that
\begin{equation}\label{eq:def_Sob_via_der}
|{\sf L}_f(b)|\leq G|b|\quad\text{ for every }b\in{\rm Der}^q_q(\X).
\end{equation}
\item[\(\rm ii)\)] It holds that
\[
\int {\sf L}_f(b)\,\d\mm=-\int f\,\div(b)\,\d\mm\quad\text{ for every }b\in{\rm Der}^q_q(\X).
\]
\end{itemize}
Moreover, for any \(f\in W^{1,p}(\X)\), we have that \(|{\sf L}_f(b)|\leq|Df||b|\) for every \(b\in{\rm Der}^q_q(\X)\) and
\begin{equation}\label{eq:formula_mwug_via_der}
|Df|=\bigwedge\big\{G\in L^p(\mm)^+\;\big|\;\eqref{eq:def_Sob_via_der}\text{ holds}\big\}.
\end{equation}
\end{theorem}
\begin{proof}
The statement follows from \cite[Theorem 7.1]{Amb:Iko:Luc:Pas:24} (taking \cite[Definition 2.1]{Amb:Iko:Luc:Pas:correction} into account).
\end{proof}
Given any function \(f\in W^{1,p}(\X)\), it follows from \eqref{eq:def_Sob_via_der} that the operator \({\sf L}_f\) can be uniquely extended to an element
of \(L^q_\Lip(T\X)^*\), which we still denote by \({\sf L}_f\). Furthermore, \eqref{eq:formula_mwug_via_der} gives that
\begin{equation}\label{eq:Sob_seminorm_L_f}
\||Df|\|_{L^p(\mm)}=\|{\sf L}_f\|_{L^q_\Lip(T\X)^*}\quad\text{ for every }f\in W^{1,p}(\X).
\end{equation}
Letting \(\bar{\sf L}_f\in L^q_\Lip(T\X)'\) be defined as \(\bar{\sf L}_f(b)\coloneqq\int{\sf L}_f(b)\,\d\mm=\textsc{Int}_{L^q_\Lip(T\X)}({\sf L}_f)(b)\)
for every \(b\in L^q_\Lip(T\X)\), we deduce from Proposition \ref{prop:map_Int} and \eqref{eq:Sob_seminorm_L_f} that the map
\begin{equation}\label{eq:prop_bar_L_f}
W^{1,p}(\X)\ni f\mapsto(f,\bar{\sf L}_f)\in L^p(\mm)\times L^q_\Lip(T\X)'\quad\text{ is a linear isometry,}
\end{equation}
where we equip the target space \(L^p(\mm)\times L^q_\Lip(T\X)'\) with the \(p\)-norm
\[
\|(h,\bar{\sf L})\|_p\coloneqq\big(\|h\|_{L^p(\mm)}^p+\|\bar{\sf L}\|_{L^q_\Lip(T\X)'}^p\big)^{1/p}\quad\text{ for every }(h,\bar{\sf L})\in L^p(\mm)\times L^q_\Lip(T\X)'.
\]
Next, we introduce the concepts of \emph{acyclic derivation} and \emph{cyclic derivation}, by mimicking the
corresponding notions in the setting of currents:
\begin{definition}
Let \((\X,\sfd,\mm)\) be a metric measure space and \(q\in[1,\infty]\). Fix \(b\in{\rm Der}^q(\X)\). Then:
\begin{itemize}
\item[\(\rm i)\)] We say that \(s\in{\rm Der}^q(\X)\) is a \textbf{subderivation} of \(b\) provided it holds that \(|b-s|+|s|\leq|b|\).
We denote by \(S(b)\) the set of all subderivations of \(b\).
\item[\(\rm ii)\)] We say that \(c\in{\rm Der}^q_{\mathfrak M}(\X)\) is a \textbf{cycle} if \(\boldsymbol\div(c)=0\).
If, in addition, \(c\) is a subderivation of \(b\), then we say that \(c\) is a \textbf{subcycle} of \(b\). We denote by $C(b)$ the set of all subcycles of $b$. 
\item[\(\rm iii)\)] We say that \(b\) is \textbf{acyclic} if $C(b) = \left\{0\right\}$.
\end{itemize}
\end{definition}

Given that the inequality \(|b|\leq|b-s|+|s|\) is always satisfied, we see that
\begin{equation}\label{eq:equiv_subder}
s\in S(b)\quad\Longleftrightarrow\quad|b-s|+|s|=|b|.
\end{equation}
In analogy with currents, each derivation can be written as the sum of an acyclic subderivation and
a subcycle, as shown by the following result, which is inspired by \cite[Proposition 3.8]{Pao:Ste:12}.
\begin{proposition}[Decomposition of a derivation]\label{prop:acycl_cyl_decomp}
Let \((\X,\sfd,\mm)\) be a metric measure space and \(q\in[1,\infty]\). Let \(b\in{\rm Der}^q(\X)\) be given.
Then we can write $b=a+c$, where \(c\) is a cycle and \(a\in {\rm Der}^q(\X)\) is an acyclic subderivation of \(b\).
\end{proposition}
\begin{proof}
Fix a finite Borel measure \(\tilde\mm\) on \(\X\) such that \(\mm\ll\tilde\mm\leq\mm\);
for example, fix a Borel partition \((E_n)_n\) of \(\X\) into sets with \(\mm(E_n)<+\infty\)
and define \(\tilde\mm\coloneqq\sum_{n=1}^\infty 2^{-n}(\mm(E_n)\vee 1)^{-1}\mm|_{E_n}\). Now let us define
\[
m(b)\coloneqq\sup_{c\in C(b)}\||c|\|_{L^1(\tilde\mm)}\quad\text{ for every }b\in{\rm Der}^q(\X).
\]
Now fix any \(b\in{\rm Der}^q(\X)\). Recursively, we can construct a sequence \((c_n)_n\subseteq{\rm Der}^q(\X)\) such that
\[
c_n\in C(b_{n-1}),\quad\||c_n|\|_{L^1(\tilde\mm)}\geq\frac{m(b_{n-1})}{2}\quad\text{ for every }n\in\N,
\]
where we set \(b_0\coloneqq b\) and \(b_n\coloneqq b-\sum_{i=1}^n c_i=b_{n-1}-c_n\) for every \(n\in\N\). We claim that
\begin{equation}\label{eq:acycl_cyl_decomp_1}
\tilde c+c_n\in C(b_{n-1})\quad\text{ for every }n\in\N\text{ and }\tilde c\in C(b_n).
\end{equation}
To prove it, observe that cycles form a vector subspace of \({\rm Der}^q(\X)\), thus \(\tilde c+c_n\) is a cycle. Moreover,
\begin{equation}\label{eq:acycl_cyl_decomp_2}
|\tilde c+c_n|+|b_{n-1}-(\tilde c+c_n)|\leq|c_n|+|\tilde c|+|b_n-\tilde c|=|c_n|+|b_n|=|c_n|+|b_{n-1}-c_n|=|b_{n-1}|,
\end{equation}
which shows that \(\tilde c+c_n\) is a subderivation of \(b_{n-1}\) and thus \(\tilde c+c_n\in C(b_{n-1})\), proving \eqref{eq:acycl_cyl_decomp_1}.
Given that \(|b_{n-1}|\leq|\tilde c+c_n|+|b_{n-1}-(\tilde c+c_n)|\), we also deduce that all inequalities in \eqref{eq:acycl_cyl_decomp_2} are in
fact equalities, thus in particular we have that \(|c_n|+|\tilde c|=|\tilde c+c_n|\) for every \(n\in\N\) and \(\tilde c\in C(b_n)\),
which implies that \(\tilde c\in C(\tilde c+c_n)\) for every \(\tilde c\in C(b_n)\) and $n \in \N$. Moreover, we can estimate
\[\begin{split}
\||\tilde c|\|_{L^1(\tilde\mm)}=\||\tilde c+c_n|\|_{L^1(\tilde\mm)}-\||c_n|\|_{L^1(\tilde\mm)}
\leq m(b_{n-1})-\frac{m(b_{n-1})}{2}=\frac{m(b_{n-1})}{2}
\end{split}\]
thanks to \eqref{eq:acycl_cyl_decomp_1} and to the assumptions on \((c_n)_n\). Taking the supremum over \(\tilde c\in C(b_n)\), we get
\begin{equation}\label{eq:acycl_cyl_decomp_3}
m(b_n)\leq\frac{m(b_{n-1})}{2}\quad\text{ for every }n\in\N.
\end{equation}
The recursion also gives that $c_{n+1} + \dots + c_{m} \in C( b_n )$ for every $0 \leq n < m$. Therefore, induction shows that
\begin{equation}\label{eq:pointwisecontrol}
    \sum_{ k = n+1 }^{ m } | c_{k} |
    =
    \left| \sum_{ k = n + 1 }^{ m } c_k \right|
    \leq
    | b_n |
    \leq 
    |b|
    \quad\text{for $0 \leq n < m$.}
\end{equation}
Observe that the upper bound in \eqref{eq:pointwisecontrol} is in $L^{q}( \mm )$. Using this observation, we define now $c \colon \LIP_{bs}( \X ) \rightarrow L^{q}( \mm )$ using the series $c(f) \coloneqq \sum_{ k = 1 }^{ \infty } c_{k}(f)$ convergent in $L^{1}( \mm )$ and in $L^{1}( \widetilde{\mm} )$. Indeed, by \eqref{eq:pointwisecontrol},
\begin{equation}\label{eq:cauchysequence}
    \left| \sum_{ k = n+1 }^{ m } c_k(f) \right|
    \leq
    \sum_{ k = n+1 }^{ m } | c_{k} | \lip_a(f)
    \leq
    | b_n | \lip_a(f)
    \leq
    | b | \lip_a(f)
    \quad\text{for $0 \leq n < m \in \N$,}
\end{equation}
so the claim follows by dominated convergence. The convergence of $( \sum_{ k = 1 }^{ m }c_k(f) )_{ m = 1 }^{ \infty }$ to $c(f)$ in $L^{1}( \mm )$ and \eqref{eq:pointwisecontrol} imply that $c \in \mathrm{Der}^{q}( \X )$ with $|c| = \sum_{ k = 1 }^{ \infty } | c_k |$, the latter decomposition valid in the strong topology of $L^{1}( \widetilde{\mm} )$. Moreover, we have that
\begin{align*}
    \Big| |b-c| - |b- \sum_{k=1}^{n} c_k| \Big|
    \leq
    \Big| c - \sum_{k=1}^{n} c_k \Big|
    \leq
    \sum_{ k = n+1 }^{ \infty } |c_k|
    \leq
    |b|,
\end{align*}
so, in the strong topology in $L^{1}( \widetilde{\mm} )$,
\begin{align*}
    | b - c | + | c |
    =
    \lim_{ n \rightarrow \infty } \left|b- \sum_{k=1}^{n} c_k\right| + \left| \sum_{ k = 1 }^{ n } c_k \right|
    =
    \lim_{ n \rightarrow \infty } |b|
    =
    |b|.
\end{align*}
This implies that $c$ is a subderivation of $b$. The fact that $c$ is a subcycle of $b$ follows from \eqref{eq:cauchysequence} and dominated convergence in $L^{1}( \mm )$.

It remains to prove that $a = b-c$ is acyclic. Since \(c_n\in C(b_{n-1})\) and \(b_n=b_{n-1}-c_n\), we have that \(b_n\) is a subderivation of \(b_{n-1}\) for every \(n\in\N\).
In particular, \(|b_n-b_k|+|b_k|=|b_n|\) for every \(n,k\in\N\) with \(n<k\). Since \(|b_k|\to|a|\) and \(|b_n-b_k|\to|b_n-a|\)
strongly in \(L^1(\tilde\mm)\) as \(k\to\infty\), we deduce that \(|b_n-a|+|a|=|b_n|\), which means that \(a\) is a subderivation
of \(b_n\) for every \(n\in\N\). Hence, for any given \(\tilde c\in C(a)\), we have that \(\tilde c\in C(b_n)\) for every
\(n\in\N\). Thus \(\||\tilde c|\|_{L^1(\tilde\mm)}\leq m(b_n)\leq 2^{-n}m(b)\) holds for every \(n\in\N\)
by \eqref{eq:acycl_cyl_decomp_3}. Letting \(n\to\infty\), we conclude that \(\tilde c=0\), meaning that \(a\) is acyclic.
The proof is complete.
\end{proof}

We point out that the cyclic-acyclic decomposition in Proposition \ref{prop:acycl_cyl_decomp} can fail to be unique. For example, consider (the derivation induced by) the curve obtained by concatenating a rectifiable simple closed curve \(\sigma\colon[0,1]\to\R^2\) and a rectifiable simple curve \(\gamma\colon[0,1]\to\R^2\)
satisfying \(\gamma_0=\sigma_0\), \(\gamma_1=\sigma_{1/2}\), and \(\gamma((0,1))\cap\sigma([0,1])=\varnothing\), and consider on $\R^2$ the restriction of the $1$-dimensional Hausdorff measure
on the images of $\sigma$ and $\gamma$.
\subsubsection*{Plans, Lipschitz derivations, and normal \texorpdfstring{\(1\)}{1}-currents}
Let us now show that the dynamic plans with barycenter in \(L^q(\mm)\) can be identified with the integrable Lipschitz \(q\)-derivations having
\(q\)-divergence; this is the content of Theorem \ref{thm:plans_vs_der_vs_curr} below. One inclusion follows from Proposition \ref{prop:der_induced_by_plan}.
To prove the converse inclusion, we employ the superposition principle for normal \(1\)-currents (Theorem \ref{thm:Paolini-Stepanov}). En route,
we prove that dynamic plans with barycenter in \(L^q(\mm)\) can be identified with a suitable class of normal \(1\)-currents, which we denote
by \({\bf N}_1^q(\X;\mm)\). Namely, given a metric measure space \((\X,\sfd,\mm)\) and \(q\in(1,\infty]\), we define
\[
{\bf N}_1^q(\X;\mm)\coloneqq\bigg\{T\in{\bf N}_1(\X)\;\bigg|\;\|T\|\ll\mm\text{ and }\frac{\d\|T\|}{\d\mm}\in L^q(\mm)\bigg\}.
\]
With the notation of Definition~\ref{def:barycenter}, the following result holds.
\begin{theorem}[Dynamic plans, Lipschitz derivations, and normal \(1\)-currents]\label{thm:plans_vs_der_vs_curr}
Let \((\X,\sfd,\mm)\) be a metric measure space and \(q\in(1,\infty]\). Let us define the maps
\[
\mathcal B_q(\X)\overset{\mathfrak p}{\longrightarrow}{\rm Der}^q_{\mathcal M}(\X)\cap{\rm Der}^1(\X)\overset{\mathfrak t}{\longrightarrow}{\bf N}_1^q(\X;\mm)
\]
as \(\mathfrak p(\ppi)\coloneqq b_\sppi\) for every \(\ppi\in\mathcal B_q(\X)\) (where \(b_\sppi\) is the derivation
given by Proposition \ref{prop:der_induced_by_plan}) and
\[
\mathfrak t(b)(f,g)\coloneqq\int f\,b(g)\,\d\mm
\]
for every \(b\in{\rm Der}^q_{\mathcal M}(\X)\cap{\rm Der}^1(\X)\) and \((f,g)\in\LIP_b(\X)\times\LIP(\X)\). Then the following hold:
\begin{itemize}
\item[\(\rm i)\)] The map \(\mathfrak p\) is additive and positively \(1\)-homogeneous. Given any plan \(\ppi\in\mathcal B_q(\X)\), it holds that \(|\mathfrak p(\ppi)|\leq{\rm Bar}(\ppi)\)
and \({\bf div}(\mathfrak p(\ppi))=(\e_0)_\#\ppi-(\e_1)_\#\ppi\). Moreover, the map \(\mathfrak p\) is surjective and, more
precisely, for any \(b\in{\rm Der}^q_{\mathcal M}(\X)\cap{\rm Der}^1(\X)\) there
exists a plan \(\ppi_b\in\mathfrak p^{-1}(b)\) such that \({\rm Bar}(\ppi_b)=|b|\).
\item[\(\rm ii)\)] The operator \(\mathfrak t\) is linear and bijective. Given any \(b\in{\rm Der}^q_{\mathcal M}(\X)\cap{\rm Der}^1(\X)\),
it holds that
\begin{equation}\label{eq:def_t_Der_curr}
\|\mathfrak t(b)\|=|b|\mm,\qquad\partial(\mathfrak t(b))=-{\bf div}(b).
\end{equation}
\end{itemize}
\end{theorem}
\begin{proof}
First, we know from Proposition \ref{prop:der_induced_by_plan} that \(\mathfrak p(\ppi)=b_\sppi\in{\rm Der}^q_{\mathcal M}(\X)\)
satisfies \(|\mathfrak p(\ppi)|\leq{\rm Bar}(\ppi)\) and \({\bf div}(\mathfrak p(\ppi))=(\e_0)_\#\ppi-(\e_1)_\#\ppi\) for every
\(\ppi\in\mathcal B_q(\X)\). Moreover, it follows directly from \eqref{eq:der_induced_by_plan} that \(\mathfrak p\) is additive and
positively \(1\)-homogeneous. Since \({\rm Bar}(\ppi)\in L^1(\mm)\), we also have that \(\mathfrak p(\ppi)\in{\rm Der}^1(\X)\).

Now fix \(b\in{\rm Der}^q_{\mathcal M}(\X)\cap{\rm Der}^1(\X)\). Let us check that the functional
\(\mathfrak t(b)\colon\LIP_b(\X)\times\LIP(\X)\to\R\) defined in \eqref{eq:def_t_Der_curr} is a normal \(1\)-current.
First, \(\mathfrak t(b)\) is bilinear by construction. Furthermore:
\begin{itemize}
\item Assume that \(f\in\LIP_b(\X)\) and \((g_n)_{n\in\N\cup\{\infty\}}\subseteq\LIP(\X)\) satisfy \(L\coloneqq\sup_{n\in\N}\Lip(g_n)<+\infty\)
and \(g_n(x)\to g_\infty(x)\) for every \(x\in\X\). Then Proposition \ref{prop:cont_Lip_der} ensures that \(b(g_n)\rightharpoonup b(g_\infty)\)
weakly\(^*\) in \(L^q(\mm)\). Fix any \(\bar x\in\X\) and define \(f_k\coloneqq\big(1-\sfd(\cdot,B_k(\bar x))\big)^+ f\) for every \(k\in\N\). Then
\[
\bigg|\int f\,b(g_n)\,\d\mm-\int f\,b(g_\infty)\,\d\mm\bigg|\leq 2L\|(f_k-f)|b|\|_{L^1(\mm)}+\bigg|\int f_k\,b(g_n)\,\d\mm-\int f_k\,b(g_\infty)\,\d\mm\bigg|
\]
by H\"{o}lder's inequality. Since \(\|(f_k-f)|b|\|_{L^1(\mm)}\to 0\) by the dominated convergence theorem and \(f_k\in\LIP_{bs}(\X)\),
by first letting \(n\to\infty\) and then \(k\to\infty\) we get \(\mathfrak t(b)(f,g_n)\to\mathfrak t(b)(f,g_\infty)\).
\item If \(f\in\LIP_b(\X)\) is given and \(g\in\LIP(\X)\) is constant on a neighbourhood of \(\{f\neq 0\}\), then
\(\lip_a(g)=0\) on \(\{f\neq 0\}\), so that \(\big|\int f\,b(g)\,\d\mm\big|\leq\int|f||b|\lip_a(g)\,\d\mm=0\) and thus \(\mathfrak t(b)(f,g)=0\).
\item Since \(|\mathfrak t(b)(f,g)|\leq\Lip(g)\int|f||b|\,\d\mm\) holds for every \((f,g)\in\LIP_b(\X)\times\LIP(\X)\), we have proved that
\(\mathfrak t(b)\in{\bf M}_1(\X)\) and \(\|\mathfrak t(b)\|\leq|b|\mm\).
\item Observe that \(\partial(\mathfrak t(b))(f)=\int b(f)\,\d\mm=-\int f\,\d{\bf div}(b)\) for every \(f\in\LIP_b(\X)\), whence it follows
that \(\mathfrak t(b)\in{\bf N}_1^q(\X;\mm)\) and \(\partial(\mathfrak t(b))=-{\bf div}(b)\).
\end{itemize}
Hence, \(\mathfrak t\colon{\rm Der}^q_{\mathcal M}(\X)\cap{\rm Der}^1(\X)\to{\bf N}_1^q(\X;\mm)\) is well-defined, and \(\|\mathfrak t(b)\|\leq|b|\mm\) and
\(\partial(\mathfrak t(b))=-{\bf div}(b)\) hold for every \(b\in{\rm Der}^q_{\mathcal M}(\X)\cap{\rm Der}^1(\X)\). Notice also that \(\mathfrak t\) is linear and injective by construction.

To conclude, fix \(T\in{\bf N}_1^q(\X;\mm)\). Theorem \ref{thm:Paolini-Stepanov} yields the existence of a plan \(\ppi_T\) on \(\X\) such that
\begin{equation}\label{eq:appl_PS}\begin{split}
T(f,g)&=\int\!\!\!\int_0^1 f(\gamma^{\sf cs}(t))\frac{\d}{\d t}g(\gamma^{\sf cs}(t))\,\d t\,\d\ppi_T(\gamma),\\
\int f\,\d\|T\|&=\int\ell(\gamma)\int_0^1 f(\gamma^{\sf cs}(t))\,\d t\,\d\ppi_T(\gamma)
\end{split}\end{equation}
for every \(f\in\LIP_b(\X)\) and \(g\in\LIP(\X)\). The second identity in \eqref{eq:appl_PS} implies that \(\ppi_T\in\mathcal B_q(\X)\) and \({\rm Bar}(\ppi_T)=\frac{\d\|T\|}{\d\mm}\).
We then have that \(T=(\mathfrak t\circ\mathfrak p)(\ppi_T)\), thus in particular the map \(\mathfrak t\) is bijective. If \(b\in{\rm Der}^q_{\mathcal M}(\X)\cap{\rm Der}^1(\X)\),
then \(\ppi_{\mathfrak t(b)}\in\mathfrak p^{-1}(b)\) and \(|b|=|\mathfrak p(\ppi_{\mathfrak t(b)})|\leq{\rm Bar}(\ppi_{\mathfrak t(b)})=\frac{\d\|\mathfrak t(b)\|}{\d\mm}\leq|b|\),
which forces the identities \({\rm Bar}(\ppi_{\mathfrak t(b)})=|b|=\frac{\d\|\mathfrak t(b)\|}{\d\mm}\). All in all, the proof is complete.
\end{proof}
\begin{remark}[The case of local currents] {\rm
If we require only $L^q(\mm)$ integrability of the divergence, the operator $\mathfrak p$ maps plans with barycenter in $L^q(\mm)$ to 
${\rm Der}^q_{\mathcal M}(\X)$ and the operator $\mathfrak t$ maps the larger set ${\rm Der}^q_{\mathfrak M}(\X)$ to locally normal currents 
(according to Remark~\ref{rem_local_currents}) with a mass measure absolutely continuous with respect to $\mm$ and with an $L^{q}(\mm)$-integrable density.
}\end{remark}

\begin{remark}\label{rmk:improved_PS}{\rm
Under the assumptions of Theorem \ref{thm:plans_vs_der_vs_curr}, it holds that if \(b\in{\rm Der}^q_{\mathcal M}(\X)\cap{\rm Der}^1(\X)\)
is acyclic, then the plan \(\ppi_b\) as in Theorem \ref{thm:plans_vs_der_vs_curr} i) can be also chosen so that
\[
(\e_0)_\#\ppi_b={\bf div}(b)^+,\qquad(\e_1)_\#\ppi_b={\bf div}(b)^-.
\]
This follows from Theorem \ref{thm:Paolini-Stepanov} by taking into account the fact that for \(b\in{\rm Der}^q_{\mathcal M}(\X)\cap{\rm Der}^1(\X)\),
\[
b\text{ is an acyclic derivation}\quad\Longleftrightarrow\quad\mathfrak t(b)\text{ is an acyclic current,}
\]
which can be easily obtained as a consequence of \eqref{eq:def_t_Der_curr} and of the linearity of \(\mathfrak t\).
}\end{remark}
\subsubsection*{Sobolev derivations}
We now discuss yet another notion of derivation, which we call `Sobolev derivation' since it is defined
in duality with Sobolev functions. As such, it is not an auxiliary tool for defining a Sobolev space, but it is
rather the result of an already-established Sobolev theory. The interplay between Lipschitz and Sobolev derivations
is tighly connected to the reflexivity properties of the Sobolev space, as we will see in Section \ref{ss:reflexivity_properties}.
Let us now give a definition of Sobolev derivation that we borrow from \cite[Definition 2.3.2]{Gig:18},
and that extends the latter to all exponents \(q\in(1,\infty]\).
\begin{definition}[Sobolev derivation]\label{def:Sob_der}
Let \((\X,\sfd,\mm)\) be a metric measure space and \(q\in(1,\infty]\). Then we say that a linear functional
\(\delta\colon W^{1,p}(\X)\to L^1(\mm)\) is a \textbf{Sobolev \(q\)-derivation} on \(\X\) if there exists
a function \(g\in L^q(\mm)^+\) such that
\begin{equation}\label{eq:def_Sob_der}
|\delta(f)|\leq g|Df|\quad\text{ for every }f\in W^{1,p}(\X).
\end{equation}
To the derivation \(\delta\), we associate the function \(|\delta|\in L^q(\mm)^+\), which we define as
\begin{equation}\label{eq:ptwse_norm_LqSob}
|\delta|\coloneqq\bigwedge\Big\{g\in L^q(\mm)^+\;\Big|\;g\text{ satisfies \eqref{eq:def_Sob_der}}\Big\}
=\bigvee_{f\in W^{1,p}(\X)}\1_{\{|Df|>0\}}^\mm\frac{|\delta(f)|}{|Df|}.
\end{equation}
We denote by \(L^q_{\rm Sob}(T\X)\) the space of all Sobolev \(q\)-derivations on \(\X\).
\end{definition}

Given any \(\delta,\tilde\delta\in L^q_{\rm Sob}(T\X)\) and \(h\in L^\infty(\mm)\), we define \(\delta+\tilde\delta,h\delta\colon W^{1,p}(\X)\to L^1(\mm)\) as
\begin{equation}\label{eq:mod_struct_LqSob}
(\delta+\tilde\delta)(f)\coloneqq\delta(f)+\tilde\delta(f),\qquad(h\delta)(f)\coloneqq h\,\delta(f)
\end{equation}
for every \(f\in W^{1,p}(\X)\). It can be readily checked that \(L^q_{\rm Sob}(T\X)\) is a module over \(L^\infty(\mm)\) if endowed with the operations
defined in \eqref{eq:mod_struct_LqSob}. Moreover, \(L^q_{\rm Sob}(T\X)\) is an \(L^q(\mm)\)-Banach \(L^\infty(\mm)\)-module if equipped with the pointwise norm operator
\(\delta\mapsto|\delta|\) defined in \eqref{eq:ptwse_norm_LqSob}. Consequently, we are entitled to call \(L^q_{\rm Sob}(T\X)\) the \textbf{Sobolev \(q\)-tangent module} of \(\X\).
\medskip

The following result states that the Sobolev \(q\)-tangent module \(L^q_{\rm Sob}(T\X)\) can be canonically identified with the \(q\)-tangent module \(L^q(T\X)\).
Its proof extends \cite[Theorem 2.3.3]{Gig:18} to all \(q\neq 2\).
\begin{theorem}[Tangent module and Sobolev derivations]\label{thm:Lq(TX)=LqSob(TX)}
Let \((\X,\sfd,\mm)\) be a metric measure space and \(q\in(1,\infty]\). Then for every \(\delta\in L^q_{\rm Sob}(T\X)\)
there exists a unique \({\rm I}(\delta)\in L^q(T\X)\) such that
\[\begin{tikzcd}
W^{1,p}(\X) \arrow[r,"\d"] \arrow[rd,swap,"\delta"] & L^p(T^*\X) \arrow[d,"{\rm I}(\delta)"] \\
& L^1(\mm)
\end{tikzcd}\]
is a commutative diagram. Moreover, the resulting map \({\rm I}\colon L^q_{\rm Sob}(T\X)\to L^q(T\X)\)
is an isomorphism of \(L^q(\mm)\)-Banach \(L^\infty(\mm)\)-modules.
\end{theorem}
\begin{proof}
Since \(\delta\colon W^{1,p}(\X)\to L^1(\mm)\) is linear and \(|\delta(f)|\leq|\delta|\psi_p(f)=|\delta||Df|\)
for every \(f\in W^{1,p}(\X)\) (see Definition \ref{def:cotg_mod}), Proposition \ref{prop:univ_prop_mod_gen} ensures that there exists a unique
\(L^\infty(\mm)\)-linear operator \({\rm I}(\delta)\colon L^p(T^*\X)\to L^1(\mm)\) such that \({\rm I}(\delta)\circ\d=\delta\),
and it holds that \(|{\rm I}(\delta)(\omega)|\leq|\delta||\omega|\) for every \(\omega\in L^p(T^*\X)\).
Therefore, it holds that \({\rm I}(\delta)\in L^q(T\X)\) and \(|{\rm I}(\delta)|\leq|\delta|\). On the other
hand, for every \(v\in L^q(T\X)\) we have that the map \((v\circ\d)(f)=\d f(v)\) maps linearly \(W^{1,p}(\X)\) to \(L^1(\mm)\) and
\(|(v\circ\d)(f)|=|\d f(v)|\leq|v||Df|\) for all \(f\in W^{1,p}(\X)\), so that \(v\circ\d\in L^q_{\rm Sob}(T\X)\)
and \(|v\circ\d|\leq|v|\). Noticing that \({\rm I}(v\circ\d)=v\) and thus \(|{\rm I}(v\circ\d)|=|v|\geq|v\circ\d|\),
we conclude that \({\rm I}\) is an isomorphism of \(L^q(\mm)\)-Banach \(L^\infty(\mm)\)-modules.
\end{proof}

In view of Theorem \ref{thm:Lq(TX)=LqSob(TX)} and Remarks \ref{remark:Dirichletspace:differential} and \ref{rmk:consist_cotg_mod}, we could have equivalently
defined Sobolev derivations as maps over the Dirichlet spaces $D^{1,p}( \X )$ or \(D^{1,p}_{\mm}(\X)\), or
\(S^{1,p}(\X)\) considered in \cite{Gig:18} (which consists of locally-integrable functions having a \(p\)-weak
upper gradient). In particular, Definition \ref{def:Sob_der} is fully consistent with \cite[Definition 2.3.2]{Gig:18}.
Our choice of defining Sobolev derivations over the Sobolev space \(W^{1,p}(\X)\) instead is due to the fact that this
simplifies the investigation of their relation with Lipschitz derivations.
\medskip

As a consequence of Theorems \ref{thm:Lq(TX)=LqSob(TX)} and \ref{thm:calculus_rules_diff} iii),
Sobolev derivations satisfy the \textbf{Leibniz rule}:
\begin{equation}\label{eq:Leibniz_Sob_der}
\delta(fg)=f\,\delta(g)+g\,\delta(f)\quad\text{ for every }f,g\in W^{1,p}(\X)\cap L^\infty(\mm).
\end{equation}
Furthermore, in light of the identification \(L^q(T\X)\cong L^q_{\rm Sob}(T\X)\) we presented in Theorem
\ref{thm:Lq(TX)=LqSob(TX)}, we can unambiguously denote by \(D(\boldsymbol\div_q;\X)\) the space of all
\(\delta\in L^q_{\rm Sob}(T\X)\) for which there exists a measure \(\boldsymbol\div_q(\delta)\in\mathfrak M(\X)\),
called the \textbf{\(q\)-divergence} of \(\delta\), satisfying the following identity:
\[
\int\delta(f)\,\d\mm=-\int f\,\d\boldsymbol\div_q(\delta)\quad\text{ for every }f\in\LIP_{bs}(\X).
\]
Similarly for \(D(\div_q;\X)\). These definitions are fully consistent with Definitions \ref{def:vf_with_div}
and \ref{def:vf_with_Lq_div}.
\medskip

Let us now study the relation between Lipschitz and Sobolev derivations. First:
\begin{lemma}\label{lem:Sob_der_induces_Lip_der}
Let \((\X,\sfd,\mm)\) be a metric measure space and \(q\in(1,\infty]\). The restriction operator \(\varrho\colon L^q_{\rm Sob}(T\X)\to{\rm Der}^q(\X)\) defined by
\[
\varrho(\delta)(f)\coloneqq\delta(f)\quad\text{ for every }( f, \delta ) \in \LIP_{bs}(\X) \times L^q_{\rm Sob}(T\X),
\]
is well-defined and a homomorphism of \(L^q(\mm)\)-Banach \(L^\infty(\mm)\)-modules
satisfying \(|\varrho(\delta)|\leq|\delta|\) for every \(\delta\in L^q_{\rm Sob}(T\X)\). Moreover, for every
\(\delta\in D(\boldsymbol\div_q;\X)\), it holds that \(\varrho(\delta)\in{\rm Der}^q_{\mathfrak M}(\X)\) and \(\boldsymbol\div(\varrho(\delta))=\boldsymbol\div_q(\delta)\).
\end{lemma}
\begin{proof}
Let \(\delta\in L^q_{\rm Sob}(T\X)\) be given. Then \(\varrho(\delta)\) is linear and
\(|\varrho(\delta)(f)|\leq|\delta||Df|\leq|\delta|\,\lip_a(f)\) for every \(f\in\LIP_{bs}(\X)\).
Since \(\varrho(\delta)\) satisfies also the Leibniz rule thanks to \eqref{eq:Leibniz_Sob_der},
we have shown that \(\varrho(\delta)\in{\rm Der}^q(\X)\) and \(|\varrho(\delta)|\leq|\delta|\).
Moreover, \(\varrho\) is clearly \(L^\infty(\mm)\)-linear, thus accordingly it is a homomorphism
of \(L^q(\mm)\)-Banach \(L^\infty(\mm)\)-modules. Finally, if \(\delta\in D(\boldsymbol\div_q;\X)\) is given,
then we have that \(\int\varrho(\delta)(f)\,\d\mm=-\int f\,\d\boldsymbol\div_q(\delta)\) for every \(f\in\LIP_{bs}(\X)\),
which proves the last claim.
\end{proof}

We anticipate that, in fact, up to an isometric embedding by a right-inverse of $\rho$, it holds that
\[
L^q_\Lip(T\X)\quad\text{ is an }L^q(\mm)\text{-Banach }L^\infty(\mm)\text{-submodule of }L^q_{\rm Sob}(T\X).
\]
We postpone its verification -- which requires some results from Section \ref{ss:divergence_measures}
-- to Theorem \ref{thm:Lip_and_Sob_tg_mod}.
\section{Applications to potential theory}
In this section, we study the dual energy functional (defined on the space of boundedly-finite signed Borel measures) of the so-called pre-Cheeger energy.
More precisely, we investigate the relation of the dual energy functional with the Sobolev \(p\)-capacity. This will enable us to show that measures in the finiteness domain
of the dual energy functional are divergences of derivations in a suitable class. As a consequence, we refine previous results concerning the relation between Lipschitz
and Sobolev tangent modules. Based on the above, we propose and study a notion of \(p\)-Laplacian in the general setting of metric measure spaces (under no additional assumptions).
\subsection{Divergence measures}\label{ss:divergence_measures}
Given a metric measure space \((\X,\sfd,\mm)\), \(p\in [1,\infty)\), and \(\mu\in\mathfrak M(\X)\), we define \({\sf F}_p(\mu)\in[0,\infty]\) as
\begin{equation}\label{eq:predual:energy}
{\sf F}_p(\mu)\coloneqq\sup\bigg\{\int f\,\d\mu\;\bigg|\;f\in\LIP_{bs}(\X)\text{ with }\int\lip_a(f)^p\,\d\mm\leq 1\bigg\}.
\end{equation}
We denote \(D({\sf F}_p)\coloneqq\big\{\mu\in\mathfrak M(\X)\,:\,{\sf F}_p(\mu)<+\infty\big\}\).
We say that \({\sf F}_p(\mu)\) is the \textbf{dual \(p\)-energy} of \(\mu\).
\begin{remark}\label{rmk:comments_def_F_p}{\rm
Let us collect here a few observations concerning the functional \({\sf F}_p\):
\begin{itemize}
\item[\(\rm i)\)] The term `dual \(p\)-energy' comes from the fact that \({\sf F}_p\) coincides with the conjugate of the convex \(p\)-homogeneous functional \(p\,\mathcal E_{p,\lip}\),
where \(\mathcal E_{p,\lip}\) denotes the pre-Cheeger energy functional defined in \eqref{eq:pre-Cheeger}. Namely, for every \(\mu\in\mathfrak M(\X)\), it holds that
\[
{\sf F}_p(\mu)=(p\,\mathcal E_{p,\lip})_*(\mu)\coloneqq\sup\bigg\{\int f\,\d\mu\;\bigg|\;f\in\LIP_{bs}(\X)\text{ with }p\,\mathcal E_{p,\lip}(f)\leq 1\bigg\}.
\]
\item[\(\rm ii)\)] Given any \(\mu\in D({\sf F}_p)\), we have that
\begin{equation}\label{eq:conseq_def_Fp_mu}
\int f\,\d\mu\leq{\sf F}_p(\mu)\|\lip_a(f)\|_{\mathcal L^p(\mm)}\quad\text{ for every }f\in\LIP_{bs}(\X).
\end{equation}
Indeed, let \(f\in\LIP_{bs}(\X)\) be a given function. If \(\|\lip_a(f)\|_{\mathcal L^p(\mm)}\neq 0\),
then we can consider the function \(\tilde f\coloneqq\|\lip_a(f)\|_{\mathcal L^p(\mm)}^{-1}f\in\LIP_{bs}(\X)\) which satisfies \(\int\lip_a(\tilde f)^p\,\d\mm=1\)
and thus accordingly \(\|\lip_a(f)\|_{\mathcal L^p(\mm)}^{-1}\int f\,\d\mu=\int\tilde f\,\d\mu\leq{\sf F}_p(\mu)\). If \(\|\lip_a(f)\|_{\mathcal L^p(\mm)}=0\), then, for every $\varepsilon \in \mathbb{R} \setminus \{0\}$, it holds that
\begin{align*}
    \int_{\X} \frac{1}{\varepsilon} f\,\d\mu \leq {\sf F}_p( \mu ).
\end{align*}
Thus \eqref{eq:conseq_def_Fp_mu} is trivially satisfied in this case. 
\item[\(\rm iii)\)]
If \((\X,\sfd,\mm)\) 
is such that \(\mm\) is a finite measure,
then \(\mu(\X)=0\) holds for
every \(\mu\in D({\sf F}_p)\cap\mathcal M(\X)\).
Indeed, for any \(k\in\N\), consider a non-decreasing sequence \((\eta^k_n)_n\subseteq\LIP_{bs}(\X;[0,k])\)
such that \(\lim_n\eta_n^k(x)=k\) for every \(x\in\X\) and \(\|\lip_a(\eta^k_n)\|_{\mathcal L^p(\mm)}\leq 1\) for every
\(n\in\N\). 
The monotone convergence theorem and
\eqref{eq:conseq_def_Fp_mu} then give
\[
k|\mu(\X)|=\lim_{n\to\infty}\bigg|\int\eta^k_n\,\d\mu^+ -\int\eta^k_n\,\d\mu^-\bigg|\leq{\sf F}_p(\mu)\quad\text{ for every }k\in\N,
\]
whence it follows that \(\mu(\X)=0\).
\end{itemize}
}\end{remark}
\begin{proposition}\label{prop:mu_ll_Cap}
Let \((\X,\sfd,\mm)\) be a metric measure space and \(p\in[1,\infty)\). Then it holds that
\[
\mu\ll{\rm Cap}_p\quad\text{ for every }\mu\in D({\sf F}_p).
\]
\end{proposition}
\begin{proof}
Up to repeating the same argument for \(-\mu\) in place of \(\mu\), it suffices to show that \(\mu^+\ll{\rm Cap}_p\).
To this aim, we fix a Borel set \(P\subseteq\X\) such that \(\mu^+(\X\setminus P)=\mu^-(P)=0\).
Since \(\mu^+\) is inner regular, it is enough to verify that
\begin{equation}\label{eq:div_ll_Cap_aux}
\mu^+(K)=0\quad\text{ for every compact set }K\subseteq P\text{ such that }{\rm Cap}_p(K)=0.
\end{equation}
By virtue of Lemma \ref{lem:Cap_with_Lip}, we can find \((f_n)_n\subseteq\LIP_{bs}(\X;[0,1])\) such that \(f_n=1\) on \(K\) and
\begin{equation}\label{eq:div_ll_Cap_aux2}
\int|f_n|^p\,\d\mm+\int\lip_a(f_n)^p\,\d\mm\to 0\quad\text{ as }n\to\infty.
\end{equation}
The compactness of \(K\) ensures the existence of \(r_j>0\) such that the closed \(r_j\)-neighbourhood \(K_j\) of \(K\)
satisfies \(\mu^-(K_j\setminus K)\leq 1/j\). The cut-off function
\(\varphi_j\coloneqq(1-r_j^{-1}\sfd(\cdot,K))\vee 0\in\LIP_{bs}(\X;[0,1])\) is \(r_j^{-1}\)-Lipschitz and
its support is contained in \(K_j\). Since \(\lip_a(\varphi_j f_n)\leq\varphi_j\lip_a(f_n)+f_n\lip_a(\varphi_j)\), we deduce
that \(\|\lip_a(\varphi_j f_n)\|_{\mathcal L^p(\mm)}\leq\|\lip_a(f_n)\|_{\mathcal L^p(\mm)}+r_j^{-1}\|f_n\|_{\mathcal L^p(\mm)}\), thus \eqref{eq:conseq_def_Fp_mu} gives
\[\begin{split}
\mu^+(K)&\leq\int\varphi_j f_n\,\d\mu^+=\int\varphi_j f_n\,\d\mu +\int_{K_j\setminus K}\varphi_j f_n\,\d\mu^-\\
&\leq{\sf F}_p(\mu)\|\lip_a(\varphi_j f_n)\|_{\mathcal L^p(\mm)}+\mu^-(K_j\setminus K)\\
&\leq{\sf F}_p(\mu)\|\lip_a(f_n)\|_{\mathcal L^p(\mm)}+\frac{1}{r_j}{\sf F}_p(\mu)\|f_n\|_{\mathcal L^p(\mm)}+\frac{1}{j}.
\end{split}\]
Thanks to \eqref{eq:div_ll_Cap_aux2}, by first letting \(n\to\infty\), then \(j\to\infty\) we get 
\(\mu^+(K)=0\). This proves the validity of \eqref{eq:div_ll_Cap_aux}, whence the statement follows.
\end{proof}
\begin{theorem}\label{thm:L_mu}
Let \((\X,\sfd,\mm)\) be a metric measure space and \(p\in[1,\infty)\). Let \(\mu\in D({\sf F}_p)\) be given.
Define \(\tilde L_\mu\colon\LIP_{bs}(\X)\to\R\) as \(\tilde L_\mu(f)\coloneqq\int f\,\d\mu\) for every \(f\in\LIP_{bs}(\X)\).
Then there exists a unique linear map \(L_\mu\colon W^{1,p}(\X)\to\R\) that is continuous from \((W^{1,p}(\X),\sfd_{\rm en})\) to \(\R\)
and for which
\begin{equation}\label{eq:diagr_Lmu}\begin{tikzcd}
\LIP_{bs}(\X) \arrow[d,swap,"\pi_\mm"] \arrow[r,"\tilde L_\mu"] & \R \\
W^{1,p}(\X) \arrow[ur,swap,"L_\mu"] &
\end{tikzcd}\end{equation}
is a commutative diagram. It also holds that
\begin{equation}\label{eq:cont_Lmu}
|L_\mu(f)|\leq{\sf F}_p(\mu)\||Df|\|_{L^p(\mm)}\quad\text{ for every }f\in W^{1,p}(\X).
\end{equation}
Moreover, \(L_\mu\) is the unique linear map such that the diagram \eqref{eq:diagr_Lmu} commutes, \eqref{eq:cont_Lmu} holds, and
\begin{equation}\label{Lmu_consistency}
L_\mu(f) = \int {\sf qcr}(f)\,\d\mu\quad\text{ for every }f\in\mathcal W,
\end{equation}
where \(\mathcal W\) denotes the vector space of all functions in \( W^{1,p}(\X)\cap L^\infty(\mm)\) having bounded support.
\end{theorem}
\begin{proof}
\ \\
\textbf{Step 1.} Recall from Lemma \ref{lem:trunc_cut-off} that \(\mathcal W\) is a strongly dense vector subspace of \(W^{1,p}(\X)\).
For any \(f\in\mathcal W\), consider any \((f_n)_n\subseteq\LIP_{bs}(\X)\) such that \(\sfd_{\rm en}(f_n^\mm,f)\to 0\)
and \(\|\lip_a(f_n)^\mm-|Df|\|_{L^p(\mm)}\to 0\). We can also assume that for some bounded Borel set \(B\subseteq\X\)
and for some \(C>0\), it holds that \({\rm spt}(f_n)\subseteq B\) and \(|f_n|\leq C\) for every \(n\in\N\).
Proposition \ref{prop:good_repres_Sob} ensures that, up to a non-relabelled subsequence, \(f_n\to{\sf qcr}(f)\) holds
\({\rm Cap}_p\)-a.e.\ on \(\X\), thus in particular \(f_n\to{\sf qcr}(f)\) holds \(\mu\)-a.e.\ on \(\X\) by Proposition
\ref{prop:mu_ll_Cap}. Since \(|\int f_n\,\d\mu|\leq{\sf F}_p(\mu)\|\lip_a(f_n)\|_{\mathcal L^p(\mm)}\) by \eqref{eq:conseq_def_Fp_mu}
and \(\int{\sf qcr}(f)\,\d\mu=\lim_n\int f_n\,\d\mu\) by the dominated convergence theorem, we get
\(|\int{\sf qcr}(f)\,\d\mu|\leq{\sf F}_p(\mu)\||Df|\|_{L^p(\mm)}\). Now,  for \(f\in\mathcal W\), define \(L_\mu(f)\coloneqq\int{\sf qcr}(f)\,\d\mu\),
so that \(L_\mu\) is linear, \(|L_\mu(f)|\leq{\sf F}_p(\mu)\||Df|\|_{L^p(\mm)}\)
for every \(f\in\mathcal W\), and \(L_\mu\circ\pi_\mm=\tilde L_\mu\).
Given that \(\||Df|\|_{L^p(\mm)}\leq\|f\|_{W^{1,p}(\X)}\)
for all \(f\in W^{1,p}(\X)\) and \(\mathcal W\) is strongly dense in \(W^{1,p}(\X)\), we can extend \(L_\mu\)
to a map \(L_\mu\colon W^{1,p}(\X)\to\R\) satisfying \(|L_\mu(f)|\leq{\sf F}_p(\mu)\||Df|\|_{L^p(\mm)}\) for every
\(f\in W^{1,p}(\X)\).\\
\textbf{Step 2.} We now show that \(L_\mu\) is \(\sfd_{\rm en}\)-continuous.
Pick a sequence \((f_n)_{n\in \N\cup\{\infty\}}\subseteq W^{1,p}(\X)\) such that \(f_n\to f_\infty\) and \(|Df_n|\to |Df_\infty|\) strongly in \(L^p(\mm)\) as \(n\to +\infty\).
Then up to a non-relabelled subsequence, we can find \(h\in L^p(\mm)^+\) such that \(|f_n|+|Df_n|\leq h\)
for every \(n\in\N\cup\{\infty\}\). By virtue of Proposition \ref{prop:good_repres_Sob}, we can also assume that
\({\sf qcr}(f_n)(x)\to{\sf qcr}(f_\infty)(x)\) for \({\rm Cap}_p\)-a.e.\ \(x\in\X\). Now consider \(\eta_k\), \(\psi_k\),
\(\mathcal T_k f_n\) as in Lemma \ref{lem:trunc_cut-off}. Thanks to \eqref{eq:trunc_cut-off_estimates}, we have
\(|D(\mathcal T_k f_n-f_n)|\leq 2\1_{E_k}^\mm h\) for every
\(n\in\N\cup\{\infty\}\), where \(E_k\coloneqq\{|h|>k\}\cup(\X\setminus B_k(\bar x))\). Given any \(\varepsilon>0\),
we find \(k(\varepsilon)\in\N\) so that \(\|2\1_{E_{k(\varepsilon)}}^\mm h\|_{L^p(\mm)}\leq\varepsilon\).
Since \(\eta_{k(\varepsilon)}\) and \(\psi_{k(\varepsilon)}\) are continuous, we have that
\[
{\sf qcr}(\mathcal T_{k(\varepsilon)}f_n)(x)\to{\sf qcr}(\mathcal T_{k(\varepsilon)}f_\infty)(x)\quad\text{ for }{\rm Cap}_p\text{-a.e.\ }x\in\X.
\]
Recalling that \(|\mathcal T_{k(\varepsilon)}f_n|\leq k(\varepsilon)\) and
\({\rm spt}(\mathcal T_{k(\varepsilon)}f_n)\subseteq B_{k(\varepsilon)+1}(\bar x)\) for every \(n\in\N\), we get that
\[
\lim_{n\to\infty}L_\mu(\mathcal T_{k(\varepsilon)}f_n)=L_\mu(\mathcal T_{k(\varepsilon)}f_\infty)
\]
by \textbf{Step 1} (cf.\ Remark \ref{rmk:bdd_cont_L_mu}). Therefore, by letting \(n\to\infty\) in
\[\begin{split}
\big|L_\mu(f_n)-L_\mu(f_\infty)\big|&\leq\big|L_\mu(f_n-\mathcal T_{k(\varepsilon)}f_n)\big|
+\big|L_\mu(\mathcal T_{k(\varepsilon)}f_n)-L_\mu(\mathcal T_{k(\varepsilon)}f_\infty)\big|
+\big|L_\mu(\mathcal T_{k(\varepsilon)}f_\infty-f_\infty)\big|\\
&\leq 2\,{\sf F}_p(\mu)\|2\1_{E_{k(\varepsilon)}}^\mm h\|_{L^p(\mm)}
+\big|L_\mu(\mathcal T_{k(\varepsilon)}f_n)-L_\mu(\mathcal T_{k(\varepsilon)}f_\infty)\big|\\
&\leq 2\,{\sf F}_p(\mu)\varepsilon+\big|L_\mu(\mathcal T_{k(\varepsilon)}f_n)-L_\mu(\mathcal T_{k(\varepsilon)}f_\infty)\big|,
\end{split}\]
we obtain that \(\lims_n|L_\mu(f_n)-L_\mu(f_\infty)|\leq 2\,{\sf F}_p(\mu)\varepsilon\), thus accordingly
\(L_\mu(f_n)\to L_\mu(f_\infty)\) as \(n\to\infty\) thanks to the arbitrariness of \(\varepsilon\). Since
the limit \(L_\mu(f_\infty)\) is independent of the chosen subsequence of \((f_n)_n\), we conclude that
also the original sequence satisfies \(L_\mu(f_n)\to L_\mu(f_\infty)\), proving the continuity with respect to \(\sfd_{\rm en}\). The claimed uniqueness follows thanks to the density of \(\LIP_{bs}(\X)^\mm\)
in \((W^{1,p}(\X),\sfd_{\rm en})\). 
Notice also that the last part of the statement holds true by construction and the strong density of \(\mathcal W\) in \(W^{1,p}(\X)\).
\end{proof}
\begin{remark}\label{rmk:bdd_cont_L_mu}{\rm
From \eqref{Lmu_consistency} and the dominated convergence theorem, we infer that \(L_\mu\colon\mathcal W\to\R\) enjoys the following continuity
property: if \((f_n)_n\subseteq\mathcal W\) is a sequence with \(\sup_n\|f_n\|_{L^\infty(\mm)}<+\infty\)
such that \(\bigcup_n{\rm spt}(f_n)\) is bounded and \({\sf qcr}(f_n)\to{\sf qcr}(f)\) holds \({\rm Cap}_p\)-a.e.\ for some
\(f\in W^{1,p}\), then \(f\in\mathcal W\) and it holds that \(L_\mu(f_n)\to L_\mu(f)\).
}\end{remark}

The following result shows that the map \(L_\mu\) associated with a \(\mu\in D({\sf F}_p)\) enjoys a stronger continuity property,
where the $\mu$-a.e. convergence of the quasicontinuous representative is not part of the assumption.
\begin{corollary}\label{cor:L_mu_cont_in_energy}
Let \((\X,\sfd,\mm)\) be a metric measure space and \(p\in[1,\infty)\). Let \(\mu\in D({\sf F}_p)\) be given.
Let \(L_\mu\) be as in 
Theorem \ref{thm:L_mu}. 
Assume that \((f_n)_{n\in\N\cup\{\infty\}}\subseteq W^{1,p}(\X)\) is a sequence such that $f_n \rightharpoonup f_\infty$ weakly in $L^{p}( \mm )$ and $\{|Df_n|\,:\,n\in\N\}$
is relatively weakly compact in $L^{p}( \mm )$. Then
\begin{equation}\label{eq:L_mu_cont_in_energy}
L_\mu(f_n)\to L_\mu(f_\infty)\quad\text{ as }n\to\infty.
\end{equation}
\end{corollary}
\begin{proof}
First of all, notice that the statement holds true under the stronger assumption that \(f_n\to f_\infty\) strongly in \(L^p(\mm)\) and 
$|Df_n| \leq g_n$ for all $n \in \N$, for some sequence $( g_n )_{ n \in \N \cup \{\infty\} } \subseteq L^{p}( \mm )$ such that
\(g_n\to g_\infty\) 
strongly in \(L^p(\mm)\). The proof follows along the same lines as in \textbf{Step 2} of the proof of Theorem \ref{thm:L_mu}.

Now, let us prove the statement under the weaker assumption that \(f_n\rightharpoonup f_\infty\) and \(g_n\rightharpoonup g_\infty\)
weakly in \(L^p(\mm)\). Since \(\sup_n|L_\mu(f_n)|\leq{\sf F}_p(\mu)\sup_n\|g_n\|_{L^p(\mm)}<+\infty\)
we have (up to a non-relabelled subsequence) that \(L_\mu(f_n)\to\lambda\)
for some \(\lambda\in\R\). By Mazur's lemma, we can find \(h_n\in{\rm conv}(\{f_i\,:\,i\geq n\})\) and
\(\rho_n\in L^p(\mm)^+\) with \(|Dh_n|\leq\rho_n\) such that \(h_n\to f_\infty\) and \(\rho_n\to g_\infty\)
strongly in \(L^p(\mm)\). Therefore, the first part of the proof ensures that \(L_\mu(h_n)\to L_\mu(f_\infty)\).
Notice also that \(L_\mu(h_n)\in{\rm conv}(\{L_\mu(f_i)\,:\,i\geq n\})\) thanks to the linearity of the operator \(L_\mu\),
so that \(L_\mu(f_\infty)=\lim_nL_\mu(h_n)=\lambda\) and therefore \(L_\mu(f_n)\to L_\mu(f_\infty)\) for the original sequence.

Finally, suppose that \(f_n\rightharpoonup f_\infty\) weakly in $L^{p}( \mm )$ and that $\{|Df_n|\}_{ n \in \N }$ is relatively weakly compact in
\(L^p(\mm)\). As $( L_\mu( f_n ) )_{ n = 1 }^{ \infty }$ is bounded, it suffices to prove that an arbitrary subsequential limit is equal to $L_\mu( f_\infty )$. To this end, if 
\(L_\mu( f_{n_j} ) \rightarrow \lambda\) for some subsequence $(n_j)_{j=1}^{\infty}$, then a further subsequence $( n_{j_k} )_{k=1}^{\infty}$ has the property that $|Df_{n_{j_k}}|$ converges weakly in \(L^p( \mm )\). Therefore, $L_\mu(f_\infty) = \lambda$ by the previous case above.
\end{proof}
\begin{remark}\label{rmk:L_mu_cont_in_energy_p>1}{\rm
In case \(p \in (1,\infty)\), a necessary and sufficient condition for the assumptions in Corollary \ref{cor:L_mu_cont_in_energy} 
to be fulfilled is that \(f_n\rightharpoonup f_\infty\) weakly in $L^p(\mm)$ and \(\sup_n\||Df_n|\|_{L^p(\mm)}<+\infty\), by the reflexivity of $L^{p}( \mm )$. 
In case \(p=1\), thanks to the Dunford--Pettis theorem, a necessary and sufficient condition is that \(f_n\rightharpoonup f_\infty\) weakly in $L^{1}( \mm )$ and the family $\mathcal{F} = \left\{ |Df_n| \colon n \in \mathbb{N} \right\}$ is bounded in $L^{1}( \mm )$ and equi-integrable; the latter property on $\mathcal{F}$ holds, for example, if there exists a bounded sequence $( g_n )_{ n \in \mathbb{N} } \subseteq L^{1}( \mm )$ such that $|Df_n| \leq g_n$ for $n \in \mathbb{N}$ and $g_n \rightharpoonup g$ in $L^{1}( \mm )$ for some $g \in L^{1}( \mm )$,
see e.g.\ \cite[Chapter 4, Theorem 4.7.20]{Bog:07} for more details.
}\end{remark}
\begin{remark}\label{rmk:L_mu_for_mu_finite}{\rm
When \(\mu\in D({\sf F}_p)\cap\mathcal M(\X)\), \eqref{Lmu_consistency} can be improved to
\[
L_\mu(f)=\int{\sf qcr}(f)\,\d\mu\quad\text{ for every }f\in W^{1,p}(\X)\cap L^\infty(\mm).
\]
The validity of this claim follows from Corollary \ref{cor:L_mu_cont_in_energy} and the proof of Theorem \ref{thm:L_mu}.
}\end{remark}
\subsubsection*{Dual dynamic cost}
If \((\X,\sfd,\mm)\) is a metric measure space and \(q\in(1,\infty]\), we define the \textbf{dual dynamic cost} \({\sf D}_q(\mu)\) as
\begin{equation}\label{eq:def_D_q}
{\sf D}_q(\mu)\coloneqq\inf\Big\{\|b\|_{{\rm Der}^q(\X)}\;\Big|\;b\in{\rm Der}^q_{\mathfrak M}(\X),\,{\bf div}(b)=\mu\Big\}\in[0,\infty]
\end{equation}
for every \(\mu\in\mathfrak M(\X)\). We denote the finiteness domain of \({\sf D}_q\) by
\[
D({\sf D}_q)\coloneqq\big\{\mu\in\mathfrak M(\X)\;\big|\;{\sf D}_q(\mu)<+\infty\big\}.
\]

\begin{lemma}\label{lem:vector_field_induced_by_mu}
Let \((\X,\sfd,\mm)\) be a metric measure space and \(p\in[1,\infty)\). Let \(\mu\in D({\sf F}_p)\) be given.
Then there exists a vector field \(v_\mu\in D(\boldsymbol\div_q;\X)\) such that \(\|v_\mu\|_{L^q(T\X)}\leq{\sf F}_p(\mu)\) and
\begin{equation}\label{eq:vector_field_induced_by_mu}
L_\mu(f)=-\int\d f(v_\mu)\,\d\mm\quad\text{ for every }f\in W^{1,p}(\X),
\end{equation}
where \(L_\mu\colon W^{1,p}(\X)\to\R\) is the functional given by Theorem \ref{thm:L_mu}. In particular, \(\boldsymbol\div_q(v_\mu)=\mu\).
\end{lemma}
\begin{proof}
Given that \(\d\colon W^{1,p}(\X)\to L^p(T^*\X)\) is linear, the set \(\mathbb V\coloneqq\{\d f\,:\,f\in W^{1,p}(\X)\}\) is a
vector subspace of \(L^p(T^*\X)\). Let us define \(\Phi\colon\mathbb V\to\R\) as \(\Phi(\d f)\coloneqq-L_\mu(f)\)
for every \(f\in W^{1,p}(\X)\). Notice that \(|\Phi(\d f)|\leq{\sf F}_p(\mu)\|\d f\|_{L^p(T^*\X)}\) for every
\(f\in W^{1,p}(\X)\) by Theorem \ref{thm:L_mu}, thus in particular \(\Phi\) is well-defined.
By the Hahn--Banach theorem, we can extend \(\Phi\) to an element \(\Phi\in L^p(T^*\X)'\) such that
\(\|\Phi\|_{L^p(T^*\X)'}\leq{\sf F}_p(\mu)\). Therefore, the element \(v_\mu\coloneqq\textsc{Int}_{L^p(T^*\X)}^{-1}(\Phi)\in L^q(T\X)\)
satisfies \(\|v_\mu\|_{L^q(T\X)}\leq{\sf F}_p(\mu)\) and \(\int\d f(v_\mu)\,\d\mm=\Phi(\d f)=-L_\mu(f)\) for every \(f\in W^{1,p}(\X)\).
In particular, we have that \(\int\d f(v_\mu)\,\d\mm=-L_\mu(f)=-\int f\,\d\mu\) for every \(f\in\LIP_{bs}(\X)\),
so that \(\boldsymbol\div_q(v_\mu)=\mu\).
\end{proof}
The next result provides the converse inclusion and establishes the equivalence between \({\sf D}_q\) and the dual \(p\)-energy \({\sf F}_p\).
\begin{theorem}[\({\sf F}_p={\sf D}_q\)]\label{thm:F_p=D_q}
Let \((\X,\sfd,\mm)\) be a metric measure space and \(p\in[1,\infty)\). Then it holds
\begin{equation}\label{eq:F_p=D_q}
{\sf F}_p(\mu)={\sf D}_q(\mu)\quad\text{ for every }\mu\in\mathfrak M(\X).
\end{equation}
In particular, \(D({\sf F}_p)\) coincides with the image of \({\bf div}\colon{\rm Der}^q_{\mathfrak M}(\X)\to\mathfrak M(\X)\).
Moreover, for every \(\mu\in D({\sf D}_q)\) the infimum in \eqref{eq:def_D_q} is a minimum, and each minimiser
of \eqref{eq:def_D_q} is an acylic derivation. In particular, every vector field \(v_\mu\) as in Lemma
\ref{lem:vector_field_induced_by_mu} satisfies \(\|v_\mu\|_{L^q(T\X)}={\sf F}_p(\mu)={\sf D}_q(\mu)\).
\end{theorem}
\begin{proof}
Let us first prove that \({\sf F}_p(\mu)\leq{\sf D}_q(\mu)\) for every \(\mu\in\mathfrak M(\X)\). To this aim, fix any \(f\in\LIP_{bs}(\X)\) with
\(\int\lip_a(f)^p\,\d\mm\leq 1\) and any \(b\in{\rm Der}^q_{\mathfrak M}(\X)\) with \({\bf div}(b)=\mu\). Therefore, we have that
\[
\int f\,\d\mu=\int f\,\d{\bf div}(b)=-\int b(f)\,\d\mm\leq\|b\|_{{\rm Der}^q(\X)}\|\lip_a(f)\|_{\mathcal L^p(\mm)}\leq\|b\|_{{\rm Der}^q(\X)}.
\]
Thanks to the arbitrariness of \(f\) and \(b\), we deduce that \({\sf F}_p(\mu)\leq{\sf D}_q(\mu)\). Let us now verify the
converse inequality \({\sf D}_q(\mu)\leq{\sf F}_p(\mu)\). This inequality is trivially verified if \({\sf F}_p(\mu)=+\infty\),
so let us assume that \({\sf F}_p(\mu)<+\infty\). Consider \(v_\mu\in L^q(T\X)\) as in Lemma \ref{lem:vector_field_induced_by_mu}.
Letting \(\delta_\mu\coloneqq{\rm I}^{-1}(v_\mu)\in L^q_{\rm Sob}(T\X)\) (with \(\rm I\) as in Theorem \ref{thm:Lq(TX)=LqSob(TX)}),
we then consider \(b_\mu\coloneqq\varrho(\delta_\mu)\in{\rm Der}^q(\X)\) as in Lemma \ref{lem:Sob_der_induces_Lip_der}. Notice that
\(b_\mu\in{\rm Der}^q_{\mathfrak M}(\X)\) and \(\boldsymbol\div(b_\mu)=\mu\),
whence it follows that \({\sf D}_q(\mu)\leq\|b_\mu\|_{{\rm Der}^q(\X)}\leq{\sf F}_p(\mu)\). This yields
\({\sf F}_p(\mu)={\sf D}_q(\mu)=\|b_\mu\|_{{\rm Der}^q(\X)}\). This shows that the claimed identity \eqref{eq:F_p=D_q} holds,
and that the infimum in \eqref{eq:def_D_q} is in fact  a minimum whenever \({\sf D}_q(\mu)<+\infty\).

Finally, assume that \(b\in{\rm Der}^q_{\mathfrak M}(\X)\) satisfies \(\|b\|_{{\rm Der}^q(\X)}={\sf D}_q({\bf div}(b))\).
Let \(a\in{\rm Der}^q(\X)\) be as in Proposition \ref{prop:acycl_cyl_decomp}. Since \(b-a\in{\rm Der}^q_{\mathfrak M}(\X)\) and \({\bf div}(b-a)=0\),
we get \(a=b-(b-a)\in{\rm Der}^q_{\mathfrak M}(\X)\) and \({\bf div}(a)={\bf div}(b)\), in other words \(a\) is a competitor
for \({\sf D}_q({\bf div}(b))\) and thus \(\|b\|_{{\rm Der}^q(\X)}\leq\|a\|_{{\rm Der}^q(\X)}\). Since \(|a|\leq|a|+|b-a|=|b|\),
we conclude that \(|b-a|=0\), which means that \(b=a\) is acyclic.
\end{proof}
\begin{remark}\label{rmk:integrability_issue_D_q}{\rm
In the case where the reference measure \(\mm\) is infinite and \(q\in(1,\infty]\), it can happen that there exists a finite measure
\(\mu\in\mathcal M(\X)\) with \({\sf D}_q(\mu)<+\infty\) such that
\[
\big\{b\in{\rm Der}^1_{\mathcal M}(\X)\;\big|\;{\bf div}(b)=\mu\big\}=\varnothing. 
\]
For example, for \(n\in\N\) we call \(I_n\) the interval joining \((0,n)\in\R^2\) and \((n^2,n)\in\R^2\),
and we consider the space \(\X\coloneqq\bigcup_{n\in\N}I_n\subseteq\R^2\) endowed with the (restriction of the) Euclidean distance and
with the boundedly-finite Borel measure \(\mm\coloneqq\sum_{n\in\N}\mathcal H^1|_{I_n}\), where \(\mathcal H^1\) stands for the \(1\)-dimensional
Hausdorff measure. Now consider the finite signed Borel measure
\[
\mu\coloneqq\sum_{n\in\N}\frac{1}{n^2}\delta_{(0,n)}-\sum_{n\in\N}\frac{1}{n^2}\delta_{(n^2,n)}\in\mathcal M(\X).
\]
The plan \(\ppi\coloneqq\sum_{n\in\N}\frac{1}{n^2}\delta_{\gamma^n}\) on \(\X\) (where the curve \(\gamma^n\colon[0,1]\to\X\) is defined as \(\gamma^n_t\coloneqq(tn^2,n)\)
for every \(t\in[0,1]\)) satisfies \((\e_0)_\#\ppi=\mu^+\) and \((\e_1)_\#\ppi=\mu^-\). Observe also that \({\rm Bar}(\ppi)=\sum_{n\in\N}\frac{1}{n^2}\1_{I_n}^\mm\) and thus
\(\|{\rm Bar}(\ppi)\|_{L^2(\mm)}^2=\sum_{n\in\N}\frac{1}{n^2}<+\infty\). Hence, the derivation \(b_\sppi\in{\rm Der}^2_{\mathcal M}(\X)\)
induced by \(\ppi\) as in Proposition \ref{prop:der_induced_by_plan} satisfies \({\bf div}(b_\sppi)=\mu\), which implies
that \({\sf D}_2(\mu)<+\infty\). Let us now show that the set of \(b\in{\rm Der}^1_{\mathcal M}(\X)\) satisfying
\({\bf div}(b)=\mu\) is empty. We argue by contradiction: assume that such a derivation \(b\) exists. Letting \(a\) be
as in Proposition \ref{prop:acycl_cyl_decomp}, we know from Theorem \ref{thm:plans_vs_der_vs_curr} and Remark \ref{rmk:improved_PS} that there
is a plan \(\tilde\ppi\) on \(\X\) concentrated on non-constant Lipschitz curves of constant speed that
satisfies \({\rm Bar}(\tilde\ppi)=|b|\in L^1(\mm)\), \((\e_0)_\#\tilde\ppi={\bf div}(b)^+=\mu^+\), and
\((\e_1)_\#\tilde\ppi={\bf div}(b)^-=\mu^-\). Since $\|{\rm Bar}(\tilde\ppi)\|_{L^1(\mm)}$ is larger than
the optimal transport cost (with $c=\sfd$) from $\mu^+$ to $\mu^-$, to reach a contradiction it suffices to show that
this cost is infinite. This follows immediately from Kantorovich duality, using the Lipschitz function $\phi(x,y)=x$.
}\end{remark}

We previously proved in \Cref{prop:mu_ll_Cap} that the divergence measure is absolutely continuous with respect to $\mathrm{Cap}_p$. The following theorem refines this conclusion using $p$-exceptional sets. To this end, recall that a set $E$ is $p$-exceptional if $\Mod_p( \Gamma_E ) = 0$ where $\Gamma_E$ is the family of nonconstant curves intersecting $E$. We also recall that a set $E \subseteq \X$ satisfies $\mathrm{Cap}_p(E) = 0$ if and only if $\mm( E ) = 0$ and $E$ is $p$-exceptional, cf. \cite[Proposition 7.2.8]{HKST:15}.
\begin{theorem}\label{thm:refined_F_p_Cap}
Let \((\X,\sfd,\mm)\) be a metric measure space and \(q\in(1,\infty]\). Let \(b\in{\rm Der}^q_{\mathfrak M}(\X)\)
be a given derivation. Let \(E\subseteq\X\) be a Borel set. Then it holds that
\[\begin{split}
&{\bf div}(b)^-(E)>0\quad\Longrightarrow\quad\Mod_p\big(\big\{\gamma\in\e_0^{-1}(E)\;\big|\;\ell(\gamma)>0\big\}\big)>0,\\
&{\bf div}(b)^+(E)>0\quad\Longrightarrow\quad\Mod_p\big(\big\{\gamma\in\e_1^{-1}(E)\;\big|\;\ell(\gamma)>0\big\}\big)>0.
\end{split}\]
In particular, if $E$ is $p$-exceptional, then $|{\bf div}(b)|( E ) = 0$.
\end{theorem}
\begin{proof}
Fix \(i\in\{0,1\}\). Since \({\bf div}(b)^\pm\) are inner regular and \(\Mod_p\) is monotone, it suffices to
check the validity of the claim for \(E\) compact. Thus, let us assume that \(E\) is compact. Suppose that
\[
\Mod_p(\Gamma_i)=0,\quad\text{ where we set }\Gamma_i\coloneqq\big\{\gamma\in\e_i^{-1}(E)\;\big|\;\ell(\gamma)>0\big\}.
\]
Fix \(\varphi\in\LIP_{bs}(\X)\) satisfying \(\varphi=1\) on a neighbourhood of \(E\). Then
\(\varphi b\in{\rm Der}^q_{\mathcal M}(\X)\cap{\rm Der}^1(\X)\) satisfies \({\bf div}(\varphi b)=\varphi\,{\bf div}(b)+b(\varphi)\mm\)
by \eqref{eq:Leibniz_Lip_der}. Since \(|b(\varphi)|\leq|b|\lip_a(\varphi)=0\) \(\mm\)-a.e.\ on \(E\), we have
\[
{\bf div}(\varphi b)^\pm(E)=(\varphi\,{\bf div}(b)^\pm)(E)={\bf div}(b)^\pm(E).
\]
Remark \ref{rmk:improved_PS} yields a plan \(\ppi\in\mathcal B_q(\X)\) with \((\e_0)_\#\ppi={\bf div}(\varphi b)^-\)
and \((\e_1)_\#\ppi={\bf div}(\varphi b)^+\). We can assume \(\ppi\) is concentrated on non-constant curves.
Since \(\ppi\ll\Mod_p\) by \eqref{eq:plan_ll_Mod} and \(\Mod_p(\Gamma_i)=0\),
\[\begin{split}
\boldsymbol\div(b)^-(E)={\bf div}(\varphi b)^-(E)=((\e_0)_\#\ppi)(E)=\ppi(\Gamma_0)=0\quad\text{ if }i=0,\\
\boldsymbol\div(b)^+(E)={\bf div}(\varphi b)^+(E)=((\e_1)_\#\ppi)(E)=\ppi(\Gamma_1)=0\quad\text{ if }i=1.
\end{split}\]
Either way, the statement is proved.
\end{proof}

We now introduce another functional on the space of finite signed Borel measures. Given a metric measure space \((\X,\sfd,\mm)\),
an exponent \(q\in(1,\infty]\), and measures \(\mu_0,\mu_1\in\mathcal M(\X)\), we denote by \(\Pi(\mu_1,\mu_0)\) the set of
plans \(\ppi\) on \(\X\) such that \((\e_0)_\#\ppi=\mu_0\) and \((\e_1)_\#\ppi=\mu_1\). We then define
\begin{equation}\label{eq:def_B_q_mu}
{\sf B}_q(\mu)\coloneqq\inf\Big\{\|{\rm Bar}(\ppi)\|_{L^q(\mm)}\;\Big|\;\ppi\in\Pi(\mu^+,\mu^-),\,\exists\,{\rm Bar}(\ppi)\in L^1(\mm)\cap L^q(\mm)\Big\}\in[0,\infty]
\end{equation}
for every \(\mu\in\mathcal M(\X)\). The next result, which is ultimately due to the superposition principle
(Theorem \ref{thm:Paolini-Stepanov}), states that \({\sf B}_q={\sf D}_q\).
\begin{proposition}[\({\sf B}_q={\sf D}_q\)]\label{prop:B_q=D_q}
Let \((\X,\sfd,\mm)\) be a metric measure space with \(\mm\) finite 
and \(q\in(1,\infty]\). Then it holds that
\[
{\sf B}_q(\mu)={\sf D}_q(\mu)\quad\text{ for every }\mu\in\mathcal M(\X).
\]
\end{proposition}
\begin{proof}
Let \(\ppi\in\Pi(\mu^+,\mu^-)\) with \({\rm Bar}(\ppi)\in L^q(\mm)\) be given. Let \(b_\sppi\in{\rm Der}^q_{\mathcal M}(\X)\) be the associated
derivation as in Proposition \ref{prop:der_induced_by_plan}. Given that \({\bf div}(-b_\sppi)=(\e_1)_\#\ppi-(\e_0)_\#\ppi=\mu^+ -\mu^-=\mu\) and
\(|b_\sppi|\leq{\rm Bar}(\ppi)\), we deduce that \({\sf D}_q(\mu)\leq\|b_\sppi\|_{{\rm Der}^q(\X)}\leq\|{\rm Bar}(\ppi)\|_{L^q(\mm)}\). Thanks
to the arbitrariness of \(\ppi\), the inequality \({\sf D}_q(\mu)\leq{\sf B}_q(\mu)\) is proved. Conversely, let \(b\in{\rm Der}^q_{\mathcal M}(\X)\)
with \({\bf div}(b)=\mu\) be given. Let \(a\in{\rm Der}^q_{\mathcal M}(\X)\) be as in Proposition \ref{prop:acycl_cyl_decomp}.
Since \(\mm\) is finite, we have that \(a\in{\rm Der}^1(\X)\), thus Remark \ref{rmk:improved_PS} yields the existence of a plan
\(\ppi_a\in\mathcal B_q(\X)\) such that \({\rm Bar}(\ppi_a)=|a|\), \((\e_{1})_\#\ppi_a=\mu^-\), and \((\e_0)_\#\ppi_a=\mu^+\).
Hence, \(\ppi_a\in\Pi(\mu^+,\mu^-)\) and \({\sf B}_q(\mu)\leq\|{\rm Bar}(\ppi_a)\|_{L^q(\mm)}\leq\|a\|_{{\rm Der}^q(\X)}\leq\|b\|_{{\rm Der}^q(\X)}\). By the arbitrariness of \(b\), we conclude that \({\sf B}_q(\mu)\leq{\sf D}_q(\mu)\). Therefore, the statement is proved.
\end{proof}
\begin{remark}\label{rmk:m_inf_B_q_diff}{\rm
If we modify the definition of ${\sf B}_q(\mu)$ by removing the $L^1$ integrability on the barycenter, namely
$$
{\sf B}^*_q(\mu)\coloneqq\inf\Big\{\|{\rm Bar}(\ppi)\|_{L^q(\mm)}\;\Big|\;\ppi\in\Pi(\mu^+,\mu^-),\,\exists\,{\rm Bar}(\ppi)\in L^q(\mm)\Big\}
\leq {\sf B}_q(\mu),
$$
the stronger inequality \({\sf D}_q(\mu)\leq{\sf B}^*_q(\mu)\) holds, even without finiteness assumptions on $\mm$, with the same proof.
However, when \(\mm\) is infinite it can happen that \({\sf B}^*_q(\mu)<{\sf B}_q(\mu)=+\infty\).
This is the case for \((\X,\sfd,\mm)\) and \(\mu\) as in Remark \ref{rmk:integrability_issue_D_q}.

When $\mm$ has infinite measure, in light of the superposition principle in \cite{Amb:Renzi:Vitillaro:2026}, the natural extension is to plans which are concentrated on locally rectifiable curves and which have a barycenter on $L^{q}( \mm )$. We leave the details for future work.
}\end{remark}
\subsubsection*{Relation between Lipschitz and Sobolev tangent modules}
The above results from this section will lead to a deeper understanding of the interplay between \(L^q_\Lip(T\X)\)
and \(L^q_{\rm Sob}(T\X)\), as we are going to discuss. This way, we refine the results of Section \ref{ss:Lip_and_Sob_der}.
\begin{theorem}[Lipschitz and Sobolev tangent modules]\label{thm:Lip_and_Sob_tg_mod}
Let \((\X,\sfd,\mm)\) be a metric measure space and \(q\in(1,\infty]\). Then there exists a unique homomorphism
of \(L^q(\mm)\)-Banach \(L^\infty(\mm)\)-modules
\[
\iota\colon L^q_\Lip(T\X)\to L^q_{\rm Sob}(T\X)
\]
such that \(\iota(b)(f)=b(f)\) for every \(b\in L^q_{\rm Lip}(T\X)\) and \(f\in\LIP_{bs}(\X)\). Moreover, $\iota$ preserves pointwise norms. That is, it holds that $|\iota(b)|=|b|$ for every $b\in L^q_\Lip(T\X).$
\end{theorem}
\begin{proof}
Let \(b\in{\rm Der}^q_{\mathfrak M}(\X)\) and \(h\in\LIP_{bs}(\X)\). We know from Theorem \ref{thm:F_p=D_q} that
\({\sf F}_p({\bf div}(hb))<+\infty\), thus we can consider the functional \(L_{{\bf div}(hb)}\colon W^{1,p}(\X)\to\R\)
given by Theorem \ref{thm:L_mu}. Let us define \(\mathcal L_f(b)(h)\coloneqq-L_{{\bf div}(hb)}(f)\) for every
\(f\in W^{1,p}(\X)\). Now fix any \(f\in W^{1,p}(\X)\) and pick a sequence \((f_n)_n\subseteq\LIP_{bs}(\X)\) such that
\(\sfd_{\rm en}(f_n^\mm,f)\to 0\) and \(\lip_a(f_n)^\mm\to|Df|\) strongly in \(L^p(\mm)\). Theorem \ref{thm:L_mu}
ensures that \(\mathcal L_{f_n}(b)(h)\to\mathcal L_f(b)(h)\) as \(n\to\infty\). Letting
\(\nu_f\coloneqq|b||Df|\mm\in\mathcal M_+(\X)\), we have
\[\begin{split}
|\mathcal L_f(b)(h)|&=\lim_{n\to\infty}|\mathcal L_{f_n}(b)(h)|=\lim_{n\to\infty}\bigg|\int f_n\,\d{\bf div}(hb)\bigg|
=\lim_{n\to\infty}\bigg|\int h\,b(f_n)\,\d\mm\bigg|\\
&\leq\lim_{n\to\infty}\int|h||b|\lip_a(f_n)\,\d\mm=\int|h||b||Df|\,\d\mm=\|h\|_{L^1(\nu_f)}.
\end{split}\]
Since \(\mathcal L_f(b)\colon\LIP_{bs}(\X)\to\R\) is linear and \(\LIP_{bs}(\X)^{\nu_f}\) is dense in \(L^1(\nu_f)\),
we deduce that there exists a unique \(\ell_f(b)\in L^\infty(\nu_f)\cong L^1(\nu_f)'\) with
\(\|\ell_f(b)\|_{L^\infty(\nu_f)}\leq 1\) such that \(\mathcal L_f(b)(h)=\int h\,\ell_f(b)\,\d\nu_f\) for
every \(h\in\LIP_{bs}(\X)\). Now let us define \(\iota(b)(f)\coloneqq\ell_f(b)|b||Df|\in L^1(\mm)\). Notice that
\[
\int h\,\iota(b)(f)\,\d\mm=\int h\,\ell_f(b)\,\d\nu_f=\mathcal L_f(b)(h)=-L_{{\bf div}(hb)}(f)=-\int f\,\d{\bf div}(hb)
=\int h\,b(f)\,\d\mm
\]
for every \(h\in\LIP_{bs}(\X)\), whence it follows that \(\iota(b)(f)=b(f)\). Since \(\iota(b)\colon W^{1,p}(\X)\to L^1(\mm)\)
is linear and \(|\iota(b)(f)|\leq|b||Df|\) for every \(f\in W^{1,p}(\X)\), we have that \(\iota(b)\in L^q_{\rm Sob}(T\X)\)
and \(|\iota(b)|\leq|b|\). On the other hand, \(|b(f)|=|\iota(b)(f)|\leq|\iota(b)||Df|\leq|\iota(b)|\lip_a(f)^\mm\) for every
\(f\in\LIP_{bs}(\X)\), whence it follows that \(|b|\leq|\iota(b)|\), thus \(|\iota(b)|=|b|\) for $b \in {\rm Der}^q_{\mathfrak M}(\X)$.  Since
\(\iota\colon{\rm Der}^q_{\mathfrak M}(\X)\to L^q_{\rm Sob}(T\X)\) is linear and \({\rm Der}^q_{\mathfrak M}(\X)\)
generates \(L^q_\Lip(T\X)\), the map \(\iota\) can be uniquely extended to a homomorphism of \(L^q(\mm)\)-Banach
\(L^\infty(\mm)\)-modules \(\iota\colon L^q_\Lip(T\X)\to L^q_{\rm Sob}(T\X)\) by Corollary \ref{cor:conseq_univ_prop}.
The extension satisfies \(|\iota(b)|=|b|\) for every $b \in L^q_\Lip(T\X)$ due to Corollary \ref{cor:conseq_univ_prop} and as the identity holds for every $b \in {\rm Der}^q_{\mathfrak M}(\X)$.
The extension also satisfies \(\iota(b)(f)=b(f)\) for every \(b\in L^q_\Lip(T\X)\) and \(f\in\LIP_{bs}(\X)\), as follows from a density argument.
\end{proof}
\begin{corollary}\label{cor:extensionfromLIPtoDIR}
There exists an isometric embedding ${\rm J} \colon L^q_\Lip(T\X)\to L^q(T\X)$ such that 
\[\d f( {\rm J}( b ) )= b(f)\quad 
\text{for every } b\in L^q_{\rm Lip}(T\X) \text{ and } f\in\LIP_{bs}(\X).\]
The embedding is a homomorphism of \(L^q(\mm)\)-Banach \(L^\infty(\mm)\)-modules preserving pointwise norms. That is, $|{\rm J}(b)| = |b|$ for every $b \in L^{q}_{\Lip}(T\X)$.
\end{corollary}
\begin{proof}
The claim follows as ${\rm J}$ is the composition of the homomorphism $\iota \colon L^q_\Lip(T\X) \to L^q_{\rm Sob}(T\X)$ from \Cref{thm:Lip_and_Sob_tg_mod} and the isomorphism ${\rm I} \colon L^{q}_{ \rm Sob}( T\X ) \rightarrow L^{q}( T\X )$ from \Cref{thm:Lq(TX)=LqSob(TX)}.
\end{proof}
As the proof of Theorem \ref{thm:Lip_and_Sob_tg_mod} shows, it holds that
\begin{equation}\label{eq:descr_iota(b)}
-L_{\boldsymbol\div(b)}(f)=\int\iota(b)(f)\,\d\mm\quad\text{ for every }b\in{\rm Der}^q_{\mathfrak M}(\X)\text{ and }f\in W^{1,p}(\X).
\end{equation}
As a consequence of Theorem \ref{thm:Lip_and_Sob_tg_mod}, we now prove that \({\rm Der}^q_{\mathfrak M}(\X)\) and \(D(\boldsymbol\div_q;\X)\) can be identified. In the following two Lemmas, we identify $L^{q}_{ {\rm Sob} }(T\X)$ and $L^{q}(T\X)$ using the isomorphism ${\rm I} \colon L^{q}_{ {\rm Sob} }(T\X) \to L^{q}(T\X)$.
\begin{lemma}[Lipschitz/Sobolev divergence]\label{lem:Lip_vs_Sob_div}
Let \((\X,\sfd,\mm)\) be a metric measure space and \(q\in(1,\infty]\). Let
\(\iota\colon L^q_\Lip(T\X)\to L^q_{\rm Sob}(T\X)\) be the isometric embedding from Theorem \ref{thm:Lip_and_Sob_tg_mod}. Then
\begin{equation}\label{eq:Lip_vs_Sob_der}
    \iota(b) \in D( \boldsymbol\div_q; \X ) \quad\text{and}\quad \boldsymbol\div_q(\iota(b))=\boldsymbol\div(b)\quad\text{ for every }b\in{\rm Der}^q_{\mathfrak M}(\X).
\end{equation}
Moreover, if $v \in D(\boldsymbol\div_q; \X )$, it holds that 
\(
\varrho(v)\in {\rm Der}^q_{\mathfrak M}(\X)\) and 
\(v = \iota( \varrho(v))
\)
for the restriction operator \(\varrho\) from Lemma \ref{lem:Sob_der_induces_Lip_der}.
\end{lemma}
\begin{proof}
If \(b\in{\rm Der}^q_{\mathfrak M}(\X)\) is given, then \(\int\iota(b)(f)\,\d\mm=\int b(f)\,\d\mm=-\int f\,\d\boldsymbol\div(b)\) holds for every \(f\in\LIP_{bs}(\X)\),
which implies that \(\iota(b)\in D(\boldsymbol\div_q;\X)\), thus in particular \(\iota({\rm Der}^q_{\mathfrak M}(\X))\subseteq D(\boldsymbol\div_q;\X)\).

In the converse direction, consider $v \in D(\boldsymbol\div_q;\X)$ and recall from Lemma \ref{lem:Sob_der_induces_Lip_der} that
$\varrho(v) \in {\rm Der}^{q}_{ \mathfrak M }( \X )$ with the same divergence as $v$. The claim is finished if we prove that $v = \iota( \varrho(v) )$. We argue by contradiction: there exists $f \in W^{1,p}( \X )$ such that $( v - \iota( \varrho(v) ) )(f) \neq 0$. Up to replacing $f$ by $-f$ and fixing representatives, there exists $\varepsilon > 0$ such that the set $\left\{ ( v - \iota( \varrho(v) ) )(f) > \varepsilon \right\}$ has positive measure. In particular, there exists a compact set $K \subseteq \left\{ ( v - \iota( \varrho(v) ) )(f) > \varepsilon \right\}$ for which
\begin{equation*}
    \int_{K} \big( v - \iota( \varrho(v) ) \big)(f) \,\d \mm > 0.
\end{equation*}
By approximating the characteristic function $\1_{K}$ from above using elements of $\LIP_{bs}( \X )$ and by applying dominated convergence, we find $\psi \in \LIP_{bs}( \X )$, with $0 \leq \psi \leq 1$ and $\psi|_{K} = \1_{K}$, such that
\begin{equation*}
    \int ( \psi v - \iota( \varrho(\psi v) ) )(f) \,\d \mm = \int \psi ( v - \iota( \varrho(v) ) )(f) \,\d\mm > 0.
\end{equation*}
Now given \((f_n)_n\subseteq\LIP_{bs}(\X)\) with \(\sfd_{\rm en}(f_n^\mm,f)\to 0\), we know from Theorem \ref{thm:L_mu} that
\begin{equation*}
    \lim_{ n \rightarrow \infty } \int ( \psi v - \iota( \varrho( \psi v ) ) )(f_n) \,\d\mm = \int ( \psi v - \iota( \varrho( \psi v ) ) )(f)\,\d\mm > 0.
\end{equation*}
However, we have $( \psi v - \iota( \varrho( \psi v ) ) )(f_n) = ( \rho( \psi v ) - \varrho( \psi v ) )( f_n ) = 0$ for every $n \in \mathbb{N}$ by definition of $\varrho$ and $\iota$, leading to a contradiction. So no such $f$ exists and thus $v = \iota( \varrho(v) )$.
\end{proof}

In view of Lemma \ref{lem:Lip_vs_Sob_div}, the dual dynamic cost \({\sf D}_q(\mu)\) can be alternatively expressed as
\[
{\sf D}_q(\mu)=\inf\Big\{\|v\|_{L^q(T\X)}\;\Big|\;v\in D(\boldsymbol\div_q;\X),\,\boldsymbol\div_q(v)=\mu\Big\}
\quad\text{ for every }\mu\in\mathfrak M(\X).
\]
We finally conclude this section with one last result concerning the different notions of divergence:
\begin{lemma}[Relation between \(D(\div_q;\X)\) and \({\rm Der}^q_q(\X)\)]\label{lem:relation_div_q}
Let \((\X,\sfd,\mm)\) be a metric measure space and \(q\in(1,\infty]\). 
Let \(\iota\colon L^q_\Lip(T\X)\to L^q_{\rm Sob}(T\X)\) be the isometric embedding from Theorem \ref{thm:Lip_and_Sob_tg_mod}. Then it holds that
\[
D(\div_q;\X) = \iota({\rm Der}^q_q(\X))
\]
and $\iota(v)$ and $v$ have the same divergence for every $v \in {\rm Der}^{q}_q( \X )$.
\end{lemma}
\begin{proof}
Fix any \(\delta\in D(\div_q;\X)\). Letting \(\varrho(\delta)\) be as in Lemma \ref{lem:Sob_der_induces_Lip_der}, we have that
\[
\int\varrho(\delta)(f)\,\d\mm=\int\delta(f)\,\d\mm=-\int f\,\d(\div_q(\delta)\mm)\quad\text{ for every }f\in\LIP_{bs}(\X),
\]
which ensures that \(\varrho(\delta)\in{\rm Der}^q_q(\X)\) and \(\div(\varrho(\delta))=\div_q(\delta)\). Since \(\iota(\varrho(\delta))=\delta\),
the first part of the statement is proved.

Now, let \(b\in{\rm Der}^q_q(\X)\) be fixed. By \Cref{lem:Lip_vs_Sob_div}, we have $\iota(b) \in D( \boldsymbol\div_q; \X )$ with the same divergence as $b$. Moreover, given \(f\in W^{1,p}(\X)\) and \((f_n)_n\subseteq\LIP_{bs}(\X)\) with \(\sfd_{\rm en}(f_n^\mm,f)\to 0\), we know from Theorem  \ref{thm:L_mu} that
\[
L_{\div(b)\mm}(f)=\lim_{n\to\infty}L_{\div(b)\mm}(f_n)=\lim_{n\to\infty}\int f_n\div(b)\,\d\mm=\int f\div(b)\,\d\mm.
\]
Then \eqref{eq:descr_iota(b)} gives that
\begin{equation*}
    L_{\div(b)\mm}(f) = - \int_\X \iota(b)(f) \,\d\mm,
\end{equation*}
whence it follows that \(\int \iota(b)( f )\,\d\mm=-\int f\div(b)\,\d\mm\) holds for every \(f\in W^{1,p}(\X)\). This means that \(\iota(b) \in D(\div_q;\X)\) and \(\div_q(\iota(b))=\div(b)\).
Taking Theorem \ref{thm:Lq(TX)=LqSob(TX)} into account, the proof is complete.
\end{proof}
\subsection{Gradient vector fields}\label{ss:grad_vector_fields}
In this section, we focus on a specific class of vector fields, i.e.\ on gradients of Sobolev functions.
The notions of infinitesimal strict convexity and infinitesimal Hilbertianity will be of particular interest, as they guarantee that the set of gradients of a Sobolev function is a singleton.  
Most of the presented material below is inspired by \cite{Gig:15}, where the two latter notions have been introduced. We also refer to \cite[Section 7.2.1]{Amb:Iko:Luc:Pas:24} for examples of spaces having the above-mentioned properties.
\subsubsection*{Infinitesimal strict convexity}
In this section, we adopt the convention that if \(p=1\) and \(a\geq 0\), then
\[
a^{p-1}=a^{p/q}\coloneqq\left\{\begin{array}{ll}
1\\
0
\end{array}\quad\begin{array}{ll}
\text{ if }a>0,\\
\text{ if }a=0.
\end{array}\right.
\]
\begin{lemma}[The functions \(S^\pm_{f,\omega}\)]\label{lem:tech_for_grad}
Let \((\X,\sfd,\mm)\) be a metric measure space and \(p\in[1,\infty)\). Fix any \(f\in D^{1,p}_{\mm}(\X)\) and \(\omega\in L^p(T^*\X)\).
Let us define the functions \(S^\pm_{f,\omega}\in L^0_{\rm ext}(\mm)\) as
\[
S^+_{f,\omega}\coloneqq\bigwedge_{\varepsilon>0}\frac{|\d f+\varepsilon\omega|^p-|\d f|^p}{p\varepsilon},\qquad
S^-_{f,\omega}\coloneqq\bigvee_{\varepsilon<0}\frac{|\d f+\varepsilon\omega|^p-|\d f|^p}{p\varepsilon}.
\]
Then \(S^-_{f,\omega}\leq S^+_{f,\omega}\) and \(S^\pm_{f,\omega}\in L^1(\mm)\). Moreover, it holds that \(\1_{\{|Df|>0\}}^\mm|S^\pm_{f,\omega}|\leq|\omega||Df|^{p-1}\).
\end{lemma}
\begin{proof}
Fix \(f\in D^{1,p}_{\mm}(\X)\) and \(\omega\in L^p(T^*\X)\). 
Given any \(\varepsilon_0,\varepsilon_1\in\mathbb Q\) and \(q\in\mathbb Q\cap[0,1]\),
we have that \(|\d f+(q\varepsilon_0+(1-q)\varepsilon_1)\omega|\leq q|\d f+\varepsilon_0\omega|+(1-q)|\d f+\varepsilon_1\omega|\) holds \(\mm\)-a.e.\ on \(\X\).
Hence, once a representative \(F_\varepsilon\in\mathcal L^p(\mm)\) of \(|\d f+\varepsilon\omega|\) is chosen for every \(\varepsilon\in\mathbb Q\), we can find
a Borel set \(N\subseteq\X\) such that \(\mm(N)=0\) and
\[
F_{q\varepsilon_0+(1-q)\varepsilon_1}(x)\leq q F_{\varepsilon_0}(x)+(1-q)F_{\varepsilon_1}(x)\quad\text{ for every }
\varepsilon_0,\varepsilon_1\in\mathbb Q,\,q\in\mathbb Q\cap[0,1],\text{ and }x\in\X\setminus N.
\]
It follows that \(\mathbb Q\ni\varepsilon\mapsto F_\varepsilon(x)\) is locally Lipschitz for every \(x\in\X\setminus N\), thus for any \(\varepsilon\in\R\setminus\mathbb Q\)
there exists a unique function \(F_\varepsilon\in\mathcal L^0(\mm)\) such that \(F_\varepsilon(x)=\lim_{\mathbb Q\ni\tilde\varepsilon\to\varepsilon}F_{\tilde\varepsilon}(x)\)
for every \(x\in\X\setminus N\), and \(F_\varepsilon(x)=0\) for every \(x\in N\). Since \(\R\ni\varepsilon\mapsto|\d f+\varepsilon\omega|\in L^p(\mm)\) is strongly continuous,
we deduce that \(F_\varepsilon\in\mathcal L^p(\mm)\) is a representative of \(|\d f+\varepsilon\omega|\) for every \(\varepsilon\in\R\setminus\mathbb Q\). Notice also that
\(\R\ni\varepsilon\mapsto F_\varepsilon(x)\) is convex for every \(x\in\X\setminus N\). In particular, letting \(H_\varepsilon\coloneqq\frac{1}{p}F_\varepsilon^p\in\mathcal L^1(\mm)\)
for every \(\varepsilon\in\R\), we have that \(\R\ni\varepsilon\mapsto H_\varepsilon(x)\) is convex for every \(x\in\X\setminus N\). Therefore, for every
\(\varepsilon_0,\varepsilon_1\in\R\setminus\{0\}\) with \(\varepsilon_0<\varepsilon_1\), we have that \(\frac{H_{\varepsilon_0}-H_0}{\varepsilon_0}\leq\frac{H_{\varepsilon_1}-H_0}{\varepsilon_1}\)
in the \(\mm\)-a.e.\ sense, so that
\[
H_{-1}^\mm-H_0^\mm\leq S^-_{f,\omega}\leq S^+_{f,\omega}\leq H_1^\mm-H_0^\mm.
\]
In particular, \(S^\pm_{f,\omega}\in L^1(\mm)\). Finally, define the auxiliary function \(\phi\colon[0,\infty)\to[0,\infty)\) as
\[
\phi(t)\coloneqq\frac{(1+t)^p-1-pt}{t}\quad\text{ for every }t>0
\]
and \(\phi(0)\coloneqq 0\). Notice that \(\phi\) is continuous. For every \(\varepsilon>0\) and \(\mm\)-a.e.\ \(x\in\{|Df|>0\}\) we have
\[\begin{split}
\frac{|\d f+\varepsilon\omega|^p(x)-|\d f|^p(x)}{p\varepsilon}&\leq|\omega|(x)|\d f|^{p-1}(x)+\frac{|\omega|(x)|\d f|^{p-1}(x)}{p}\phi\bigg(\frac{\varepsilon|\omega|(x)}{|\d f|(x)}\bigg),\\
\frac{|\d f-\varepsilon\omega|^p(x)-|\d f|^p(x)}{-p\varepsilon}&\geq-|\omega|(x)|\d f|^{p-1}(x)-\frac{|\omega|(x)|\d f|^{p-1}(x)}{p}\phi\bigg(\frac{\varepsilon|\omega|(x)}{|\d f|(x)}\bigg).
\end{split}\]
Given that \(\phi(t)\to 0\) as \(t\to 0\), by letting \(\varepsilon\to 0\) in the above estimates we obtain that
\[
-|\omega||\d f|^{p-1}\leq\1_{\{|Df|>0\}}^\mm S^-_{f,\omega}\leq\1_{\{|Df|>0\}}^\mm S^+_{f,\omega}\leq|\omega||\d f|^{p-1},
\]
whence it follows that \(\1_{\{|Df|>0\}}^\mm|S^\pm_{f,\omega}|\leq|\omega||Df|^{p-1}\). Therefore, the statement is achieved.
\end{proof}
It follows from (the proof of) Lemma \ref{lem:tech_for_grad} and the dominated convergence theorem that
\begin{equation}\label{eq:tech_for_grad_conseq}\begin{split}
&\lim_{\varepsilon\searrow 0}\bigg\|\1_{\{|Df|>0\}}^\mm\frac{|\d f+\varepsilon\omega|^p-|\d f|^p}{p\varepsilon}-\1_{\{|Df|>0\}}^\mm S^+_{f,\omega}\bigg\|_{L^1(\mm)}=0,\\
&\lim_{\varepsilon\nearrow 0}\bigg\|\1_{\{|Df|>0\}}^\mm\frac{|\d f+\varepsilon\omega|^p-|\d f|^p}{p\varepsilon}-\1_{\{|Df|>0\}}^\mm S^-_{f,\omega}\bigg\|_{L^1(\mm)}=0
\end{split}\end{equation}
for every \(f\in D^{1,p}_{\mm}(\X)\) and \(\omega\in L^p(T^*\X)\).
\begin{definition}[The functions \(D^\pm g(\nabla f)\)]
Let \((\X,\sfd,\mm)\) be a metric measure space and \(p\in[1,\infty)\). Let \(f,g\in D^{1,p}_{\mm}(\X)\) be given.
Then we define the functions \(D^\pm g(\nabla f)\in L^0_{\rm ext}(\mm)\) as
\[\begin{split}
D^+g(\nabla f)&\coloneqq\bigwedge_{\varepsilon>0}\1_{\{|Df|>0\}}^\mm\frac{|D(f+\varepsilon g)|^p-|Df|^p}{p\varepsilon|Df|^{p-2}}
=\1_{\{|Df|>0\}}^\mm\frac{S^+_{f,\d g}}{|Df|^{p-2}},\\
D^-g(\nabla f)&\coloneqq\bigvee_{\varepsilon<0}\1_{\{|Df|>0\}}^\mm\frac{|D(f+\varepsilon g)|^p-|Df|^p}{p\varepsilon|Df|^{p-2}}
=\1_{\{|Df|>0\}}^\mm\frac{S^-_{f,\d g}}{|Df|^{p-2}}.
\end{split}\]
\end{definition}
The above definition is from \cite[Definition 3.1]{Gig:15}. In the ensuing lemma, we collect basic properties of the
functions \(D^\pm g(\nabla f)\), which are direct consequences of Lemma \ref{lem:tech_for_grad} and \eqref{eq:tech_for_grad_conseq}.
\begin{proposition}\label{prop:main_Dg(nablaf)}
Let \((\X,\sfd,\mm)\) be a metric measure space and \(p\in[1,\infty)\). Then it holds that
\[
D^-g(\nabla f)\leq D^+g(\nabla f),\qquad|D^\pm g(\nabla f)|\leq|Df||Dg|
\]
for every \(f,g\in D^{1,p}_{\mm}(\X)\). In particular, \(D^\pm g(\nabla f)\in L^0(\mm)\) and \(D^\pm g(\nabla f)|Df|^{p-2}\in L^1(\mm)\). Also,
\begin{equation}\label{eq:est_Dg(nablaf)_for_dct}\begin{split}
&\lim_{\varepsilon\searrow 0}\bigg\|\1_{\{|Df|>0\}}^\mm\frac{|D(f+\varepsilon g)|^p-|Df|^p}{p\varepsilon}-D^+g(\nabla f)|Df|^{p-2}\bigg\|_{L^1(\mm)}=0,\\
&\lim_{\varepsilon\nearrow 0}\bigg\|\1_{\{|Df|>0\}}^\mm\frac{|D(f+\varepsilon g)|^p-|Df|^p}{p\varepsilon}-D^-g(\nabla f)|Df|^{p-2}\bigg\|_{L^1(\mm)}=0.
\end{split}\end{equation}
\end{proposition}
The following definition has been proposed in \cite[Definition 3.3 and (3.11)]{Gig:15}.
\begin{definition}[Infinitesimal strict convexity]\label{def:inf_strict_conv}
Let \((\X,\sfd,\mm)\) be a metric measure space. Then we say that \((\X,\sfd,\mm)\) is \textbf{\(q\)-infinitesimally strictly convex} for some exponent \(q\in(1,\infty)\) provided
\[
D^-g(\nabla f)=D^+g(\nabla f)\quad\text{ for every }f,g\in D^{1,p}_{\mm}(\X).
\]
\end{definition}
\subsubsection*{Gradients of Dirichlet functions}
\begin{definition}[Gradient]\label{def:gradient}
Let \((\X,\sfd,\mm)\) be a metric measure space and \(p\in[1,\infty)\). Fix any function \(f\in D^{1,p}_{\mm}(\X)\). Then we say that a vector field \(v\in L^q(T\X)\)
is a \textbf{\(q\)-gradient} of \(f\) if
\begin{equation}\label{eq:def_grad}
\d f(v)=|Df|^p,\qquad|v|=|Df|^{p/q}.
\end{equation}
We denote by \({\sf Grad}_q(f)\subseteq L^q(T\X)\) the set of all \(q\)-gradients of \(f\). Whenever \({\sf Grad}_q(f)\) is a singleton, we denote by \(\nabla_q f\) its unique element.
In the case where \(p=q=2\), we write \(\nabla f\) instead of \(\nabla_2 f\).
\end{definition}
In particular, \(\|v\|_{L^q(T\X)}=\|\d f\|_{L^p(T^*\X)}^{p/q}\) for every \(v\in{\sf Grad}_q(f)\). Remark \ref{rmk:conseq_HB} implies that
\begin{equation}\label{eq:grad_exist}
{\sf Grad}_q(f)\neq\varnothing\quad\text{ for every }f\in D^{1,p}_{\mm}(\X).
\end{equation}
However, \({\sf Grad}_q(f)\) might contain more than one element.
Assuming \({\sf Grad}_q(f)\) is a singleton for every \(f\in D^{1,p}_{\mm}(\X)\), it can still happen that \(\nabla_q\colon D^{1,p}_{\mm}(\X)\to L^q(T\X)\) is not linear (even for \(p=2\)).
\begin{remark}\label{rmk:equiv_cond_grad}{\rm
In the case \(p>1\), the defining property \eqref{eq:def_grad} of gradients can be equivalently rewritten in the more familiar form
\(\d f(v)=|Df|^p=|v|^q\) (in the limit case \(p=1\), one has to interpret \eqref{eq:def_grad} as follows: \(\d f(v)=|Df|\) and \(|v|=\1_{\{|Df|>0\}}^\mm\)).
Moreover, a variational characterization based on Young's inequality is the following
\begin{equation}\label{eq:equiv_grad_Young}
v\in{\sf Grad}_q(f)\quad\Longleftrightarrow\quad\int\d f(v)\,\d\mm\geq\frac{1}{p}\int|Df|^p\,\d\mm+\frac{1}{q}\int|v|^q\,\d\mm,
\end{equation}
where we adopt the convention that \(\frac{1}{q}\int|v|^q\,\d\mm\coloneqq 0\) if \(q=\infty\).
}\end{remark}
\begin{proposition}\label{prop:Dg(nablaf)_and_Grad}
Let \((\X,\sfd,\mm)\) be a metric measure space, \(p\in(1,\infty)\), and \(f,g\in D^{1,p}_{\mm}(\X)\). Then
\begin{equation}\label{eq:rel_grad_Dg(nablaf)}
D^-g(\nabla f)|Df|^{p-2}\leq\d g(v)\leq D^+g(\nabla f)|Df|^{p-2}\quad\text{ for every }v\in{\sf Grad}_q(f).
\end{equation}
In particular, for any given function \(f\in D^{1,p}_{\mm}(\X)\), we have that the set \({\sf Grad}_q(f)\) is a singleton if and only if \(D^-g(\nabla f)=D^+g(\nabla f)\)
for every \(g\in D^{1,p}_{\mm}(\X)\). In this case, it holds that
\begin{equation}\label{eq:rel_grad_Dg(nablaf)_2}
\d g(\nabla_q f)=D^+g(\nabla f)|Df|^{p-2}=D^-g(\nabla f)|Df|^{p-2}\quad\text{ for every }g\in D^{1,p}_{\mm}(\X).
\end{equation}
\end{proposition}
\begin{proof}
We know from \eqref{eq:def_grad} that \(\frac{1}{p}|Df|^p\leq\d f(v)-\frac{1}{q}|v|^q\). Also, for every \(\varepsilon>0\) and \(\sigma\in\{1,-1\}\)
we have \(\frac{1}{p}|D(f+\sigma\varepsilon g)|^p\geq\d(f+\sigma\varepsilon g)(v)-\frac{1}{q}|v|^q\) thanks to Young's inequality. It follows that
\[
|Df|^{p-2}\bigg(\1_{\{|Df|>0\}}^\mm\frac{|D(f+\sigma\varepsilon g)|^p-|Df|^p}{p\sigma\varepsilon|Df|^{p-2}}\bigg)\geq\frac{\d(f+\sigma\varepsilon g)(v)-\d f(v)}{\sigma\varepsilon}=\sigma\,\d g(v),
\]
where we used that \(\{|v|>0\}=\{|Df|>0\}\). Since \(\varepsilon\), \(\sigma\) are arbitrary, we deduce that \eqref{eq:rel_grad_Dg(nablaf)} holds.

Now let \(f\in D^{1,p}_{\mm}(\X)\) be fixed. On the one hand, assume that \({\sf Grad}_q(f)\) contains two distinct elements \(v\) and \(w\). Since \(\{\d g\,:\,g\in D^{1,p}_{\mm}(\X)\}\)
generates \(L^p(T^*\X)\), we can find \(g\in D^{1,p}_{\mm}(\X)\) such that \(\d g(v)\neq\d g(w)\), thus there exists a Borel set \(P\subseteq\X\) with $\mm(P)>0$ such that
\(\d g(v)\wedge\d g(w)<\d g(v)\vee\d g(w)\) holds \(\mm\)-a.e.\ on \(P\). Recalling \eqref{eq:rel_grad_Dg(nablaf)}, we deduce that \(D^-g(\nabla f)|Df|^{p-2}<D^+g(\nabla f)|Df|^{p-2}\)
holds \(\mm\)-a.e.\ on \(P\). Since \(\{|v|>0\}=\{|w|>0\}=\{|Df|>0\}\), we deduce that \(P\subseteq\{|Df|>0\}\) up to \(\mm\)-null sets, so that \(D^-g(\nabla f)\neq D^+g(\nabla f)\).
On the other hand, assume that \(D^-g(\nabla f)=D^+g(\nabla f)\) for every \(g\in D^{1,p}_{\mm}(\X)\). Given any \(v,w\in{\sf Grad}_q(f)\), we know from \eqref{eq:rel_grad_Dg(nablaf)} that
\(\d g(v)=\d g(w)\) for every \(g\in D^{1,p}_{\mm}(\X)\). Since \(\{\d g\,:\,g\in D^{1,p}_{\mm}(\X)\}\) generates \(L^p(T^*\X)\), we conclude that \(v=w\), namely that \({\sf Grad}_q(f)\) is
a singleton. Finally, when \(\#{\sf Grad}_q(f)=1\), \eqref{eq:rel_grad_Dg(nablaf)_2} immediately follows from \eqref{eq:rel_grad_Dg(nablaf)}.
\end{proof}
Clearly, Proposition \ref{prop:Dg(nablaf)_and_Grad} gives the following characterisation of infinitesimal strict convexity:
\begin{corollary}\label{cor:equiv_inf_strict_conv}
Let \((\X,\sfd,\mm)\) be a metric measure space and \(p\in(1,\infty)\). Then the space \((\X,\sfd,\mm)\) is \(q\)-infinitesimally strictly convex if and only if
\({\sf Grad}_q(f)\) is a singleton for every \(f\in D^{1,p}_{\mm}(\X)\).
\end{corollary}
By suitably adapting the proof of \cite[Proposition 2.3.7]{Gig:18} (or of \cite[Proposition 4.2.17]{Gig:Pas:20}) -- where only the case \(p=q=2\) is treated --
it is possible to prove that for any given functions \(f,g\in D^{1,p}_{\mm}(\X)\) there exist \(v^g_+,v^g_-\in{\sf Grad}_q(f)\) such that
\[
\d f(v^g_+)=D^+g(\nabla f)|Df|^{p-2},\qquad\d f(v^g_-)=D^-g(\nabla f)|Df|^{p-2}.
\]
However, we do not need this fact, thus we do not prove it.
\begin{proposition}
Let \((\X,\sfd,\mm)\) be a metric measure space and \(p\in[1,\infty)\). Then
\begin{equation}\label{eq:wstar_dense_tg_mod}
\bigg\{\sum_{i=1}^n\1_{E_i}^\mm v_i\;\bigg|\;n\in\N,\,(E_i)_{i=1}^n\text{ Borel subsets of }\X,\,(v_i)_{i=1}^n\subseteq\bigcup_{f\in D^{1,p}_{\mm}(\X)}{\sf Grad}_q(f)\bigg\}
\end{equation}
is weakly\(^*\) dense in \(L^q(T\X)\). In particular, if the tangent module \(L^q(T\X)\) is reflexive, then it holds that the set \(\bigcup_{f\in D^{1,p}_{\mm}(\X)}{\sf Grad}_q(f)\)
generates \(L^q(T\X)\) in the sense of \(L^q(\mm)\)-Banach \(L^\infty(\mm)\)-modules.
\end{proposition}
\begin{proof}
The family \(\mathcal V\) in \eqref{eq:wstar_dense_tg_mod} is a vector subspace of \(L^q(T\X)\) (since the sets \(E_i\) are not required to be disjoint). Moreover, let us denote
by \(\mathcal W\) the set of all those elements \(\omega\) of \(L^p(T^*\X)\) that can be written as \(\omega=\sum_{i=1}^n\1_{E_i}^\mm\d f_i\), where \(n\in\N\), \((E_i)_{i=1}^n\)
is a Borel partition of \(\X\), and \((f_i)_{i=1}^n\subseteq D^{1,p}_{\mm}(\X)\). Now let \(\omega=\sum_{i=1}^n\1_{E_i}^\mm\d f_i\in\mathcal W\) be fixed. For \(i=1,\ldots,n\), consider
any \(v_i\in{\sf Grad}_q(f_i)\). Then
\[\begin{split}
\omega\bigg(\sum_{i=1}^n\1_{E_i}^\mm v_i\bigg)&=\sum_{i=1}^n\1_{E_i}^\mm\,\d f_i(v_i)=\sum_{i=1}^n\1_{E_i}^\mm|Df_i|^p=|\omega|^p,\\
\bigg|\sum_{i=1}^n\1_{E_i}^\mm v_i\bigg|&=\sum_{i=1}^n\1_{E_i}^\mm|v_i|=\sum_{i=1}^n\1_{E_i}^\mm|Df_i|^{p/q}=|\omega|^{p/q}.
\end{split}\]
In particular, we have shown that
\begin{equation}\label{eq:wstar_dense_tg_mod_aux}
\forall\omega\in\mathcal W\quad\exists\,v\in\mathcal V:\quad\omega(v)=|\omega|^p,\quad|v|=|\omega|^{p/q}.
\end{equation}
We now claim that \(\mathcal V\) separates the points of \(L^p(T^*\X)\), which is equivalent to the weak\(^*\) density of \(\mathcal V\) in \(L^q(T\X)\).
To this aim, let \(\omega\in L^p(T^*\X)\setminus\{0\}\) be fixed. Since \(\mathcal W\) is strongly dense in \(L^p(T^*\X)\) by Remark \ref{rmk:comments_Ban_mod} iv),
we find a sequence \((\omega_n)_n\subseteq\mathcal W\) such that \(\omega_n\to\omega\) strongly in \(L^p(T^*\X)\). By virtue of \eqref{eq:wstar_dense_tg_mod_aux},
for any \(n\in\N\), there exists \(v_n\in\mathcal V\) such that \(\omega_n(v_n)=|\omega_n|^p\) and \(|v_n|=|\omega_n|^{p/q}\). Now observe that
\[
\bigg|\int(\omega-\omega_n)(v_n)\,\d\mm\bigg|\leq\|\omega-\omega_n\|_{L^p(T^*\X)}\|\omega_n\|_{L^p(T^*\X)}^{p/q}\quad\text{ for every }n\in\N
\]
by H\"{o}lder's inequality. As \(\|\omega_n\|_{L^p(T^*\X)}^{p/q}\to\|\omega\|_{L^p(T^*\X)}^{p/q}\) if \(p>1\), while \(\sup_n\|\omega_n\|_{L^p(T^*\X)}^{p/q}\leq 1\)
if \(p=1\), we deduce that \(\int(\omega-\omega_n)(v_n)\,\d\mm\to 0\). Therefore, we can conclude that
\[\begin{split}
\lim_{n\to\infty}\int\omega(v_n)\,\d\mm&=\lim_{n\to\infty}\bigg(\int(\omega-\omega_n)(v_n)\,\d\mm+\int\omega_n(v_n)\,\d\mm\bigg)
=\lim_{n\to\infty}\int\omega_n(v_n)\,\d\mm\\
&=\lim_{n\to\infty}\int|\omega_n|^p\,\d\mm=\int|\omega|^p\,\d\mm>0.
\end{split}\]
It follows that \(\int\omega(v_n)\,\d\mm\neq 0\) for \(n\in\N\) sufficiently large, thus \(\mathcal V\) separates the points of \(L^p(T^*\X)\).

Finally, let us assume that \(L^q(T\X)\) is reflexive. Then \(\mathcal V\) is weakly dense in \(L^q(T\X)\) by the first part of the statement. By Mazur's lemma,
it follows that \(\mathcal V\) is strongly dense in \(L^q(T\X)\), which implies that \(\bigcup_{f\in D^{1,p}_{\mm}(\X)}{\sf Grad}_q(f)\) generates \(L^q(T\X)\).
All in all, the statement is achieved.
\end{proof}
\subsubsection*{Infinitesimal Hilbertianity}
We recall next the definition of infinitesimal Hilbertianity, introduced in \cite[Definition 2.3.17]{Gig:18}.
\begin{definition}[Infinitesimal Hilbertianity]\label{def:inf_Hilb}
Let \((\X,\sfd,\mm)\) be a metric measure space. Then we say that \((\X,\sfd,\mm)\) is \textbf{infinitesimally Hilbertian}
provided \(W^{1,2}(\X)\) is a Hilbert space.
\end{definition}
By slightly adapting the proof of \cite[Proposition 2.3.17]{Gig:18}, one can prove the following result:
\begin{proposition}
Let \((\X,\sfd,\mm)\) be a metric measure space. Then the following are equivalent:
\begin{itemize}
\item[\(\rm i)\)] The space \((\X,\sfd,\mm)\) is infinitesimally Hilbertian.
\item[\(\rm ii)\)] The Cheeger \(2\)-energy functional \(\mathcal E_2\colon L^2(\mm)\to[0,\infty]\) is a quadratic form.
\item[\(\rm iii)\)] The cotangent module \(L^2(T^*\X)\) is a Hilbert module.
\item[\(\rm iv)\)] The tangent module \(L^2(T\X)\) is a Hilbert module.
\item[\(\rm v)\)] \((\X,\sfd,\mm)\) is \(2\)-infinitesimally strictly convex and \(\d g(\nabla f)=\d f(\nabla g)\) for every \(f,g\in D^{1,2}_\mm(\X)\).
\item[\(\rm vi)\)] \((\X,\sfd,\mm)\) is \(2\)-infinitesimally strictly convex and \(\nabla\colon D^{1,2}_\mm(\X)\to L^2(T\X)\) is linear.
\end{itemize}
\end{proposition}
\subsection{Laplacian measures}\label{ss:laplacian}
Building on our study of divergence measures and dual energies from the previous sections, we now propose a notion of a \(p\)-Laplacian for metric Sobolev functions.
\begin{definition}[Laplacian]\label{def:laplacian}
Let \((\X,\sfd,\mm)\) be a metric measure space and \(p\in[1,\infty)\). Fix a function \(f\in D^{1,p}_{\mm}(\X)\).
Then we say that a measure \(\mu\in\mathfrak M(\X)\) is a \textbf{\(p\)-Laplacian} of \(f\) if
\begin{equation}\label{eq:def_laplacian}
-\int\psi\,\d\mu\leq\int_{\{|Df|>0\}}\frac{|D(f+\psi)|^p-|Df|^p}{p}\,\d\mm\quad\text{ for every }\psi\in\LIP_{bs}(\X).
\end{equation}
We denote by \(\boldsymbol\Delta_p(f)\subseteq\mathfrak M(\X)\) the set of all \(p\)-Laplacians of \(f\).
Let us also define
\[
D(\boldsymbol\Delta_p)\coloneqq\big\{f\in D^{1,p}_{\mm}(\X)\;\big|\;\boldsymbol\Delta_p(f)\neq\varnothing\big\}.
\]
\end{definition}
After proving basic properties for the $p$-Laplacian, we comment on the definition in Remarks \ref{rmk:equiv_Lapl_p>1} and \ref{rem:consistency:definition}.
\begin{lemma}\label{lem:Delta_implies_div}
Let \((\X,\sfd,\mm)\) be a metric measure space and \(p\in[1,\infty)\). Fix a function \(f\in D(\boldsymbol\Delta_p)\)
and some Laplacian measure \(\mu\in\boldsymbol\Delta_p(f)\). Then
\begin{equation}\label{eq:Delta_implies_div}
-\int\psi\,\d\mu\leq\int|D\psi||Df|^{p-1}\,\d\mm\quad\text{ for every }\psi\in\LIP_{bs}(\X).
\end{equation}
In particular, it holds that \({\sf D}_q(\mu)\leq\||Df|\|_{L^p(\mm)}^{p/q}<+\infty\).
\end{lemma}
\begin{proof}
Fix any \((\varepsilon_n)_n\subseteq(0,1)\) with \(\varepsilon_n\searrow 0\). Given any \(\psi\in\LIP_{bs}(\X)\), we deduce from \eqref{eq:est_Dg(nablaf)_for_dct} that
\[\begin{split}
-\int\psi\,\d\mu&=-\lim_{n\to\infty}\frac{1}{\varepsilon_n}\int\varepsilon_n\psi\,\d\mu
\leq\lim_{n\to\infty}\int_{\{|Df|>0\}}\frac{|D(f+\varepsilon_n\psi)|^p-|Df|^p}{p\varepsilon_n}\,\d\mm\\
&=\int D^+\psi(\nabla f)|Df|^{p-2}\,\d\mm,
\end{split}\]
whence \eqref{eq:Delta_implies_div} follows by the \(\mm\)-a.e.\ inequality \(D^+\psi(\nabla f)\leq|Df||D\psi|\), proved in Proposition \ref{prop:main_Dg(nablaf)}.

Next, let us observe that \eqref{eq:Delta_implies_div} and an application of H\"{o}lder's inequality yield
\[
\int\psi\,\d\mu\leq\||D\psi|\|_{L^p(\mm)}\||Df|^{p-1}\|_{L^q(\mm)}\leq\|\lip_a(\psi)\|_{\mathcal L^p(\mm)}\||Df|\|_{L^p(\mm)}^{p/q}\quad\text{ for all }\psi\in\LIP_{bs}(\X).
\]
Recalling Remark \ref{rmk:comments_def_F_p} ii) and Theorem \ref{thm:F_p=D_q}, we conclude that \({\sf D}_q(\mu)={\sf F}_p(\mu)\leq\||Df|\|_{L^p(\mm)}^{p/q}\) 
($\leq 1$ in the case $p=1$).
\end{proof}

Lemma \ref{lem:Delta_implies_div} has the following immediate consequence:
\begin{corollary}\label{cor:incl_dom_Lapl_Dq}
Let \((\X,\sfd,\mm)\) be a metric measure space and \(p\in[1,\infty)\). Then it holds that
\[
\bigcup_{f\in D^{1,p}_{\mm}(\X)}\boldsymbol\Delta_p(f)\subseteq D({\sf D}_q).
\]
\end{corollary}
The main result of this section (i.e.\ Theorem \ref{thm:Lapl_are_div} below) states that the Laplacian of a function \(f\) is the divergence of a gradient of \(f\).
Its proof shares some similarities with that of Lemma \ref{lem:vector_field_induced_by_mu}.
\begin{theorem}[Laplacians are divergences of gradients]\label{thm:Lapl_are_div}
Let \((\X,\sfd,\mm)\) be a metric measure space and \(p\in[1,\infty)\). Fix any \(f\in D(\boldsymbol\Delta_p)\) and \(\mu\in\boldsymbol\Delta_p(f)\).
Then there exists \(w_\mu\in D(\boldsymbol\div_q;\X)\) such that
\[
\boldsymbol\div_q(w_\mu)=\mu,\qquad w_\mu\in{\sf Grad}_q(f).
\]
\end{theorem}
\begin{proof}
Let us consider the vector subspace \(\mathbb V\coloneqq\{\d\psi\,:\,\psi\in\LIP_{bs}(\X)\}\) of \(L^p(T^*\X)\). We define the function \(\Phi\colon\mathbb V\to\R\)
as \(\Phi(\d\psi)\coloneqq-\int\psi\,\d\mu\) for every \(\psi\in\LIP_{bs}(\X)\). Moreover, we define the function \(s\colon L^p(T^*\X)\to\R\) as
\(s(\omega)\coloneqq\int_{\{|Df|>0\}}S^+_{f,\omega}\,\d\mm\) for every \(\omega\in L^p(T^*\X)\) where \(S^+_{f,\omega}\) is defined as in Lemma \ref{lem:tech_for_grad}.
Observe that \(s\) is a sublinear function and that we can estimate
\begin{equation}\label{eq:Lapl_are_div_aux}
|s(\omega)|\leq\int|\omega||Df|^{p-1}\,\d\mm\leq\|\omega\|_{L^p(T^*\X)}\||Df|\|_{L^p(\mm)}^{p/q}\quad\text{ for every }\omega\in L^p(T^*\X).
\end{equation}
Recalling that \(\mu\in\boldsymbol\Delta_p(f)\) and applying the first line of \eqref{eq:est_Dg(nablaf)_for_dct}, we thus obtain that
\[
\Phi(\d\psi)=-\lim_{t\searrow 0}\frac{1}{t}\int t\psi\,\d\mu\leq s(\d\psi)\leq\|\d\psi\|_{L^p(T^*\X)}\||Df|\|_{L^p(\mm)}^{p/q}
\quad\text{ for every }\psi\in\LIP_{bs}(\X),
\]
which ensures that \(\Phi\) is well-defined. Since \(\Phi\) is linear, an application of the Hahn--Banach
theorem guarantees that \(\Phi\) can be extended to a linear functional \(\Phi\colon L^p(T^*\X)\to\R\) satisfying
\(\Phi(\omega)\leq s(\omega)\) for every \(\omega\in L^p(T^*\X)\). In light of \eqref{eq:Lapl_are_div_aux}, we deduce
that \(|\Phi(\omega)|\leq\|\omega\|_{L^p(T^*\X)}\||Df|\|_{L^p(\mm)}^{p/q}\) for every \(\omega\in L^p(T^*\X)\), whence
it follows that \(\Phi\in L^p(T^*\X)'\) and \(\|\Phi\|_{L^p(T^*\X)'}\leq\||Df|\|_{L^p(\mm)}^{p/q}\). Since
\[
s(-\d f)=\int_{\{|Df|>0\}}S^+_{f,-\d f}\,\d\mm
=\inf_{\varepsilon\in(0,1)}\frac{(1-\varepsilon)^p-1}{p\varepsilon}\int|Df|^p\,\d\mm=-\int|Df|^p\,\d\mm,
\]
we obtain \(\Phi(\d f)\geq -s(-\d f)=\int|Df|^p\,\d\mm=\|\d f\|_{L^p(T^*\X)}\||Df|\|_{L^p(\mm)}^{p/q}\),
which forces the identities
\[
\|\Phi\|_{L^p(T^*\X)'}=\||Df|\|_{L^p(\mm)}^{p/q},\qquad\Phi(\d f)=\|\d f\|_{L^p(T^*\X)}^p.
\]
Now let us define \(w_\mu\coloneqq\textsc{Int}_{L^p(T^*\X)}^{-1}(\Phi)\in L^q(T\X)\). Hence, we conclude that
\(\|w_\mu\|_{L^q(T\X)}=\||Df|\|_{L^p(\mm)}^{p/q}\) and
\[
\int\d\psi(w_\mu)\,\d\mm=\Phi(\d\psi)=-\int\psi\,\d\mu\quad\text{ for every }\psi\in\LIP_{bs}(\X).
\]
In particular, we have that \(w_\mu\in D(\boldsymbol\div_q;\X)\) and \(\boldsymbol\div_q(w_\mu)=\mu\).
Finally, it holds that
\[
\int\d f(w_\mu)\,\d\mm=\|\d f\|_{L^p(T^*\X)}^p=\frac{1}{p}\int|Df|^p\,\d\mm+\frac{1}{q}\int|w_\mu|^q\,\d\mm,
\]
where \(\frac{1}{q}\int|w_\mu|^q\,\d\mm\coloneqq 0\) if \(q=\infty\). This implies that \(w_\mu\in{\sf Grad}_q(f)\) thanks to Remark \ref{rmk:equiv_cond_grad}.
\end{proof}
\begin{remark}\label{rmk:equiv_Lapl_p>1}{\rm
When \( p \in (1,\infty) \), the defining property \eqref{eq:def_laplacian} of the \(p\)-Laplacian is equivalent to
\begin{equation}\label{eq:def_laplacian_bis}
-\int\psi\,\d\mu\leq\mathcal E_p(f+\psi)-\mathcal E_p(f)\quad\text{ for every }\psi\in\LIP_{bs}(\X).
\end{equation}
The fact that \eqref{eq:def_laplacian} implies \eqref{eq:def_laplacian_bis} is clear since $| D( f + \psi ) |^p - | D f |^p = |D\psi|^p$ almost everywhere on $\left\{ |Df| = 0 \right\}$. For the other implication,
the key observations are that
\[
\bigg|\int_{\{|Df|=0\}}\frac{|D(f+\varepsilon\psi)|^p-|Df|^p}{p\varepsilon}\,\d\mm\bigg|\leq\varepsilon^{p-1}\mathcal E_p(\psi)\to 0\quad\text{ as }\varepsilon\searrow 0
\]
and \((0,1)\ni\varepsilon\mapsto\int_{\{|Df|>0\}}\frac{|D(f+\varepsilon\psi)|^p-|Df|^p}{p\varepsilon}\,\d\mm\)
is a non-decreasing function, whence it follows that
\[\begin{split}
-\int\psi\,\d\mu&=-\lim_{\varepsilon\searrow 0}\frac{1}{\varepsilon}\int\varepsilon\psi\,\d\mu
\overset{\eqref{eq:def_laplacian_bis}}\leq\limi_{\varepsilon\searrow 0}\int_{\{|Df|>0\}}\frac{|D(f+\varepsilon\psi)|^p-|Df|^p}{p\varepsilon}\,\d\mm\\
&\leq\int_{\{|Df|>0\}}\frac{|D(f+\psi)|^p-|Df|^p}{p}\,\d\mm,
\end{split}\]
proving \eqref{eq:def_laplacian}.
Note, however, that condition \eqref{eq:def_laplacian} is more restrictive than \eqref{eq:def_laplacian_bis} when \(p=1\).}
\end{remark}

We remark that on \(q\)-infinitesimally strictly convex spaces the \(p\)-Laplacian is unique:
\begin{corollary}\label{cor:str_convex_Delta_implies_div}
Let \((\X,\sfd,\mm)\) be a \(q\)-infinitesimally strictly convex metric measure space, for some \(q\in(1,\infty)\).
Let \(f\in D(\boldsymbol\Delta_p)\) be given. Then \(\nabla_q f\in D(\boldsymbol\div_q;\X)\) and \(\boldsymbol\Delta_p(f)=\{\boldsymbol\div_q(\nabla_q f)\}\).
\end{corollary}

The following result is the converse to Theorem \ref{thm:Lapl_are_div}:
\begin{lemma}\label{lem:grad_impl_Lapl}
Let \((\X,\sfd,\mm)\) be a metric measure space and \(q\in(1,\infty]\). Let \(f\in D^{1,p}_{\mm}(\X)\) be given.
Assume that there exists a vector field \(v\in{\sf Grad}_q(f)\cap D(\boldsymbol\div_q;\X)\). Then
\(\boldsymbol\div_q(v)\in\boldsymbol\Delta_p(f)\).
\end{lemma}
\begin{proof}
For any \(\psi\in\LIP_{bs}(\X)\), by Young's inequality and the fact that \(v\in{\sf Grad}_q(f)\), we get
\[\begin{split}
-\int\psi\,\d\boldsymbol\div_q(v)&=\int\d\psi(v)\,\d\mm=\int_{\{|Df|>0\}}\d(f+\psi)(v)\,\d\mm-\int_{\{|Df|>0\}}\d f(v)\,\d\mm\\
&\leq\int_{\{|Df|>0\}}\frac{|D(f+\psi)|^p-|Df|^p}{p}\,\d\mm.
\end{split}\]
This proves that \(\boldsymbol\div_q(v)\in\boldsymbol\Delta_p(f)\), as desired.
\end{proof}
\begin{remark}\label{rmk:Gigli_laplacian}
{\rm 
A notion of \(p\)-Laplacian that is similar in spirit to the one proposed in this work has been introduced in \cite{Gig:15}, and in the case \(p=2\) the two notions coincide. In the case \(p\neq 2\), Proposition \ref{prop:Dg(nablaf)_and_Grad} shows that our notion is consistent with the 
usual notion of \(p\)-Laplacian in the Euclidean setting, namely with the identity `\(\Delta_p(f) ={\rm div}(|\nabla f|^{p-2}\nabla f)\)'.
}\end{remark}

It is of interest to consider also sub- and supersolutions for the $p$-Laplace equation. Thus, we propose the following definitions.
\begin{definition}[Sub/supersolutions of the \(p\)-Laplace equation]
Let \((\X,\sfd,\mm)\) be a metric measure space and \(p\in[1,\infty)\). Given any \(\mu\in\mathfrak M(\X)\), we say that \emph{$\mu$ has a \(p\)-subsolution $f \in D^{1,p}_{\mm}(\X)$}
provided
\begin{equation}\label{eq:def_laplacian:subsolution}
    -\int\psi\,\d\mu\leq\int_{\{|Df|>0\}}\frac{|D(f+\psi)|^p-|Df|^p}{p}\,\d\mm\quad\text{ for every nonpositive }\psi\in\LIP_{bs}(\X).
\end{equation}
Similarly, \emph{$\mu$ has a \(p\)-supersolution $f \in D^{1,p}_{\mm}(\X)$} provided
\begin{equation}\label{eq:def_laplacian:supersolution}
-\int\psi\,\d\mu\leq\int_{\{|Df|>0\}}\frac{|D(f+\psi)|^p-|Df|^p}{p}\,\d\mm\quad\text{ for every nonnegative }\psi\in\LIP_{bs}(\X).
\end{equation}
\end{definition}

In the lemma below, we show that measures \(\mu\in \mathfrak M(\X)\) admitting \(p\)-sub- or \(p\)-supersolutions are in the domain of the functional \({\sf D}_q\).
\begin{lemma}
Let \((\X, \sfd, \mm)\) be a metric measure space. Suppose that \(\mu\in \mathfrak M(\X)\) admits a \(p\)-subsolution (or a \(p\)-supersolution) \(f\in D^{1,p}_{\mm}(\X)\) for some \(p\in [1,\infty)\). Then
\[
{\sf D}_q(\mu)\leq\||Df|\|_{L^p(\mm)}^{p/q}<+\infty.
\]
\end{lemma}
\begin{proof}
We are going to prove the claim in the case \(\mu\) admits a \(p\)-subsolution \(f\in D^{1,p}_{\mm}(\X)\). The other case can be proved in the same fashion. By arguing as in the proof of \Cref{lem:Delta_implies_div}, \(\mu\) having a \(p\)-subsolution implies that
\begin{equation*}
    - \int \psi \,\d\mu \leq \int |Df|^{p-1} |D\psi| \,\d\mm
    \quad\text{for nonpositive $\psi \in \LIP_{bs}( \X )$.}
\end{equation*}
By dividing $\psi \in \LIP_{bs}( \X )$ into its positive $\psi^{+}$ and negative $\psi^{-}$ parts and recalling $|D\psi| = | D\psi^{+} | + |D\psi^{-}|$, we conclude that
\begin{equation*}
    \left| \int \psi \,\d\mu \right| \leq \int |Df|^{p-1}|D\psi|\,\d\mm
    \quad\text{for every $\psi \in \LIP_{bs}( \X )$.}
\end{equation*}
This implies that \({\sf D}_q(\mu)\leq\||Df|\|_{L^p(\mm)}^{p/q}\) ($\leq 1$ in the case $p=1$), as claimed.
\end{proof}
The existence of a sub- or a supersolution gives a simpler way to verify whether a measure is in the domain of the divergence.
\begin{remark}
{\rm
For the sake of consistency, we observe that whenever $\mu \in \mathfrak{M}( \X )$ 
is such that there exists \(f\in D^{1,p}_{\mm}(\X)\) that is both \(p\)-sub and \(p\)-supersolution of \(\mu\), then \(f\in D(\boldsymbol \Delta_p)\) and \(\mu \in \boldsymbol\Delta_p(f) \). Indeed, if $\varphi \in \LIP_{bs}( \X )$, then
\begin{align*}
    \1_{ \left\{ |Df| > 0 \right\} }( | D( f + \varphi ) |^p - | Df |^p )
    =
    \1_{ \left\{ |Df| > 0 \right\} \cap \left\{ \varphi \neq 0 \right\} }
    ( | D( f + \varphi ) |^p - | Df |^p ) 
\end{align*}
since $|Df| = |D(f+\varphi)|$ holds \(\mm\)-a.e.\ in $\left\{ \varphi = 0 \right\}$.

In fact, the positive part $\varphi^{+}$ satisfies $| D( f + \varphi ) | = | D( f + \varphi^{+} ) |$ \(\mm\)-a.e.\ on $\left\{ |Df| > 0 \right\} \cap \left\{ \varphi >0 \right\}$ and similarly $| D( f + \varphi ) | = | D( f - \varphi^{-} ) |$ \(\mm\)-a.e.\  on $\left\{ |Df| > 0 \right\} \cap \left\{ \varphi < 0 \right\}$ for the negative part $\varphi^{-}$. These observations and \eqref{eq:def_laplacian:subsolution} and \eqref{eq:def_laplacian:supersolution} imply the claim.
}
\end{remark}

\begin{remark}\label{rem:consistency:definition}
{\rm 
\Cref{def:laplacian} and \eqref{eq:def_laplacian_bis} imply that the following properties are equivalent when $p \in (1,\infty)$ and $(\X, \sfd)$ is a proper metric space:
\begin{itemize}
    \item \(0 \in \boldsymbol\Delta_p(f)\) for \(f\in D^{1,p}_{\mm}(\X)\);
    \item \(f\in D^{1,p}_{\mm}(\X)\) is a minimiser of the \(p\)-Cheeger energy \(\mathcal E_p\) in the sense of \cite[Definition 7.7]{Bj:Bj:11}.
\end{itemize}
Similarly, under the assumptions, the notion of sub/superminimisers in \cite[Definition 7.7]{Bj:Bj:11} coincide with the notion of sub/supersolutions in \eqref{eq:def_laplacian:subsolution} and \eqref{eq:def_laplacian:supersolution} for $\mu = 0$. The properness assumption is only used to make sure that the classes of compactly-supported Lipschitz functions and boundedly-supported Lipschitz functions coincide.

Further investigations in this direction are out of the scope of the present paper, but we point out the above relations for the sake of possible future developments. We refer the interested reader e.g.\ to \cite{Bj:Bj:11,Gig:Mon:13,Hei:Kil:Mar:12} for further reading on the topic.
}\end{remark}

\subsection{Condenser capacities}\label{ss:condenser_capacity}
As an application of our study of divergence measures and of the notion of \(p\)-Laplacian, 
we now address the minimisation problem arising in the definition of \(p\)-condenser capacity.
\medskip

Let \((\X,\sfd,\mm)\) be a metric measure space and \(p\in[1,\infty)\). Given two disjoint subsets \(E\), \(F\) of \(\X\),
we define the curve family \(\Gamma(E,F)\subseteq\mathscr C(\X)\) as
\[
\Gamma(E,F)\coloneqq\big\{\gamma\in\mathscr C(\X)\;\big|\;\gamma(a_\gamma)\in E,\,\gamma(b_\gamma)\in F\big\}.
\]
Moreover, we define the family \(\mathcal A(E,F)\subseteq W^{1,p}_{loc}(\X)\)
of \emph{admissible functions} for \(E\), \(F\) as
\[
\mathcal A(E,F)\coloneqq\pi_\mm(\bar{\mathcal A}(E,F)),
\]
where we set
\[
\bar{\mathcal A}(E,F)\coloneqq\Big\{\bar f\in \bar{N}^{1,p}_{loc}(\X)\;\Big|\;0\leq\bar f\leq 1,\,\bar f=0\text{ on }E,\,\bar f=1\text{ on }F\Big\}.
\]
\begin{definition}[Condenser capacity]\label{def:condens_cap}
Let \((\X,\sfd,\mm)\) be a metric measure space and \(p\in[1,\infty)\). Let \(E\), \(F\) be two disjoint subsets of \(\X\).
Then we define the \textbf{\(p\)-condenser capacity} of \((E,F)\) as
\[
{\rm Cap}_p(E,F)\coloneqq\inf_{f\in\mathcal A(E,F)}p\,\mathcal E_p(f)\in[0,\infty].
\]
\end{definition}
\begin{remark}{\rm
Observe that, in the case where \(\mm(\X)<+\infty\), we have that
\[
{\rm Cap}_p(E,F)=\inf\bigg\{p\,\mathcal E_p(\pi_\mm(\bar f))\;\bigg|\;\bar f\in \bar{N}^{1,p}(\X),\,0\leq\bar f\leq 1,\,\bar f=0\text{ on }E,\,\bar f=1\text{ on }F\bigg\}
\]
for every couple \(E\), \(F\) of disjoint subsets of \(\X\).
}\end{remark}
To address the main result of this section (i.e.\ Theorem \ref{thm:condens_cap}), we need some auxiliary results:
\begin{lemma}\label{lem:limit_A(E,F)}
Let \((\X,\sfd,\mm)\) be a metric measure space and \(p\in(1,\infty)\).
Let \(E\) and \(F\) be given disjoint subsets of \(\X\). Assume that \((f_n)_n\subseteq\mathcal A(E,F)\) satisfies \(\limi_n\mathcal E_p(f_n)<+\infty\).
Then there exists a function \(f\in\mathcal A(E,F)\) such that \(\mathcal E_p(f)\leq\limi_n\mathcal E_p(f_n)\).
\end{lemma}
\begin{proof}
Fix \(\bar x\in\X\). Notice that \(\|\1_{B_k}^\mm f_n\|_{L^p(\mm)}\leq\mm(B_k)^{1/p}\) for every \(k,n\in\N\), where \(B_k\coloneqq B_k(\bar x)\).
Since \(L^p(\mm)\) is reflexive, we can thus find \(f\in L^p_{loc}(\mm)\) and \(\rho\in L^p(\mm)^+\) such that, up to a non-relabelled subsequence,
\(\1_{B_k}^\mm f_n\rightharpoonup\1_{B_k}^\mm f\) and \(|Df_n|\rightharpoonup\rho\) weakly in \(L^p(\mm)\) for every \(k\in\N\).
Notice that \(0\leq f\leq 1\) and \(\|\rho\|_{L^p(\mm)}^p\leq C\coloneqq\limi_n p\,\mathcal E_p(f_n)\). For every \(n\in\N\), pick some function \(\bar f_n\in\bar{\mathcal A}(E,F)\)
such that \(\pi_\mm(\bar f_n)=f_n\). By Mazur's lemma and a diagonalisation argument, we can find \(h_n\in{\rm conv}(\{\bar f_i\,:\,i\geq n\})\) and \(\rho_n\in L^p(\mm)^+\)
with \(\rho_n\geq|Dh_n|\) such that \(\1_{B_k}^\mm\pi_\mm(h_n)\to\1_{B_k}^\mm f\) and \(\rho_n\to\rho\) strongly in \(L^p(\mm)\) for every \(k\in\N\). Since
\({\sf qcr}(\pi_\mm(h_n))=h_n\) in the \({\rm Cap}_p\)-a.e.\ sense, by applying Corollary \ref{cor:good_repres_Sob_loc} we deduce that \(h_n\to{\sf qcr}(f)\)
holds \({\rm Cap}_p\)-a.e.\ on \(\X\). 
Having fixed some representative \(\tilde f\in \bar{N}^{1,p}_{loc}(\X)\) of \(f\), which satisfies \(\tilde f={\sf qcr}(f)\) in the \({\rm Cap}_p\)-a.e.\ sense, we have that
\[
N\coloneqq\big\{x\in E\;\big|\;\tilde f(x)\neq 0\big\}\cup\big\{x\in F\;\big|\;\tilde f(x)\neq 1\big\}\cup\big\{x\in\X\;\big|\;\tilde f(x)\notin[0,1]\big\}
\]
satisfies \({\rm Cap}_p(N)=0\). Therefore, \(\bar f\coloneqq\1_{\X\setminus N}\tilde f+\1_{F\cap N}\in \bar{N}^{1,p}_{loc}(\X)\)
(thanks to Remark \ref{rmk:N1p_Cap-ae_inv}) and \(f=\pi_\mm(\bar f)\in\mathcal A(E,F)\). Finally, we have that
\(p\,\mathcal E_p(f)\leq\|\rho\|_{L^p(\mm)}^p\leq C\), concluding the proof.
\end{proof}
We recall some basic measure theory for the following couple of lemmas. A topological space $\Y$ is Polish if $\Y$ is homeomorphic to a complete and separable metric space. A subset $E \subseteq \Y$ is \emph{Souslin} if there exists a Polish space $\Z$ and a continuous map $f \colon \Z \to \Y$ for which $f(\Z) = E$. Moreover, 
for \(E\subseteq Y\) Souslin, recall that if $\W$ is a Polish space and $g \colon \W \to \Y$ and $h \colon \Y \to \W$ are continuous, then $g^{-1}(E) \subseteq \W$ and $h(E) \subseteq \W$ are Souslin as well. We recall that a Borel set in a Polish space is Souslin and that countable intersections and unions of Souslin sets are Souslin sets. Furthermore, every Souslin set in a metric measure space $( \X, \sfd, \mm )$ is $\mm$-measurable; see e.g. \cite[Chapter 6]{Bog:07} for these facts.
\begin{lemma}\label{lem:aux_upper_grad}
Let \((\X,\sfd,\mm)\) be a metric measure space and \(p\in[1,\infty)\). Let \(E,F\subseteq\X\) be disjoint sets with \(E\) being Souslin.
Let \(\rho\in\mathcal L^p(\mm)^+\) be a given competitor for \({\rm Mod}_p(\Gamma(E,F))\). Then there exists a function \(f\in\mathcal A(E,F)\)
such that \(|Df|\leq\rho\) holds \(\mm\)-a.e.\ on \(\X\).
\end{lemma}
\begin{proof}
Given any \(y\in\X\), we denote \(\Gamma(y)\coloneqq\big\{\gamma\in C([0,1];\X)\,:\,\gamma_0\in E,\,\gamma_1=y\big\}\). We also define
\[
f(y)\coloneqq 1\wedge\inf_{\gamma\in\Gamma(y)}\int_\gamma\rho\,\d s\quad\text{ for every }y\in\X.
\]
We verify that the resulting function \(f\colon\X\to[0,1]\) is \(\mm\)-measurable; our argument is inspired by \cite[Corollary 1.10]{Ja:Ja:Ro:Ro:Sha:07}.
Our goal is to show that \(\{f<\lambda\}\) is an \(\mm\)-measurable subset of \(\X\) for any given \(\lambda\in(0,1]\). Recall that \(C([0,1];\X)\) is a complete
and separable metric space, and that \(C([0,1];\X)\ni\gamma\mapsto{\rm \pi}_\rho(\gamma)\coloneqq\int_\gamma\rho\,\d s\) is a Borel function (cf.\ \eqref{eq:meas_path_int}). Observe that
\[
\{f<\lambda\}=\e_1\big(\e_0^{-1}(E)\cap{\rm \pi}_\rho^{-1}([0,\lambda))\big),
\]
where the evaluation maps \(\e_0\), \(\e_1\) are defined as in \eqref{eq:def_e_t}. Since \(\e_0\), \(\e_1\) are continuous, we deduce that \(\{f<\lambda\}\) is Souslin, and thus \(\mm\)-measurable. It also follows that $\{ f = 1 \} = \X \setminus \{ f < 1 \}$ is $\mm$-measurable. We conclude that $f$ is $\mm$-measurable.

Now notice that if \(y\in F\), then \(\Gamma(y)\subseteq\Gamma(E,F)\), so that \(\int_\gamma\rho\,\d s\geq 1\) for every
\(\gamma\in\Gamma(y)\) and thus \(f=1\) on \(F\). On the other hand, if \(y\in E\), then \(\Gamma(y)\) contains the constant curve at \(y\), so that \(f(y)=0\)
and thus \(f=0\) on \(E\). We now claim that \(\rho\) is an upper gradient of \(f\), namely that
\begin{equation}\label{eq:aux_upper_grad}
|f(\sigma(b_\sigma))-f(\sigma(a_\sigma))|\leq\int_\sigma\rho\,\d s\quad\text{ for every }\sigma\in\mathscr R(\X).
\end{equation}
To this end, fix any \(\sigma\in\mathscr R(\X)\) and assume that \(\int_\sigma\rho\,\d s<+\infty\). If \(\Gamma(\sigma(a_\sigma))=\varnothing\), then \(\Gamma(\sigma(b_\sigma))=\varnothing\)
and thus \(f(\sigma(a_\sigma))=f(\sigma(b_\sigma))=1\). Now, if \(\Gamma(\sigma(a_\sigma))\neq\varnothing\), then for every \(\gamma\in\Gamma(\sigma(a_\sigma))\) we can estimate
\(f(\sigma(b_\sigma))\leq 1\wedge\int_\gamma\rho\,\d s+\int_\sigma\rho\,\d s\), whence it follows that \(f(\sigma(b_\sigma))\leq f(\sigma(a_\sigma))+\int_\sigma\rho\,\d s\).
Repeating the same argument with \(b_\sigma\) instead of \(a_\sigma\), we obtain that \(f(\sigma(a_\sigma))\leq f(\sigma(b_\sigma))+\int_\sigma\rho\,\d s\). All in all,
\eqref{eq:aux_upper_grad} is proved. In particular, \(f\in \bar{N}^{1,p}(\X)\) and \(|Df|\leq\rho\) holds \(\mm\)-a.e.\ on \(\X\).
\end{proof}

Notice that the proof of Lemma \ref{lem:aux_upper_grad} shows something stronger than what is claimed in its statement, namely that there exists
a function \(\bar f\in\bar{\mathcal A}(E,F)\) such that \(\rho\) is an upper gradient of \(\bar f\).
\begin{proposition}\label{prop:Mod_inf_E_p}
Let \((\X,\sfd,\mm)\) be a metric measure space and \(p\in[1,\infty)\).
Let \(E,F\subseteq\X\) be disjoint sets with \(E\) being Souslin. Then it holds that
\begin{equation}\label{eq:Mod_equal_inf_Ep}
{\rm Mod}_p(\Gamma(E,F))={\rm Cap}_p(E,F).
\end{equation}
\end{proposition}
\begin{proof}
On the one hand, let \(f=\pi_\mm(\bar f)\in\mathcal A(E,F)\) be fixed, where \(\bar f\in\bar{\mathcal A}(E,F)\). Let \(\rho\in\mathcal L^p(\mm)^+\) be any weak \(p\)-upper gradient of \(\bar f\).
Then there exists a family of curves \(\mathcal N\) in \(\X\) such that \({\rm Mod}_p(\mathcal N)=0\) and \(|\bar f(\gamma(b_\gamma))-\bar f(\gamma(a_\gamma))|\leq\int_\gamma\rho\,\d s\)
for every \(\gamma\in\mathscr R(\X)\setminus\mathcal N\). In particular,
\[
1=\bar f(\gamma(b_\gamma))-\bar f(\gamma(a_\gamma))\leq\int_\gamma\rho\,\d s\quad\text{ for every }\gamma\in\Gamma(E,F)\setminus\mathcal N.
\]
This means that \(\rho\) is a competitor for \({\rm Mod}_p(\Gamma(E,F)\setminus\mathcal N)\), thus accordingly
\[
{\rm Mod}_p(\Gamma(E,F))={\rm Mod}_p(\Gamma(E,F)\setminus\mathcal N)\leq\|\rho\|_{\mathcal L^p(\mm)}^p.
\]
Choosing \(\rho\) of minimal \(\mathcal L^p(\mm)\)-seminorm among all weak \(p\)-upper gradients of \(\bar f\), we thus deduce that \({\rm Mod}_p(\Gamma(E,F))\leq p\,\mathcal E_p(f)\).
As \(f\in\mathcal A(E,F)\) is arbitrary, the inequality \(\leq\) in \eqref{eq:Mod_equal_inf_Ep} is proved.

On the other hand, given any competitor \(\rho\in\mathcal L^p(\mm)^+\) for \({\rm Mod}_p(\Gamma(E,F))\), by Lemma \ref{lem:aux_upper_grad}, we can find \(f\in\mathcal A(E,F)\) such that \(|Df|\leq\rho\) holds \(\mm\)-a.e.\ on \(\X\). Then \({\rm Cap}_p(E,F)\leq\|\rho\|_{\mathcal L^p(\mm)}\),
whence the validity of the inequality \(\geq\) in \eqref{eq:Mod_equal_inf_Ep} follows thanks to the arbitrariness of \(\rho\).
\end{proof}
\begin{theorem}\label{thm:condens_cap}
Let \((\X,\sfd,\mm)\) be a metric measure space and \(p\in(1,\infty)\). Let \(E,F\subseteq\X\) be disjoint Souslin sets with \(0<{\rm Mod}_p(\Gamma(E,F))<+\infty\). Then there exists a function
\(f\in\mathcal A(E,F)\) such that
\begin{equation}\label{eq:condens_cap_min_Ch}
p\,\mathcal E_p(f)={\rm Cap}_p(E,F)={\rm Mod}_p(\Gamma(E,F)).
\end{equation}
Assuming, in addition, that \(\mm\) is a finite measure, there exists
a \(p\)-Laplacian \(\mu\in\boldsymbol\Delta_p(f)\cap\mathcal M(\X)\) of the function \(f\) such that \(\mu^+\) and \(\mu^-\) are concentrated on \(E\) and \(F\), respectively.
\end{theorem}
\begin{proof}
Consider a sequence \((f_n)_n\subseteq\mathcal A(E,F)\) such that \(p\,\mathcal E_p(f_n)\to{\rm Mod}_p(\Gamma(E,F))\),
whose existence is ensured by Proposition \ref{prop:Mod_inf_E_p}. Then by Lemma \ref{lem:limit_A(E,F)} we can find a
function \(f\in\mathcal A(E,F)\) such that \(p\,\mathcal E_p(f)\leq\lim_n p\,\mathcal E_p(f_n)={\rm Mod}_p(\Gamma(E,F))\),
whence the identities in \eqref{eq:condens_cap_min_Ch} follow.

Let us pass to the verification of the last part of the statement. Let $$\widetilde{\Gamma}(E,F) = \left\{ \gamma \in C([0,1];\X) \colon \gamma \in \e_0^{-1}(E) \cap \e_1^{-1}(F) \right\}$$ for the evaluation maps \(\e_0\), \(\e_1\) as in \eqref{eq:def_e_t}. Then $\widetilde{\Gamma}(E,F) \subseteq C( [0,1]; \X )$ is a Souslin set. It also follows that $\bar{\Gamma}(E,F) \coloneqq \widetilde{\Gamma}(E,F) \cap \left\{ 0 < \ell( \gamma ) < \infty \right\}$ is a Souslin subset of $C( [0,1]; \X )$ and we have
\begin{align*}
    \Mod_p( \Gamma(E,F) ) = \Mod_p( \bar{\Gamma}(E,F) )
\end{align*}
since we recall that $\Mod_p$ is invariant under reparametrizations; $\Gamma(E,F)$ does not contain any constant curves since $\Mod_p( \Gamma(E,F) ) < \infty$; and the family of unrectifiable curves is negligible for $\Mod_p$.

Now by the modulus-plans duality \cite[Lemma 6.7 and Theorem 7.2]{Amb:Mar:Sav:15} (see also \cite[Theorem 4.3.2]{Sav:22}), there exists a plan \(\ppi\) on \(\X\) with barycenter in \(L^q(\mm)\) that is concentrated on \(\bar{\Gamma}(E,F) \) and satisfies
\[
\|{\rm Bar}(\ppi)\|_{L^q(\mm)}^q=\ppi(\bar\Gamma(E,F))={\rm Mod}_p(\Gamma(E,F))=p\,\mathcal E_p(f).
\]
Fix a quasicontinuous representative \(\bar f\in \bar{N}^{1,p}(\X)\) of \(f\) and a representative \(\rho\in\mathcal L^p(\mm)^+\) of \(|Df|\).
Since \(\ppi\ll{\rm Mod}_p\) (by \eqref{eq:plan_ll_Mod}) and \((\gamma_0,\gamma_1)\in E\times F\) for \(\ppi\)-a.e.\ \(\gamma\),
we have \(1=\bar f(\gamma_1)-\bar f(\gamma_0)\)
for \(\ppi\)-a.e.\ \(\gamma\). Integrating with respect to \(\ppi\)  we get that
\(
\ppi(\bar\Gamma(E,F))=\int\bar f(\gamma_1)-\bar f(\gamma_0)\,\d\ppi(\gamma)
\).
Together with the above proven equality \(\ppi(\bar\Gamma(E,F))=p\mathcal E_p(f)\), the latter gives the validity of the identity
\begin{equation}\label{eq:condens_cap_1}
\int\bar f(\gamma_1)-\bar f(\gamma_0)\,\d\ppi(\gamma)=p\,\mathcal E_p(f).
\end{equation}
Denote \(\mu\coloneqq(\e_0)_\#\ppi-(\e_1)_\#\ppi\in\mathcal M(\X)\) for brevity. Thanks to Proposition \ref{prop:B_q=D_q}, Theorem \ref{thm:F_p=D_q},
and Lemma \ref{lem:vector_field_induced_by_mu}, we find a vector field \(v_\mu\in D(\boldsymbol\div_q;\X)\subseteq L^q(T\X)\) such that \(\boldsymbol\div_q(v_\mu)=\mu\),
\[
\|v_\mu\|_{L^q(T\X)}={\sf B}_q(\mu)\leq\|{\rm Bar}(\ppi)\|_{L^q(\mm)}=(p\,\mathcal E_p(f))^{1/q},
\]
and \(\int\d f(v_\mu)\,\d\mm=-L_\mu(f)\). Therefore, by \eqref{eq:condens_cap_1} and Remark \ref{rmk:L_mu_for_mu_finite}, we obtain that
\[\begin{split}
\int\d f(v_\mu)\,\d\mm&=-\int\bar f\,\d\mu=\int\bar f(\gamma_1)-\bar f(\gamma_0)\,\d\ppi(\gamma)=p\,\mathcal E_p(f)\\
&=\mathcal E_p(f)+\frac{p\,\mathcal E_p(f)}{q}\geq\frac{1}{p}\int|Df|^p\,\d\mm+\frac{1}{q}\int|v_\mu|^q\,\d\mm.
\end{split}\]
Recalling Remark \ref{rmk:equiv_cond_grad}, we deduce that \(v_\mu\in{\sf Grad}_q(f)\). Hence, Lemma \ref{lem:grad_impl_Lapl}
gives that \(\mu\in\boldsymbol\Delta_p(f)\). Finally, notice that \(\mu^+=(\e_0)_\#\ppi\) and \(\mu^-=(\e_1)_\#\ppi\) are concentrated on \(E\) and \(F\), respectively.
\end{proof}
\begin{remark}{\rm
We assume a finite reference measure in Theorem \ref{thm:condens_cap} due to the 
fact that the modulus-plan duality results in \cite[Theorem 7.2]{Amb:Mar:Sav:15} and in \cite[Theorem 4.3.2]{Sav:22} have been proven in this setting.
}\end{remark}
\section{Duality of Sobolev spaces}\label{ss:reflexivity_properties}
In this section, we investigate duals and preduals of metric Sobolev spaces. As an application, we will provide new characterisations of the reflexivity
of the Sobolev space \(W^{1,p}(\X)\) for \(p\in(1,\infty)\) (Theorem \ref{thm:predual_W1p}), but first we need to recall some basic definitions and
results in convex analysis.
\medskip

Let \(\mathbb V\) be a normed space. Recall that the \textbf{Fenchel conjugate} \(f^*\colon\mathbb V'\to[-\infty,+\infty]\) of a given function \(f\colon\mathbb V\to(-\infty,+\infty]\) is defined as
\[
f^*(\omega)\coloneqq\sup\big\{\langle\omega,x\rangle-f(x)\big|\;x\in\mathbb V\big\}\quad\text{ for every }\omega\in\mathbb V'.
\]
\begin{theorem}[Conjugate of the sum {\cite[Theorem 20]{Ro:74}}]\label{thm:conj_sum}
Let \(\mathbb V\) be a normed space. Let \(f,g\colon\mathbb V\to(-\infty,+\infty]\) be convex functionals.
Assume that there exists \(x_0\in\mathbb V\) such that \(g(x_0)<+\infty\) and \(f\) is continuous at \(x_0\). Then it holds that
\[
(f+g)^*(\omega)=\inf\big\{f^*(\omega-\eta)+g^*(\eta)\;\big|\;\eta\in\mathbb V'\big\}\quad\text{ for every }\omega\in\mathbb V'.
\]
\end{theorem}
Let us also recall that the \textbf{adjoint} of a bounded linear operator \(A\colon\mathbb W\to\mathbb V\) between
normed spaces \(\mathbb W\), \(\mathbb V\) is defined as the unique map \(A^*\colon\mathbb V'\to\mathbb W'\) such that
\[
\langle A^*(\eta),x\rangle=\langle\eta,A(x)\rangle\quad\text{ for every }x\in\mathbb W\text{ and }\eta\in\mathbb V'.
\]
Moreover, \(A^*\) is a bounded linear operator whose operator norm coincides with the one of \(A\).
For the proof of the next result we refer to \cite[Theorem 5.1]{BBS} (see also \cite[Theorem 17]{Ro:74} for a more general version).
\begin{proposition}[Conjugate of the composition]\label{prop:conj_compos}
Let \(\mathbb W\), \(\mathbb V\) be Banach spaces. Let \(A\colon\mathbb W\to\mathbb V\) be a bounded linear operator.
Let \(f\colon\mathbb V\to(-\infty,+\infty]\) be a convex, lower semicontinuous functional. Assume that there exists
\(x_0\in\mathbb W\) such that \(f\) is continuous at \(A(x_0)\). Then it holds that
\[
(f\circ A)^*(\omega)=\inf\big\{f^*(\eta)\;\big|\;\eta\in\mathbb V',\,A^*(\eta)=\omega\big\}\quad\text{ for every }\omega\in\mathbb W'.
\]
\end{proposition}
\subsection{Duals and preduals of Sobolev spaces}
\begin{definition}[The spaces \(W^{-1,q}(\X)\) and \(W^{-1,q}_{pd}(\X)\)]\label{def:pd_Sob}
Let \((\X,\sfd,\mm)\) be a metric measure space and \(p\in[1,\infty)\). Then we define \(W^{-1,q}(\X)\) as the dual
of the Sobolev space \(W^{1,p}(\X)\), i.e.\ we set
\[
W^{-1,q}(\X)\coloneqq W^{1,p}(\X)'.
\]
Moreover, we define the vector subspace
\[
W^{-1,q}_{pd}(\X)\subseteq W^{-1,q}(\X)
\]
as the set of \(L\in W^{-1,q}(\X)\) having the following stronger continuity
property: if \((f_n)_{n\in\N\cup\{\infty\}}\subseteq W^{1,p}(\X)\) satisfies \(\|f_n-f_\infty\|_{L^p(\mm)}\to 0\)
and \(\sup_{n\in\N}\|f_n\|_{W^{1,p}(\X)}<+\infty\), then \(L(f_n)\to L(f_\infty)\).
\end{definition}

We claim that \(W^{-1,q}_{pd}(\X)\) is in fact a closed vector subspace of \(W^{-1,q}(\X)\). To prove it, assume
that \((L^k)_k\subseteq W^{-1,q}_{pd}(\X)\) and \(L\in W^{-1,q}(\X)\) satisfy \(\|L^k-L\|_{W^{-1,q}(\X)}\to 0\).
Given any bounded sequence \((f_n)_{n\in\N\cup\{\infty\}}\subseteq W^{1,p}(\X)\) with \(\|f_n-f_\infty\|_{L^p(\mm)}\to 0\),
we can estimate
\[
|L(f_n)-L(f_\infty)|\leq\|L-L^k\|_{W^{-1,q}(\X)}\big(\|f_n\|_{W^{1,p}(\X)}+\|f_\infty\|_{W^{1,p}(\X)}\big)+|L^k(f_n)-L^k(f_\infty)|.
\]
Letting first \(n\to\infty\) and then \(k\to\infty\), we conclude that \(L(f_n)\to L(f_\infty)\). Hence \(L\in W^{-1,q}_{pd}(\X)\).
\begin{remark}\label{rmk:W-1q_pd_cont_in_nrg}{\rm
Notice that \(L\colon(W^{1,p}(\X),\sfd_{\rm en})\to\R\) is continuous for every \(L\in W^{-1,q}_{pd}(\X)\).
}\end{remark}
Given any parameter \(\kappa\in[0,\infty)\), we define the following weighted (semi)norms:
\begin{equation}\label{eq:def_W1pk_norms}\begin{split}
\|f\|_{W^{1,p}_\kappa(\X)}&\coloneqq\kappa\|f\|_{L^p(\mm)}+\||Df|\|_{L^p(\mm)}\quad\text{ for every }f\in W^{1,p}(\X),\\
\|L\|_{op,\kappa}&\coloneqq\sup\big\{L(f)\;\big|\,f\in W^{1,p}(\X),\,\|f\|_{W^{1,p}_\kappa(\X)}\leq 1\big\}\,\,\text{ for every }L\in W^{-1,q}(\X).
\end{split}\end{equation}
Notice that if \(\kappa>0\), then \(\|\cdot\|_{W^{1,p}_\kappa(\X)}\) is a norm on \(W^{1,p}(\X)\) equivalent
to \(\|\cdot\|_{W^{1,p}(\X)}\), while \(\|\cdot\|_{op,\kappa}\) is a norm on \(W^{-1,q}(\X)\) equivalent to
\(\|\cdot\|_{W^{-1,q}(\X)}\). More precisely, it holds that
\[\begin{split}
(\kappa\wedge 1)\|f\|_{W^{1,p}(\X)}\leq\|f\|_{W^{1,p}_\kappa(\X)}\leq 2^{1/p}(\kappa\vee 1)\|f\|_{W^{1,p}(\X)}&
\quad\text{ for every }f\in W^{1,p}(\X),\\
\frac{1}{2^{1/q}(\kappa\vee 1)}\|L\|_{W^{-1,q}(\X)}\leq\|L\|_{op,\kappa}\leq\frac{1}{\kappa\wedge 1}\|L\|_{W^{-1,q}(\X)}&
\quad\text{ for every }L\in W^{-1,q}(\X).
\end{split}\]
A distinguished subspace of \(W^{-1,q}_{pd}(\X)\) consists of those elements of the form \(L_g\), defined as follows:
\begin{definition}[The functionals \(L_g\)]\label{def:L_g}
Let \((\X,\sfd,\mm)\) be a metric measure space and \(q\in(1,\infty]\). Let \(g\in L^q(\mm)\) be a given function.
Then we define the functional \(L_g\colon W^{1,p}(\X)\to\R\) as
\[
L_g(f)\coloneqq\int fg\,\d\mm\quad\text{ for every }f\in W^{1,p}(\X).
\]
\end{definition}

It can be readily checked that \(L_g\in W^{-1,q}_{pd}(\X)\) for every \(g\in L^q(\mm)\) and that
\[
\|L_g\|_{op,\kappa}\leq\frac{\|g\|_{L^q(\mm)}}{\kappa}\quad\text{ for every }\kappa>0.
\]
The notation \(L_g\) is consistent with the one in Theorem \ref{thm:L_mu}, thanks to the next observation:
\begin{remark}\label{rmk:consist_L_g_and_L_mu}{\rm
Fix any \(b\in{\rm Der}^q_q(\X)\). Letting \(L_{\div(b)\mm}\) be as in Theorem \ref{thm:L_mu}, we claim that
\[
L_{\div(b)}=L_{\div(b)\mm},\qquad\|L_{\div(b)}\|_{W^{-1,q}(\X)}\leq\|b\|_{{\rm Der}^q(\X)}.
\]
Indeed, we have that \(L_{\div(b)}\colon(W^{1,p}(\X),\sfd_{\rm en})\to\R\) is a continuous function (see Remark
\ref{rmk:W-1q_pd_cont_in_nrg}) and that \(L_{\div(b)}(f)=\int f\,\d(\div(b)\mm)\) for every \(f\in\LIP_{bs}(\X)\),
so that \(L_{\div(b)}=L_{\div(b)\mm}\) thanks to the uniqueness statement in Theorem \ref{thm:L_mu}.
Moreover, Theorem \ref{thm:L_mu} and Theorem \ref{thm:F_p=D_q} imply that the inequalities
\(|L_{\div(b)}(f)|\leq{\sf F}_p(\div(b)\mm)\||Df|\|_{L^p(\mm)}\leq\|b\|_{{\rm Der}^q(\X)}\|f\|_{W^{1,p}(\X)}\)
hold for every \(f\in W^{1,p}(\X)\), so that \(\|L_{\div(b)}\|_{W^{-1,q}(\X)}\leq\|b\|_{{\rm Der}^q(\X)}\).
All in all, the claim is proved.
}\end{remark}

The next result states that the elements of \(W^{-1,q}_{pd}(\X)\) can be approximated by functionals \(L_g\)'s.
\begin{lemma}\label{lem:approx_W-1q_pd}
Let \((\X,\sfd,\mm)\) be a metric measure space and \(q\in(1,\infty]\). Let \(L\in W^{-1,q}_{pd}(\X)\) be given.
Then there exists a sequence \((g_n)_n\subseteq L^q(\mm)\) such that \(\|L_{g_n}-L\|_{op,0}\to 0\) as \(n\to\infty\).
\end{lemma}
\begin{proof}
We define \(A\), \(\varphi_\alpha\), \(\psi_\beta\) for \(\alpha\in[0,\infty)\) and \(\beta\in(0,\infty)\) as
\[\begin{split}
A\colon W^{1,p}(\X)\to L^p(\mm),&\quad A(f)\coloneqq f,\\
\varphi_\alpha\colon W^{1,p}(\X)\to\R,&\quad\varphi_\alpha(f)\coloneqq\|f\|_{W^{1,p}_\alpha(\X)},\\
\psi_\beta\colon L^p(\mm)\to\R,&\quad\psi_\beta(f)\coloneqq\beta\|f\|_{L^p(\mm)}.
\end{split}\]
Observe that \(A\) is a bounded linear operator whose adjoint \(A^*\) coincides with the operator
\[
L^q(\mm)\ni g\mapsto L_g\in W^{-1,q}(\X)
\]
introduced in Definition \ref{def:L_g}. Moreover, \(\varphi_\alpha\), \(\psi_\beta\) are convex and it can be readily checked that
\[
\varphi_\alpha^*(\tilde L)=\left\{\begin{array}{ll}
0\\
+\infty
\end{array}\quad\begin{array}{ll}
\text{ if }\|\tilde L\|_{op,\alpha}\leq 1,\\
\text{ otherwise}
\end{array}\right.
\]
for every \(\tilde L\in W^{-1,q}(\X)\). Now let us define \(\lambda_k\coloneqq\|L\|_{op,k}\) for every
\(k\in\N\). Given any \(\varepsilon>0\), we can find \(f_k\in W^{1,p}(\X)\) such that \(\|f_k\|_{W^{1,p}_k(\X)}\leq 1\)
and \(L(f_k)\geq\lambda_k-\varepsilon\). In particular, we have that \(\|f_k\|_{L^p(\mm)}\leq 1/k\) and
\(\|f_k\|_{W^{1,p}(\X)}\leq 1\) for every \(k\in\N\). This means that \(f_k\to 0\) strongly in \(L^p(\mm)\)
and \(\sup_k\|f_k\|_{W^{1,p}(\X)}<+\infty\), whence it follows that \(L(f_k)\to L(0)=0\) thanks to the assumption
that \(L\in W^{-1,q}_{pd}(\X)\). Therefore we deduce that \(0\leq\lims_k\lambda_k\leq\varepsilon+\lim_k L(f_k)=\varepsilon\)
and thus \(\lambda_k\to 0\) by arbitrariness of \(\varepsilon\). Consequently, for every \(n\in\N\), there exists
\(k(n)\in\N\) such that \(\|nL\|_{op,k(n)}=n\lambda_{k(n)}\leq 1\). By applying Theorem \ref{thm:conj_sum} and
Proposition \ref{prop:conj_compos}, we thus get
\[
0=\varphi_{k(n)}^*(nL)=(\varphi_0+\psi_{k(n)}\circ A)^*(nL)=
\inf_{g\in L^q(\mm)}\big(\varphi_0^*(nL-L_g)+\psi_{k(n)}^*(g)\big).
\]
In particular, there exists \(\tilde g_n\in L^q(\mm)\) such that \(\|nL-L_{\tilde g_n}\|_{op,0}\leq 1\).
Letting \(g_n\coloneqq\tilde g_n/n\in L^q(\mm)\), we conclude that \(\|L-L_{g_n}\|_{op,0}\leq 1/n\),
whence the statement follows.
\end{proof}
\begin{definition}[The spaces \(W^{-1,q}_0(\X)\) and \(W^{-1,q}_{pd,0}(\X)\)]
Let \((\X,\sfd,\mm)\) be a metric measure space and \(q\in(1,\infty]\). Then we define the spaces
\(W^{-1,q}_0(\X)\) and \(W^{-1,q}_{pd,0}(\X)\) as
\[\begin{split}
W^{-1,q}_0(\X)&\coloneqq\big\{L\in W^{-1,q}(\X)\;\big|\;\|L\|_{op,0}<+\infty\big\},\\
W^{-1,q}_{pd,0}(\X)&\coloneqq W^{-1,q}_{pd}(\X)\cap W^{-1,q}_0(\X).
\end{split}\]
\end{definition}

One can readily check that \((W^{-1,q}_0(\X),\|\cdot\|_{op,0})\) and \((W^{-1,q}_{pd,0}(\X),\|\cdot\|_{op,0})\) are Banach
spaces. Since \(\|\cdot\|_{op,1}\leq\|\cdot\|_{op,0}\), the ensuing result follows directly from Lemma \ref{lem:approx_W-1q_pd}:
\begin{corollary}\label{cor:L_g_dense_W-1q_pd}
Let \((\X,\sfd,\mm)\) be a metric measure space and \(q\in(1,\infty]\). Then 
\[
W^{-1,q}_{pd}(\X)=\overline{\{L_g\,:\,g\in L^q(\mm)\}}^{W^{-1,q}(\X)}
\] 
and \(W^{-1,q}_{pd,0}(\X)\) is the closure of \(\{L_g\,:\,g\in L^q(\mm)\}\cap W^{-1,q}_0(\X)\) in \(W^{-1,q}_0(\X)\).
\end{corollary}
\begin{remark}\label{rmk:L_mu_W-1q_pd}{\rm
When \(p\in(1,\infty)\), for every \(\mu\in D({\sf F}_p)\), we have that
\[
L_\mu\in W^{-1,q}_{pd,0}(\X),\qquad\|L_\mu\|_{op,0}\leq{\sf F}_p(\mu),
\]
where \(L_\mu\) is the functional given by Theorem \ref{thm:L_mu}.
These properties follow from Remark \ref{rmk:L_mu_cont_in_energy_p>1} and from the observation
that \(|L_\mu(f)|\leq{\sf F}_p(\mu)\|f\|_{W^{1,p}_0(\X)}\) for every \(f\in W^{1,p}(\X)\).
}\end{remark}

The following result illustrates a deep relation between \(W^{-1,q}_0(\X)\), \(W^{-1,q}_{pd,0}(\X)\)
and the tangent modules \(L^q_{\rm Sob}(T\X)\), \(L^q_{\rm Lip}(T\X)\) when $q \in (1,\infty)$. By a \textbf{linear submetry}
\(T\colon\mathbb B\to\mathbb V\) between two Banach spaces \(\mathbb B\), \(\mathbb V\) we mean a linear
operator with operator norm at most $1$ such that \(\|y\|_{\mathbb V}=\min\big\{\|x\|_{\mathbb B}\,:\,x\in\mathbb B,\,T(x)=y\big\}\) for every \(y\in\mathbb V\).
In particular, a linear submetry is surjective.
\begin{theorem}[Derivations and \(W^{-1,q}_0\)]\label{thm:der_and_W-1q-0}
Let \((\X,\sfd,\mm)\) be a metric measure space and \(q\in(1,\infty)\). Then
\(\pi\colon L^q_{\rm Sob}(T\X)\to W^{-1,q}_0(\X)\), defined by
\[
\pi(\delta)(f)\coloneqq\int\delta(f)\,\d\mm\quad\text{ for every }f\in W^{1,p}(\X),
\]
is a linear submetry and
\begin{equation}\label{eq:der_and_W-1q-0_aux}
\delta_1-\delta_2\in D(\div_q;\X),\quad\div_q(\delta_1-\delta_2)=0
\quad\text{ whenever }\delta_1,\delta_2\in L^q_{\rm Sob}(T\X)\text{ and }\pi(\delta_1)=\pi(\delta_2).
\end{equation}
Moreover, the embedding \(\iota\colon L^q_\Lip(T\X)\to L^q_{\rm Sob}(T\X)\) from Theorem \ref{thm:Lip_and_Sob_tg_mod} satisfies
\begin{equation}\label{eq:der_and_W-1q-0}
\pi^{-1}(W^{-1,q}_{pd,0}(\X))=\iota(L^q_\Lip(T\X)).
\end{equation}
\end{theorem}
\begin{proof}
Given any \(\delta\in L^q_{\rm Sob}(T\X)\) and \(f\in W^{1,p}(\X)\), we can estimate
\[
|\pi(\delta)(f)|\leq\int|\delta||Df|\,\d\mm\leq\|\delta\|_{L^q_{\rm Sob}(T\X)}\|f\|_{W^{1,p}_0(\X)}.
\]
Since \(\pi(\delta)\) is also linear, we deduce that \(\pi(\delta)\in W^{-1,q}_0(\X)\)
and \(\|\pi(\delta)\|_{op,0}\leq\|\delta\|_{L^q_{\rm Sob}(T\X)}\). Now fix any \(L\in W^{-1,q}_0(\X)\).
Let \({\rm I}\colon L^q_{\rm Sob}(T\X)\to L^q(T\X)\) be the isomorphism from Theorem \ref{thm:Lq(TX)=LqSob(TX)}
and denote by \(\mathbb V\) the vector subspace \(\{\d f\,:\,f\in W^{1,p}(\X)\}\) of \(L^p(T^*\X)\).
The linear functional
\[
\Phi_L\colon\mathbb V\to\R,\qquad\Phi_L(\d f)\coloneqq L(f)
\]
is well-defined and continuous by the validity of the inequality \(|\Phi_L(\d f)|\leq\|L\|_{op,0}\|\d f\|_{L^p(T^*\X)}\)
for all \(f\in W^{1,p}(\X)\). Furthermore, the Hahn--Banach theorem ensures that \(\Phi_L\)
can be extended to some \(\Phi_L\in L^p(T^*\X)'\) satisfying \(\|\Phi_L\|_{L^p(T^*\X)'}\leq\|L\|_{op,0}\).
Letting 
\[
\delta_L\coloneqq(\textsc{Int}_{L^p(T^*\X)}\circ{\rm I})^{-1}(\Phi_L)\in L^q_{\rm Sob}(T\X),
\]
we have that \(\|\delta_L\|_{L^q_{\rm Sob}(T\X)}\leq\|L\|_{op,0}\) and \(\pi(\delta_L)(f)=\int\delta_L(f)\,\d\mm=\Phi_L(\d f)=L(f)\)
for every \(f\in W^{1,p}(\X)\). This shows that \(\pi(\delta_L)=L\) and \(\|L\|_{op,0}=\|\delta_L\|_{L^q_{\rm Sob}(T\X)}\),
which implies that \(\pi\) is a linear submetry. Moreover, given any \(\delta_1,\delta_2\in L^q_{\rm Sob}(T\X)\) with
\(\pi(\delta_1)=\pi(\delta_2)\), we have that \(\int(\delta_1-\delta_2)(f)\,\d\mm=0\) for every \(f\in W^{1,p}(\X)\), which
yields \(\delta_1-\delta_2\in D(\div_q;\X)\) and \(\div_q(\delta_1-\delta_2)=0\). This proves \eqref{eq:der_and_W-1q-0_aux}.

Let us now pass to the verification of \eqref{eq:der_and_W-1q-0}. Fix any \(b\in{\rm Der}^q_{\mathfrak M}(\X)\).
Letting \(L_{\boldsymbol\div(b)}\) be as in Theorem \ref{thm:L_mu}, we know from Remark \ref{rmk:L_mu_W-1q_pd}
that \(L_{\boldsymbol\div(b)}\in W^{-1,q}_{pd,0}(\X)\). Since \(L_{\boldsymbol\div(b)}=-\pi(\iota(b))\)
by \eqref{eq:descr_iota(b)}, we deduce that \((\pi\circ\iota)(b)\in W^{-1,q}_{pd,0}(\X)\). The map
\(\pi\circ\iota\colon L^q_\Lip(T\X)\to W^{-1,q}_0(\X)\) is continuous, \({\rm Der}^q_{\mathfrak M}(\X)\)
is dense in \(L^q_\Lip(T\X)\), and \(W^{-1,q}_{pd,0}(\X)\) is closed in \(W^{-1,q}_0(\X)\), thus we can conclude
that \(\iota(L^q_\Lip(T\X))\subseteq\pi^{-1}(W^{-1,q}_{pd,0}(\X))\).

Conversely, fix any \(L\in W^{-1,q}_{pd,0}(\X)\).
By Corollary \ref{cor:L_g_dense_W-1q_pd}, there exists \((g_n)_n\subseteq L^q(\mm)\) such that
\(\|L_{g_n}-L\|_{op,0}\leq 2^{-n-2}\) for every \(n\in\N\). Define \(g_0\coloneqq 0\in L^q(\mm)\). Then
\[
    \left| \int f\,\d(g_n-g_{n-1})\mm \right| \leq\|L_{g_n}-L_{g_{n-1}}\|_{op,0}\||Df|\|_{L^p(\mm)}\leq\frac{1}{2^n}+\1_{\{1\}}(n)\|L\|_{op,0}
\]
for every \(n\in\N\) and \(f\in\LIP_{bs}(\X)\) with \(\int\lip_a(f)^p\,\d\mm\leq 1\). This means that
\((g_n-g_{n-1})\mm\in D({\sf F}_p)\) and \({\sf F}_p((g_n-g_{n-1})\mm)\leq 2^{-n}+\1_{\{1\}}(n)\|L\|_{op,0}\),
thus Theorem \ref{thm:F_p=D_q} yields the existence of a derivation \(b_n\in{\rm Der}^q_{\mathfrak M}(\X)\) such that
\(\boldsymbol\div(b_n)=(g_n-g_{n-1})\mm\) and \(\|b_n\|_{{\rm Der}^q(\X)}\leq 2^{-n}+\1_{\{1\}}(n)\|L\|_{op,0}\).
Given that \(L^q_\Lip(T\X)\) is a Banach space, we get that
\[
\exists\,b\coloneqq\sum_{n=1}^\infty b_n\in L^q_\Lip(T\X).
\]
Thanks to Remark \ref{rmk:consist_L_g_and_L_mu}, we have that \((\pi\circ\iota)(b_n)=L_{g_{n-1}}-L_{g_n}\) for every \(n\in\N\)
with \(n\geq 2\) and that \((\pi\circ\iota)(b_1)=-L_{g_1}\). Hence, a telescopic argument gives
\[
L=\lim_{n\to\infty}L_{g_n}=-\sum_{n=1}^\infty(\pi\circ\iota)(b_n)=-(\pi\circ\iota)(b)\quad\text{ with respect to }\|\cdot\|_{op,0}.
\]
Now consider any \(\delta\in L^q_{\rm Sob}(T\X)\) such that \(\pi(\delta)=L\). We know from
\eqref{eq:der_and_W-1q-0_aux} that \(\delta+\iota(b)\in D(\div_q;\X)\), so that
\(\delta+\iota(b)\in\iota({\rm Der}^q_{\mathfrak M}(\X))\) by Lemma \ref{lem:Lip_vs_Sob_div}.
We conclude that \(\delta=-\iota(b)+(\delta+\iota(b))\in\iota(L^q_\Lip(T\X))\), whence the inclusion
\(\pi^{-1}(W^{-1,q}_{pd,0}(\X))\subseteq\iota(L^q_\Lip(T\X))\) follows. Therefore, also \eqref{eq:der_and_W-1q-0} is proved.
\end{proof}
\begin{corollary}\label{cor:density_vf_with_div}
Let \((\X,\sfd,\mm)\) be a metric measure space and \(q\in(1,\infty)\). Then \(D(\div_q;\X)\) is dense in \(\iota(L^q_\Lip(T\X))\),
where \(\iota\colon L^q_\Lip(T\X)\to L^q_{\rm Sob}(T\X)\) is the embedding given by Theorem \ref{thm:Lip_and_Sob_tg_mod}.
In particular, the space \({\rm Der}^q_q(\X)\) is dense in \(L^q_\Lip(T\X)\). 
\end{corollary}
\begin{proof}
Fix \(b\in L^q_\Lip(T\X)\) and denote \(\delta\coloneqq\iota(b)\). Since \(\pi(\delta)\in W^{-1,q}_{pd,0}(\X)\) by Theorem
\ref{thm:der_and_W-1q-0}, there exists a sequence \((g_n)_n\subseteq L^q(\mm)\) such that \(\|\pi(\delta)-L_{g_n}\|_{op,0}\to 0\)
by Corollary \ref{cor:L_g_dense_W-1q_pd}. Recalling that \(\pi\) is a linear submetry by Theorem \ref{thm:der_and_W-1q-0},
for every \(n\in\N\) we can find \(\tilde\delta_n\in\pi^{-1}(\pi(\delta)-L_{g_n})\) such that
\(\|\tilde\delta_n\|_{L^q_{\rm Sob}(T\X)}=\|\pi(\delta)-L_{g_n}\|_{op,0}\). Now define
\(\delta_n\coloneqq\delta-\tilde\delta_n\in L^q_{\rm Sob}(T\X)\). Then \(\pi(\delta_n)=L_{g_n}\), which ensures that
\(\delta_n\in D(\div_q;\X)\) and \(\div_q(\delta_n)=-g_n\). Moreover, we have that
\[
\|\delta-\delta_n\|_{L^q_{\rm Sob}(T\X)}=\|\tilde\delta_n\|_{L^q_{\rm Sob}(T\X)}=\|\pi(\delta)-L_{g_n}\|_{op,0}\to 0\quad\text{ as }n\to\infty.
\]
This proves that \(D(\div_q;\X)\) is dense in \(\iota(L^q_\Lip(T\X))\). Since \(D(\div_q;\X) = \iota({\rm Der}^q_q(\X))\) by Lemma \ref{lem:relation_div_q}, we also deduce that \({\rm Der}^q_q(\X)\) is dense in \(L^q_\Lip(T\X)\). The proof is complete.
\end{proof}
\begin{theorem}[Predual of \(W^{1,p}(\X)\)]\label{thm:char_pred_Sob}
Let \((\X,\sfd,\mm)\) be a metric measure space and \(p\in(1,\infty)\). Then the dual of the Banach space \(W^{-1,q}_{pd}(\X)\)
is isomorphic to \(W^{1,p}(\X)\). More precisely, letting \(\mathscr L\colon L^q(\mm)\to W^{-1,q}_{pd}(\X)\)
be defined as \(\mathscr L(g)\coloneqq L_g\) for every \(g\in L^q(\mm)\), its adjoint operator \(\mathscr L^*\colon W^{-1,q}_{pd}(\X)'\to L^p(\mm)\) satisfies
\[
\mathscr L^*(W^{-1,q}_{pd}(\X)')=W^{1,p}(\X),\qquad\|\mathscr L^*(\Omega)\|_{W^{1,p}(\X)}=
\|\Omega\|_{W^{-1,q}_{pd}(\X)'}\quad\text{ for all }\Omega\in W^{-1,q}_{pd}(\X)'.
\]
\end{theorem}
\begin{proof}
Since \(\mathscr L\) is \(1\)-Lipschitz linear, also its adjoint \(\mathscr L^*\) is \(1\)-Lipschitz linear. 
Next, let us check that \(\mathscr L^*(\Omega)\in W^{1,p}(\X)\) holds
for any given \(\Omega\in W^{-1,q}_{pd}(\X)'\). Denote \(f\coloneqq\mathscr L^*(\Omega)\in L^p(\mm)\) for brevity.
Now let us define
\[
A_f(b)\coloneqq -\int f\,\div(b)\,\d\mm\quad\text{ for every }b\in{\rm Der}^q_q(\X).
\]
Notice that \(A_f\colon{\rm Der}^q_q(\X)\to\R\) is a linear operator. For any \(b\in{\rm Der}^q_q(\X)\) we also have that
\[
A_f(b)=-\langle\mathscr L^*(\Omega),\div(b)\rangle=-\langle\Omega,\mathscr L(\div(b))\rangle=-\langle\Omega,L_{\div(b)}\rangle,
\]
so that \(|A_f(b)|\leq C\|b\|_{{\rm Der}^q(\X)}\) by Remark \ref{rmk:consist_L_g_and_L_mu}, where we set
\(C\coloneqq\|\Omega\|_{W^{-1,q}_{pd}(\X)'}\). Since \({\rm Der}^q_q(\X)\) is dense in \(L^q_\Lip(T\X)\) by Corollary
\ref{cor:density_vf_with_div}, we can uniquely extend \(A_f\) to an element \(A_f\in L^q_\Lip(T\X)'\) such that
\(\|A_f\|_{L^q_\Lip(T\X)'}\leq C\). Let \(\omega_f\coloneqq\textsc{Int}_{L^q_\Lip(T\X)}^{-1}(A_f)\in L^q_\Lip(T\X)^*\).
Then \(\||\omega_f|\|_{L^p(\mm)}\leq C\) and \(\int\omega_f(b)\,\d\mm=A_f(b)\) for all \(b\in L^q_\Lip(T\X)\).
In particular, \({\sf L}_f\coloneqq\omega_f|_{{\rm Der}^q_q(\X)}\colon{\rm Der}^q_q(\X)\to L^1(\mm)\) is a \(\LIP_b(\X)\)-linear operator satisfying
\(|{\sf L}_f(b)|\leq|\omega_f||b|\) and \(\int {\sf L}_f(b)\,\d\mm=-\int f\,\div(b)\,\d\mm\) for every \(b\in{\rm Der}^q_q(\X)\).
Therefore Theorem \ref{thm:Sob_via_der} ensures that the function \(\mathscr L^*(\Omega)=f\) belongs to \(W^{1,p}(\X)\), as desired.

Let us now check that, letting \(f\) and \(C\) be as above, we also have that \(\|f\|_{W^{1,p}(\X)}\leq C\). Recall from \eqref{eq:prop_bar_L_f}
that \(\|f\|_{W^{1,p}(\X)}\) coincides with the \(p\)-norm of \((f,\bar{\sf L}_f)\in L^p(\mm)\times L^q_\Lip(T\X)'\). Hence
\[
\|f\|_{W^{1,p}(\X)}=\sup\bigg\{\int fg\,\d\mm+\bar{\sf L}_f(b)\;\bigg|\,(g,b)\in L^q(\mm)\times L^q_\Lip(T\X),\,\|g\|_{L^q(\mm)}^q+\|b\|_{{\rm Der}^q(\X)}^q\leq 1\bigg\}.
\]
Given \((g,b)\in L^q(\mm)\times L^q_\Lip(T\X)\) with \(\|(g,b)\|_q\coloneqq(\|g\|_{L^q(\mm)}^q+\|b\|_{{\rm Der}^q(\X)}^q)^{1/q}\leq 1\), we have that
\[
\int fg\,\d\mm+\bar{\sf L}_f(b)=\langle\mathscr L^*(\Omega),g\rangle+\int{\sf L}_f(b)\,\d\mm=\langle\Omega,L_g\rangle-\int f\,\div(b)\,\d\mm=\langle\Omega,L_{g-\div(b)}\rangle.
\]
Next, observe that
\[\begin{split}
\|L_{g-\div(b)}\|_{W^{-1,q}(\X)}&=\sup\big\{L_{g-\div(b)}(h)\;\big|\;h\in W^{1,p}(\X),\,\|h\|_{W^{1,p}(\X)}\leq 1\big\}\\
&=\sup\bigg\{\int gh\,\d\mm+\bar{\sf L}_h(b)\;\bigg|\;h\in W^{1,p}(\X),\,\|h\|_{W^{1,p}(\X)}\leq 1\bigg\}\\
&\leq\sup\bigg\{\int gh\,\d\mm+\bar{\sf L}(b)\;\bigg|\;(h,\bar{\sf L})\in L^p(\mm)\times L^q_\Lip(T\X)',\,\|(h,\bar{\sf L})\|_p\leq 1\bigg\}\\
&=\|(g,b)\|_q\leq 1.
\end{split}\]
All in all, we conclude that \(\int fg\,\d\mm+\bar{\sf L}_f(b)=\langle\Omega,L_{g-\div(b)}\rangle\leq C\|L_{g-\div(b)}\|_{W^{-1,q}(\X)}\leq C\),
thus \(\|f\|_{W^{1,p}(\X)}\leq C\). This shows that \(\|\mathscr L^*(\Omega)\|_{W^{1,p}(\X)}\leq\|\Omega\|_{W^{-1,q}_{pd}(\X)'}\)
for every \(\Omega\in W^{-1,q}_{pd}(\X)'\).

Finally, it remains to check that \(\mathscr L^*\colon W^{-1,q}_{pd}(\X)'\to W^{1,p}(\X)\) is surjective and that for every
\(\Omega\in W^{-1,q}_{pd}(\X)'\) it holds that \(\|\mathscr L^*(\Omega)\|_{W^{1,p}(\X)}\geq\|\Omega\|_{W^{-1,q}_{pd}(\X)'}\) .
To this aim, fix any \(f\in W^{1,p}(\X)\). Letting \({\rm j}\colon W^{1,p}(\X)\to W^{-1,q}(\X)'\) be the canonical embedding map
into the bidual, we consider the element \(\Omega\coloneqq{\rm j}(f)|_{W^{-1,q}_{pd}(\X)}\in W^{-1,q}_{pd}(\X)'\). For every function
\(g\in L^q(\mm)\), we have that
\[
\langle\mathscr L^*(\Omega),g\rangle=\langle\Omega,L_g\rangle=\langle\,{\rm j}(f),L_g\rangle=L_g(f)=\int fg\,\d\mm=\langle f,g\rangle,
\]
which gives that \(\mathscr L^*(\Omega)=f\) and thus accordingly \(\mathscr L^*\colon W^{-1,q}_{pd}(\X)'\to W^{1,p}(\X)\) is surjective.
Notice also that \(\|\Omega\|_{W^{-1,q}_{pd}(\X)'}\leq\|{\rm j}(f)\|_{W^{-1,q}(\X)'}=\|f\|_{W^{1,p}(\X)}\). Therefore the proof is complete.
\end{proof}
\subsection{Reflexivity of Sobolev spaces and tangent modules}
\begin{proposition}\label{prop:reform_reflex_W1p}
Let \((\X,\sfd,\mm)\) be a metric measure space and \(p\in(1,\infty)\). Then the following conditions are equivalent:
\begin{itemize}
\item[\(\rm i)\)] The Sobolev space \(W^{1,p}(\X)\) is reflexive.
\item[\(\rm ii)\)] Given \((f_n)_n\subseteq W^{1,p}(\X)\) bounded, there exist a subsequence \((f_{n_i})_i\) and \(f\in W^{1,p}(\X)\) such that
\[\begin{split}
f_{n_i}\rightharpoonup f&\quad\text{ weakly in }L^p(\mm),\\
\d f_{n_i}\rightharpoonup\d f&\quad\text{ weakly in }L^p(T^*\X).
\end{split}\]
\end{itemize}
In particular, if \(L^p(T^*\X)\) is reflexive, then \(W^{1,p}(\X)\) is reflexive.
\end{proposition}
\begin{proof}
Let \(\phi\colon W^{1,p}(\X)\to L^p(\mm)\times L^p(T^*\X)\) be given by \(\phi(f)\coloneqq(f,\d f)\) for every \(f\in W^{1,p}(\X)\). Then \(\phi\)
is a linear isometry if we endow its codomain with the norm
\[
\|(f,\omega)\|\coloneqq\big(\|f\|_{L^p(\mm)}^p+\|\omega\|_{L^p(T^*\X)}^p\big)^{1/p}\quad\text{ for every }(f,\omega)\in L^p(\mm)\times L^p(T^*\X).
\]
Hence, ii) is equivalent to the weak sequential compactness of all closed balls of \(W^{1,p}(\X)\). Thanks to the Eberlein--\v{S}mulian theorem
(weak compactness and weak sequential compactness coincide) and to Kakutani's theorem (reflexivity is equivalent to the weak compactness of all
closed balls), we conclude that ii) is equivalent to i). The last part of the statement readily follows.
\end{proof}

By \cite[Proposition 44]{Amb:Col:DiMa:15} there are examples of metric measure spaces such that $W^{1,p}(\X)$ is not reflexive, in particular
\(L^p(T^*\X)\) is not reflexive.
On the positive side, many sufficient conditions for the reflexivity of $W^{1,p}(\X)$ are known; cf.\ \cite[Section 7.2.1]{Amb:Iko:Luc:Pas:24}.
It is not known whether the reflexivity of \(W^{1,p}(\X)\) implies that of \(L^p(T^*\X)\), cf.\ \cite[Remark 2.2.11]{Gig:18}.
\begin{theorem}[Characterisations of the reflexivity of \(W^{1,p}(\X)\)]\label{thm:predual_W1p}
Let \((\X,\sfd,\mm)\) be a metric measure space and \(p\in(1,\infty)\). Then the following conditions are equivalent:
\begin{itemize}
\item[\(\rm i)\)] \(W^{1,p}(\X)\) is reflexive.
\item[\(\rm ii)\)] \(W^{-1,q}_{pd}(\X)=W^{-1,q}(\X)\).
\item[\(\rm iii)\)] The isometric embedding \(\iota\colon L^q_\Lip(T\X)\to L^q_{\rm Sob}(T\X)\) from Theorem \ref{thm:Lip_and_Sob_tg_mod}
is surjective (thus an isometric isomorphism).
\item[\(\rm iv)\)] \(D(\div_q;\X)\) is dense in \(L^q(T\X)\).
\end{itemize}
\end{theorem}
\begin{proof}
\ \\
\(\bf i)\Longrightarrow ii)\) Assume that \(W^{1,p}(\X)\) is reflexive. The inclusion map \(\phi\colon W^{-1,q}_{pd}(\X)\to W^{-1,q}(\X)\) is a linear isometry,
thus its adjoint \(\phi^*\colon W^{-1,q}(\X)'\to W^{-1,q}_{pd}(\X)'\) is a linear submetry. We claim that \(\phi^*\) is injective. To prove it, fix any \(\Omega\in W^{-1,q}(\X)'\)
such that \(\phi^*(\Omega)=0\). Hence, \(\langle\Omega,L\rangle=0\) for every \(L\in W^{-1,q}_{pd}(\X)\). Now, since \(W^{1,p}(\X)\) is assumed to be reflexive and
\(W^{-1,q}(\X)'\) is the bidual of \(W^{1,p}(\X)\), there exists a (unique) function \(f\in W^{1,p}(\X)\) such that \(\langle L,f\rangle=\langle\Omega,L\rangle\)
for every \(L\in W^{-1,q}(\X)\). In particular, \(\int fg\,\d\mm=\langle L_g,f\rangle=\langle\Omega,L_g\rangle=0\) for every \(g\in L^q(\mm)\), whence it follows
that \(f=0\) and thus \(\Omega=0\). This proves that \(\phi^*\) is injective. It follows that \(\phi^*\) is an isomorphism, so that \(\phi\) is an isomorphism as well.
In particular, \(W^{-1,q}_{pd}(\X)=W^{-1,q}(\X)\).\\
\(\bf ii)\Longrightarrow iii)\) Assume that \(W^{-1,q}_{pd}(\X)=W^{-1,q}(\X)\). Then we have that \(W^{-1,q}_{pd,0}(\X)=W^{-1,q}_0(\X)\), which yields \(\iota(L^q_\Lip(T\X))=L^q_{\rm Sob}(T\X)\)
by Theorem \ref{thm:der_and_W-1q-0}.\\
\(\bf iii)\Longrightarrow iv)\) If \(\iota(L^q_\Lip(T\X))=L^q_{\rm Sob}(T\X)\), then \({\rm cl}_{L^q(T\X)}(D(\div_q;\X))=L^q(T\X)\) by Corollary \ref{cor:density_vf_with_div} 
and Theorem~\ref{thm:Lq(TX)=LqSob(TX)}.\\
\(\bf iv)\Longrightarrow i)\) Assume that \(D(\div_q;\X)\) is dense in \(L^q(T\X)\). Let \((f_n)_n\subseteq W^{1,p}(\X)\) be a given bounded sequence. Up to a passing to a subsequence and relabeling, we thus have that \(f_n\rightharpoonup f\) weakly in \(L^p(\mm)\) for some \(f\in W^{1,p}(\X)\). We claim that \(\d f_n\rightharpoonup\d f\) weakly in \(L^p(T^*\X)\). To this end,
pick any element \(v\in L^q(T\X)\) and a sequence \((v_k)_k\subseteq D(\div_q;\X)\) such that \(v_k\to v\) strongly in \(L^q(T\X)\). Notice that
\[
\bigg|\int\d f_n(v)\,\d\mm-\int\d f(v)\,\d\mm\bigg|\leq\bigg|\int f\,\div_q(v_k)\,\d\mm-\int f_n\,\div_q(v_k)\,\d\mm\bigg|+2M\|v-v_k\|_{L^q(T\X)}
\]
for every \(n,k\in\N\), where we set \(M\coloneqq\sup_{m\in\N}\||Df_m|\|_{L^p(\mm)}<+\infty\). By letting first \(n\to\infty\) and then \(k\to\infty\), we deduce
that \(\int\d f_n(v)\,\d\mm\to\int\d f(v)\,\d\mm\). In light of Proposition \ref{prop:map_Int}, this means that \(\d f_n\rightharpoonup\d f\) weakly in \(L^p(T^*\X)\).
Therefore, \(W^{1,p}(\X)\) is reflexive by Proposition \ref{prop:reform_reflex_W1p}.
\end{proof}
\section{Bibliographical notes}\label{sec:bib_notes}
\subsection*{Preliminaries}
The presentation of the preliminary notions regarding metric measure spaces and metric Sobolev spaces is mainly from \cite{Amb:Iko:Luc:Pas:24}. The results related to the Choquet integration are taken from \cite{Deb:Gig:Pas:21},
see also \cite{Den:10}. For the theory of \(L^p\)-Banach \(L^\infty\)-modules and the related results we refer to
\cite{Gig:18,Gig:17} and to \cite{Luc:Pas:23} (see also the references therein). Let us also mention that the theory
of \(L^p\)-Banach \(L^\infty\)-modules, which was introduced by Gigli in \cite{Gig:18} to provide a differential
calculus for metric measure spaces, is fully consistent with the theory of random normed modules, originating in the
study of random metric spaces \cite{Sw:Skl:11} -- see the introduction of \cite{Gu:Mu:Tu:24} for an overview on this topic and a survey paper \cite{Guo:10}. See also
the monograph \cite{HLR91} for the essentially equivalent notion of randomly normed space. For the results about metric
currents, see \cite{Amb:Kir:00} and \cite{Pao:Ste:12,Pao:Ste:13}.
\subsection*{Quasicontinuity of Sobolev functions}
For the notion of Sobolev \(p\)-capacity in the classical setting, we refer to \cite{Hei:Kil:Mar:12}. In the metric
setting, this notion has been thoroughly investigated e.g.\ in \cite{Kin:Mar:00}, \cite{HKST:15}, and \cite{Bj:Bj:11}. 
Our presentation in this section is mainly based on results from the above-mentioned sources, on \cite{Deb:Gig:Pas:21}
(where quasicontinuity and quasiuniform convergence are studied in the case \(p=2\)), and on \cite{EB:PC:23} (where
the equivalence of several notions of capacity appearing in the above-listed literature is proved). 
\subsection*{Derivations and tangent modules}
The results in this section are built upon the theories of derivations developed in \cite{Gig:18} and in \cite{DiMaPhD:14},
and upon further refinements of the latter that we obtained in \cite{Amb:Iko:Luc:Pas:24}. Some results in the spirit of Theorem
\ref{thm:plans_vs_der_vs_curr} can be found in the papers \cite{Sch:16:der, Sch:16}, where the relation between derivations and
currents is provided via their mutual relation to Alberti representations. We also point out that the several variants of
derivations considered in this paper (and in the above-mentioned works) are strongly inspired by the work of
Weaver (e.g.\ \cite{Wea:00}).
\subsection*{Applications to potential theory}
\subsubsection*{Section \ref{ss:divergence_measures}}
Our study of the dual energy functional of the pre-Cheeger energy functional has been inspired by the paper
\cite{Amb:Sav:21} and the manuscript \cite{Sav:22}. In the former, analogous results to Proposition \ref{prop:mu_ll_Cap} have
been provided in the setting of \emph{extended metric measure spaces} (assuming that the reference measure
is finite and \(p\in(1,\infty)\)) -- see \cite[Theorems 5.6 and 5.7]{Amb:Sav:21}. Furthermore, in the same setting
the equality \({\sf F}_p={\sf B}_q\) has been proven in \cite[Theorem 11.8]{Sav:22}. 
We also mention that a duality between (versions of) \({\sf F}_p\) and \({\sf D}_q\) appeared also in the study of mass optimization problems in e.g.\ \cite{Bo:Bu:01, Bo:Bol:22, GBL:23} and, together with a notion of a gradient of Sobolev function in the weighted Euclidean space, played an important role in characterization of the solutions via a PDE involving a notion of \(p\)-Laplacian.
\subsubsection*{Section \ref{ss:grad_vector_fields}}
For the notions of gradient vector field, of infinitesimal strict convexity, and of infinitesimal Hilbertianity, 
we mainly follow \cite{Gig:15,Gig:18}; see also \cite{Gig:Pas:20}. For more information concerning infinitesimally
Hilbertian spaces, see \cite[Section 7.2.1]{Amb:Iko:Luc:Pas:24} and references therein.
\subsubsection*{Section \ref{ss:laplacian}}
In the general setting of metric measure spaces, notions of \(p\)-Laplacian have been proposed for the case \(p=2\) in \cite{Amb:Gig:Sav:08} (as elements of the subdifferential
of the Cheeger energy functional), see e.g.\ \cite{Gig:Mon:13}, \cite{Gor:Maz:22}, and \cite{Ca:Mo:20} for some applications of this theory. In \cite{Gig:15}, all exponents
\(p \in (1,\infty)\) are considered (cf.\ Remark \ref{rmk:Gigli_laplacian}). The study of \(p\)-harmonic functions is one of the main topics of interest in potential theory. In the setting
of metric measure spaces with doubling measure and supporting a \(p\)-Poincar\' e inequality, notions of sub/super minimisers of the Dirichlet energy functional, and several
applications to Dirichlet boundary problems and obstacle problems have been investigated in \cite{Kin:Mar:00}, \cite{Bj:Bj:11} (see also the references therein) and \cite{Ma:Sha:18},
just to name a few; see \cite{EB:Sarsa:25} for recent progress on the structure theory $p$-harmonic functions on $p$-PI spaces. Another relevant literature regarding notions of \(p\)-Laplacian is the one related to the study of harmonic functions and their  regularity properties in the RCD setting (see e.g.\ \cite{Gig:Vio:23, Ben:Vio:24, Vio:25}).
\subsubsection*{Section \ref{ss:condenser_capacity}}
The $2$-condenser capacity played already a prominent role in the work of Ahlfors and Beurling in the removability problems for holomorphic and conformal homeomorphisms \cite{AB:50} and was also of foundational importance in the regularity theory of quasiconformal mappings in higher dimensions \cite{Geh:62}. The study of condenser capacities has played a prominent role in the duality of capacities (resp. moduli) results on metric surfaces \cite{Raj:17,RR:19,Es:Iko:Raj} and their higher-dimensional analogs \cite{Geh:62,Zie:67,EB:PC:22:reciprocal}. Such duality results have recently been investigated in metric settings as well \cite{Loh:21,Loh:23,Loh:Raj:21,kangasniemi2022modulilipschitzhomologyclasses}. See also Friedman and He \cite{Freed:He:91} for applications to the linking problem.

In the 90's, the study of $p$-condenser capacities in metric measure spaces was initiated in \cite{Hei:Kos:98} and, in fact, \Cref{prop:Mod_inf_E_p} extends \cite[Proposition 2.17]{Hei:Kos:98} from $E$ and $F$ being disjoint and closed to $E$ and $F$ being disjoint and $E$ being Souslin.
\subsection*{Duality of Sobolev spaces}
The contents of the first half of this section closely follow that of \cite{Amb:Sav:21}, where duals and preduals of the Sobolev space have been investigated in the setting of
extended metric measure spaces, assuming a finite reference measure. For the classical tools in convex analysis needed to address the mentioned questions, we refer to \cite{Ro:74}.
The characterization of the reflexivity of the metric Sobolev space in Proposition \ref{prop:reform_reflex_W1p} has been proved in \cite{Gig:18}. Another description
of an isometric predual of the space \(W^{1,p}(\X)\) for \(p\in(1,\infty)\) has been recently obtained in \cite{Pas:Tai:25} (in the context of extended metric measure spaces).
It is not known (in those cases where \(W^{1,p}(\X)\) is not reflexive) whether the predual of \(W^{1,p}(\X)\) is unique.
\section*{Main notation}
\halign{#\quad&#\hfil\cr
${\mathscr L}^n$&Lebesgue measure in $\R^n$\cr
$\1_E$&Characteristic function of $E$\cr
$(\X,\sfd,\mm)$&metric measure space\cr
$\Lip(f)$&Global Lipschitz constant of $f$, \eqref{eq:global_lip}\cr
$\lip(f)$&Local Lipschitz constant of $f$, \eqref{eq:local_asym_lip}\cr
$\lip_a(f)$&Asymptotic Lipschitz constant of $f$, 
\eqref{eq:local_asym_lip}\cr
${\mathcal E}_{p,\lip}$&Pre-Cheeger $p$-energy functional,
\eqref{eq:pre-Cheeger}\cr
$L^0_{\rm ext}(\mm)$&Space of extended measurable functions, \eqref{eq:ext_measurable}\cr
$L^0(\mm)$&Space of real-valued measurable functions, \eqref{eq:ext_measurable}\cr
$L^p(\mm)$&Lebesgue spaces, \eqref{eq:Lebesgue_space}\cr
$\Sigma_\mm$&$\sigma$-algebra of $\mm$-measurable subsets\cr
$\bar\mm$&Completion of $\mm$\cr
${\mathcal M}(\X)$&Signed measures with finite total variation on $\X$\cr
${\mathfrak M}(\X)$&Boundedly finite signed measures on $\X$\cr
$\mathscr C(\X)$&Continuous curves from an interval of $\R$ to $\X$\cr
$\e_t(\gamma)$&Evaluation of a curve $\gamma$ at $t$\cr
$\ell(\gamma)$&Length of a curve $\gamma$, \eqref{eq:def_e_t}\cr
${\rm Mod}_p(\Gamma)$&$p$-modulus of a family $\Gamma\subseteq\mathscr
C(\X)$, Definition~\ref{def:modulus}\cr
${\rm Bar}(\ppi)$&Barycenter of a plan $\ppi$,
Definition~\ref{def:barycenter}\cr
$\mathcal B_q(\X)$&Plans with barycenter in $L^q(\mm)$,
Definition~\ref{def:barycenter}\cr
$\text{\sc Int}_{\mathscr M}$&Integral operator identifying $\mathscr
M^*$ with $\mathscr M'$, Proposition~\ref{prop:map_Int}\cr
$[\![\ppi]\!]$&normal $1$-current induced by a plan $\ppi$, \eqref{eq:current_induced_by_plan}\cr
${\rm WUG}_p(f)$&Weak $p$-upper gradients of $f$,
Definition~\ref{def:wug}\cr
\(D^{1,p}_\mm(\X)\) & The Dirichlet space, Definition \ref{def:Sobolev}\cr
${\mathcal E}_p$&Cheeger $p$-energy functional,
Definition~\ref{def:Cheeger_energy}\cr
${\rm Cap}_p(E)$&Sobolev $p$-capacity of $E\subseteq\X$,
Definition~\ref{def:cap_p}\cr
$\mathcal{QC}(\X)$&Quasicontinuous maps on $\X$\cr
${\sf qcr}(f)$&Quasicontinuous representative of $f\in W^{1,p}(\X)$,
Theorem~\ref{thm:qc_repr}\cr
$L^p(T^*(\X))$&$p$-cotangent module, Definition~\ref{def:cotg_mod}\cr
$L^q(T(\X))$&$q$-tangent module, Definition~\ref{def:tg_mod}\cr
$\boldsymbol\div_q(v)$&$q$-divergence in ${\mathfrak M}(\X)$ of $v\in
L^q(T(\X))$, Definition~\ref{def:vf_with_div}\cr
$\div_q(v)$&$L^q$-divergence of $v\in L^q(T(\X))$,
Definition~\ref{def:vf_with_Lq_div}\cr
$\boldsymbol\div(b)$&$q$-divergence in ${\mathfrak M}(\X)$ of a
derivation $b$, Definition~\ref{def:Lip_derivation_divergence}\cr
${\rm Der}^q_{\mathfrak M}(\X)$&$L^q$ derivations with divergence in
${\mathfrak M}$, Definition~\ref{def:Lip_derivation_divergence}\cr
$L^q_\Lip(T\X)$&Lipschitz $q$-tangent submodule of ${\rm Der}^q(\X)$,
Definition~\ref{def:Lip_tg_mod}\cr
$L^q_{\rm Sob}(T\X)$&Sobolev $q$-tangent module,
Definition~\ref{def:Sob_der}\cr
${\sf F}_p(\mu)$&Dual $p$-energy of $\mu$, \eqref{eq:predual:energy}\cr
${\sf D}_q(\mu)$&Dual dynamic cost of $\mu$, \eqref{eq:def_D_q}\cr
${\sf Grad}_q(f)$& $q$-gradients of $f\in D^{1,p}_{\mm}(\X)$,
Definition~\ref{def:gradient}\cr
$\boldsymbol\Delta_p(f)$& $p$-Laplacians of $f\in D^{1,p}_{\mm}(\X)$,
Definition~\ref{def:laplacian}\cr
${\rm Cap}_p(E,F)$&$p$-condenser capacity of $(E,F)$,
Definition~\ref{def:condens_cap}\cr}
\bigskip

\small

\begin{thebibliography}{10}

\bibitem{AB:50}
{\sc L.~Ahlfors and A.~Beurling}, {\em Conformal invariants and function-theoretic null-sets}, Acta Math., 83 (1950), pp.~101--129.

\bibitem{AmbICM}
{\sc L.~Ambrosio}, {\em Calculus, heat flow and curvature-dimension bounds in metric measure spaces}, in Proceedings of the {I}nternational {C}ongress of {M}athematicians---{R}io de {J}aneiro 2018. {V}ol. {I}. {P}lenary lectures, World Sci. Publ., Hackensack, NJ, 2018, pp.~301--340.

\bibitem{Amb:Col:DiMa:15}
{\sc L.~Ambrosio, M.~Colombo, and S.~Di~Marino}, {\em Sobolev spaces in metric measure spaces: reflexivity and lower semicontinuity of slope}, in Variational methods for evolving objects, vol.~67 of Adv. Stud. Pure Math., Math. Soc. Japan, [Tokyo], 2015, pp.~1--58.

\bibitem{Amb:Mar:Sav:15}
{\sc L.~Ambrosio, S.~Di~Marino, and G.~Savar\'{e}}, {\em On the duality between {$p$}-modulus and probability measures}, J. Eur. Math. Soc. (JEMS), 17 (2015), pp.~1817--1853.

\bibitem{Amb:Gig:Sav:08}
{\sc L.~Ambrosio, N.~Gigli, and G.~Savar{\'e}}, {\em Gradient flows in metric spaces and in the space of probability measures}, Lectures in Mathematics ETH Z\"urich, Birkh\"auser Verlag, Basel, second~ed., 2008.

\bibitem{Amb:Gig:Sav:13}
{\sc L.~Ambrosio, N.~Gigli, and G.~Savar\'{e}}, {\em Density of {L}ipschitz functions and equivalence of weak gradients in metric measure spaces}, Rev. Mat. Iberoam., 29 (2013), pp.~969--996.

\bibitem{Amb:Gig:Sav:14}
\leavevmode\vrule height 2pt depth -1.6pt width 23pt, {\em Calculus and heat flow in metric measure spaces and applications to spaces with {R}icci bounds from below}, Invent. Math., 195 (2014), pp.~289--391.

\bibitem{Amb:Iko:Luc:Pas:24}
{\sc L.~Ambrosio, T.~Ikonen, D.~Lu{\v{c}}i{\'{c}}, and E.~Pasqualetto}, {\em Metric {S}obolev {S}paces {I}: {E}quivalence of {D}efinitions}, Milan Journal of Mathematics, 92 (2024), pp.~255--347.

\bibitem{Amb:Iko:Luc:Pas:correction}
{\sc L.~Ambrosio, T.~Ikonen, D.~Lu{\v{c}}i{\'c}, and E.~Pasqualetto}, {\em Correction to: Metric {S}obolev {S}paces {I}: {E}quivalence of {D}efinitions}, Milan Journal of Mathematics,  (2025), pp.~1--5.

\bibitem{Amb:Iko:Luc:Pas:3}
{\sc L.~Ambrosio, T.~Ikonen, D.~Lu\v{c}i\'{c}, and E.~Pasqualetto}, {\em Surjectivity of the \(p\)-{L}aplacian and modulus-plan duality}.
\newblock In preparation.

\bibitem{Amb:Kir:00}
{\sc L.~Ambrosio and B.~Kirchheim}, {\em Currents in metric spaces}, Acta Math., 185 (2000), pp.~1--80.

\bibitem{AK:00}
\leavevmode\vrule height 2pt depth -1.6pt width 23pt, {\em Rectifiable sets in metric and {B}anach spaces}, Math. Ann., 318 (2000), pp.~527--555.

\bibitem{Amb:Renzi:Vitillaro:2026}
{\sc L.~Ambrosio, F.~Renzi, and F.~Vitillaro}, {\em The superposition principle for local 1-dimensional currents}, Nonlinear Analysis, 262 (2026), p.~113913.

\bibitem{Amb:Sav:21}
{\sc L.~Ambrosio and G.~Savar\'{e}}, {\em Duality properties of metric {S}obolev spaces and capacity}, Math. Eng., 3 (2021), pp.~Paper No. 1, 31.

\bibitem{Amb:Til:04}
{\sc L.~Ambrosio and P.~Tilli}, {\em Topics on analysis in metric spaces}, vol.~25 of Oxford Lecture Series in Mathematics and its Applications, Oxford University Press, Oxford, 2004.

\bibitem{Ben:Vio:24}
{\sc L.~Benatti and I.~Y. Violo}, {\em Second-order estimates for the $ p $-{L}aplacian in {RCD} spaces}, Journal of Differential Equations, 439 (2025), p.~113398.

\bibitem{Bj:Bj:11}
{\sc A.~Bj\"{o}rn and J.~Bj\"{o}rn}, {\em Nonlinear potential theory on metric spaces}, vol.~17 of EMS Tracts in Mathematics, European Mathematical Society (EMS), Z\"{u}rich, 2011.

\bibitem{Bog:07}
{\sc V.~I. Bogachev}, {\em Measure theory. {V}ol. {I}, {II}}, Springer-Verlag, Berlin, 2007.

\bibitem{Bo:Bol:22}
{\sc K.~Bo{\l}botowski and G.~Bouchitt{\'e}}, {\em Optimal design versus maximal {M}onge--{K}antorovich metrics}, Archive for Rational Mechanics and Analysis, 243 (2022), pp.~1449--1524.

\bibitem{Bo:Bu:01}
{\sc G.~Bouchitt\'{e} and G.~Buttazzo}, {\em Characterization of optimal shapes and masses through {M}onge--{K}antorovich equation}, Journal of the European Mathematical Society, 003 (2001), pp.~139--168.

\bibitem{BBS}
{\sc G.~Bouchitt\'e, G.~Buttazzo, and P.~Seppecher}, {\em Energies with respect to a measure and applications to low dimensional structures}, Calc. Var. Partial Differential Equations, 5 (1997), pp.~37--54.

\bibitem{Bre:Gig:24}
{\sc C.~Brena and N.~Gigli}, {\em Calculus and {F}ine {P}roperties of {F}unctions of {B}ounded {V}ariation on {R}{C}{D} {S}paces}, The Journal of Geometric Analysis, 34 (2024).

\bibitem{Bre:Gig:24b}
\leavevmode\vrule height 2pt depth -1.6pt width 23pt, {\em Local vector measures}, Journal of Functional Analysis, 286 (2024), p.~110202.

\bibitem{GBL:23}
{\sc G.~Buttazzo, M.~S. Gelli, and D.~Lu\v{c}i\'{c}}, {\em Mass optimization problem with convex cost}, SIAM Journal on Mathematical Analysis, 55 (2023), pp.~5617--5642.

\bibitem{Ca:Mo:20}
{\sc F.~Cavalletti and A.~Mondino}, {\em {New formulas for the Laplacian of distance functions and applications}}, Analysis $\&$ PDE, 13 (2020), pp.~2091--2147.

\bibitem{Ch:99}
{\sc J.~Cheeger}, {\em Differentiability of {L}ipschitz functions on metric measure spaces}, Geom. Funct. Anal., 9 (1999), pp.~428--517.

\bibitem{Deb:Gig:Pas:21}
{\sc C.~Debin, N.~Gigli, and E.~Pasqualetto}, {\em Quasi-{C}ontinuous {V}ector {F}ields on {RCD} spaces}, Potential Analysis, 54 (2021), pp.~183--211.

\bibitem{Den:10}
{\sc D.~Denneberg}, {\em Non-{A}dditive {M}easure and {I}ntegral}, Theory and Decision Library B, Springer Netherlands, 2010.

\bibitem{DiMaPhD:14}
{\sc S.~Di~Marino}, {\em Recent advances on {BV} and {S}obolev spaces in metric measure spaces}, PhD thesis, Scuola Normale Superiore (Pisa), 2014.
\newblock CVGMT preprint.

\bibitem{DiMar:14}
\leavevmode\vrule height 2pt depth -1.6pt width 23pt, {\em Sobolev and {BV} spaces on metric measure spaces via derivations and integration by parts}.
\newblock Preprint, arXiv:1409.5620, 2014.

\bibitem{DiMar:Gig:Pas:Sou:20}
{\sc S.~Di~Marino, N.~Gigli, E.~Pasqualetto, and E.~Soultanis}, {\em {Infinitesimal Hilbertianity of Locally {CAT}($\kappa$)-Spaces}}, J. Geom. Anal.,  (2020), pp.~1--65.

\bibitem{DM:Lu:Pas:2021}
{\sc S.~Di~Marino, D.~Lu{\v{c}}i{\'c}, and E.~Pasqualetto}, {\em Representation theorems for normed modules}, Rev. Real Acad. Cienc. Exactas Fis. Nat. Ser. A-Mat., 119 (2025).

\bibitem{EB:PC:22:reciprocal}
{\sc S.~Eriksson-Bique and P.~Poggi-Corradini}, {\em On the sharp lower bound for duality of modulus}, Proc. Amer. Math. Soc., 150 (2022), pp.~2955--2968.

\bibitem{EB:PC:23}
\leavevmode\vrule height 2pt depth -1.6pt width 23pt, {\em Density of continuous functions in {S}obolev spaces with applications to capacity}, Trans. Amer. Math. Soc. Ser. B, 11 (2024), pp.~901--944.

\bibitem{EB:Sarsa:25}
{\sc S.~{Eriksson-Bique} and S.~{Sarsa}}, {\em {Duality for the gradient of a $p$-harmonic function and the existence of gradient curves}}.
\newblock Preprint, arXiv:2502.01301, 2025.

\bibitem{Es:Iko:Raj}
{\sc B.~Esmayli, T.~Ikonen, and K.~Rajala}, {\em Coarea inequality for monotone functions on metric surfaces}, Trans. Amer. Math. Soc., 376 (2023), pp.~7377--7406.

\bibitem{Freed:He:91}
{\sc M.~H. Freedman and Z.-X. He}, {\em Divergence-free fields: energy and asymptotic crossing number}, Ann. of Math. (2), 134 (1991), pp.~189--229.

\bibitem{Fre:12}
{\sc D.~H. Fremlin}, {\em Measure {T}heory: {B}road {F}oundations. {V}olume 2}, Measure Theory, Torres Fremlin, 2012.

\bibitem{Geh:62}
{\sc F.~W. Gehring}, {\em Extremal length definitions for the conformal capacity of rings in space}, Michigan Math. J., 9 (1962), pp.~137--150.

\bibitem{Gig:15}
{\sc N.~Gigli}, {\em On the differential structure of metric measure spaces and applications}, Mem. Amer. Math. Soc., 236 (2015).

\bibitem{Gig:17}
\leavevmode\vrule height 2pt depth -1.6pt width 23pt, {\em {L}ecture notes on differential calculus on $\sf {R}{C}{D}$ spaces}, Publ. RIMS Kyoto Univ., 54 (2018).

\bibitem{Gig:18}
\leavevmode\vrule height 2pt depth -1.6pt width 23pt, {\em Nonsmooth differential geometry---an approach tailored for spaces with {R}icci curvature bounded from below}, Mem. Amer. Math. Soc., 251 (2018).

\bibitem{Gig:23}
{\sc N.~Gigli}, {\em De {G}iorgi and {G}romov working together}.
\newblock Preprint, arXiv:2306.14604, 2023.

\bibitem{Gig:Luc:Pas:22}
{\sc N.~Gigli, D.~Lu\v{c}i\'{c}, and E.~Pasqualetto}, {\em Duals and pullbacks of normed modules}, Israel Journal of Mathematics,  (2025), pp.~1--46.

\bibitem{Gig:Mon:13}
{\sc N.~Gigli and A.~Mondino}, {\em A {P}{D}{E} approach to nonlinear potential theory in metric measure spaces}, Journal de Math\'{e}matiques Pures et Appliqu\'{e}es, 100 (2013), pp.~505--534.

\bibitem{Gig:Pas:19}
{\sc N.~Gigli and E.~Pasqualetto}, {\em Differential structure associated to axiomatic {S}obolev spaces}, Expositiones Mathematicae, 38 (2020), pp.~480--495.

\bibitem{Gig:Pas:20}
{\sc N.~Gigli and E.~Pasqualetto}, {\em {Lectures on Nonsmooth Differential Geometry}}, Springer International Publishing, Cham, Switzerland, 2020.

\bibitem{Gig:Pas:21}
{\sc N.~Gigli and E.~Pasqualetto}, {\em Behaviour of the reference measure on $\sf{R}{C}{D}$ spaces under charts}, Commun. Anal. Geom., 29 (2021), pp.~1391--1414.

\bibitem{Gig:Vio:23}
{\sc N.~Gigli and I.~Y. Violo}, {\em Monotonicity formulas for harmonic functions in {RCD(0,N)} spaces}, The Journal of Geometric Analysis, 33 (2023), p.~100.

\bibitem{Guo:10}
{\sc T.~Guo}, {\em Recent progress in random metric theory and its applications to conditional risk measures}, Science China Mathematics, 54 (2010), pp.~633--660.

\bibitem{Gu:Mu:Tu:24}
{\sc T.~Guo, X.~Mu, and Q.~Tu}, {\em {The relations among the notions of various kinds of stability and their applications}}, Banach Journal of Mathematical Analysis, 18 (2024).

\bibitem{Gor:Maz:22}
{\sc W.~Górny and J.~M. Mazón}, {\em On the p-{L}aplacian evolution equation in metric measure spaces}, Journal of Functional Analysis, 283 (2022), p.~109621.

\bibitem{HLR91}
{\sc R.~Haydon, M.~Levy, and Y.~Raynaud}, {\em Randomly normed spaces}, vol.~41 of Travaux en Cours [Works in Progress], Hermann, Paris, 1991.

\bibitem{Hei:Kil:Mar:12}
{\sc J.~Heinonen, T.~Kilpel{\"a}inen, and O.~Martio}, {\em Nonlinear {P}otential {T}heory of {D}egenerate {E}lliptic {E}quations}, Dover Books on Mathematics, Dover Publications, 2012.

\bibitem{Hei:Kos:98}
{\sc J.~Heinonen and P.~Koskela}, {\em Quasiconformal maps in metric spaces with controlled geometry}, Acta Mathematica, 181 (1998), pp.~1--61.

\bibitem{HKST:15}
{\sc J.~Heinonen, P.~Koskela, N.~Shanmugalingam, and J.~T. Tyson}, {\em Sobolev spaces on metric measure spaces. An approach based on upper gradients}, vol.~27 of New Mathematical Monographs, Cambridge University Press, Cambridge, 2015.

\bibitem{Ja:Ja:Ro:Ro:Sha:07}
{\sc E.~J\"{a}rvenp\"{a}\"{a}, M.~J\"{a}rvenp\"{a}\"{a}, K.~Rogovin, S.~Rogovin, and N.~Shanmugalingam}, {\em Measurability of equivalence classes and {${\rm MEC}_p$}-property in metric spaces}, Rev. Mat. Iberoam., 23 (2007), pp.~811--830.

\bibitem{kangasniemi2022modulilipschitzhomologyclasses}
{\sc I.~Kangasniemi and E.~Prywes}, {\em On the {M}oduli of {L}ipschitz {H}omology {C}lasses}, J. Geom. Anal., 35 (2025).

\bibitem{Kin:Mar:00}
{\sc J.~Kinnunen and O.~Martio}, {\em Choquet property for the {S}obolev capacity in metric spaces}, in Proceedings on Analysis and Geometry (Russian) (Novosibirsk Akademgorodok, 1999), Izdat. Ross. Akad. Nauk Sib. Otd. Inst. Mat., Novosibirsk, 2000, p.~285–290.

\bibitem{Lang}
{\sc U.~Lang}, {\em Local currents in metric spaces}, J. Geom. Anal., 21 (2011), pp.~683--742.

\bibitem{La:We:11}
{\sc U.~Lang and S.~Wenger}, {\em The pointed flat compactness theorem for locally integral currents}, Comm. Anal. Geom., 19 (2011), pp.~159--189.

\bibitem{Loh:21}
{\sc A.~Lohvansuu}, {\em Duality of moduli in regular toroidal metric spaces}, Ann. Fenn. Math., 46 (2021), pp.~3--20.

\bibitem{Loh:23}
\leavevmode\vrule height 2pt depth -1.6pt width 23pt, {\em On the modulus duality in arbitrary codimension}, Int. Math. Res. Not. IMRN,  (2023), pp.~21068--21085.

\bibitem{Loh:Raj:21}
{\sc A.~Lohvansuu and K.~Rajala}, {\em Duality of moduli in regular metric spaces}, Indiana Univ. Math. J., 70 (2021), pp.~1087--1102.

\bibitem{Lott-Villani}
{\sc J.~Lott and C.~Villani}, {\em Ricci curvature for metric-measure spaces via optimal transport}, Ann. of Math. (2), 169 (2009), pp.~903--991.

\bibitem{Lu:Pas:Voj:24}
{\sc M.~Lu{\v{c}}i{\'c}, E.~Pasqualetto, and I.~Vojnovi{\'c}}, {\em On the reflexivity properties of {B}anach bundles and {B}anach modules}, Banach Journal of Mathematical Analysis, 18 (2024), p.~7.

\bibitem{Lu:Pa:19}
{\sc D.~Lu\v{c}i\'{c} and E.~Pasqualetto}, {\em The {S}erre-{S}wan theorem for normed modules}, Rendiconti del Circolo Matematico di Palermo Series 2, 68 (2019), pp.~385--404.

\bibitem{Luc:Pas:23}
{\sc D.~Lu\v{c}i\'{c} and E.~Pasqualetto}, {\em An axiomatic theory of normed modules via {R}iesz spaces}, The Quarterly Journal of Mathematics, 75 (2024), pp.~1429--1479.

\bibitem{Ma:Sha:18}
{\sc L.~Malý and N.~Shanmugalingam}, {\em Neumann problem for p-{L}aplace equation in metric spaces using a variational approach: Existence, boundedness, and boundary regularity}, Journal of Differential Equations, 265 (2018), pp.~2431--2460.

\bibitem{Pao:Ste:12}
{\sc E.~Paolini and E.~Stepanov}, {\em Decomposition of acyclic normal currents in a metric space}, J. Funct. Anal., 263 (2012), pp.~3358--3390.

\bibitem{Pao:Ste:13}
\leavevmode\vrule height 2pt depth -1.6pt width 23pt, {\em Structure of metric cycles and normal one-dimensional currents}, J. Funct. Anal., 264 (2013), pp.~1269--1295.

\bibitem{Pas:24-2}
{\sc E.~Pasqualetto}, {\em Projective and injective tensor products of {B}anach ${L}^0$-modules}, Ann. Funct. Anal, 15 (2024).

\bibitem{Pas:24}
{\sc E.~Pasqualetto}, {\em Limits and colimits in the category of {B}anach ${L}^0$-modules}, Rend. Semin. Mat. Univ. Padova, 154 (2025), pp.~105--142.

\bibitem{Pas:Tai:25}
{\sc E.~Pasqualetto and J.~Taipalus}, {\em Derivations and {S}obolev functions on extended metric-measure spaces}.
\newblock Preprint, arXiv:2503.02596, 2025.

\bibitem{Raj:17}
{\sc K.~Rajala}, {\em Uniformization of two-dimensional metric surfaces}, Invent. Math., 207 (2017), pp.~1301--1375.

\bibitem{RR:19}
{\sc K.~Rajala and M.~Romney}, {\em Reciprocal lower bound on modulus of curve families in metric surfaces}, Ann. Acad. Sci. Fenn. Math., 44 (2019), pp.~681--692.

\bibitem{Ro:74}
{\sc R.~T. Rockafellar}, {\em Conjugate duality and optimization}, SIAM, 1974.

\bibitem{Sav:22}
{\sc G.~Savar\'{e}}, {\em Sobolev spaces in extended metric-measure spaces}, in New trends on analysis and geometry in metric spaces, vol.~2296 of Lecture Notes in Math., Springer, Cham, [2022] \copyright 2022, pp.~117--276.

\bibitem{Sch:16:der}
{\sc A.~Schioppa}, {\em Derivations and {A}lberti representations}, Adv. Math., 293 (2016), pp.~436--528.

\bibitem{Sch:16}
\leavevmode\vrule height 2pt depth -1.6pt width 23pt, {\em Metric currents and {A}lberti representations}, J. Funct. Anal., 271 (2016), pp.~3007--3081.

\bibitem{Sw:Skl:11}
{\sc B.~Schweizer and A.~Sklar}, {\em Probabilistic Metric Spaces}, Dover Publications, 2011.

\bibitem{Sha:00}
{\sc N.~Shanmugalingam}, {\em Newtonian spaces: an extension of {S}obolev spaces to metric measure spaces}, Rev. Mat. Iberoamericana, 16 (2000), pp.~243--279.

\bibitem{Smi:93}
{\sc S.~K. Smirnov}, {\em Decomposition of solenoidal vector charges into elementary solenoids, and the structure of normal one-dimensional flows}, Algebra i Analiz, 5 (1993), pp.~206--238.

\bibitem{Vio:25}
{\sc I.~Y. Violo}, {\em An overview of regularity results for the {L}aplacian and p-{L}aplacian in metric spaces}, Indagationes Mathematicae,  (2025).

\bibitem{Wea:00}
{\sc N.~Weaver}, {\em Lipschitz algebras and derivations. {II}. {E}xterior differentiation}, J. Funct. Anal., 178 (2000), pp.~64--112.

\bibitem{Zie:67}
{\sc W.~P. Ziemer}, {\em Extremal length and conformal capacity}, Trans. Amer. Math. Soc., 126 (1967), pp.~460--473.

\end{thebibliography}
%
%

%
%
\end{document}